\documentclass[11pt, a4paper]{amsart}

\usepackage{amssymb,amsfonts,amsmath,amsthm,amsopn,amstext,amscd,mathtools}
\usepackage[colorlinks=true, citecolor=magenta, linkcolor=blue, urlcolor=blue1, filecolor=cyan]{hyperref}

\usepackage[margin=1in]{geometry}
\usepackage{xy}
\xyoption{all}

\usepackage[usenames,dvipsnames]{xcolor}
\definecolor{blue}{RGB}{51,103, 255}

\usepackage{tabularx}
\usepackage{latexsym}

\usepackage[only, llbracket,rrbracket]{stmaryrd}
\usepackage{todonotes}

\usepackage[shortlabels]{enumitem} 
\usepackage{bbm}
\usepackage{pdfpages}
\usepackage{titletoc}
\usepackage{booktabs}
\usepackage{mathtools}
\usepackage{mathrsfs}
\usepackage{thmtools}
\usepackage{hhline}
\usepackage{caption}
\usepackage{multirow, url}
\usepackage{textcomp}
\DeclareMathAlphabet{\cmcal}{OMS}{cmsy}{m}{n}

\usepackage{arydshln}
\usepackage{libertine}
\usepackage[utf8]{inputenc}
\usepackage[OT2,T1]{fontenc}
\DeclareMathOperator*{\motimes}{\text{\raisebox{0.4ex}{\scalebox{0.6}{$\bigotimes$}}}}
\DeclareMathOperator*{\moplus}{\text{\raisebox{0.4ex}{\scalebox{0.6}{$\bigoplus$}}}}

\DeclareSymbolFont{euletters}{U}{eur}{m}{n}
\DeclareSymbolFont{eufrakletters}{U}{euf}{m}{n}
\DeclareFontFamily{U}{wncy}{}
    \DeclareFontShape{U}{wncy}{m}{n}{<->wncyr10}{}
    \DeclareSymbolFont{mcy}{U}{wncy}{m}{n}
    \DeclareMathSymbol{\Sha}{\mathord}{mcy}{"58}
    
\definecolor{blue1}{rgb}{0, 0, 2}
\definecolor{sky}{rgb}{0, 0.2, 0.8}

\newtheorem{theorem}{Theorem}[section]
\newtheorem{corollary}[theorem]{Corollary}
\newtheorem{lemma}[theorem]{Lemma}
\newtheorem{proposition}[theorem]{Proposition}

\newtheorem{conjecture}[theorem]{Conjecture}

\theoremstyle{definition}
\newtheorem{definition}[theorem]{Definition}

\theoremstyle{remark}
\newtheorem{remark}[theorem]{Remark}
\newtheorem{assumption}[theorem]{Assumption}

\newtheorem{notation}[theorem]{Notation}

\numberwithin{equation}{section} 
\numberwithin{table}{subsection}
	
\newtheorem*{theorem*}{\bf{Theorem}}
\newtheorem*{claim*}{\bf{Claim}}
\newtheorem*{remark*}{\bf{Remark}}
\newtheorem*{remarks*}{\bf{Remarks}}
\newtheorem*{example*}{\bf{Example}}
\newtheorem*{examples*}{\bf{Examples}}
\newtheorem*{observation*}{\bf{Observation}}

\newcommand{\bA}{{\mathbf{A}}}
\newcommand{\bB}{{\mathbf{B}}}
\newcommand{\C}{{\mathbf{C}}}
\newcommand{\bD}{{\mathbf{D}}}
\newcommand{\bE}{{\mathbf{E}}}

\newcommand{\bP}{{\mathbf{P}}}
\newcommand{\Q}{{\mathbf{Q}}}

\newcommand{\bU}{{\mathbf{U}}}
\newcommand{\bV}{{\mathbf{V}}}
\newcommand{\bW}{{\mathbf{W}}}
\newcommand{\bX}{{\mathbf{X}}}
\newcommand{\bY}{{\mathbf{Y}}}
\newcommand{\Z}{{\mathbf{Z}}}

\newcommand{\fg}{{\mathfrak{g}}}
\newcommand{\fh}{{\mathfrak{h}}}

\newcommand{\fm}{{\mathfrak{m}}}
\newcommand{\fn}{{\mathfrak{n}}}

\newcommand{\fp}{{\mathfrak{p}}}
\newcommand{\fq}{{\mathfrak{q}}}

\newcommand{\fM}{{\mathfrak{M}}}
\newcommand{\fN}{{\mathfrak{N}}}

\newcommand{\fY}{{\mathfrak{Y}}}
\newcommand{\fZ}{{\mathfrak{Z}}}

\newcommand{\cD}{{\cmcal{D}}}
\newcommand{\cE}{{\cmcal{E}}}
\newcommand{\cF}{{\cmcal{F}}}
\newcommand{\cG}{{\cmcal{G}}}
\newcommand{\cH}{{\cmcal{H}}}
\newcommand{\cI}{{\cmcal{I}}}

\newcommand{\cS}{{\cmcal{S}}}
\newcommand{\cT}{{\cmcal{T}}}

\newcommand{\scC}{{\mathscr{C}}}

\newcommand{\scG}{{\mathscr{G}}}
\newcommand{\scH}{{\mathscr{H}}}

\newcommand{\bbA}{\mathbb{A}}
\newcommand{\bbB}{\mathbb{B}}

\newcommand{\bbD}{\mathbb{D}}

\newcommand{\bbM}{\mathbb{M}}

\newcommand{\bbO}{\mathbb{O}}

\newcommand{\bbV}{\mathbb{V}}

\newcommand{\lla}{\left\langle}
\newcommand{\rra}{\right\rangle}

\newcommand{\tn}{\textnormal}

\newcommand{\zmod}[1]{{\Z/{#1}\Z}}

\newcommand{\arinj}{\ar@{^(->}}
\newcommand{\arsurj}{\ar@{->>}}
\newcommand{\arsub}{\ar@{}[r]|-*[@]{\subset}}
\newcommand{\arsup}{\ar@{}[r]|-*[@]{\supset}}
\newcommand{\arcap}{\ar@{}[d]|-*[@]{\subset}}
\newcommand{\arcup}{\ar@{}[u]|-*[@]{\subset}}
\newcommand{\arin}{\ar@{}[u]|-*[@]{\in}}

\newcommand{\Gal}{{\textnormal{Gal}}}
\newcommand{\Div}{{\textnormal{Div}}}
\newcommand{\Pic}{{\textnormal{Pic}}}

\newcommand{\GL}{{\textnormal{GL}}}

\newcommand{\SL}{{\textnormal{SL}}}

\newcommand{\sHom}{{\mathscr{H}\kern-.5pt om}}
\newcommand{\sExt}{{\mathscr{E}\kern-.5pt xt}}

\newcommand{\tor}{{\textnormal{tors}}}

\newcommand{\num}{{\textnormal{numerator}}}

\newcommand{\sqf}{{\textnormal{sf}}}
\newcommand{\nsqf}{{\textnormal{nsf}}}
\newcommand{\rad}{{\textnormal{rad}}}

\newcommand{\val}{\textnormal{val}}

\renewcommand{~}{\hspace*{0.5mm}}

\newcommand{\ov}{\overline}

\newcommand{\ms}{\medskip}
\newcommand{\sm}{\smallsetminus}

\newcommand{\Gcd}{\mathsf{GCD}}
\newcommand{\pw}{\mathsf{Pw}}
\newcommand{\pow}{\mathsf{Power}}
\newcommand{\rpw}{{{\bf r}\hyp\mathsf{Pw}}}

\mathchardef\hyp="2D
\newcommand{\xyv}[1]{\xymatrixrowsep{#1 pc}}
\newcommand{\xyh}[1]{\xymatrixcolsep{#1 pc}}

\newcommand{\qa}{{\quad \textnormal{and} \quad}}
\newcommand{\qqa}{{~ \textnormal{ and } ~}}

\newcommand{\mat}[4]{
 \left(  \begin{smallmatrix} #1 & #2 \\ #3 & #4 \end{smallmatrix} \right)}
 
\newcommand{\bmat}[4]{
  \left( \begin{array}{cc} #1 & #2 \\ #3 & #4 \end{array} \right)}

\newcommand{\vect}[2]{
 \left(  \begin{smallmatrix} #1 \\ #2 \end{smallmatrix} \right)}

\newcommand{\bect}[2]{
  \left[ \begin{smallmatrix} #1  \\ #2 \end{smallmatrix} \right]}

\newcommand{\br}[1]{\left\langle #1 \right\rangle}

\newcommand{\leg}[2]{\genfrac(){}{}{#1}{#2}}

\newcommand{\Qdiv}[1]{\Div^0_{\textnormal{cusp}}(X_0(#1))({\Q})}
\newcommand{\Qdivv}[1]{\Div_{\textnormal{cusp}}(X_0(#1))({\Q})}

\newcommand{\vtl}{\vartriangleleft}
\newcommand{\pmodo}[1]{~~(\textnormal{mod}~~ {#1})}
\newcommand{\modo}[1]{~~\textnormal{mod}~~ {#1}}
\renewcommand{\pmod}[1]{~~(\textnormal{mod}~~ {#1})}

\renewcommand{\gcd}{\textnormal{gcd}}
\renewcommand{\max}{\textnormal{max}}
\renewcommand{\min}{\textnormal{min}}
\renewcommand{\det}{\textnormal{det}}

\begin{document}

\title{The rational cuspidal divisor class group of $X_0(N)$}


\author{Hwajong Yoo}
\address{College of Liberal Studies and Research Institute of Mathematics, Seoul National University, Seoul 08826, South Korea}
\email{\textnormal{hwajong@snu.ac.kr}}
\thanks{}

\date{}




\begin{abstract}
For any positive integer $N$, we completely determine the structure of the rational cuspidal divisor class group of $X_0(N)$, which is conjecturally equal to the rational torsion subgroup of $J_0(N)$. 
More specifically, for a given prime $\ell$, we construct a rational cuspidal divisor $Z_\ell(d)$ for any non-trivial divisor $d$ of $N$. Also, we compute the order of the linear equivalence class of $Z_\ell(d)$ and show that the $\ell$-primary subgroup of the rational cuspidal divisor class group of $X_0(N)$ is isomorphic to the direct sum of the cyclic subgroups generated by the linear equivalence classes of $Z_\ell(d)$.
\end{abstract}

\maketitle

\setcounter{tocdepth}{1}
\tableofcontents

\section{Introduction}\label{chapter1}
Let $N$ be a positive integer and let $\Gamma_0(N)$ be the congruence subgroup of $\SL_2(\Z)$ consisting of matrices that are upper-triangular modulo $N$. The complete modular curve $X_0(N)_{\C}$ is the union of the affine modular curve $Y_0(N)_{\C}=\Gamma_0(N)\backslash \cH$ and a finite set of cusps, where $\cH$ is the complex upper half plane and $\Gamma_0(N)$ acts on $\cH$ by linear fractional transformation. The curve $X_0(N)_{\C}$ has a canonical nonsingular projective model $X_0(N)$ defined over $\Q$ (cf.  \cite[Ch. 6]{Shi71}), and in this model the set of cusps is invariant under the action of $\Gal(\ov{\Q}/\Q)$, the absolute Galois group of $\Q$. 

Let $J_0(N):=\Pic^0(X_0(N))$ be the Jacobian variety of $X_0(N)$ and $J_0(N)(\Q)_\tor$ its rational torsion subgroup. We would like to understand the group $J_0(N)(\Q)_\tor$ for any positive integer $N$, but there is no systematic way for doing it yet. On the other hand, if $N$ is a prime, Mazur proved the following, which was known as Ogg's conjecture \cite[Th. 1]{M77}.
\begin{theorem}[Mazur]
Let $N\geq 5$ be a prime number, and let $n=\num\left(\frac{N-1}{12}\right)$. 
The rational torsion subgroup $J_0(N)(\Q)_\tor$ is a cyclic group of order $n$, generated by the linear equivalence class of the difference of the two cusps $(0)-(\infty)$.
\end{theorem}

Let $\scC_N$ be the \textit{cuspidal subgroup} of $J_0(N)$, which is defined as a subgroup of $J_0(N)(\ov{\Q})$ generated by the linear equivalence classes of the differences of cusps. By the theorem of Manin \cite[Cor. 3.6]{Ma72} and Drinfeld \cite{Dr73}, we have $\scC_N(\Q) \subseteq J_0(N)(\Q)_\tor$, where $\scC_N(\Q)$ is the group of the rational points on $\scC_N$, called the \textit{rational cuspidal subgroup} of $J_0(N)$. As a natural generalization of Mazur's theorem, we expect the following.
\begin{conjecture}[Generalized Ogg's conjecture]\label{conjecture: GOC}
For any positive integer $N$, we have
\begin{equation*}
J_0(N)(\Q)_\tor=\scC_N(\Q).
\end{equation*}
\end{conjecture}

Let $\Div^0_{\textnormal{cusp}}(X_0(N))$ be the group of the degree $0$ cuspidal divisors\footnote{By a \textit{cuspidal divisor}, we mean a divisor on $X_0(N)$ supported only on cusps.}  on $X_0(N)$. 
By definition, there is an exact sequence:
\begin{equation}\label{0}
\xymatrix{
0 \ar[r] & U_N \ar[r] & \Div^0_{\textnormal{cusp}}(X_0(N)) \ar[r] & \scC_N \ar[r] & 0,
}
\end{equation}
where $U_N$ is the group of the divisors of modular units\footnote{By a \textit{modular unit on $X_0(N)$}, we mean a meromorphic function on $X_0(N)_{\C}$ which does not have zeros and poles on $Y_0(N)_{\C}$.}. 
In fact, $\Gal(\ov{\Q}/\Q)$ acts naturally on the objects in the exact sequence and (\ref{0}) is an exact sequence of $\Gal(\ov{\Q}/\Q)$-modules. Taking Galois cohomology we get
\begin{equation*}
\xymatrix{
\Qdiv N \ar[r]^-{\pi} & \scC_N(\Q) \ar[r] & H^1(\Gal(\ov{\Q}/\Q),  \, U_N),
}
\end{equation*}
where
\begin{equation*}
\Qdiv N:=H^0(\Gal(\ov{\Q}/\Q), \, \Div^0_{\textnormal{cusp}}(X_0(N)))
\end{equation*}
is the group of the degree $0$ rational cuspidal divisors\footnote{By a \textit{rational cuspidal divisor}, we mean a cuspidal divisor fixed by the action of $\Gal(\ov{\Q}/\Q)$.} on $X_0(N)$.
The image of $\pi$ is called the \textit{rational cuspidal divisor class group} of $X_0(N)$, denoted by $\scC(N)$. In other words, $\scC(N)$ is a subgroup of $J_0(N)(\Q)$ generated by the linear equivalence classes of the degree $0$ rational cuspidal divisors on $X_0(N)$. Ken Ribet asked whether the map $\pi$ is surjective, or more generally $H^1(\Gal(\ov{\Q}/\Q), \, U_N)=0$.\footnote{For a discussion of a similar problem, see \cite{Ol70} and \cite[Th. 4.13]{Ta14}.}
Motivated by his question and the result of Toshikazu Takagi \cite[Th. 1.1]{Ta14}, we propose the following.

\begin{conjecture}\label{conjecture: Yoo and Ribet}
For any positive integer $N$, we have
\begin{equation*}
\scC(N)=\scC_N(\Q).
\end{equation*}
\end{conjecture}


Thus, it is worth understanding the structure of the group $\scC(N)$ for any positive integer $N$. Although the group $\scC(N)$ is a very explicit object, not much is known about
its precise structure due to a lack of efficient tools. To the author's best knowledge, the structure of $\scC(N)$ for a composite integer $N$ (which is large enough) has been computed for the following cases:
\begin{enumerate}
\item
$N$ is the product of two distinct primes by Chua and Ling \cite{CL97}.
\item
$N$ is a power of an odd prime $p$ by Lorenzini \cite{Lo95} and Ling \cite{Li97}.
\item
$N$ is a power of $2$  by Rouse and Webb \cite[Th. 10]{RW15}.
\item
$N$ is squarefree by Takagi \cite[Th. 6.1]{Ta97}.\footnote{In fact, Takagi computed the precise order of the group $\scC(N)$, but he did not determine the structure of the $2$-primary subgroup of $\scC(N)$, which is a motivation of this paper.}
\end{enumerate}

\begin{remark}
There are some partial results for Conjectures \ref{conjecture: GOC} and \ref{conjecture: Yoo and Ribet}, e.g.  \cite{Box, Li97, Lo95, Og73, Oh14, OS19, Po87, Re18, WY20, Yoo5}.
For a thorough discussion and a new result of the conjectures above, see \cite{Yoo10}.
\end{remark}

\ms
\subsection{Main result}
In this paper, we completely determine the structure of the $\ell$-primary subgroup of the group $\scC(N)$ for any positive integer $N$ and any prime $\ell$. Throughout this paper, we denote by $A[\ell^\infty]$ the $\ell$-primary subgroup of a finite abelian group $A$.

\ms
For any non-trivial divisor\footnote{By a \textit{non-trivial divisor} of $N$, we mean a positive divisor of $N$ different from $1$.} $d$ of $N$, we try to find a degree $0$ rational cuspidal divisor $D_d$ such that
\begin{equation*}
\scC(N) \simeq \moplus_{d \in \cD_N^0} \br{\ov{D_d}},
\end{equation*}
where $\cD_N^0$ is the set of all non-trivial divisors of $N$ and $\ov{D_d}$ denotes the linear equivalence class of $D_d$. 
Although it is possible when $N$ is a prime power, 
it seems already difficult if $N$ is the product of two primes. Nonetheless, we almost solve this problem by leaving the ``squarefree part'' aside. Let $\cD_N^\sqf$ be the set of all non-trivial squarefree divisors of $N$ and let $\cD_N^\nsqf$ be the set of all non-squarefree divisors of $N$.
\begin{theorem}\label{theorem: main theorem 1}
Let $N$ be a positive integer.
For any non-trivial divisor $d$ of $N$, there is a degree $0$ rational cuspidal divisor $Z(d)$ such that
\begin{equation*}
\scC(N) \simeq \br{\ov{Z(d)} : d \in \cD_N^\sqf} \moplus\left(\moplus_{d \in \cD_N^\nsqf} \br{\ov{Z(d)}}\right).
\end{equation*}
Also, the order of $\ov{Z(d)}$ is $\fn(N, d)$, which is defined in Section \ref{section: intro notation}.
\end{theorem}

Let $\scC(N)^\sqf:=\br{\ov{Z(d)} : d \in \cD_N^\sqf}$ be the ``squarefree part'' of the group $\scC(N)$. 
Since it seems difficult to directly find the decomposition of $\scC(N)^\sqf$ into cyclic groups, we deal with the decomposition of its $\ell$-primary subgroup instead.

\begin{theorem}\label{theorem: main theorem 2}
Let $N$ be a positive integer and let $\ell$ be any given prime.
For any non-trivial squarefree divisor $d$ of $N$, there is a degree $0$ rational cuspidal divisor $Y^2(d)$ such that
\begin{equation*}
\scC(N)^\sqf[\ell^\infty] \simeq \moplus_{d \in \cD_N^\sqf} \br{\ov{Y^2(d)}}[\ell^\infty].
\end{equation*}
Also, the order of $\ov{Y^2(d)}$ is $\fN(N, d)$, which is defined in Section \ref{section: intro notation}.
\end{theorem}

Thus, by Theorems \ref{theorem: main theorem 1} and \ref{theorem: main theorem 2}, we easily have the following.
\begin{theorem}
Let $N$ be a positive integer and let $\ell$ be any given prime. For any $d \in \cD_N^0$, there is a rational cuspidal divisor\footnote{
Let $a(d)$ and $b(d)$ be the prime-to-$\ell$ parts of $\fn(N, d)$ and $\fN(N, d)$, respectively. Then we define $Z_\ell(d):=a(d) \cdot Z(d)$ if $d$ is not squarefree, and $Z_\ell(d)=b(d) \cdot Y^2(d)$ otherwise.} $Z_\ell(d)$  such that
\begin{equation*}
\scC(N)[\ell^\infty] \simeq \moplus_{d \in \cD_N^0} \br{\ov{Z_\ell(d)}}.
\end{equation*}
\end{theorem}

As an application, we can determine the structure of the cuspidal group $\scC_N$ when $N=4M$ or $8M$ for $M$ odd squarefree because two groups $\scC(N)$ and $\scC_N$ are equal.

\ms
\subsection{Construction of the divisors $Z(d)$ and $Y^2(d)$}\label{section: intro strategy}
Before describing our method for computing the group $\scC(N)$, we review previous methods and point out some difficulties in their generalization.
For a divisor $d$ of $N$, there is a rational cuspidal divisor $(P_d)$ on $X_0(N)$ whose degree is equal to $\varphi(z)$, where $z=\gcd(d, N/d)$ and $\varphi(n)$ is the Euler's totient function. By Lemma \ref{lemma: group of rational cuspidal divisors}, the group $\Qdiv N$ is generated by 
\begin{equation*}
C_d:=\varphi(z)\cdot (P_1)-(P_d) ~~ \text{ for any } d\in \cD_N^0.
\end{equation*}
A starting point to understand the group $\scC(N)$ is to compute the order\footnote{By the \textit{order} of a degree $0$ rational cuspidal divisor on $X_0(N)$, we mean the order of its linear equivalence class in $J_0(N)$, which is always finite by the theorem of Manin and Drinfeld.} of $C_d$.
This can be easily done by Ligozat's method if $N$ is either a prime power \cite{Li97} or the product of two primes \cite{CL97}. However, as you can see in \cite[Th. 3.2.16]{Li75}, the formula for the order of $C_N$ looks very complicated in general. So we need a better understanding of the computation of the order of $C_d$ for any $d \in \cD_N^0$, which is done in Section \ref{chapter3}.

After that, to determine the structure of $\scC(N)$, it suffices to find all relations\footnote{We say that \textit{there is a relation among rational cuspidal divisors $C_d$} if there are integers $a_d$ such that $\sum a_d \cdot \ov{C_d}=0 \in J_0(N)$ and $a_d \cdot \ov{C_d} \neq 0$ for some $d$.} among $C_d$. 
As noticed on \cite[pg. 31]{Li97}, it seems difficult even in a simple case where $N$ is a power of $2$, and the main problem is that there is essentially only one method for finding relations among $C_d$, which is just a rephrasing of the definition: 
\begin{equation*}
\text{If there is a relation $\textstyle\sum a_d \cdot \ov{C_d}=0 \in J_0(N)$, then the order of $X=\textstyle\sum a_d \cdot C_d$ is $1$.}
\end{equation*}
Although the order of $X$ can be computed by Ligozat's method in principle once $a_d$ are given, it is very hard to find such integers $a_d$. Thus, it seems necessary to develop a new method for computing the group $\scC(N)$.

\ms
From now on, we illustrate our strategy for constructing the divisors $Z(d)$ and $Y^2(d)$. Before proceeding, we fix some notations, which will be used throughout the whole paper. 
\begin{notation}\label{notation: S1 and S2 in introduction}
For $k=1$ or $2$, let $\cS_k(N)_\Q$ be the $\Q$-vector space of dimension $\sigma_0(N)$, indexed by the divisors of $N$, and let $\cS_k(N)$ be the $\Z$-lattice of $\cS_k(N)_\Q$ consisting of integral vectors.
In other words,
\begin{equation*}
\begin{split}
\cS_k(N)&:=\textstyle\left\{ \sum_{d\mid N} a_d \cdot {\bf e}(N)_d  : a_d \in \Z \right\}\qa \\
\cS_k(N)_\Q&:=\textstyle\left\{ \sum_{d\mid N} a_d \cdot {\bf e}(N)_d  : a_d \in \Q \right\},
\end{split}
\end{equation*}
where ${\bf e}(N)_d$ is the unit vector in $\cS_k(N)$ whose $d$-th entry is $1$ and all other entries are zero. Also, let 
\begin{equation*}
\textstyle \cS_2(N)^0:=\left\{ \sum_{d\mid N} a_d \cdot {\bf e}(N)_d  \in \cS_2(N) : \sum_{d \mid N} a_d \cdot \varphi(\gcd(d, N/d))=0 \right\}.
\end{equation*}

Let $\Qdivv N$ be the group of the rational cuspidal divisors on $X_0(N)$. Then by Lemma \ref{lemma: group of rational cuspidal divisors}, we have a tautological isomorphism as abelian groups\footnote{The restriction of $\Phi_N$ to $\Qdiv N$ induces an isomorphism with $\cS_2(N)^0$.}
\begin{equation*}
\xymatrix{
\Phi_N : \Qdivv N \ar[r]^-{\sim} & \cS_2(N)
}
\end{equation*}
sending $(P_d)$ to ${\bf e}(N)_d$. Let $\Upsilon(N)=(\Upsilon(N)_{\delta d}) \in M_{n\times n}(\Z)$ be a square matrix of size $n=\sigma_0(N)$ (indexed by the divisors of $N$), defined in Section \ref{section: matrix upsilon}. We regard this matrix as a linear map from $\cS_2(N)_\Q$ to $\cS_1(N)_\Q$.
\end{notation} 

We use a capital Roman letter for a rational cuspidal divisor on $X_0(N)$ and the corresponding capital bold Roman letter for its image in $\cS_2(N)$ by $\Phi_N$.
For example, $D_d$ and $\bD_d$, $B_p(r, f)$ and $\bB_p(r, f)$, $A_p(r, f)$ and $\bA_p(r, f)$, $Z(d)$ and $\Z(d)$, $Y^i(d)$ and $\bY^i(d)$, respectively.

\ms
\noindent \textbf{Step 1}: First, we elaborate Ligozat's method and provide a simple algorithm for computing the order of a degree $0$ rational cuspidal divisor as follows:
For a rational cuspidal divisor 
\begin{equation*} 
C=\textstyle\sum_{d\mid N} a_d \cdot (P_d) \in \Qdiv N,
\end{equation*}
we compute an integral vector $V(C)=\sum_{\delta \mid N} V(C)_\delta \cdot {\bf e}(N)_\delta  \in \cS_1(N)$ defined as
\begin{equation*}
V(C):=\Upsilon(N) \times \Phi_N(C).
\end{equation*}
Let $\Gcd(C)$ be the greatest common divisor of the entries of $V(C)$ and let
\begin{equation*}
\bbV(C):=\Gcd(C)^{-1} \cdot V(C) \in \cS_1(N).
\end{equation*}
Although this computation is very easy, the vector $\bbV(C)$ plays a crucial role throughout the whole paper.  For each prime $p$, let
\begin{equation*}
\pw_p(C):=\textstyle\sum_{\tn{val}_p(\delta) \not\in 2\Z} \bbV(C)_\delta,
\end{equation*}
where the sum runs over the divisors of $N$ whose $p$-adic valuations are odd. Let
\begin{equation*}
\fh(C):=\begin{cases}
1 & \text{ if  }~~ \pw_p(C) \in 2\Z ~~\text{ for all primes } p,\\
2 & \text{ if  }~~ \pw_p(C) \not\in 2\Z ~~\text{ for some prime } p.
\end{cases}
\end{equation*}
Finally, let $\kappa(N)=N\prod_{p \mid N} (p-p^{-1})$. 
Then the order of $C$ is equal to
\begin{equation*}
\num\left(\frac{\kappa(N) \cdot \fh(C)}{24\cdot \Gcd(C)}\right)=\frac{\kappa(N)}{\gcd(\kappa(N), ~24\cdot \Gcd(C)\cdot \fh(C)^{-1})}.
\end{equation*}

As an application, we simplify Ligozat's formula and compute the order of $C_d$ for any non-trivial divisor $d$ of $N$. More specifically, we prove that
\begin{equation*}
\Gcd(C_d)=\fg(N, d) \qa \fh(C_d)=\fh(N, d),
\end{equation*}
where $\fg(N, d)$ and $\fh(N, d)$ are defined in Section \ref{section: Example II: The order of Cd}. Note that in most cases we have $\fg(N, d)=\fh(N, d)=1$, and so the order of $C_d$ is $\frac{\kappa(N)}{24}$.

\ms
\noindent \textbf{Step 2}: Since finding relations among $C_d$ is quite difficult, we
try to find a criterion for proving ``linear independence'' among rational cuspidal divisors instead. As a result, we have the following.
\begin{theorem}\label{theorem: main criterion}
Let $C_i \in \Qdiv N$ for all $1\leq i\leq k$. Suppose that there is a divisor $\delta$ of $N$ such that
\begin{equation*}
|\bbV(C_k)_\delta|=1 \qa \bbV(C_i)_\delta=0 ~~\text{ for all } ~ i < k.
\end{equation*}
If $\fh(C_k)=1$, then we have $\br{\ov{C_i} : 1\leq i \leq k-1} \cap \br{\ov{C_k}}= 0$, or equivalently 
\begin{equation*}
\br{\ov{C_i} : 1\leq i \leq k} \simeq \br{\ov{C_i} : 1\leq i \leq k-1} \moplus \br{\ov{C_k}}.
\end{equation*}
\end{theorem}
Applying this criterion successively, we can easily deduce the following: 
For rational cuspidal divisors $D_d$ on $X_0(N)$, we consider a square matrix 
\begin{equation*}
\fM=(|\bbV(D_d)_{\delta}|)_{d \delta}
\end{equation*}
indexed by the non-trivial divisors of $N$. If $\fM$ is lower-unipotent\footnote{A square matrix is \textit{lower-unipotent} if it is lower-triangular and all its diagonal entries are $1$.} (with respect to suitable orderings on $\cD_N^0$), and $\fh(D_d)=1$ for all but one that corresponds to the first row, then we have 
\begin{equation*}
\br{\ov{D_d} : d \in \cD_N^0} \simeq \moplus_{d \in \cD_N^0} \br{\ov{D_d}}.
\end{equation*}
Additionally, if the divisors $D_d$ generate $\Qdiv N$, or equivalently 
\begin{equation*}
\cS_2(N)^0=\br{\bD_d : d \in \cD_N^0},
\end{equation*}
then we easily have 
\begin{equation*}
\scC(N) \simeq \moplus_{d \in \cD_N^0} \br{ \ov{D_d} }.
\end{equation*}
Although the required assumptions are pretty strong, we can verify them in many cases, and this strategy is quite useful in our computation.

\ms
\noindent \textbf{Step 3}: We first apply our strategy when $N=p^r$ is a prime power. 
To do so, we have to find a rational cuspidal divisor $D_d$ such that most of the entries of $\bbV(D_d)$ are zeros, which can be constructed from lower levels (cf. Proposition \ref{prop: degeneracy maps revisited}). Thus, we first find ``nice'' vectors $\bbB_p(r, f) \in \cS_1(N)$ for any $1\leq f \leq r$ so that the matrix 
\begin{equation*}
\fM=(|\bbB_p(r, f)_{p^k}|)_{1 \leq f, k \leq r}
\end{equation*}
is lower-unipotent with respect to suitable orderings on $\{1, 2, \dots, r\}$, and then compute $\Upsilon(p^r)^{-1}\times \bbB_p(r, f)$. 
By removing the denominators of the entries, we obtain ``nice'' vectors $\bB_p(r, f) \in \cS_2(N)^0$ such that 
$\Upsilon(p^r) \times \bB_p(r, f)$ is a scalar multiple of $\bbB_p(r, f)$. Furthermore, we prove that the vectors $\bB_p(r, f)$ (integrally) generate $\cS_2(p^r)^0$. If $p$ is odd, then we have $\fh(B_p(r, f))=1$ for any $f\geq 2$, and so
\begin{equation*}
\scC(p^r) \simeq \moplus_{f=1}^r \br{\ov{B_p(r, f)}}.
\end{equation*}

However, if $p=2$, then the arguments above break down since $\fh(B_2(r, f))=2$ for some $f\geq 2$. Nonetheless, we find another one $\bB^2(r, f)$ and show that
\begin{equation*}
\scC(2^r) \simeq \moplus_{f=3}^r \br{\ov{B^2(r, f)}}.
\end{equation*}
(Our proof relies on two facts: one is that the genus of $X_0(16)$ is zero, and the other is that the group $\scC(2^r)$ is a $2$-group.)

\ms
\noindent \textbf{Step 4}: We then apply our strategy for any positive integer $N$. 
For simplicity, let $N=Mp^r$ with $\gcd(M, p)=1$. (Here, $p$ denotes a prime as above.) 
As already mentioned, we need a rational cuspidal divisor $D_d$ such that most of the entries of $\bbV(D_d)$ are zeros, which can be constructed using tensor product.\footnote{By the Chinese remainder theorem, we have $\cS_2(N)_\Q \simeq \cS_2(M)_\Q \motimes \cS_2(p^r)_\Q$ (Remark \ref{remark: decomposition of RCD}).} In general, the dimension of $\cS_2(N)^0$ is larger than the product of the dimensions of $\cS_2(M)^0$ and $\cS_2(p^r)^0$, and so we need more vectors, which are not of degree $0$. Motivated by the discussion in Section \ref{subsection: odd t=2} (Remark \ref{remark: A0 and A1}), we construct vectors $\bA_p(r, 0)$ and $\bA_p(r, 1)$ in $\cS_2(p^r)$. Note that $\bA_p(r, 0)$ is constructed when we regard a divisor $C$ in level $M$ as one in level $Mp^r$, and $\bA_p(r, 1)$ is obtained by applying the degeneracy map from level $M$ to level $Mp^r$, which may be regarded as a vector (in level $p^r$) ``coming from level $1$''. 
By letting $\bA_p(r, f):=\bB_p(r, f)$ for any $2\leq f\leq r$, we can prove that the vectors $\bA_p(r, f)$ (integrally) generate $\cS_2(p^r)$.
Using these vectors, we now define a vector $\Z^1(d)$ for any non-trivial divisor $d$ of $N$ as follows.
\begin{definition}
Let $N=\prod_{i=1}^t p_i^{r_i}$ and $d=\prod_{i=1}^t p_i^{f_i}$. Then we define a vector $\Z^1(d) \in \cS_2(N)^0$ as
\begin{equation*}
\Z^1(d):=\begin{cases}
\motimes_{i=1}^t \bA_{p_i}(r_i, f_i)  & \text{ if }~~ d \in \cD_N^\nsqf,\\
\motimes_{i=1, \, i\neq m}^t \bA_{p_i}(r_i, f_i) \motimes \bB_{p_m}(r_m, 1) & \text{ if } ~~ d\in \cD_N^\sqf,\\
\end{cases}
\end{equation*}
where $m$ is the smallest positive integer such that $f_m=1$.
\end{definition}

If $N$ is odd, we can prove Theorem \ref{theorem: main theorem 1} as in Step 3. On the other hand, if $N$ is divisible by $32$, there are some problematic vectors that we cannot apply our strategy as in the case of level $2^r$. Using the vectors $\bB^2(r, f)$, we then define a vector $\Z(d)$ as follows:
\begin{equation*}
\Z(d):=\begin{cases}
\motimes_{i=1, \, i\neq u} \bA_{p_i}(r_i, 1) \motimes \bB^2(r_u, f_u) & \text{ if } ~~(f_1, \dots, f_t) \in \cT_u,\\
\Z^1(d) & \text{ otherwise},
\end{cases}
\end{equation*}
where $\cT_u:=\emptyset$ if either $u=0$ or $r_u\leq 4$, and otherwise 
\begin{equation*}
\cT_u:= \{ I=(f_1, \dots, f_t) \in \square(t) : 3\leq f_u \leq r_u, ~~ f_i=1 \text{ for all } i \neq u\}.
\end{equation*}
(Here, $u$ denotes the index such that $p_u=2$.)
Unfortunately, we cannot prove Theorem \ref{theorem: main theorem 1} directly when $N$ is divisible by $32$. So we prove a partial result first (Theorem \ref{theorem: main theorem in 6.4}) and finish the proof in Section \ref{section: final step}.

\begin{remark}
In Section \ref{section: definition of Zd}, we fix an ordering of the prime divisors of $N$  (using Assumption \ref{assumption 1.14} below) and define $\Z(d)$. 
On the other hand, here and in Theorem \ref{theorem: main theorem 1}, we do not choose a specific ordering of the prime divisors of $N$ in the definition of $\Z(d)$. 
This does not cause any problem because
\begin{enumerate}
\item
for any $d \in \cD_N^\nsqf$, the definition of $Z(d)$ does not depend on the ordering of the prime divisors of $N$, and
\item
the squarefree part $\scC(N)^\sqf$ does not depend on the ordering of the prime divisors of $N$ (Remark \ref{remark: independence of ell}).
\end{enumerate}
\end{remark}

\ms
\noindent \textbf{Step 5}: For any given prime $\ell$, we try to understand the group $\scC(N)^\sqf[\ell^\infty]$. Let $N=p_1^{r_1} p_2^{r_2}$ be the product of two prime powers. Then we construct the following vector in $\cS_2(N)$, which is \textit{not defined by tensors}\footnote{For its definition, see Section \ref{section: Example I: The divisors defined by tensors}.}:
\small
\begin{equation*}
\bD(p_1^{r_1}, p_2^{r_2}):=\gcd(\gamma_1, \gamma_2)^{-1}(\gamma_2\cdot \bB_{p_1}(r_1, 1) \motimes \bA_{p_2}(r_2, 0)-\gamma_1\cdot \bA_{p_1}(r_1, 0) \motimes \bB_{p_2}(r_2, 1)),
\end{equation*}\normalsize
where $\gamma_i=p_i^{r_i-1}(p_i+1)$. Using these vectors, we define a vector $\bY^0(d)$ as follows.
\begin{definition}
Let $N=\prod_{i=1}^t p_i^{r_i}$ be a positive integer. For a given prime $\ell$, 
by appropriately ordering the prime divisors of $N$, we make Assumption \ref{assumption 1.14} below. Then for a non-trivial squarefree divisor $d=\prod_{i=1}^t p_i^{f_i}$ of $N$, we define a vector $\bY^0(d)$ in $\cS_2(N)^0$ as
\begin{equation*}
\bY^0(d):=\begin{cases}
\motimes_{i=1, \, i\neq m}^t \bA_{p_i}(r_i, f_i) \motimes \bB_{p_m}(r_m, 1) & \text{ if }~~f_i=1 \text{ for all } i\geq m,\\
\motimes_{i=1, \, i\neq m, n}^t \bA_{p_i}(r_i, f_i) \motimes \bD(p_m^{r_m}, p_n^{r_n})  & \text{ otherwise},
\end{cases}
\end{equation*}
where $m$ is as above and $n$ is the smallest integer such that $n>m$ and $f_n=0$.
\end{definition}

By its construction (and Assumption \ref{assumption 1.14}), it is not difficult to show that
\begin{equation*}
\textstyle\br{\Z(d) : d \in \cD_N^\sqf} \motimes_\Z \Z_\ell = \br{ \bY^0(d) : d \in \cD_N^\sqf} \motimes_\Z \Z_\ell.
\end{equation*}
Also, the matrix $\fM_0:=(\bbV(Y^0(d))_{\delta})_{d\delta}$ is lower-triangular (with respect to suitable orderings on $\cD_N^\sqf$). Moreover, if $N$ is odd, then the diagonal entries of $\fM_0$ are $\ell$-adic units. Thus, by an $\ell$-adic variant of Theorem \ref{theorem: main criterion}, we prove that
\begin{equation*}
\scC(N)^\sqf [\ell^\infty] \simeq \moplus_{d \in \cD_N^\sqf} \br{ \ov{Y^0(d)}} [\ell^\infty].
\end{equation*}

\ms
\noindent \textbf{Step 6}: However, some of the previous arguments break down if $N$ is even. As in the case of level $2^r$, we know exactly where the problems occur. So by replacing problematic elements by new ones (Remark \ref{remark: Y0 Y1 Y2}), we construct a rational cuspidal divisor $Y^1(d)$ and show that $\fM_0:=(\bbV(Y^1(d))_{\delta})_{d \delta}$ is lower-triangular and all its diagonal entries are $\ell$-adic units. In this case, it is not obvious that
\begin{equation*}
\textstyle\br{\Z(d) : d \in \cD_N^\sqf} \motimes_\Z \Z_\ell = \br{ \bY^1(d) : d \in \cD_N^\sqf} \motimes_\Z \Z_\ell,
\end{equation*}
which is proved in Section \ref{section: generation}. Hence for an odd prime $\ell$, we prove that
\begin{equation*}
\scC(N)^\sqf [\ell^\infty] \simeq \moplus_{d \in \cD_N^\sqf} \br{ \ov{Y^1(d)}} [\ell^\infty].
\end{equation*}
But still, we have a problem if $\ell=2$. Thus, we finally construct a rational cuspidal divisor $Y^2(d)$ and for any prime $\ell$, we prove that
\begin{equation*}
\scC(N)^\sqf [\ell^\infty] \simeq \moplus_{d \in \cD_N^\sqf} \br{ \ov{Y^2(d)}} [\ell^\infty].
\end{equation*}

\begin{remark}
Using our criteria for linear independence, we may easily guess which divisors are linearly independent. However, they frequently fail to generate the whole group $\Qdiv N$.
The main achievement of this paper is that we actually succeed in finding such ($\ell$-adic) generators that satisfy strong assumptions in our criteria. Thus, we obtain the decomposition of the $\ell$-primary subgroup of $\scC(N)$ for any positive integer $N$ and any prime $\ell$. 
\end{remark}

\ms
\subsection{Notation and Convention}\label{section: intro notation}
In order to avoid excessive repetition, we adhere to some conventions throughout the whole paper.
\begin{itemize}[--]
\item
$p$, $p_i$ and $\ell$ : prime numbers (unless otherwise mentioned).
\item
$r$, $r_i$, $f$ and $f_i$ : non-negative integers, which are exponents of primes
(unless otherwise mentioned).
\item
$N$ : a positive integer (unless otherwise mentioned).

\item
$\rad(N)$ : the \textit{radical} of $N$, the largest squarefree divisor of $N$, i.e., $\rad(N):=\prod_{p\mid N} p$.
\item
$\tn{val}_p(N)$ : the (normalized) $p$-adic valuation of $N$, i.e., $N$ is divisible by $p^{\tn{val}_p(N)}$ but not by $p^{\tn{val}_p(N)+1}$.

\item
$\kappa(N):=N\prod_{p\mid N} (p-p^{-1})=\frac{N}{\rad(N)}\prod_{p\mid N}(p^2-1)$.
\item
$\varphi(N):=N\prod_{p\mid N}(1-p^{-1})=\frac{N}{\rad(N)}\prod_{p\mid N}(p-1)$.
\item
$\cD_N$ : the set of all (positive) divisors of $N$.
\item
$\sigma_0(N):=\# \cD_N$ : the number of all divisors of $N$.
\item
$\cD_N^0:=\cD_N \sm \{1\}$ : the set of all non-trivial divisors of $N$.
\item
$\cD_N^\nsqf$ : the set of all non-squarefree divisors of $N$.
\item
$\cD_N^\sqf:=\cD_N^0 \sm \cD_N^\nsqf$ : the set of all non-trivial squarefree divisors of $N$.
\end{itemize}

If we write $N=\prod_{i=1}^t p_i^{r_i}$ for some $t\geq 1$ and $r_i\geq 1$, then we use the following.
\begin{itemize}[--]
\item
$\Omega(t):=\{(f_1, \dots, f_t) \in \Z^t : 0\leq f_i \leq r_i \text{ for all } i \text{ and } f_i\neq 0 \text{ for some } i \}$.
\item
$\Delta(t):=\{ (f_1, \dots, f_t) \in \Z^t : 0\leq f_i \leq 1 \text{ for all } i \text{ and } f_i\neq 0 \text{ for some }i \}$.
\item
$\square(t):=\{ (f_1, \dots, f_t) \in \Omega(t) : f_i\geq 2 \text{ for some } i\}=\Omega(t)\sm \Delta(t)$.
\item
$\fp_I:=\prod_{i=1}^t p_i^{f_i}$ for any $I=(f_1, \dots, f_t) \in \Omega(t)$.\footnote{By definition, $\cD_N^0=\{ \fp_I : I \in \Omega(t)\}$, $\cD_N^\sqf=\{\fp_I : I \in \Delta(t) \}$ and $\cD_N^\nsqf = \{\fp_I : I \in \square(t)\}$.}
\end{itemize}

If either $u=0$ or $r_u\leq 4$, then let $\cT_u:=\emptyset$. Otherwise, let
\begin{equation*}
\cT_u:= \{ I=(f_1, \dots, f_t) \in \square(t) : 3\leq f_u \leq r_u \qqa f_i=1 \text{ for all } i \neq u\}.
\end{equation*}

For an element $I=(f_1, \dots, f_t) \in \Delta(t)$, let 
\begin{itemize}[--]
\item
$m(I)$ : the smallest positive integer $m$ such that $f_m=1$.
\item
$n(I)$ : the smallest integer $n$ such that $n>m(I)$ and $f_n=0$.
\item
$k(I)$ : the smallest integer $k$ such that $k>n(I)$ and $f_k=0$.
\end{itemize}
Here, we set $n(I):=t+1$ (resp. $k(I):=t+1$) if $f_i=1$ for all $i>m(I)$ (resp. $i>n(I)$).

For any integer $1\leq k \leq t$, let
\begin{itemize}[--]
\item
$A(k):=(f_1, \dots, f_t)$ such that $f_i=0$ for all $i<k$ and $f_j=1$ for all $j\geq k$.
\item
$E(k):=(f_1, \dots, f_t)$ such that $f_k=0$ and $f_i=1$ for all $i\neq k$.
\item
$F(k):=(f_1, \dots, f_t)$ such that $f_k=1$ and $f_i=0$ for all $i\neq k$.
\end{itemize}
Also, for a given integer $1\leq u \leq t$ and any integer $1\leq k \leq t$ different from $u$, let
\begin{itemize}[--]
\item
$E_u(k):=(f_1, \dots, f_t)$ such that $f_k=f_u=0$ and $f_i=1$ for all $i\neq k, u$.
\item
$F_u(k):=(f_1, \dots, f_t)$ such that $f_k=f_u=1$ and $f_i=0$ for all $i\neq k, u$.
\end{itemize}

Now, we define some subsets of $\Delta(t)$. Let
\begin{equation*}
\cE:=\{I \in \Delta(t) : n(I)=t+1\}=\{A(m) : 1\leq m \leq t\}.
\end{equation*}
If $u\leq 1$ then we set $\cH_u=\cH_u^1:=\emptyset$, and for any $2\leq u \leq t$, we set
\begin{itemize}[--]
\item
$\cH_u:=\{ (f_1, \dots, f_t) \in \Delta(t) : n(I)=u \qqa k(I) \leq t \}$.
\item
$\cH_u^1:=\{(f_1, \dots, f_t) \in \Delta(t) : n(I)=u \qqa k(I)=t+1 \}$.
\end{itemize}
Also, if $u=0$ then we set $\cF_u=\cF_u^1=\emptyset$, and for any $1\leq u \leq t$, we set
\begin{equation*}
\cF_u:=\{ E(n) : n \in \cI_u \} \qa \cF_u^1:=\{ E_u(n) : n \in \cI_u \},
\end{equation*}
where 
\begin{equation*}
\cI_u:=\begin{cases}
\{n \in \Z : 3\leq n \leq t\} & \text{ if } ~~ u=1, \\
\{n \in \Z : 2\leq n \leq t, ~~ n\neq u\} & \text{ otherwise}.
\end{cases}
\end{equation*}
Furthermore, for $0\leq u \leq t$, we set
\begin{equation*}
\cG_u:=\begin{cases}
\{E(2)\} & \text{ if } ~~u=1,\\
\quad \emptyset & \text{ otherwise,}
\end{cases} \qa \cG_u^1:=\begin{cases}
\{E(n) : 1\leq n \leq t\} & \text{ if }~~u=1,\\
\{E(n) : 2\leq n \leq t\} & \text{ otherwise}.
\end{cases}
\end{equation*}

\vspace{3mm}
The definitions of $\fn(N, d)$ and $\fN(N, d)$, which both depend on the ordering of the prime divisors of $N$, are a bit complicated. We first assume the following.
\begin{assumption}\label{assumption 1.14}
Let $N=\prod_{i=1}^t p_i^{r_i}$ be the prime factorization of $N$. For a given prime $\ell$, by appropriately ordering the prime divisors of $N$, we assume that
\begin{equation*}
\tn{val}_\ell(\gamma_i) \geq \tn{val}_\ell(\gamma_j) ~\text{ for any } 1\leq i < j \leq t, \text{ where } \gamma_i:=p_i^{r_i-1}(p_i+1).
\end{equation*}
Let $u$ be the smallest positive integer such that $p_u=2$ if $N$ is even, and $u=0$ otherwise. Also, let $s=0$ if $\ell$ is odd, and $s=u$ if $\ell=2$. We further assume that
\begin{equation*}
\tn{val}_\ell(p_i-1) \leq \tn{val}_\ell(p_j-1) ~~\text{ for any } 1 \leq i < j \leq t \text{ different from } s.
\end{equation*}
\end{assumption}

Next, we define the following.
\begin{definition}
For a positive integer $r$, let
\begin{equation*}
\cG_p(r, f):=\begin{cases}
p^{r-1}(p^2-1) & \text{ if }~~ f=0,\\
1 & \text{ if }~~ f=1, \\
p^2-1 & \text{ if }~~ f=2, \\
p^{r-1-j}(p^2-1) & \text{ if }~~ 3\leq f \leq r,
\end{cases}
\end{equation*}
where $j=[\frac{r+1-f}{2}]$. Also, let
\begin{equation*}
\cG(p_i^{r_i}, p_j^{r_j}) :=\frac{(p_i-1)(p_j-1) \cdot \gcd(\gamma_i, \, \gamma_j)}{\gcd(p_i-1, \, p_j-1)}.
\end{equation*}
\end{definition}

Then, we define $\fn(N, d)$ for any $d \in \cD_N^0$ as follows.
\begin{definition}
For any $I=(f_1, \dots, f_t) \in \Omega(t)$, let
\begin{equation*}
\cG(N, \fp_I):=\begin{cases}
\prod_{i=1}^t \cG_{p_i}(r_i, f_i) & \text{ if }~~ I \in \square(t),\\
\prod_{i=1, \, i\neq m}^t \cG_{p_i}(r_i, f_i) \times (p_m-1) & \text{ if }~~ I \in \Delta(t),
\end{cases}
\end{equation*}
where $m=m(I)$. Also, let $\cH(N, \fp_I):=2$ if one of the following holds, and $\cH(N, \fp_I):=1$ otherwise.
\begin{enumerate}
\item
$I=A(1)$.
\item
$u\geq 1$ and $I=E(u)$.
\item
$u\geq 1$, $3\leq r_u\leq 4$, $f_u=3$ and $f_i=1$ for all $i \neq u$.
\item
$u\geq 1$, $r_u\geq 5$, $f_u=r_u+1 -\gcd(2, r_u)$ and $f_i=1$ for all $i \neq u$.
\end{enumerate}
Furthermore, let
\begin{equation*}
\fn(N, d):=\num \left( \frac{\cG(N, d) \times \cH(N, d)}{24} \right).
\end{equation*}
\end{definition}

Lastly, we define $\fN(N, d)$ for any $d \in \cD_N^\sqf$ as follows.
\begin{definition}
For any $I=(f_1, \dots, f_t) \in \Delta(t)$ with $m=m(I)$, $n=n(I)$ and $k=k(I)$, let
\begin{equation*}
\scG(N, \fp_I):=\begin{cases}
\prod_{i=1, \, i\neq x}^t \cG_{p_i}(r_i, f_i) \times (p_x-1) & \text{ if }~~ I \in \cE,\\
\cG(p_y^{r_y}, p_n^{r_n})  & \text{ if } ~~I \in \cF_s,\\
\cG_{p_2}(r_2, 0) & \text{ if }~~ I \in \cG_s,\\
\prod_{i=1, \, i\neq m, k}^t \cG_{p_i}(r_i, f_i) \times \cG(p_m^{r_m}, p_k^{r_k})  & \text{ if }~~ I \in \cH_u,\\
\prod_{i=1, \, i\neq m, n}^t \cG_{p_i}(r_i, f_i) \times \cG(p_m^{r_m}, p_n^{r_n}) & \text{ otherwise},
\end{cases}
\end{equation*}
where $x=\max(m, u)$ and $y=\max(1, 3-s)$. Also, let 
\begin{equation*}
\scH(N, \fp_I):=\begin{cases}
2 & \text{ if }~~ I \in (\cF_u^1 \cup \cG_u^1 \cup \{A(1)\}) \sm (\cF_s \cup \cG_s),\\
1 & \text{ otherwise}.
\end{cases}
\end{equation*}
Furthermore, let
\begin{equation*}
\fN(N, d):=\num \left( \frac{\scG(N, d) \times \scH(N, d)}{24} \right).
\end{equation*}
\end{definition}

\vspace{10mm}
\section{The cusps of $X_0(N)$}\label{chapter2}
In this section, we review the results about the cusps of $X_0(N)$, which are well-known to the experts. Although there is no new result in this section, we provide detailed proofs as elementary and self-contained as possible for the convenience of the readers. 

\ms
In this section (except Section \ref{section: Rational cuspidal divisors}), we only consider the modular curves over $\C$ (not over $\Q$) and regard them as compact Riemann surfaces (given by explicit charts). By explicit and concrete methods, we obtain various results on the cusps without further digression on algebraic theory.\footnote{For more general discussion on the modular curves over $\C$, see Chapters $2$ and $3$ of \cite{DS05}.}
For the algebraic description of the cusps (using generalized elliptic curves), see \cite{DR73}, \cite{Con07} or \cite{Ces17}.

As usual, let $Y_0(N)_{\C}:=\Gamma_0(N) \backslash \cH$, where $\cH=\{ z \in \C : \tn{Im } z>0\}$ and $\Gamma_0(N)\subset \SL_2(\Z)$ acts on $\cH$ by linear fractional transformations. Also, let $X_0(N)_{\C}$ denote the classical (analytic) modular curve, the ``canonical'' compactification of $Y_0(N)_{\C}$. The \textit{cusps} of $X_0(N)$ are the points added for the compactification, which can be naturally identified with the equivalence classes of $\bP^1(\Q)$ modulo $\Gamma_0(N)$, i.e.,
\begin{equation*}
\{\text{cusps of $X_0(N)$}\}:=X_0(N)_{\C} \sm Y_0(N)_{\C} \simeq \Gamma_0(N) \backslash \bP^1(\Q),
\end{equation*}
where $\Gamma_0(N)$ acts on $\bP^1(\Q)$ by linear fractional transformation.

In Sections \ref{section: degeneracy maps} and \ref{section: AL operators}, we study the degeneracy maps and the Atkin--Lehner operators on modular curves. To investigate their properties, it is often useful to consider them as holomorphic maps\footnote{They also have ``moduli interpretations'', so there exist corresponding algebraic morphisms. But we do not discuss ``moduli interpretations'' here as they are not used. For such discussions, see \cite[Sec. 13]{MR91} and \cite[Sec. 1]{Oh14}.} between compact Riemann surfaces.  More specifically, let $A$ be a positive integer and $B$ its (positive) divisor. Let
\begin{equation*}
\gamma=\mat a b c d \in M_2(\Z) \cap \GL_2^+(\Q)
\end{equation*}
be a matrix satisfying $\Gamma:=\gamma\Gamma_0(A) \gamma^{-1} \subset \Gamma_0(B)$. Then since 
\begin{equation*}
\gamma(\Gamma_0(A)\tau)=(\gamma \Gamma_0(A) \gamma^{-1})(\gamma \cdot \tau)=\Gamma(\gamma \cdot \tau),
\end{equation*}
we have natural maps
\begin{equation*}
\xyh{4.4}
\xyv{1.5}
\xymatrix{
Y_0(A)_{\C} \ar[r]^-{\sim}_-{\times \gamma} & \Gamma \backslash \cH \ar@{->>}[r]_-{\text{taking modulo $\Gamma_0(B)$}~~} & Y_0(B)_{\C}\\
[\tau \modo {\Gamma_0(A)}] ~~\ar@{|->}[r] \arin & ~~ [\gamma \cdot \tau \modo {\Gamma}] \arin \ar@{|->}[r] ~~& ~~ [\gamma \cdot \tau \modo {\Gamma_0(B)}] \arin.
}
\end{equation*}
The composition of the two maps naturally extends to a holomorphic map from $X_0(A)_{\C}$ to $X_0(B)_{\C}$, denoted by $F_\gamma$. Note that for any $p, q \in \Z$ with $\gcd(p, q)=1$, we have 
\begin{equation*}
\gamma \cdot \frac{p}{q}=\frac{a(p/q)+b}{c(p/q)+d}=\frac{ap+bq}{cp+dq}=\frac{(ap+bq)/g}{(cp+dq)/g},
\end{equation*}
where $g=\gcd(ap+bq, cp+dq)$, and therefore 
\begin{equation}\label{equation: cusp image by usual matrix gamma}
F_\gamma\left(\left[\frac{p}{q} \modo {\Gamma_0(A)} \right]\right)=\left[\frac{(ap+bq)/g}{(cp+dq)/g} \modo \Gamma_0(B)\right].
\end{equation}
Note that $g$ is a divisor of $\det(\gamma)=ad-bc$ by Lemma \ref{lemma: determinant and gcd} below.

\begin{remark}
The results about the Atkin--Lehner operators and the Hecke operators are not used in this paper, but we include them for the sake of the readers.
\end{remark}

\ms
\subsection{Representatives of the cusps}\label{section: representatives of the cusps}
Let
\begin{equation*}
(\Z^2)':=\left\{\vect a b \in \Z^2 : \gcd(a, b)=1 \right\},
\end{equation*}
and we define an equivalence relation on $(\Z^2)'$ by
\begin{equation*}
\vect a b \sim \vect {a'}{b'} \iff  \vect {a'}{b'}=\vect{ra+ub}{va+wb} \text{ for some } \mat r u v w \in \Gamma_0(N).
\end{equation*}
We denote by $\bect a b^{N}$ (or simply $\bect a b$ if there is no confusion) an equivalence class of $\vect a b \in (\Z^2)'$. If we write the notation $\bect a b$, we \textit{always} assume that $a$ and $b$ are relatively prime integers. 

\ms
A cusp of $X_0(N)$ can be regarded as an equivalence class in $(\Z^2)'/\sim$ (as in \cite[Sec. 1.3]{Ste82}). Thus, we simply denote a cusp of $X_0(N)$ by $\bect a b$. 
For a (positive) divisor $d$ of $N$, we say that a cusp of $X_0(N)$ is \textit{of level $d$} if it is equivalent to $\vect x d$ for some $\vect x d \in (\Z^2)'$. Thus, a cusp of level $d$ is written as $\bect x d$ for some integer $x$ relatively prime to $d$.

\begin{theorem}\label{theorem: cusp representation}
For any $\vect a b \in (\Z^2)'$, we have $\bect a b=\bect x d$ for some integer $x$, where $d=\gcd(b, N)$. Also, for two divisors $d$ and $d'$ of $N$, we have
\begin{equation*}
\bect x d = \bect {x'}{d'} \iff d=d' \qqa x\equiv x' \pmodo {\gcd(d, N/d)}.
\end{equation*}
\end{theorem}
As a corollary, we have the following.
\begin{corollary}\label{corollary: cusp representation}
The set of the cusps of $X_0(N)$ can be written as
\begin{equation*}
\left\{ \bect x d ~:~ 1\leq d \mid N, ~\gcd(x, d)=1 \text{ and $x$ taken modulo $\gcd(d, N/d)$} \right\}.
\end{equation*}
\end{corollary}

\begin{remark}\label{remark: cusp representative}
Since the map $(\zmod N)^\times \to (\zmod d)^\times$ is surjective for any divisor $d$ of $N$, it is often useful to take the set of the cusps of $X_0(N)$ by
\begin{equation*}
\left\{ \bect x d ~:~ 1\leq d \mid N, ~\gcd(x, N)=1 \text{ and $x$ taken modulo $\gcd(d, N/d)$} \right\}.
\end{equation*}
\end{remark}

To begin with, we show three types of equivalences of cusps.
\begin{lemma}\label{lemma: three equivalences}
We have
\begin{equation}\tag{1}
\bect a b^N = \bect {a+jb}{b}^N,
\end{equation}
\begin{equation}\tag{2}
\bect a b^N = \bect a {b+kN}^N,
\end{equation}
\begin{equation}\tag{3}
\bect{a}{yb}^N= \bect{ya}{b}^N.
\end{equation}
Here, $j$ can be any integer, $k$ can be any integer satisfying $\gcd(a, b+kN)=1$ and $y$ can be any integer satisfying $\gcd(y, abN)=1$.
\end{lemma}
\begin{proof}
The first assertion easily follows because $\mat 1 j 0 1 \in \Gamma_0(N)$. To prove the second assertion, we find $v, w \in \Z$ such that
$av+b(b+kN)w=k$, which is possible because $\gcd(a, b(b+kN))=1$. Now, we consider the following matrix:
\begin{equation*}
\gamma=\mat {1-Nbw}{Naw}{Nv}{1+N(b+kN)w} \in \Gamma(N) \subset \Gamma_0(N),
\end{equation*}
where $\Gamma(N)$ is the \textit{principal congruence subgroup of level $N$}, which is defined as the kernel of the natural homomorphism $\SL_2(\Z) \to \SL_2(\zmod N)$ induced by the reduction modulo $N$.
By direct computation, we have $\gamma\vect a b=\vect {a}{b+kN}$.\footnote{More generally, $\Gamma(N)\vect a b =\Gamma(N)\vect{a'}{b'}$ if and only if $\vect a b \equiv \pm \vect {a'}{b'} \pmod N$.}

For the last one, we can find $v, w \in \Z$ such that $\gamma=\mat y v N w \in \Gamma_0(N)$ because $\gcd(y, N)=1$. Since $wy=1+Nv \equiv 1 \pmod N$, we have
\begin{equation*}
\bect {a}{yb}=\left[ \gamma \vect {a}{yb} \right]=\bect {ya+vyb}{Na+wyb} \overset{(2)}{=} \bect {ya+vyb}{b} \overset{(1)}{=} \bect {ya}{b},
\end{equation*}
as desired. This completes the proof.
\end{proof}

\begin{proof}[Proof of Theorem \ref{theorem: cusp representation}]
Let $d=\gcd(b, N)$. Also, let $b'=b/d$ and $N'=N/d$. 

To prove the first assertion, we first suppose that $\gcd(b', N)=1$. Then we have
\begin{equation*}
\bect a b=\bect {a}{b'd}\overset{(3)}{=} \bect {b'a}{d}.
\end{equation*}
Suppose next that $\gcd(b', N)\neq 1$. Since $\gcd(b', N')=1$, there is an integer\footnote{This can be proved directly by taking $k$ as the product of all prime divisors of $ad$ not dividing $b'$, or by Dirichlet's theorem on arithmetic progression.} $k$ such that $\gcd(b'+kN', ad)=1$. 
Also, since $\gcd(b'+kN', N')=\gcd(b', N')=1$, we have $\gcd(b'+kN', aN)=1$. 
Finally, since $\gcd(b+kN, a)=1$, we have
\begin{equation*}
\bect a b\overset{(2)}{=} \bect {a}{b+kN} \overset{(3)}{=} \bect {(b'+kN')a}{d}.
\end{equation*}
This completes the proof of the first assertion.

Next, we prove the second assertion. For simplicity, let $z=\gcd(d, N/d)$.
For any $\gamma=\mat r u {Nv} w \in \Gamma_0(N)$ and $\vect a b \in (\Z^2)'$, we have $\gcd(Nva+wb, N)=\gcd(wb, N)=\gcd(b, N)$. Thus, any two equivalent cusps have the same level. Now, we claim that $\bect x d = \bect {x'} d$ if and only if $x \equiv x' \pmod z$.
 
Suppose first that $\bect x d=\bect {x'}{d}$, i.e., there is a matrix $\gamma=\mat r u {Nv} w \in \Gamma_0(N)$ such that $\gamma \vect x d=\vect {x'}{d}$. 
Then we have $d=Nvx+dw$ and $rw -Nuv=1$. Thus, we have $w\equiv 1 \pmod {N'}$ and $rw \equiv 1 \pmod N$. This implies that $r\equiv 1 \pmod {N'}$. Since $z$ is a divisor of $N'$, we also have $r \equiv 1 \pmod z$, and so $x'=rx+du \equiv x \pmod z$, as wanted.

Conversely, suppose that $x=x'+nz$ for some $n \in \Z$. Since $\gcd(x, d)=1$, we have $\gcd(x^2N', d)=z$, and so there are $v, w \in \Z$ such that $vd+wx^2N'=nz$. 
Note that $x(1-wxN')=x'+vd$ and $\gcd(xx', d)=1$. Thus, we have
\begin{equation*}
\gcd(1-wxN', d)=\gcd(x(1-wxN'), d)=\gcd(x', d)=1.
\end{equation*}
Since $\gcd(1-wxN', xN')=1$, we have $\gcd(1-wxN', xN)=1$ and so
\begin{equation*}
\bect x d\overset{(2)}{=} \bect {x}{d-wxN} \overset{(3)}{=} \bect {(1-wxN')x}{d}\overset{(1)}{=} \bect {x'}{d}.
\end{equation*}
This completes the proof.
\end{proof}

\vspace{2mm}
It is often useful to fix a choice of representatives of the cusps of $X_0(N)$ (which is neither important nor harmful). For a divisor $d$ of $N$, we take a finite subset $R(N, d)$ of $\Z$ satisfying the following conditions:
\begin{itemize}[--]
\item
The number of elements of $R(N, d)$ is $\varphi(z)$, where $z=\gcd(d, N/d)$.
\item
Any elements of $R(N, d)$ are relatively prime to $d$.
\item
The elements of $R(N, d)$ are all distinct modulo $z$.
\end{itemize}
Then by Corollary \ref{corollary: cusp representation}, a cusp of $X_0(N)$ of level $d$ can be written as $\bect x d^N$ for some $x \in R(N, d)$, and so the set of cusps of $X_0(N)$ is equal to
\begin{equation*}
\{ \bect x d^N : 1\leq d \mid N \qqa x \in R(N, d) \}.
\end{equation*}

\vspace{2mm}
There is another notation for the cusps of $X_0(N)$, which is useful to figure out the actions of the Atkin--Lehner operators and the Hecke operators. Let $\cS(N)_\Q:=\Div_{\text{cusp}}(X_0(N)) \otimes_\Z \Q$, where $\Div_{\text{cusp}}(X_0(N))$ is the group of cuspidal divisors of $X_0(N)$. 
In other words, 
\begin{equation*}
\cS(N)_\Q:=\left\{ \sum_{1\leq d \mid N, \, x \in R(N, d)} a(x, d) \cdot \bect x d^N : a(x, d) \in \Q \right\}.
\end{equation*}

\begin{lemma}\label{lemma: cusp decomposition tensor}
Let $N=Mp^r$ with $\gcd(M, p)=1$. Then there is a canonical isomorphism 
\begin{equation*}
\iota: \cS(N)_\Q \simeq \cS(M)_\Q \motimes \cS(p^r)_\Q
\end{equation*}
sending a vector $\bect x {d}^N$ to a tensor $\bect x {d'}^M \motimes \bect x {p^f}^{p^r}$, where $d'=\gcd(d, M)$ and $p^f=\gcd(d, p^r)$.
\end{lemma}
\begin{proof}
The existence of the map $\iota$ is obvious (and canonical). 
Suppose that $\iota(\bect x d^N)=\iota(\bect y \delta^N)$, i.e., 
\begin{equation*}
\bect x {d'}^M \motimes \bect x {p^f}^{p^r}=\bect y {\delta'}^M \motimes \bect y {p^{f'}}^{p^r}.
\end{equation*}
Then by definition, we have $\bect x {d'}^M=\bect y {\delta'}^M$ and $\bect x {p^f}^{p^r}=\bect y {p^{f'}}^{p^r}$. By Theorem \ref{theorem: cusp representation}, we have $d'=\delta'$ and $f=f'$, and so $d=\delta$. Also, we have $x \equiv y \pmod {z'}$ and $x \equiv y \pmod {p^{m(f)}}$, where $z'=\gcd({d'}, M/{d'})$ and $m(f)=\min (f, \, r-f)$. Since $\gcd(M, p)=1$, we have $x \equiv y \pmod z$, where $z=z'\cdot p^{m(f)}$, which is equal to $\gcd(d, N/d)$. Thus, we have $\bect x {d}^N=\bect y {\delta}^N$ and hence the map $\iota$ is injective.
By the Chinese remainder theorem, the map $\iota$ is surjective. Indeed, if $\bect a {d'}^M$ and $\bect b {p^f}^{p^r}$ are cusps of $X_0(M)$ and $X_0(p^r)$, respectively, then we can find an integer $x$ such that $x \equiv a \pmod {d'}$ and $x \equiv b \pmod {p^f}$ because $\gcd({d'}, p^f)=1$. By its construction, we have $\gcd(x, {d'} p^f)=1$, and so
$\iota(\bect x {d'p^f}^N)=\bect {x}{{d'}}^M \motimes \bect x {p^f}^{p^r}$, as claimed.
\end{proof}
From now on, we identify a cusp $\bect x {{d'} p^f}^N$ with $\bect {x}{{d'}}^M \motimes \bect x {p^f}^{p^r}$ if we write $N=Mp^r$ with $\gcd(M, p)=1$. (Here, we allow the case of $r=0$, in which case we have $\bect {x} {{d'}}^M=\bect {x}{{d'}}^M \motimes \bect 1 1^1$.)

\ms
\subsection{Degeneracy maps}\label{section: degeneracy maps}
Let $N=Mp^r$ with $\gcd(M, p)=1$. Also, let ${d'}$ be a divisor of $M$ and $0\leq f \leq r+1$. Let $\alpha_p(N)$ and $\beta_p(N)$ be two degeneracy maps from $X_0(Np)_{\C}$ to $X_0(N)_{\C}$ defined by
\begin{equation*}
\alpha_p(N)(\tau)=\tau \pmodo {\Gamma_0(N)} \qa \beta_p(N)(\tau)=p\tau \pmodo {\Gamma_0(N)},
\end{equation*}
respectively. As already mentioned at the beginning of this section, if we take $A=Np$ and $B=N$, then $\alpha_p(N)$ (resp. $\beta_p(N)$) is the map $F_\gamma$ for $\gamma=\mat 1 0 0 1$ (resp. $\gamma=\mat p 0 0 1$).\footnote{We can easily check that $\gamma \Gamma_0(Np) \gamma^{-1} \subset \Gamma_0(N)$, or see Definition \ref{definition: degeneracy map pi(A, B)} below.}
Thus, by (\ref{equation: cusp image by usual matrix gamma}) we have
\begin{equation*}
\alpha_p(N)\left[\vect x {{d'} p^f} \modo {\Gamma_0(Np)} \right]=\left[\vect x {{d'} p^f} \modo {\Gamma_0(N)}\right]=\begin{cases}
\bect {x}{{d'} p^f} & \text{ if }  ~~f\leq r,\\
\bect {px+{d'}}{{d'} p^r} & \text{ if } ~~f=r+1,
\end{cases}
 \end{equation*}
 and
 \begin{equation*}
 \beta_p(N)\left[\vect x {{d'} p^f} \modo {\Gamma_0(Np)} \right]=\left[\vect {px/g}{{d'} p^f/g} \modo {\Gamma_0(N)}\right]=\begin{cases}
\bect {px}{{d'}} & \text{ if } ~~ f=0,\\
\bect {x}{{d'} p^{f-1}} & \text{ if }~~ f\geq 1,
\end{cases}
 \end{equation*}
where $g=\gcd(px, {d'} p^f)$ is a divisor of $p$. Equivalently, we have the following.
\begin{lemma}\label{lemma: degeneracy map on cusp}
We have
\begin{equation*}
\alpha_p(N)\left(\bect x {d'} ^M\motimes \bect x {p^f}^{p^{r+1}} \right) = \begin{cases}
\bect x {d'}^M \motimes \bect x {p^f} ^{p^r} & \text{ if } ~~f \leq r,\\
\bect {px} {d'}^M \motimes \bect {1} {p^r}^{p^r} & \text{ if }~~ f=r+1,
\end{cases}
\end{equation*}
and
\begin{equation*}
\beta_p(N) \left(\bect x {d'} ^M\motimes \bect x {p^f}^{p^{r+1}} \right) = \begin{cases}
\bect {px}{{d'}}^M \motimes \bect 1 1^{p^r} & \text{ if }~~ f=0, \\
\bect x {d'} ^M \motimes \bect x {p^{f-1}}^{p^r} & \text{ if }~~ f\geq 1.
\end{cases}
\end{equation*}
\end{lemma}

Also, we have the following. 
\begin{lemma}\label{lemma: ramification index of cusp}
The map $\alpha_p(N)$ is ramified at a cusp $\bect x {{d'} p^f}^{Np}$ if and only if  $0\leq f \leq r/2$. Also, the map $\beta_p(N)$ is ramified at a cusp $\bect x {{d'} p^f}^{Np}$ if and only if $r/2+1 \leq f \leq r+1$. In particular, the ramification indices of the maps $\alpha_p(N)$ and $\beta_p(N)$ depend only on $f$, neither on $x$ nor on ${d'}$. Furthermore, the ramification indices of $\alpha_p(N)$ and $\beta_p(N)$ at a cusp are either $1$ or $p$.
\end{lemma}
This is \cite[Lem. 2.1]{Yoo6}. Here, we provide a more direct proof using the theory of compact Riemann surfaces. Before proceeding, we compute the width of a cusp.

\begin{definition}
The \textit{width} of a cusp $\bect a d$ of level $d$ of $X_0(N)$ is defined as the smallest positive integer $n$ such that 
$\mat 1 n 0 1 \in \gamma^{-1}\Gamma_0(N) \gamma$,
where $\gamma \in \SL_2(\Z)$ satisfying $\gamma \vect 1 0 = \vect a d$.
\end{definition}

\begin{lemma}\label{lemma: width of cusp}
The width of a cusp $\bect a d$ of level $d$ of $X_0(N)$ is $\frac{N}{d \cdot \gcd(d, N/d)}$.
\end{lemma}
\begin{proof}
Since $\gcd(a, d)=1$, we can find integers $b$ and $c$ such that $ac-bd=1$. We take $\gamma=\mat a b d c \in \SL_2(\Z)$ so that $\gamma \vect 1 0 = \vect a d$. Since
\begin{equation*}
\mat 1 n 0 1 \in \gamma^{-1}\Gamma_0(N) \gamma \iff \gamma \mat 1 n 0 1 \gamma^{-1}=\mat * * {-d^2n}* \in \Gamma_0(N),
\end{equation*}
we must have $-d^2 n \equiv 0 \pmod N$. Note that $\gcd(d^2, N)=d\cdot \gcd(d, N/d)$, and hence the result follows.
\end{proof}

\begin{proof}[Proof of Lemma \ref{lemma: ramification index of cusp}]
We recall the definition of the ramification index of a holomorphic function between compact Riemann surfaces. Let $X$ and $Y$ be compact Riemann surfaces, and
let $\phi : X \to Y$ be a non-constant holomorphic function. Also,
let $q$ and $q'$ denote the local parameters of $\tau \in X$ and $\phi(\tau) \in Y$, respectively. 
If we write $\phi(q)=\sum_{n\geq 1} a(n) \cdot (q')^n$, then the ramification index of $\phi$ at a point $\tau$ is defined as the smallest positive integer $n$ such that $a(n)\neq 0$.

Now, we prove the lemma. Let $\phi=\alpha_p(N)$ or $\beta_p(N)$.
Also, let $h$ (resp. $h'$) be the width of a cusp $c$ (resp. $\phi(c)$). 
By the discussion in \cite[Sec. 2.4]{DS05}, the local parameter of a cusp $c$ (resp. $\phi(c)$) of $X_0(N)$ is equal to $e^{2\pi i \tau/h}$ (resp. $e^{2\pi i \tau/{h'}}$). Since
\begin{equation*}
\phi(e^{2\pi i \tau/h})=e^{2\pi i \phi(\tau)/h}=(e^{2\pi i \tau /{h'}})^{\frac{\phi(\tau)}{\tau} \times \frac{h'}{h}},
\end{equation*}
the ramification index of $\alpha_p(N)$ (resp. $\beta_p(N)$) at a cusp $c$ is $\frac{h'}{h}$ (resp. $\frac{ph'}{h}$).\footnote{The ramification index of $\alpha_p(N)$ is computed on \cite[pg. 67]{DS05} or \cite[pg. 538]{Ste85}.}
By Lemma \ref{lemma: width of cusp}, the width of a cusp $c=\bect x {{d'} p^f}^{Np}$ is $\frac{M}{{d'} \cdot \gcd({d'}, M/{d'})} \times p^{r+1-f-m(f)}$, where $m(f)=\min(f, \, r+1-f)$. So the result easily follows by Lemma \ref{lemma: degeneracy map on cusp}. 
\end{proof}

More generally, we can define various degeneracy maps as follows. 
We use the same notation as at the beginning of the section.
For $\gamma=\mat n 0 0 1 \in M_2(\Z)$, we have
\begin{equation*}
\gamma \mat a b c d \gamma^{-1}=\mat a {nb}{n^{-1}c} {d}.
\end{equation*}
Thus, if we take $n$ as a divisor of $A/B$, then we have $\gamma\Gamma_0(A)\gamma^{-1} \subset \Gamma_0(B)$, and hence the map $F_\gamma : X_0(A)_\C \to X_0(B)_\C$ is well-defined.
\begin{definition}\label{definition: degeneracy map pi(A, B)}
We denote by $\pi_1(A, B)$ (resp. $\pi_2(A, B)$) the map $F_\gamma$ from $X_0(A)$ to $X_0(B)$ induced by $\gamma=\mat 1 0 0 1$ (resp. $\gamma=\mat {A/B} 0 0 1$). By definition, for any integer $r\geq 1$ we have 
\begin{equation*}
\begin{split}
\pi_1(Mp^r, M)=\alpha_p(Mp^{r-1}) \circ \alpha_p(Mp^{r-2}) \circ \cdots \circ \alpha_p(Mp)\circ\alpha_p(M), \\
\pi_2(Mp^r, M)=\beta_p(Mp^{r-1}) \circ \beta_p(Mp^{r-2}) \circ \cdots \circ \beta_p(Mp)\circ\beta_p(M).\\
\end{split}
\end{equation*}
Furthermore, if we take $\gamma=\mat p 0 0 1$ when $A/B=p^2$, then we obtain a map 
\begin{equation*}
F_\gamma: X_0(Np^2) \to X_0(N),
\end{equation*}
denoted by $\pi_{12}(N)$, which is equal to
\begin{equation*}
\alpha_p(Np) \circ \beta_p(N)=\beta_p(Np) \circ \alpha_p(N).
\end{equation*}
\end{definition}

\ms
\subsection{Atkin--Lehner operators}\label{section: AL operators}
As above, let $N=Mp^r$ with $\gcd(M, p)=1$. Also, let ${d'}$ be a divisor of $M$ and $0\leq f \leq r$. In this subsection, we further assume that $p$ is a divisor of $N$, i.e., $r\geq 1$. Consider any matrix of the form:
\begin{equation*}
W_{p^r} = \bmat {m_1 p^r} {m_2} {m_3 Mp^r} {m_4 p^r} \text{ for some } m_i \in \Z ~\text{ with } ~m_1 m_4 p^r-m_2m_3 M=1.
\end{equation*}
One of the properties of such matrices is that they normalize the group $\Gamma_0(N)$, so if we take $A=B=N$ and $\gamma=W_{p^r}$ at the beginning of the section, then the map $F_\gamma$ is well-defined as an endomorphism of $X_0(N)_{\C}$, which is called the \textit{Atkin--Lehner operator with respect to $p$}, denoted by $w_p$. Note that this operator does not depend on the choice of a matrix $W_{p^r}$ because all such matrices are equivalent modulo $\Gamma_0(N)$, i.e., 
for any two such matrices $W$ and $W'$ there is a matrix $U \in \Gamma_0(N)$ such that $W=U \times W'$.

\begin{lemma}\label{lemma: AL action on cusps}
Let $c=\bect x {{d'} p^f}^N$ be a cusp of $X_0(N)$. Then we have
\begin{equation*}
w_p(c)=\bect x {d'}^M \motimes \bect {-x}{p^{r-f}}^{p^r}.
\end{equation*}
\end{lemma}
\begin{proof}
During the proof, we take $x \in \Z$ so that $\gcd(x, {d'} p)=1$ (Remark \ref{remark: cusp representative}). By (\ref{equation: cusp image by usual matrix gamma}), we have
\begin{equation*}
w_p(c)=\bect {(m_1xp^r+m_2{d'} p^f)/g}{(m_3xMp^r+m_4d'p^{r+f})/g}=\bect {up^f/g}{v{d'} p^r/g},
\end{equation*}
where $u=m_1 x p^{r-f}+m_2{d'}$, $v=m_3x(M/{d'})+m_4p^f$ and $g=\gcd(up^f, v{d'} p^r)$.

Since the operator $w_p$ does not depend on the choice of a matrix $W_{p^r}$, we may choose $m_3={d'}$ and so $v=Mx+m_4p^f$. Indeed, such a matrix $W_{p^r}$ exists because we can always find integers $m_1, m_2$ and $m_4$ such that $m_1m_4p^r-m_2{d'} M=1$ by our assumption $\gcd({d'} M, p)=1$. 
Since $m_4 p^f$ is relatively prime to $M$, we have $\gcd(v, M)=\gcd(Mx+m_4p^f, M)=1$. Also, since $Mx$ is relatively prime to $p$, we have $\gcd(v, p^r)=\gcd(Mx+m_4p^f, p^r)=1$. Thus, we have $\gcd(v, N)=1$. 
Since $g$ is a divisor of $\det(W_{p^r})=p^r$, and since $p$ does not divide $u$, we have $g=p^f$. Therefore we obtain $\gcd(u, v)=1$. Since $\gcd(v, N)=1$ as well, we have
\begin{equation*}
w_p(c)= \bect {up^f/g}{v{d'} p^r/g}=\bect {u}{v{d'} p^{r-f}} \overset{(3)}{=} \bect {uv}{{d'} p^{r-f}}=\bect {uv}{{d'}}^M \motimes \bect {uv} {p^{r-f}}^{p^r}.
\end{equation*}
Note that $u\equiv m_1 x p^{r-f} \pmod {d'}$ and $v \equiv m_4p^f \pmod {d'}$. Note also that $u\equiv m_2{d'} \pmod {p^{m(f)}}$ and $v \equiv Mx \pmod {p^{m(f)}}$, where $m(f)=\min(f, \, r-f)=m(r-f)$. 
Therefore by the determinant condition $m_1m_4p^r-m_2{d'} M=1$, we have
\begin{equation*}
uv \equiv m_1m_4p^r x\equiv x \pmodo {{d'}} \qa uv\equiv m_2{d'} Mx \equiv -x \pmodo {p^{m(f)}}.
\end{equation*}
This completes the proof.
\end{proof}
\begin{lemma}\label{lemma: determinant and gcd}
Let $M=\mat a b c d \in M_2(\Z)$ and $\vect x y \in (\Z^2)'$. Then the greatest common divisor of $ax+by$ and $cx+dy$ is a divisor of the determinant of $M$.
\end{lemma}
\begin{proof}
Since $\gcd(x, y)=1$, there are integers $u$ and $v$ such that $ux+vy=1$. So we have
\begin{equation*}
(ud-vc)(ax+by)+(-ub+va)(cx+dy)=(ux+vy)(ad-bc)=ad-bc,
\end{equation*}
which proves the result.
\end{proof}

\ms
\subsection{Hecke operators}\label{section: Hecke operators}
The degeneracy maps induce maps between the divisor groups. 
Using them, we define the \textit{$p$-th Hecke operator $T_p$} by the composition
\begin{equation*}
\xyv{0.7}
\xyh{4}
\xymatrix{
& \Div(X_0(Np)) \ar[rd]^-{\beta_p(N)_*} & \\
\Div(X_0(N)) \ar@{-->}[rr]_-{T_p} \ar[ru]^-{\alpha_p(N)^*}& & \Div(X_0(N)).
}
\end{equation*}
We also denote by $T_p$ the endomorphism of $J_0(N)$ induced by the restriction of the map $T_p$ to $\Div^0(X_0(N))$, which we also call the \textit{$p$-th Hecke operator}. 
\begin{lemma}\label{lemma: Hecke action on cusps}
As above, let $N=Mp^r$ with $\gcd(M, p)=1$.  Also, let ${d'}$ be a divisor of $M$ and $0\leq f \leq r$. For a prime $p$, let $p^* \in \Z$ be chosen so that $pp^* \equiv 1 \pmod M$ and $\gcd(p^*, p)=1$.
\begin{enumerate}
\item
Suppose that $r=0$. Then we have
\begin{equation*}
T_p(\bect x {d'}^M)=p\cdot \bect {px}{{d'}}^M + \bect {p^*x}{{d'}}^M.
\end{equation*}

\item
Suppose that $r\geq 1$. Then we have
\begin{equation*}
T_p(\bect x {{d'} p^f}^N)=\begin{cases}
p \cdot \bect {px}{{d'}}^M \motimes \bect 1 1^{p^r}& \text{ if } ~ f=0,\\
\sum_{i=1}^{p-1} \bect x {d'}^M \motimes \bect i {p^{r-1}}^{p^{r}}+\bect {p^*x}{{d'}}^M \motimes \bect 1 {p^{r}}^{p^{r}} & \text{ if }~ f=r,\\
\sum_{i=0}^{p-1} \bect x {d'} ^{M} \motimes \bect {x+ip^{r-f}} {p^{f-1}}^{p^r}
 & \text{ if } ~ (r+1)/2 < f \leq r-1,\\
p \cdot \bect x {d'}^M \motimes \bect x {p^{f-1}}^{p^r} & \text{ otherwise}.
\end{cases}
\end{equation*}
\end{enumerate}
\end{lemma}
\begin{proof}
During the proof, we frequently use Lemma \ref{lemma: degeneracy map on cusp}. 

First, suppose that $r=0$. By direct computation, we have
\begin{equation*}
\alpha_p(N)^*(\bect x {d'}^M)=p \cdot \bect x {d'}^M \motimes \bect 1 1^p+ \bect {p^*x}{{d'}}^M \motimes \bect 1 p^p,
\end{equation*}
which implies the first assertion. (Note that $\bect x {d'}^M=\bect x {d'}^M \motimes \bect 1 1^1$.)

Next, suppose that $r\geq 1$. If $0\leq f\leq r/2$, then we have
\begin{equation*}
\alpha_p(N)^*(\bect x {{d'}}^M \motimes \bect x {p^f}^{p^r})=p \cdot \bect x {{d'}}^M \motimes \bect x {p^f}^{p^{r+1}},
\end{equation*}
and so we have the result. If $r/2<f \leq r-1$, then we 
have $\gcd(p^f, p^{r+i-f})=p^{r+i-f}$ for $i=0$ or $1$. Thus, we have
\begin{equation*}
\alpha_p(N)^*(\bect x {{d'}}^M \motimes \bect x {p^f}^{p^r})=\sum_{i=0}^{p-1} \bect x {d'}^M \motimes \bect {x+ip^{r-f}}{p^f}^{p^{r+1}}.
\end{equation*}
If $f > (r+1)/2$, then $f-1\geq r/2$ and so $\gcd(p^{f-1}, p^{r-(f-1)})=p^{r-f+1}$.  Thus, $x+ip^{r-f}$ (for $0\leq i \leq p-1$) are all distinct modulo $\gcd(p^{f-1}, p^{r-(f-1)})$, and hence 
the result follows. 
If $f=(r+1)/2$, then we have $\gcd(p^{f-1}, p^{r-(f-1)})=p^{f-1}=p^{r-f}$, and therefore
\begin{equation*}
T_p(\bect x {{d'}}^M \motimes \bect x {p^f}^{p^r})=\sum_{i=0}^{p-1} \bect x {d'} ^M \motimes \bect{x+ip^{r-f}}{p^{f-1}}^{p^r}=p \cdot \bect x {d'}^M \motimes \bect x {p^{f-1}}^{p^r}.
\end{equation*}
Lastly, since $\bect i {p^r}^{p^r}=\bect 1 {p^r}^{p^r}$ for any $1\leq i \leq p-1$, we have
\begin{equation*}
\alpha_p(N)^*(\bect x {{d'}}^M \motimes \bect 1 {p^r}^{p^r})=\sum_{i=1}^{p-1} \bect x {d'}^M \motimes \bect i {p^r}^{p^{r+1}}+\bect {p^*x}{{d'}}^M \motimes \bect 1 {p^{r+1}}^{p^{r+1}}.
\end{equation*}
This completes the proof.
\end{proof}

\begin{lemma}\label{lemma: Tp equals from lower level}
As above, let $N=Mp^r$ with $\gcd(M, p)=1$. For any cuspidal divisor $D$ on $X_0(N)$, we have
\begin{equation*}
\alpha_p(N/p)^*\circ \beta_p(N/p)_*(D)=\begin{cases}
(T_p+w_p)(D) & \text{ if }~~ r=1, \\
T_p(D) & \text{ if }~~ r\geq 2.
\end{cases}
\end{equation*}
\end{lemma}
\begin{proof}
Let $c=\bect x {{d'} p^f}^N=\bect x {d'}^M \motimes \bect x {p^f}^{p^r}$ be a cusp of $X_0(N)$, where $x$ is an integer relatively prime to $N$, ${d'}$ is a divisor of $M$ and $0 \leq f \leq r$.  We claim that the formula holds for $D=c$, which easily proves the result by linearity.

First, suppose that $r=1$. Then we have $\beta_p(N/p)_*(c)=\bect {px'}{{d'}}^M \motimes \bect 1 1^1$, where $x'=(p^*)^f x$.
Thus, by Lemma \ref{lemma: degeneracy map on cusp} we have
\begin{equation*}
\begin{split}
\alpha_p(N/p)^*\circ \beta_p(N/p)_*(c)&=p\cdot \bect {px'}{{d'}}^M \motimes \bect 1 1^p+\bect {x'}{{d'}}^M \motimes \bect 1 p^p\\
&=p \cdot \bect {p^{1-f}x}{{d'}}^M \motimes \bect 1 1^p + \bect {(p^*)^f x}{{d'}}^M \motimes \bect 1 p^p.
\end{split}
\end{equation*}
By Lemmas  \ref{lemma: AL action on cusps} and \ref{lemma: Hecke action on cusps}, if $f=0$ then we have $w_p(c)=\bect x {d'}^M \motimes \bect 1 p^p$ and 
$T_p(c)=p\cdot \bect {px}{d'}^M \motimes \bect 1 1 ^p$.
Also, if $f=1$ then we have $w_p(c)=\bect x {d'}^M \motimes \bect 1 1^p$ and
\begin{equation*}
T_p(c)=(p-1)\cdot \bect x {d'}^M \motimes \bect 1 1^p + \bect {p^* x}{{d'}}^M \motimes \bect 1 p^p.
\end{equation*}
Thus, the claim follows.

Next, suppose that $r\geq 2$. If $f=0$, then we have
\begin{equation*}
\alpha_p(N/p)^*\circ \beta_p(N/p)_*(c)=\alpha_p(N/p)^*(\bect {px} {d'}^M \motimes \bect 1 1^{p^{r-1}})=p\cdot \bect {px} {d'}^M \motimes \bect 1 1^{p^r}.
\end{equation*}
If $1\leq f \leq r$, then $\beta_p(N/p)_*(c)=\bect x {d'}^M \motimes \bect x {p^{f-1}}^{p^{r-1}}$. 
Thus, the result follows by Lemma \ref{lemma: degeneracy map on cusp} and \ref{lemma: Hecke action on cusps}.
\end{proof}
\begin{remark}
In general, as an endomorphism on $J_0(N)$ we have
\begin{equation*}
\alpha_p(N/p)^*\circ \beta_p(N/p)_*=\begin{cases}
T_p+w_p & \text{ if } ~~r=1, \\
T_p & \text{ if } ~~r\geq 2
\end{cases}
\end{equation*}
(cf. \cite[(2.7)]{Yoo6}). The computation above verifies this formula for cuspidal divisors.
\end{remark}

\ms
\subsection{Rational cuspidal divisor group}\label{section: Rational cuspidal divisors}
In \cite[Th. 1.3.1]{Ste82}, Glenn Stevens computed the action of $\Gal(\ov{\Q}/\Q)$ on the cusps of the modular curve $X_\Gamma$ for a congruence subgroup $\Gamma$ containing $\Gamma(N)$. As a corollary, we have the following.
\begin{theorem}\label{theorem: Galois action on cusps}
A cusp $\bect x d^N$ of level $d$ is defined over $\Q(\mu_z)$, where $z=\gcd(d, N/d)$ and the action of $\Gal(\Q(\mu_z)/\Q)$ on the set of all cusps of level $d$ is simply transitive. 
\end{theorem}
\begin{proof}
For simplicity, let $G=\Gal(\Q(\mu_N)/\Q)$ and $H=\Gal(\Q(\mu_N)/\Q(\mu_z))$. Also, let $X_d$ be the set of all cusps of $X_0(N)$ of level $d$.

First, any cusps in $X_d$ are defined over $\Q(\mu_N)$ by Theorem 1.3.1(a) of \textit{loc. cit.}

Next, for any $k \in (\zmod N)^\times$ let $\tau_k$ be an element in $G$ sending $\zeta_N$ to $\zeta_N^k$, where $\zeta_N$ is a primitive $N$-th root of unity. By Theorem 1.3.1(b) of \textit{loc. cit.}, for any cusp $\bect x d \in X_d$, we have $\tau_k\left(\bect x d \right)=\bect x {k^*d}$, where $k^* \in \Z$ is chosen so that $kk^* \equiv 1 \pmod N$ and $\gcd(k^*, x)=1$. Since $\bect x {k^*d} \overset{(3)}{=} \bect {k^*x}{d}$ by Lemma \ref{lemma: three equivalences}, there is an action of $G$ on the set $X_d$. 

Then, for any cusp $\bect {x'} d \in X_d$, we can find $k \in \Z$ such that $x'k \equiv x \pmod d$ and $\gcd(k, N)=1$. 
Since $kk^* \equiv 1 \pmod N$, we have $k^*x \equiv kk^*x' \equiv x' \pmod d$. Thus, we have 
$\bect x d ^{\tau_k}=\bect {x'}{d}$ and hence the action of $G$ on $X_d$ is transitive. 

Finally, note that $\bect x d ^{\tau_k}=\bect {x}{d}$ if and only if $k^*x \equiv x \pmod z$, or equivalently $k\equiv 1 \pmod z$. Since $H$ is equal to $\{ \tau_k \in G : k \in (\zmod N)^\times \text{ with } k \equiv 1 \pmod z \}$, the action of $G$ on $X_d$ factors through $G/H \simeq \Gal(\Q(\mu_z)/\Q)$. Therefore a cusp $\bect x d \in X_d$ is defined over $\Q(\mu_z)$ and the action of $\Gal(\Q(\mu_z)/\Q)$ on $X_d$ is simply transitive, as claimed.
\end{proof}

\begin{definition}
Let $(P(N)_d)$ denote the divisor on $X_0(N)$ defined as the sum of all cusps of level $d$ (each with multiplicity one), i.e.,
\begin{equation*}
(P(N)_d):=\sum_{c \in X_d} c=\sum_{x \in R(N, d)} \bect x d^N,
\end{equation*}
where $R(N, d)$ is defined in Section \ref{section: representatives of the cusps}.
Also, let
\begin{equation*}
C(N)_d:=\varphi(\gcd(d, N/d))\cdot (P(N)_1)-(P(N)_d).
\end{equation*}
If there is no confusion, we simply write $(P_d)$ and $C_d$.
\end{definition}

\begin{lemma}\label{lemma: group of rational cuspidal divisors}
We have
\begin{equation*}
 \Qdivv N =\br{ (P_d) : d \in \cD_N}:= \left \{ \sum_{d \in \cD_N} a_d \cdot (P_d) : a_d \in \Z \right \}
\end{equation*}
and 
\begin{equation*}
\Qdiv N=\br{ C_d : d \in \cD_N^0 }:=\left \{ \sum_{d\in \cD_N^0} a_d \cdot C_d : a_d \in \Z \right\}.
\end{equation*}
\end{lemma}
\begin{proof}
By Theorem \ref{theorem: Galois action on cusps}, $(P_d)$ is a single orbit of $\Gal(\ov{\Q}/\Q)$. Thus, the first assertion follows. Since the degree of $(P_d)$ is equal to the number of the cusps of $X_0(N)$ of level $d$, which is $\varphi(\gcd(d, N/d))$, we have
\begin{equation*}
\Qdiv N = \left \{ \sum_{d\in \cD_N} a_d \cdot (P_d) : a_d \in \Z, \sum_{d \in \cD_N} a_d \cdot \varphi(\gcd(d, N/d))=0 \right\}.
\end{equation*}
Since 
\begin{equation*}
\sum_{d \in \cD_N} a_d \cdot (P_d) = - \sum_{d \in \cD_N^0} a_d \cdot C_d + \sum_{d \in \cD_N} a_d \cdot \varphi(\gcd(d, N/d)) \cdot (P_1),
\end{equation*}
the second assertion follows.
\end{proof}

\begin{remark}\label{remark: decomposition of RCD}
As already introduced in Notation \ref{notation: S1 and S2 in introduction}, there is a tautological isomorphism 
\begin{equation*}
\xymatrix{
\Phi_N : \Qdivv N \ar[r]^-{\sim} & \cS_2(N)
}
\end{equation*}
sending $(P_d)$ to ${\bf e}(N)_d$.
As above, let $N=Mp^r$ with $\gcd(M, p)=1$. Also, let ${d'}$ be a divisor of $M$.
Using the identification $\bect {x}{{d'} p^f}^N=\bect x {d'}^M \motimes \bect x {p^f}^{p^r}$ and an isomorphism 
$(\zmod N)^\times \simeq (\zmod M)^\times \times (\zmod {p^r})^\times$, we easily have $(P(N)_{{d'} p^f}) = (P(M)_{d'}) \motimes (P(p^r)_{p^f})$. Thus, we also identify $\cS_2(N)_\Q$ with $\cS_2(M)_\Q \motimes  \cS_2(p^r)_\Q$ by letting ${\bf e}(N)_{{d'} p^f}={\bf e}(M)_{d'} \motimes {\bf e}(p^r)_{p^f}$.
Similarly, if $N=\prod_{i=1}^t p_i^{r_i}$ is the prime factorization, then  we identify $\cS_k(N)_\Q$ with $\motimes_{i=1}^t \cS_k(p_i^{r_i})_\Q$ for both $k=1$ and $k=2$. 
\end{remark}

\ms
\subsection{The actions of various operators on $\Qdivv N$}\label{section: actions on RCD}
As above, let $N=Mp^r$ with $\gcd(M, p)=1$. Also, let ${d'}$ be a divisor of $M$ and
\begin{equation*}
S(k):=\{y \in \Z : 1\leq y \leq p^k \qa \gcd(y, \, p)=1 \}.
\end{equation*}
Note that since $p$ does not divide $M$ we have
\begin{equation}\label{equation: multi by p is automorphism}
(P(M)_{d'})=\sum_{x \in R(M, {d'})} \bect x {d'}^M=\sum_{x \in R(M, {d'})} \bect {px} {d'}^M=\sum_{x \in R(M, {d'})} \bect {p^*x} {d'}^M,
\end{equation}
where $p^*$ is an integer such that $pp^* \equiv 1 \pmod M$ and $\gcd(p^*, p)=1$. 

\begin{lemma}\label{lemma: degeneracy maps pushforward on RCD}
For any $0\leq f\leq r+1$, we have
\begin{equation*}
\Phi_N(\alpha_p(N)_*(P(Np)_{{d'} p^f}))=\Phi_M(P(M)_{d'}) \motimes \Phi_{p^{r}}(A)
\end{equation*}
and
\begin{equation*}
\Phi_N(\beta_p(N)_*(P(Np)_{{d'} p^f}))=\Phi_M(P(M)_{d'}) \motimes \Phi_{p^{r}}(B),
\end{equation*}
where
\begin{equation*}
A=\begin{cases}
(P(p^r)_{p^f}) & \text{ if } ~~0\leq f \leq r/2,\\
p\cdot (P(p^r)_{p^f}) & \text{ if } ~~r/2 < f \leq r-1,\\
(p-1)\cdot (P(p^r)_{p^r}) & \text{ if } ~~f=r \text{ and }  r\geq 1, \\
(P(p^r)_{p^r}) & \text{ if } ~~f=r+1, \\
\end{cases}\\
\end{equation*}
and
\begin{equation*}
B=\begin{cases}
(P(p^r)_1) & \text{ if } ~~f=0,\\
(p-1) \cdot (P(p^r)_1) & \text{ if }~~  f=1 \text{ and } r\geq 1,\\
p\cdot (P(p^r)_{p^{f-1}}) & \text{ if } ~~2\leq f <r/2+1,\\
(P(p^r)_{p^{f-1}}) & \text{ if } ~~r/2+1\leq f\leq r+1.
\end{cases}\\
\end{equation*}
\end{lemma}
\begin{proof}
For simplicity, let $D=(P(Np)_{{d'} p^f})$. Let $k=\min(f, \, r+1-f)$. Then we have
\small
\begin{equation*}
D=\sum_{x \in R(M, {d'})} \sum_{y\in S(k)} \bect x {d'}^M \motimes \bect y {p^f}^{p^{r+1}}=\left(\sum_{x \in R(M, {d'})}  \bect x {d'}^M \right) \motimes \left(\sum_{y\in S(k)} \bect y {p^f}^{p^{r+1}} \right).
\end{equation*}\normalsize

Suppose first that $f=r+1$. Then $S(k)=\{1\}$, and by (\ref{equation: multi by p is automorphism}) we have
\begin{equation*}
\alpha_p(N)_*(D)=\sum_{x\in R(M, {d'})} \bect {px} {d'}^M \motimes \bect 1 {p^r}^{p^r}=(P(M)_{d'}) \motimes (P(p^r)_{p^r}).
\end{equation*}

Suppose next that $f\leq r$. Then we have  
\begin{equation*}
\alpha_p(N)_*(D)=\sum_{x \in R(M, {d'})} \sum_{y\in S(k)} \bect x {d'}^M \motimes \bect y {p^f}^{p^{r}}=(P(M)_{d'})\motimes \left(\sum_{y\in S(k)} \bect y {p^f}^{p^{r}}\right).
\end{equation*}
If $0\leq f\leq r/2$, then we have $\min(f, \, r-f)=k$ and $(P(p^r)_{p^f})=\sum_{y \in S(k)} \bect y {p^f}^{p^r}$. If $f=r\geq 1$, then $\# S(k)=p-1$, and 
 we have $\bect y {p^r}^{p^r}=\bect 1 {p^r}^{p^r}=(P(p^r)_{p^r})$ for any $y \in S(k)$. If $r/2<f \leq r-1$, then $\min(f, \, r-f)=k-1$ and we have
\begin{equation*}
\sum_{y \in S(k)} \bect y {p^f}^{p^r}=\sum_{y\in S(k-1)} \sum_{i=0}^{p-1} \bect {y+ip^{k-1}} {p^f}^{p^r}=p\sum_{y \in S(k-1)} \bect y {p^f}^{p^r}=p\cdot (P(p^r)_{p^f}).
\end{equation*}
Thus, we obtain the result for $\alpha_p(N)_*(D)$.

The proof for $\beta_p(N)_*$ is similar and we leave the details to the readers.
\end{proof}

\begin{lemma}\label{lemma: degeneracy maps pullback on RCD}
For any $0\leq f\leq r$, we have
\begin{equation*}
\Phi_{Np}(\alpha_p(N)^*(P(N)_{{d'} p^f}))=\Phi_M(P(M)_{d'})\motimes \Phi_{p^{r+1}}(A)
\end{equation*}
and 
\begin{equation*}
\Phi_{Np}(\beta_p(N)^*(P(N)_{{d'} p^f}))=\Phi_M(P(M)_{d'})\motimes \Phi_{p^{r+1}}(B),
\end{equation*}
where
\begin{equation*}
A=\begin{cases}
p\cdot (P(p)_1)+(P(p)_p) & \text{ if }~~ f=r=0,\\
p\cdot (P(p^{r+1})_{p^f}) & \text{ if }~~ 0\leq f \leq r/2 \text{ and } r\geq 1,\\
(P(p^{r+1})_{p^f}) & \text{ if }~~r/2<f \leq r-1, \\
(P(p^{r+1})_{p^r})+(P(p^{r+1})_{p^{r+1}}) & \text{ if } ~~f=r\geq 1,\\
\end{cases} 
\end{equation*}
and
\begin{equation*}
B=\begin{cases}
(P(p)_1)+p\cdot (P(p)_p) & \text{ if } ~~f=r=0,\\
(P(p^{r+1})_1)+(P(p^{r+1})_p) & \text{ if } ~~f=0 \text{ and } r\geq 1, \\
(P(p^{r+1})_{p^{f+1}}) & \text{ if }~~1\leq f <r/2,\\
p\cdot (P(p^{r+1})_{p^{f+1}}) & \text{ if }~~ r/2\leq f \leq r \text{ and } r\geq 1.
\end{cases}
\end{equation*}
\end{lemma}
\begin{proof}
The result follows by Lemmas \ref{lemma: degeneracy map on cusp} and \ref{lemma: ramification index of cusp}.
More specifically, suppose that $f=r=0$. Then we have
\begin{equation*}
\alpha_p(N)^*(\bect {x} {{d'}}^M \motimes \bect 1 1^1) = p\cdot \bect {x} {{d'}}^M \motimes \bect 1 1^p + \bect {p^*x} {{d'}}^M \motimes \bect 1 p^p.
\end{equation*}
By (\ref{equation: multi by p is automorphism}), the result follows. If $f=r\geq 1$, then 
\begin{equation*}
\alpha_p(N)^*(\bect {x} {{d'}}^M \motimes \bect 1 {p^r}^{p^r}) = \sum_{i=1}^{p-1} \bect {x} {{d'}}^M \motimes \bect i {p^r}^{p^{r+1}} + \bect {p^*x} {{d'}}^M \motimes \bect 1 {p^{r+1}}^{p^{r+1}}.
\end{equation*}
Also, if $0\leq f\leq r-1$, then 
\begin{equation*}
\alpha_p(N)^*(\bect x {d'}^M \motimes \bect y {p^f}^{p^r})=\begin{cases}
p\cdot \bect x {d'}^M \motimes \bect y {p^f}^{p^{r+1}} & \text{ if } ~~0\leq f\leq r/2,\\
\sum_{i=0}^{p-1} \bect x {d'}^M \motimes \bect {y+ip^{r-f}}{p^f}^{p^{r+1}} & \text{ if }~~r/2<f\leq r-1.
\end{cases}
\end{equation*}
If $0\leq f\leq r/2$, then $\min(f, \, r+1-f)=f$ and hence 
\begin{equation*}
(P(p^{r+1})_{p^f})=\sum_{y \in S(f)} \bect y {p^f}^{p^{r+1}}.
\end{equation*}
If $r/2<f \leq r$, then $\min(f, \, r+1-f)=r+1-f=\min(f, \, r-f)+1$ and therefore
\begin{equation*}
(P(p^{r+1})_{p^f})=\sum_{y \in S(r+1-f)} \bect y {p^f}^{p^{r+1}}=\sum_{y \in S(r-f)} \sum_{i=0}^{p-1} \bect {y+ip^{r-f}}{p^f}^{p^{r+1}}.
\end{equation*}
Thus, we obtain the result for $\alpha_p(N)^*(D)$.

The proof for $\beta_p(N)^*$ is similar and we leave the details to the readers.
\end{proof}

\begin{lemma}\label{lemma: Hecke and AL operators on RCD}
If $r=0$ then we have
\begin{equation*}
\Phi_M(T_p(P(M)_{{d'}}))=(p+1) \cdot \Phi_M(P(M)_{d'}).
\end{equation*}
Suppose that $r\geq 1$. Then for any $0\leq f \leq r$, we have
\begin{equation*}
\Phi_N(w_p(P(N)_{{d'} p^f}))=\Phi_M(P(M)_{d'}) \motimes \Phi_{p^r}(P(p^r)_{p^{r-f}})
\end{equation*}
and
\begin{equation*}
\Phi_N(T_p(P(N)_{{d'} p^f}))=\Phi_M(P(M)_{d'}) \motimes \Phi_{p^r}(A),
\end{equation*}
where
\begin{equation*}
A=\begin{cases}
p \cdot (P(p^r)_1)& \text{ if }~~ f=0,\\
p(p-1) \cdot (P(p^r)_1) & \text{ if } ~~ f=1 \text{ and } r\geq 2,\\
p^2 \cdot (P(p^r)_{p^{f-1}})& \text{ if }~~ 2\leq f\leq r/2,\\
p\cdot (P(p^r)_{p^{f-1}}) & \text{ if }~~ f=(r+1)/2,\\
(P(p^r)_{p^{f-1}}) & \text{ if }~~(r+1)/2<f\leq r-1,\\
(p-1) \cdot (P(p)_1)+(P(p)_p) & \text{ if } ~~f=r \text{ and } r=1,\\
(P(p^r)_{p^{r-1}})+(P(p^r)_{p^r}) & \text{ if } ~~f=r \text{ and } r\geq 2.\\
\end{cases}
\end{equation*}
\end{lemma}
\begin{proof}
For simplicity, let $D=(P(N)_{{d'} p^f})$. Then we have
\begin{equation*}
D=\sum_{x \in R(M, {d'})}\sum_{y \in S(k)} \bect x {d'}^M \motimes \bect y {p^f}^{p^r}=\left(\sum_{x \in R(M, {d'})} \bect x {d'}^M \right)\motimes \left( \sum_{y \in S(k)} \bect y {p^f}^{p^r}\right),
\end{equation*}
where $k=\min(f, \, r-f)$. Thus, by Lemma \ref{lemma: AL action on cusps} we have
\begin{equation*}
w_p(D)=\left(\sum_{x \in R(M, {d'})} \bect x {d'}^M \right)\motimes \left( \sum_{y \in S(k)} \bect {-y} {p^{r-f}}^{p^r}\right)=(P(M)_{d'}) \motimes (P(p^r)_{p^{r-f}}).
\end{equation*}
Also, by Lemma \ref{lemma: Hecke action on cusps} and (\ref{equation: multi by p is automorphism}), we have the following: If $f=r=0$, then 
\begin{equation*}
T_p(D)=\sum_{x \in R(M, {d'})} \left(p\cdot \bect {px}{d'}^M +\bect {p^* x}{{d'}}\right)=(p+1) \cdot (P(M)_{d'}).
\end{equation*}
Suppose that $r\geq 1$. If $f=0$, then 
\begin{equation*}
T_p(D)=\sum_{x \in R(M, {d'})} p\cdot \bect {px}{{d'}}^M \motimes \bect 1 1^{p^r}=p \cdot (P(M)_{d'}) \motimes (P(p^r)_1).
\end{equation*}

If $f=r\geq 1$, then we have
\begin{equation*}
\begin{split}
T_p(D)&=\sum_{x \in R(M, {d'})} \left(\sum_{i=1}^{p-1} \bect x {d'}^M \motimes \bect i {p^{r-1}}^{p^r}+\bect {p^*x}{{d'}}^M \motimes \bect 1 {p^r}^{p^r} \right)\\
&=(P(M)_{d'}) \motimes \left(\sum_{i=1}^{p-1}\bect i {p^{r-1}}^{p^r} \right)+(P(M)_{d'})\motimes (P(p^r)_{p^r}).
\end{split}
\end{equation*}
If $r=1$, then $\bect i {p^{r-1}}^{p^r}=\bect 1 1^{p}=(P(p)_1)$ for any $1\leq i\leq p-1$, and so the formula holds. If $r\geq 2$, then $(P(p^r)_{p^{r-1}})=\sum_{i=1}^{p-1}\bect i {p^{r-1}}^{p^r}$ and so the result follows. 

Now, suppose that $1\leq f\leq r-1$. If $1\leq f \leq (r+1)/2$, then 
\small
\begin{equation*}
T_p(D)=\sum_{x \in R(M, {d'})}\sum_{y \in S(k)} p \cdot \bect x {d'}^M \motimes \bect y {p^{f-1}}^{p^r}=p \cdot (P(M)_{d'}) \motimes \left(\sum_{y \in S(k)} \bect y {p^{f-1}}^{p^r}\right).
\end{equation*}\normalsize
By direct computation, we have
\begin{equation*}
\sum_{y \in S(k)} \bect y {p^{f-1}}^{p^r}=\begin{cases}
(p-1)\cdot (P(p^r)_1) & \text{ if }~~f=1 ~(\text{and }r\geq 2),\\
p \cdot (P(P^r)_{p^{f-1}}) & \text{ if } ~~ 2\leq f \leq r/2,\\
(P(p^r)_{p^{f-1}}) & \text{ if } ~~ f=(r+1)/2.
\end{cases}
\end{equation*}
Thus, the result follows. If $(r+1)/2<f \leq r-1$, then we have
\begin{equation*}
\begin{split}
T_p(D)&=\sum_{x \in R(M, {d'})}\sum_{y \in S(k)} \bect x {d'}^M \motimes \sum_{i=0}^{p-1} \bect {y+ip^{r-f}} {p^{f-1}}^{p^r}\\
&=\sum_{x \in R(M, {d'})}\bect x {d'}^M \motimes \sum_{y \in S(k)} \sum_{i=0}^{p-1} \bect {y+ip^{r-f}} {p^{f-1}}^{p^r}=(P(M)_{d'})\motimes (P(p^r)_{p^{f-1}})
\end{split}
\end{equation*}
because any element of $S(r-f+1)$ can be uniquely written as $y+ip^{r-f}$ for some $y \in S(k)=S(r-f)$ and $0\leq i \leq p-1$. This completes the proof.
\end{proof}

\vspace{2mm}
The following will be used later.
\begin{lemma}\label{lemma: A1 definition}
For any $r\geq 0$, we have
\begin{equation*}
\pi_1(p^r, 1)^*(P(1)_1)=\sum_{k=0}^r p^{\max(r-2k, \, 0)} \cdot (P(p^r)_{p^k}).
\end{equation*}
\end{lemma}
\begin{proof}
First, we have 
\begin{equation*}
\pi_1(p, 1)^*(P(1)_1)=\alpha_p(1)^*(P(1)_1)=p\cdot (P(p)_1)+(P(p)_p).
\end{equation*}
Next, suppose that the above formula holds for $r\geq 1$, i.e.,
\begin{equation*}
\pi_1(p^{r}, 1)^*(P(1)_1)=\sum_{k=0}^{[r/2]} p^{r-2k} \cdot (P(p^{r})_{p^k})+\sum_{k=[r/2]+1}^r  (P(p^r)_{p^k}).
\end{equation*}
Then by Lemma \ref{lemma: degeneracy maps pullback on RCD}, we have
\begin{equation*}
\begin{split}
\pi_1(p^{r+1}, 1)^*(P(1)_1)&=\alpha_p(p^r)^* \circ \pi_1(p^r, 1)^*(P(1)_1) \\
&=\alpha_p(p^{r})^*\left(\sum_{k=0}^{[r/2]} p^{r-2k} \cdot (P(p^{r})_{p^k})+\sum_{k=[r/2]+1}^r  (P(p^r)_{p^k}) \right)\\
&=\sum_{k=0}^{[r/2]} p^{r+1-2k} \cdot (P(p^{r+1})_{p^k})+\sum_{k=[r/2]+1}^{r+1} (P(p^{r+1})_{p^k}).
\end{split}
\end{equation*}
By induction we obtain the result.
\end{proof}

\begin{remark}\label{remark: level 1, level M}
Suppose that $M$ is a prime not divisible by $p$. As above, for any $D \in \Qdivv M$, we easily have
\begin{equation*}
\Phi_{Mp^r}(\pi_1(Mp^r, M)^*(D))=\Phi_M(D) \motimes \Phi_{p^r}(\pi_1(p^r, 1)^*(P(1)_1)).
\end{equation*}
\end{remark}

\begin{lemma}\label{lemma: B2 definition}
We have 
\begin{equation*}
\pi_{12}(1)^*(P(1)_1)=p\cdot [(P(p^2)_1)+(P(p^2)_p)+(P(P^2)_{p^2})].
\end{equation*}
If $r\geq 1$, then  we have
\begin{equation*}
\pi_{12}(p^r)^*(P(p^r)_{p^f})=\begin{cases}
p\cdot [(P(p^{r+2})_1)+(P(p^{r+2})_p)] & \text{ if }~~ f=0,\\
p\cdot (P(p^{r+2})_{p^{f+1}}) & \text{ if }~~ 1\leq f \leq r-1,\\
p\cdot [(P(p^{r+2})_{p^{r+1}})+(P(p^{r+2})_{p^{r+2}})] & \text{ if }~~ f=r.
\end{cases}
\end{equation*}
Therefore the map $\frac{1}{p}\times \pi_{12}(p^r)^*$ is well-defined as 
a linear map from the group of rational cuspidal divisors on $X_0(p^r)$ to that on $X_0(p^{r+2})$. 
\end{lemma}
\begin{proof}
We may prove the theorem by direct computation using
\begin{equation*}
\pi_{12}(p^r) \left( \bect x {p^f}^{p^r}\right)=\begin{cases}
\bect {px} 1 ^{p^{r+2}} & \text{ if } ~~f=0,\\
\bect x {p^{f-1}}^{p^{r+2}} & \text{ if } ~~1\leq f\leq r.
\end{cases}
\end{equation*}
But instead, we note that 
\begin{equation*}
\pi_{12}(p^r)=\alpha_p(p^r)\circ \beta_p(p^{r+1}) = \beta_p(p^r) \circ \alpha_p(p^{r+1}),
\end{equation*}
and therefore
\begin{equation*}
\pi_{12}(p^r)^*=\alpha_p(p^{r+1})^* \circ \beta_p(p^r)^*=\beta_p(p^{r+1})^* \circ \alpha_p(p^r)^*.
\end{equation*}

First, suppose that $r=f=0$. Then $\alpha_p(1)^*(P(1)_1)=p\cdot (P(p)_1)+(P(p)_p)$. Thus, we have
\begin{equation*}
\pi_{12}(1)^*(P(1)_1)=p\cdot [(P(p^2)_1)+(P(p^2)_p)]+p\cdot (P(p^2)_{p^2}).
\end{equation*}

Next, suppose that $r\geq 1$. Suppose further that $0\leq f\leq r/2$. Then we have
\begin{equation*}
\alpha_p(p^r)^*(P(p^r)_{p^f})=p\cdot (P(p^{r+1})_{p^f}).
\end{equation*}
Thus, we have
\begin{equation*}
\pi_{12}(p^r)^*(P(p^r)_{p^f})=\begin{cases}
p \cdot [(P(p^{r+2})_{1})+(P(p^{r+2})_{p})] & \text{ if } ~~f=0,\\
p \cdot (P(p^{r+2})_{p^{f+1}}) & \text{ otherwise}.
\end{cases}
\end{equation*}
If $(r+1)/2 \leq  f\leq r$, then we have $\beta_p(p^r)^*(P(p^r)_{p^f})=p\cdot (P(p^{r+1})_{p^{f+1}})$, and hence
\begin{equation*}
\pi_{12}(p^r)^*(P(p^r)_{p^f})=\begin{cases}
p\cdot [(P(p^{r+2})_{p^{r+1}})+(P(p^{r+2})_{p^{r+2}})] & \text{ if }~~ f=r,\\
p\cdot (P(p^{r+2})_{p^{f+1}}) & \text{ otherwise}.
\end{cases}
\end{equation*}

Lastly, since the coefficient of $(P(p^{r+2})_{p^k})$ in $\pi_{12}(p^r)^*(P(p^r)_{p^f})$ is divisible by $p$ for any $f$ and $k$, the last assertion follows.
\end{proof}

\vspace{10mm}
\section{The order of a rational cuspidal divisor}\label{chapter3}
For a degree $0$ divisor $D$ on $X_0(N)$, let $\ov{D}$ denote the linear equivalence class of $D$ in $J_0(N)$.
By \textit{the order of $D$}, we mean the order of $\ov{D}$ in $J_0(N)$, i.e., the smallest positive integer $n$ such that $n \cdot D$ is equal to the divisor of a meromorphic function on $X_0(N)$. (If there does not exist such an integer, then we say that the order is \textit{infinite}.)
In this section, we develop a method for computing the order of a degree $0$ rational cuspidal divisor on $X_0(N)$, which is a slight elaboration of Ligozat's method (Section \ref{section: algorithm}). As an application, we compute the order of $C_d$ for any non-trivial divisor $d$ of $N$ (Sections \ref{section: Example II: The order of CN} and \ref{section: Example II: The order of Cd}).

\vspace{2mm}
Before proceeding, we discuss a way to compute the order of a degree $0$ cuspidal divisor on $X_0(N)$, which may not be rational.
Let $D$ be a degree $0$ cuspidal divisor on $X_0(N)$. By Manin \cite{Ma72} and Drinfeld \cite{Dr73}, the order of $D$ is finite. In other words, there is a modular function\footnote{By a \textit{modular function} on $X_0(N)$, we mean a meromorphic function on $\cH^*$ invariant under the action of $\Gamma_0(N)$.} on $X_0(N)$ whose divisor is an integral multiple of $D$.
Such a function has no zeros and poles on $\cH$, and hence is called a \textit{modular unit}. 
Modular units for the modular curves $X(N)$ or $X_1(N)$ have been studied by various mathematicians (most notably Kubert and Lang \cite{KL}), but those for the modular curve $X_0(N)$ have not much studied before (unless $N$ is squarefree). In principle, if we know a precise description\footnote{Since modular units on $X_0(N)$ are also modular units on $X(N)$, they can be written in terms of Siegel's units. However, we need a precise description of which Siegel's units are invariant under the action of a more larger group $\Gamma_0(N)$. For some progress on modular units on $X_0(N)$, see \cite{GYYY}.} of all the modular units on $X_0(N)$, then we can compute the order of any degree $0$ cuspidal divisor on $X_0(N)$.
However, until now, there is no systematic way to compute the order of a non-rational cuspidal divisor except a method using modular symbols.

\ms
Now, let $D$ be a degree $0$ rational cuspidal divisor on $X_0(N)$. If $n$ is the order of $D$, then there is a modular function $F$ on $X_0(N)$ such that
\begin{equation*}
\tn{div} (F)=n \cdot D.
\end{equation*}
Such a function $F$ has the following properties:
\begin{enumerate}
\item
It has no zeros and poles on $\cH$.
\item
Its order of vanishing at a cusp $\bect x d^N$ of level $d$ does not depend on $x$.
\end{enumerate}
In the 1950s, Morris Newman constructed such functions using the Dedekind eta function \cite{Ne57, Ne59}. Now, they are called \textit{eta quotients} (or \textit{eta products}, depending on author's preference, cf. \cite[pg. 31]{Ko11}). 
In fact, he found a sufficient condition when an eta quotient is a modular function on $X_0(N)$. Also, in the early 1970s Andrew Ogg proved a necessary and sufficient condition for eta quotients to be modular functions on $X_0(N)$ when $N$ is either a prime or the product of two primes, and he computed the order of a degree $0$ (rational) cuspidal divisors \cite{Og73, Og74}.
Lastly, in 1975 Gerard Ligozat proved a necessary and sufficient condition for an eta quotient to be a modular function on $X_0(N)$ for any positive integer $N$ (Proposition \ref{proposition: criteria for rational modular unit}). As an application,  he computed the order of the divisor $(0)-(\infty)$ on $X_0(N)$ for any positive integer $N$ \cite[Th. 3.2.16]{Li75}. 

In 1997, Toshikazu Takagi proved that all modular units on $X_0(N)$ are eta quotients (up to constant) when $N$ is squarefree \cite[Th. 3.3]{Ta97}. We remark that this result can be obtained by the work of Ligozat. More precisely, Ligozat's result is enough to prove that any modular unit on $X_0(N)$ such that its order of vanishing at a cusp $\bect x d^N$ of level $d$ does not depend on $x$ is an eta quotient up to constant (Theorem \ref{theorem: rational modular units}). Since all cusps of $X_0(N)$ are defined over $\Q$ when $N$ is of the form $2^rM$, where $M$ is odd squarefree and $r\leq 3$, all modular units on $X_0(N)$ for such an $N$ are indeed eta quotients (up to constant).

\ms
\subsection{Eta quotients}\label{section: eta quotients}
Let $\eta : \cH \to \C$ be the Dedekind eta function defined by 
\begin{equation*}
\eta(\tau)=e^{\frac{\pi i \tau}{12}}\prod_{n=1}^{\infty}(1-e^{2\pi i n \tau})=q^{\frac{1}{24}} \prod_{n=1}^{\infty}(1-q^n), \text{ where }~~ q=e^{2\pi i \tau}.
\end{equation*}
It is well-known that for any $\gamma=\mat a b c d \in \SL_2(\Z)$ with $c\geq 0$, we have
\begin{equation*}
\eta(\gamma \tau)=\eta\left(\frac{a\tau+b}{c\tau+d}\right)=\epsilon(a, b, c, d)\sqrt{-i}(c\tau+d)^{\frac{1}{2}}\eta(\tau),
\end{equation*}
where $\epsilon(a, b, c, d)$\textemdash see \cite[Sec. 3.1]{Li75} for its definition\textemdash  is a certain $24$-th root of unity. One way to ignore this root of unity is to take its $24$-th power. The function $\Delta(\tau):=\eta(\tau)^{24}$ is then invariant under the action of $\SL_2(\Z)$, and so it is a modular form of weight $12$ for $\SL_2(\Z)$. 
As mentioned at the beginning of Section \ref{chapter2}, for any divisor $\delta$ of $N$, we have the map $F_\gamma : X_0(N) \to X_0(1)$, where $\gamma=\mat \delta 0 0 1$. Thus, the function
\begin{equation*}
\Delta_\delta(\tau):=F_\gamma^*(\Delta)(\tau)=\Delta(\delta \tau)
\end{equation*}
is a modular form of weight $12$ for $\Gamma_0(N)$, and so the ratio of such modular forms  may be used to construct a modular function on $X_0(N)$. 
Likewise, we define
\begin{equation*}
\eta_\delta(\tau):=F_\gamma^*(\eta)=\eta(\delta \tau)
\end{equation*}
and  consider the following.

\begin{definition}
A function $g: \cH \to \C$ is called an \textit{eta quotient of level $N$} if it is of the form $g=\prod_{\delta \mid N} \eta_\delta^{r_\delta}$ for some $r_\delta \in \Z$.
\end{definition}

As in Notation \ref{notation: S1 and S2 in introduction}, we consider the $\Q$-vector space $\cS_1(N)_\Q$ and for any ${\bf r}=\sum_{\delta \mid N} r_\delta \cdot {\bf e}(N)_\delta \in \cS_1(N)_\Q$, we define a \textit{generalized eta quotient of level $N$} by
\begin{equation*}
g_{\bf r}:=\prod_{\delta\mid N} \eta_\delta^{r_\delta},
\end{equation*}
which is regarded as a power series in $q$ with rational coefficients (after multiplying suitable rational power of $q$ if necessary). Our interest is to understand when such a product is a modular function on $X_0(N)$. 
\begin{lemma}\label{lemma: the order of vanishing of eta}
Let $d$ and $\delta$ be two divisors of $N$. The order of vanishing of $\Delta_\delta$ at a cusp $\bect x d^N$ of level $d$ is 
\begin{equation*}
a_N(d, \delta):=\frac{N}{\gcd(d, N/d)} \times \frac{\gcd(d, \delta)^2}{d\delta}.
\end{equation*}
\end{lemma}
\begin{proof}
Let $\gamma=\mat x a d b \in \SL_2(\Z)$ so that $\gamma \infty=\bect x d$. Then the order of $\Delta_\delta$ at a cusp $\bect x d^N$ is the smallest power of $q_h:=q^{1/h}=e^{2\pi i \tau/h}$ in the Puiseux expansion of $(d\tau+b)^{-12}\Delta_\delta(\gamma \tau)$, where $h$ is the width of a cusp $\bect x d^N$. (For example, see \cite[Sec. 3.2]{DS05}.) By Lemma  \ref{lemma: width of cusp}, we have $h=\frac{N}{d\cdot \gcd(d, N/d)}$.

Let $g=\gcd(d, \delta)$, and write $d=gd_1$ and $\delta=g\delta_1$ with $\gcd(d_1, \delta_1)=1$. 
Since $\gcd(x\delta_1, d_1)=1$, there are integers $m$ and $n$ such that $x{\delta_1}n-d_1m=1$. 
By direct computation, we have
\begin{equation*}
k:=\frac{ag-m}{x}=\frac{b-n\delta_1}{d_1}\in \Z.
\end{equation*}
Since
\begin{equation}\label{equation: aU=Ua}
\delta(\gamma \tau)=\frac{\delta x \tau+\delta a}{d\tau+b}=\bmat {x \delta_1}{m}{d_1}{n}\left(\frac{g\tau +k}{\delta_1}\right)
\end{equation}
and $\mat {x \delta_1}{m}{d_1}{n} \in \SL_2(\Z)$, we have
\begin{equation}\label{equation: aU=Ua 2}
\Delta_\delta(\gamma\tau)=\Delta(\delta(\gamma\tau))=\left(\frac{d\tau+b}{\delta_1}\right)^{12}\Delta \left(\frac{g\tau +k}{\delta_1}\right).
\end{equation}
Thus, the Puiseux expansion of $(d\tau+b)^{-12}\Delta_\delta(\gamma\tau)$ is
\begin{equation*}
\begin{split}
\delta_1^{-12} \cdot \Delta\left(\frac{g\tau +k}{\delta_1}\right)&=\delta_1^{-12} \cdot  e^{\frac{2\pi ik}{\delta_1}} \cdot q^{\delta_1^{-1}g}\prod_{n=1}^\infty(1-(e^{\frac{2\pi ik}{\delta_1}}q^{\delta_1^{-1}g})^n)^{24}\\
&=\delta_1^{-12} \cdot e^{\frac{2\pi ik}{\delta_1}} \cdot q_h^{\delta_1^{-1}gh}\prod_{n=1}^\infty(1-(e^{\frac{2\pi ik}{\delta_1}}q_h^{\delta_1^{-1}gh})^n)^{24},
\end{split}
\end{equation*} 
and therefore the order of vanishing of $\Delta_\delta$ at a cusp $\bect x d^N$ is 
\begin{equation*}
\delta_1^{-1} g h=\frac{g}{\delta_1} \times \frac{N}{d\cdot \gcd(d, N/d)}=\frac{N}{\gcd(d, N/d)}\times \frac{g^2}{d\delta}.
\end{equation*}
This completes the proof.
\end{proof}
\begin{definition}\label{definition: the matrix Lambda}
For a positive integer $N$, let 
\begin{equation*}
\Lambda(N):=\left(\frac{a_N(d, \delta)}{24}\right)_{d, \delta \mid N}
\end{equation*}
be a square matrix of size $\sigma_0(N)$, indexed by the divisors of $N$. We regard this matrix as a linear map from $\cS_1(N)_\Q$ to $\cS_2(N)_\Q$.
\end{definition}

By Lemma \ref{lemma: the order of vanishing of eta}, we have the following, which is \cite[Prop. 3.2.8]{Li75}.
\begin{lemma}\label{lemma: the divisor of eta quotient}
Let ${\bf r}=\sum_{\delta \mid N} r_\delta \cdot {\bf e}(N)_\delta \in \cS_1(N)_\Q$. Then we have
\begin{equation*}
\tn{div}(g_{\bf r})=\sum_{d\mid N} \left(\sum_{\delta\mid N} \frac{a_N(d, \delta)}{24} \times r_\delta\right) \cdot (P_d).
\end{equation*}
\end{lemma}

If $g_{\bf r}$ is a modular function on $X_0(N)$, then its order of vanishing at any cusp is an integer and the degree of its divisor is zero. Thus, we have
$\Lambda(N) \times {\bf r} \in \cS_2(N)^0$. It turns out that such properties are almost enough for an eta quotient to be a modular function.

\begin{proposition}[Ligozat]\label{proposition: criteria for rational modular unit}
Let ${\bf r}=\sum_{\delta \mid N} r_\delta \cdot {\bf e}(N)_\delta \in \cS_1(N)_\Q$ and 
$g_{\bf r}=\prod_{\delta \mid N} \eta_\delta^{r_\delta}$. Then $g_{\bf r}$ is a modular function on $X_0(N)$ if and only if all the following conditions are satisfied:
\begin{enumerate}
\setcounter{enumi}{-1}
\item
all $r_\delta$ are rational integers, i.e., ${\bf r} \in \cS_1(N)$.
\item
$\sum_{\delta \mid N} r_\delta \cdot \delta \equiv 0 \pmod {24}$.
\item
$\sum_{\delta \mid N} r_\delta \cdot (N/\delta) \equiv 0 \pmod {24}$.
\item
$\sum_{\delta \mid N} r_\delta =0$.
\item
$\prod_{\delta \mid N} \delta^{r_\delta}$ is the square of a rational number.
\end{enumerate}
\end{proposition}
\begin{proof}
The proof can be obtained from \cite[Sec. 3.2]{Li75}. For the sake of the readers, we explain this proof in detail.

First, for any $\gamma=\mat a b {Nc} d \in \Gamma_0(N)$ and $r_\delta \in \Z$, we have
\begin{equation*}
g_{\bf r}(\gamma \cdot z)=\varepsilon(\gamma) \cdot (Ncz+d)^{\frac{1}{2}\sum r_\delta} \cdot g_{\bf r}(z),
\end{equation*}
where $\varepsilon(\gamma)$ is a certain root of unity depending on $\gamma$ and ${\bf r}$. 
Indeed, by the same idea used in (\ref{equation: aU=Ua}) and (\ref{equation: aU=Ua 2}), we obtain this formula, which will not be done here.
In this case, the argument is much simpler because we may use the equality $\mat {\delta} 0 0 1 \mat a b {Nc} d = \mat a {b\delta} {Nc\delta^{-1}} {d} \mat {\delta} 0 0 1$  instead of (\ref{equation: aU=Ua}).
See page 29 of \textit{op. cit.} for more detail.

Suppose that ${\bf r}$ satisfies all the conditions above. 
To prove that $g_{\bf r}$ is a modular function on $X_0(N)$, it suffices to show that $\varepsilon(\gamma)=1$ for any $\gamma \in \Gamma_0(N)$.
In fact, by the argument on \cite[pg. 374]{Ne59}, it suffices to show that $\varepsilon(\gamma)=1$ for the matrices $\gamma=\mat a b {Nc} d \in \Gamma_0(N)$ satisfying $\gcd(a, 6)=1$, $a>0$ and $c>0$. By direct computation, we easily have 
$\varepsilon(\gamma)=1$ for such $\gamma$. (For more detail, see \cite{Ne57}.)

Conversely, suppose that $g_{\bf r}$ is a modular function on $X_0(N)$. Then the order of vanishing of $g_{\bf r}$ at a cusp must be an integer. By Lemma \ref{lemma: the divisor of eta quotient}, for any divisor $d$ of $N$, we get $\frac{1}{24}\sum_{\delta\mid N} a_N(d, \delta) \cdot r_\delta \in \Z$. By direct computation, $a_N(N, \delta)=\delta$ and $a_N(1, \delta)=N/{\delta}$, so conditions (1) and (2) are satisfied. Also, since the degree of $\tn{div}(g_{\bf r})$ is zero, condition (3) is fulfilled. 
Furthermore, by the transformation property of $g_{\bf r}$, we must have $\prod_{\delta \mid N} \leg {\delta}{a}^{r_\delta}=1$ for any $a>0$ with $\gcd(a, 6N)=1$. In particular, for any prime $p$ not dividing $6N$, we have $\leg x p=1$, where $x=\prod_{\delta \mid N} \delta^{r_\delta}$. This only holds when $x$ is a square (cf. Lemme on \cite[pg. 32]{Li75}), and hence condition (4) is satisfied. 
Thus, it suffices to show that $r_\delta \in \Z$ for all divisors $\delta$ of $N$. 
Note that since $\sum_{\delta \mid N} r_\delta =0$, $g_{\bf r}$ is a power series in $q$ with rational coefficients.
As on \cite[pg. 39]{Li75}, we easily have $r_\delta \in \Z$ for all $\delta$ if $g_{\bf r} \in \Z[[q]]$ (cf. \cite[Lem. 19]{RW15}). 
Instead, we follow the argument on \cite[pg. 458]{Og74}. Since the order of vanishing of $\Delta$ at a cusp $\bect x d^N$ of level $d$ is $\frac{N}{d\cdot \gcd(d, N/d)}$ by Lemma \ref{lemma: the order of vanishing of eta}, a function $g'=g_{\bf r} \cdot \Delta^{k}$ vanishes at all cusps for a sufficiently large integer $k$.
Thus, $g'$ is a cusp form of weight $12k$ for $\Gamma_0(N)$ with rational Fourier coefficients.
Since these coefficients have bounded denominators by \cite[Th. 3.52]{Shi71}, there is a non-zero integer $b$ such that $g''=b \cdot g' \in \Z[[q]]$. Since $g''$ does not vanish on $\cH$, we can write $g''=c \prod_{\delta \mid N} \eta_\delta^{s_\delta}$ for some $c \in \Z$ and $s_\delta \in \Z$ by \cite[Th. 7]{RW15}. Since $g''\neq 0$, we have $c\neq 0$ and hence $r_\delta \in \Z$.\footnote{For a generalization of such an argument, see Theorem 4.2 of \cite[Ch. 4]{KL} and 
the proof of the claim (d) on \cite[pg. 421]{WY20}.} (Note that $r_1=s_1-12k$ and $r_\delta=s_\delta$ for any $\delta \in \cD_N^0$.)
\end{proof}

\begin{theorem}\label{theorem: rational modular units}
Suppose that $F$ is a modular unit on $X_0(N)$ such that its order of vanishing at a cusp $\bect x d^N$ of level $d$ does not depend on $x$. Then there is a constant $\epsilon \in \C^\times$ and an eta quotient $g_{\bf r}$ of level $N$ such that $F=\epsilon\cdot g_{\bf r}$.
\end{theorem}
\begin{proof}
This basically follows from the fact that the matrix $\Lambda(N)$ is invertible (\cite[Lem. 3.2.9]{Li75} or Lemma \ref{lemma: invertible lambda} below).
By our assumption, we can write
\begin{equation*}
\tn{div}(F)=\textstyle\sum_{d\mid N} a_d \cdot (P_d) ~~\text{ for some } a_d \in \Z.
\end{equation*}
Let ${\bf r}=\sum_{\delta \mid N} r_\delta \cdot {\bf e}(N)_\delta=\Lambda(N)^{-1} \times (\sum_{d\mid N}a_d  \cdot {\bf e}(N)_d)$. 
Note that \textit{a priori} we only have $r_\delta \in \Q$, and so $g_{\bf r}$ might not be an eta quotient. 
Since
\begin{equation*}
\textstyle\Lambda(N) \times \left(\sum_{\delta \mid N} r_\delta \cdot {\bf e}(N)_\delta\right)=\Lambda(N) \times \Lambda(N)^{-1} \times \left(\sum_{d\mid N}a_d  \cdot {\bf e}(N)_d \right)=\sum_{d\mid N}a_d  \cdot {\bf e}(N)_d,
\end{equation*}
we have $\tn{div}(g_{\bf r})=\sum_{d\mid N} a_d \cdot (P_d)=\tn{div}(F)$. Thus, there is a constant $\epsilon\in \C^\times$ such that $F=\epsilon\cdot g_{\bf r}$. Since $F$ is a modular function on $X_0(N)$, so is $g_{\bf r}$. Therefore we have $r_\delta \in \Z$ by Proposition \ref{proposition: criteria for rational modular unit}, and so $g_{\bf r}$ is indeed an eta quotient of level $N$.
\end{proof}

\ms
\subsection{The matrix $\Upsilon(N)$}\label{section: matrix upsilon}
For a prime $p$ and a positive integer $r$, we define a tridiagonal matrix
$\Upsilon(p^r)$  (indexed by the divisors of $p^r$) by
\small
\begin{equation*}
\arraycolsep=4.5pt\def\arraystretch{1.2}
\begin{split}
\Upsilon(p^r):=
\left(\begin{array}{cccccccc}
p & -p & &&& \\
-1 & p^2+1 & &&& & \\
& -p &  & & & & \\ 
&  &  \ddots & & & & \\ 
& &  & -p^{m(f)}& & & \\ 
 &   & & p^{m(f)-1}(p^2+1)&  &   & \leftarrow \\
 &    & &-p^{m(f)} &&&  \\
 &  & & & \ddots & &  \\
  &  & & & & -p &  \\
 &   & & &&p^2+1&-1 \\
 &   & & \uparrow&&-p& p\\
\end{array}\right)  \begin{array}{c}
\tn{ $p^f$-th row for }\\
1\leq f \leq r-1
\end{array}
\\ 
p^f\text{-th column for }~~1\leq f \leq r-1, \phantom{aaaaaaaaaaaaaaaaaaaaaaaa}
\end{split}
\end{equation*} \normalsize
where $m(f)=\min(f, \, r-f)$. In other words, we have
\begin{equation*}
\Upsilon(p^r)_{p^i p^j}:=\begin{cases}
p & \text{ if } ~~ i=j=0 ~~\text{ or }~~ r, \\
p^{m(j)-1}(p^2+1) & \text{ if }~~1\leq i=j \leq r-1,\\
-p^{m(j)} & \text{ if }~~ |i-j|=1,\\
0 & \text{ if }~~ |i-j|\geq 2.
\end{cases}
\end{equation*}

\vspace{2mm}
If we write $N=\prod_{i=1}^t p_i^{r_i}$, then we define a matrix $\Upsilon(N)$ (indexed by the divisors of $N$) by
\begin{equation*}
\Upsilon(N):=\motimes_{i=1}^t \Upsilon(p_i^{r_i}).
\end{equation*}
In other words, if $\delta=\prod_{i=1}^t p_i^{e_i}$ and $d=\prod_{i=1}^t p_i^{f_i}$, then we have
\begin{equation*}
\Upsilon(N)_{\delta d}=\prod_{i=1}^t \Upsilon(p_i^{r_i})_{p_i^{e_i} p_i^{f_i}}.
\end{equation*}

From now on, we regard $\Upsilon(N)$ as a linear map from $\cS_2(N)_\Q$ to $\cS_1(N)_\Q$ by our identifications $\cS_k(N)_\Q = \motimes_{i=1}^t \cS_k(p_i^{r_i})_\Q$ (cf. Remark \ref{remark: decomposition of RCD}).

\begin{lemma}\label{lemma: invertible lambda}
For any positive integer $N>1$, we have
\begin{equation*}
\Upsilon(N) \times \Lambda(N) = \Lambda(N) \times \Upsilon(N)=\frac{\kappa(N)}{24}\times \tn{Id}_{\sigma_0(N)},
\end{equation*}
where $\tn{Id}_n$ is the identity matrix of size $n$. In particular, $\Lambda(N)$ is invertible.
\end{lemma}
\begin{proof}
For any integer $N$, let $\Lambda(N)'=24\times \Lambda(N)$. It suffices to prove that
\begin{equation*}
\Upsilon(N) \times \Lambda(N)' = \Lambda(N)' \times \Upsilon(N) = \kappa(N) \times \tn{Id}_{\sigma_0(N)}.
\end{equation*}

First, suppose that $N=p^r$ is a prime power. Then by direct computation (or by \cite[Prop. 3]{Li97}), the result follows.

Next, let $N=Mp^r$ with $\gcd(M, p)=1$ and $r\geq 1$. Suppose that
\begin{equation*}
\Upsilon(M) \times \Lambda(M)' = \Lambda(M)' \times \Upsilon(M)=\kappa(M)\times \tn{Id}_{\sigma_0(M)}.
\end{equation*}
Let $d$ and $\delta$ be two divisors of $M$, and $0\leq f, \, g \leq r$. 
By direct computation, we have
\begin{equation*}
a_N(dp^f, \delta p^g)=a_M(d, \delta) \times a_{p^r}(p^f, p^g),
\end{equation*}
and so $\Lambda(N)'=\Lambda(M)'\motimes \Lambda(p^r)'$. Since $\Upsilon(N)=\Upsilon(M) \motimes \Upsilon(p^r)$ by definition, we have 
\begin{equation*} 
\begin{split}
\Upsilon(N) \times \Lambda(N)'&=(\Upsilon(M)\times \Lambda(M)')\motimes (\Upsilon(p^r) \times \Lambda(p^r)')\\
&=(\kappa(M) \times \tn{Id}_{\sigma_0(M)})\motimes (\kappa(p^r) \times \tn{Id}_{r+1})=\kappa(N)\times \tn{Id}_{\sigma_0(N)}.
\end{split}
\end{equation*}
Similarly, we have $\Lambda(N)' \times \Upsilon(N)=\kappa(N) \times \tn{Id}_{\sigma_0(N)}$.

By induction, the result follows.
\end{proof}

\begin{lemma}\label{lemma: entries of Upsilon matrix}
For a divisor $d$ of $N$, let 
\begin{equation*}
\textstyle {\bf r}(d)=\sum_{\delta \mid N} {\bf r}(d)_\delta \cdot {\bf e}(N)_\delta :=\Upsilon(N) \times {\bf e}(N)_d \in \cS_1(N)
\end{equation*}
be the $d$-th column of the matrix $\Upsilon(N)$ and let $z=\gcd(d, N/d)$. Then we have the following.
\begin{enumerate}
\item
We have 
\begin{equation*}
\sum_{\delta\mid N} {\bf r}(d)_\delta =\varphi(z) \times \prod_{p\mid N}(p-1)=\frac{z}{\rad(z)} \times \prod_{p \mid N} (p-1)^{a(p)},
\end{equation*}
where $a(p)=2$ if $p$ divides $z$, and $a(p)=1$ otherwise.
\item
We have
\begin{equation*}
\sum_{\delta \mid N} {\bf r}(d)_\delta \cdot \delta = \begin{cases}
\kappa(N) & \text{ if } ~~ d=N,\\
0 & \text{ otherwise}.
\end{cases}
\end{equation*}
\item
We have
\begin{equation*}
\sum_{\delta \mid N} {\bf r}(d)_\delta \cdot (N/\delta)=\begin{cases}
\kappa(N) & \text{ if } ~~d=1,\\
0 & \text{ otherwise}.
\end{cases}
\end{equation*}
\item
We have
\begin{equation*}
\gcd({\bf r}(d)_{\delta}: \delta \mid N)= \frac{z}{\rad(z)}.
\end{equation*}
\end{enumerate}
\end{lemma}
\begin{proof}
Let $N=\prod_{i=1}^t p_i^{r_i}$ and $d=\prod_{i=1}^t p_i^{f_i}$. Since ${\bf e}(N)_d=\motimes_{i=1}^t {\bf e}(p_i^{r_i})_{p_i^{f_i}}$, we have
\begin{equation*}
{\bf r}(d)=\motimes_{i=1}^t \left( \Upsilon(p_i^{r_i}) \times {\bf e}(p_i^{r_i})_{p_i^{f_i}} \right).
\end{equation*}
Since the constant function, the identity function, the reciprocal function and the greatest common divisor function are multiplicative, both sides are multiplicative. Thus, it suffices to check the formulas when $N=p^r$ with $r\geq 1$. By direct computation, we have
\begin{equation*}
\begin{split}
\sum_{i=0}^r \Upsilon(p^r)_{p^i p^f}&=\varphi(\gcd(p^f, \, p^{r-f}))\times (p-1),\\
\sum_{i=0}^r \Upsilon(p^r)_{p^i p^f} \times p^i &=\begin{cases}
p^{r-1}(p^2-1)\phantom{1} & \text { if }~~ f=r, \\
0 & \text { if }~~ f <r,
\end{cases}\\
\sum_{i=0}^r \Upsilon(p^r)_{p^i p^f} \times p^{r-i}&=\begin{cases}
p^{r-1}(p^2-1)\phantom{1} & \text{ if }~~ f=0, \\
0 & \text{ if }~~ f>0,
\end{cases}\\
\gcd(\Upsilon(p^r)_{p^i p^f} : 0\leq i \leq r)&=\begin{cases}
p^{m(f)-1} & \text{ if }~~ m(f)\geq 1,\\
1\phantom{p^{r-1}(p^2-1)} & \text{ if }~~ m(f)=0.
\end{cases}
\end{split}
\end{equation*}
This completes the proof.
\end{proof}
\begin{remark}
The second and the third equalities easily follow from Lemma \ref{lemma: invertible lambda} because the $N$-th (resp. first) row of $24\times \Lambda(N)$ is $\sum_{d \mid N} d \cdot {\bf e}(N)_d$ (resp. $\sum_{d \mid N} (N/d) \cdot {\bf e}(N)_d$).
\end{remark}

\ms
\subsection{Algorithm for computing the order}\label{section: algorithm}
In this subsection, we elaborate Ligozat's method and provide a simple algorithm for computing the order of a degree $0$ rational cuspidal divisor $C$ on $X_0(N)$.  We first construct a generalized eta quotient $g({\bf r}(C))$ as follows:
\begin{equation*}
\xyv{1.5}
\xymatrix{
\Qdiv N \ar[r]^-{\Phi_N} & \cS_2(N)^0 \ar[r]^-{\Lambda(N)^{-1}} & \cS_1(N)_\Q \ar[r]^-{g} & \cE \\
C=\sum a_d \cdot (P_d) \arin \ar@{|->}[r] & \sum a_d \cdot {\bf e}(N)_d \ar@{|->}[r] \arin & {\bf r}(C)=\sum r_\delta \cdot {\bf e}(N)_\delta \ar@{|->}[r] \arin & g({\bf r}(C)), \arin
}
\end{equation*}
where $\cE$ is the set of generalized eta quotients of level $N$, 
\begin{equation*}
{\bf r}(C):=\Lambda(N)^{-1} \times \Phi_N(C) \qa g({\bf r}(C)):=g_{{\bf r}(C)}=\prod_{\delta\mid N} \eta_\delta^{r_\delta}.
\end{equation*}

The following is well-known (cf. \cite[pg. 36]{Li97}).
\begin{proposition}\label{proposition: computation of the order}
For a rational cuspidal divisor $C$ on $X_0(N)$, the order of $C$ is the smallest positive integer $n$ such that $g(n\cdot {\bf r}(C))$ is a modular function on $X_0(N)$, or equivalently $n\cdot {\bf r}(C)$ satisfies all the conditions in Proposition \ref{proposition: criteria for rational modular unit}.
\end{proposition}
\begin{proof}
Let $k$ be the order of $C$. Then by the definition of $k$, there is a modular function $F$ on $X_0(N)$ such that $k\cdot C=\tn{div}(F)$. Since the divisor of $g({\bf r}(C))$ is $C$ by Lemma \ref{lemma: the divisor of eta quotient}, the divisors of $F$ and $g(k\cdot {\bf r}(C))$ are equal. Therefore there is a constant $\epsilon \in \C^\times$ such that $F=\epsilon\cdot g(k\cdot {\bf r}(C))$, and so $g(k\cdot {\bf r}(C))$ is also a modular function on $X_0(N)$.
By the minimal property of $n$, we have $n\leq k$. 

Conversely, by the definition of $n$, $g(n\cdot {\bf r}(C))$ is a modular function on $X_0(N)$. Thus, by the minimal property of $k$, we have $k\leq n$ because the divisor of $g(n\cdot {\bf r}(C))$ is $n\cdot C$. This completes the proof.
\end{proof}

The following is crucial in our method.
\begin{corollary}\label{corollary: order 1 criterion}
Let $X \in \Qdiv N$ be a degree $0$ rational cuspidal divisor on $X_0(N)$. For a prime $p$, let
\begin{equation*}
\rpw_p(X):=\textstyle\sum_{\tn{val}_p(\delta)\not\in 2\Z} {\bf r}(X)_\delta,
\end{equation*} 
where the sum runs over the divisors of $N$ whose $p$-adic valuations are odd. Then the following are equivalent.
\begin{enumerate}
\item
The order of $X$ is $1$, or equivalently $\ov{X}=0 \in J_0(N)$.
\item
${\bf r}(X) \in \cS_1(N)$ and $\rpw_p(X) \in 2\Z$ for all primes $p$.
\end{enumerate}
\end{corollary}
\begin{proof}
Let $X=\sum_{d\mid N} a_d \cdot (P_d) \in \Qdiv N$. Note that conditions (1), (2) and (3) for ${\bf r}(X)$ follow from $a_N\in \Z$, $a_1 \in \Z$ and the degree of $X$ is $0$, respectively. (For instance, see the proof of Proposition \ref{proposition: criteria for rational modular unit} above.) Thus, conditions (1), (2) and (3) for $n \cdot {\bf r}(X)$ are always satisfied for any integer $n\geq 1$.
Note also that the normalized $p$-adic valuation of the product in condition (4) is $\sum_{\delta\mid N} \tn{val}_p(\delta)\cdot {\bf r}(X)_\delta$. Since
\begin{equation*}
\textstyle\sum_{\tn{val}_p(\delta)\not\in 2\Z} {\bf r}(X)_\delta \equiv \sum_{\tn{val}_p(\delta)\not\in 2\Z} \tn{val}_p(\delta) \cdot {\bf r}(X)_\delta \equiv 
\sum_{\delta\mid N} \tn{val}_p(\delta)\cdot {\bf r}(X)_\delta \pmod 2,
\end{equation*}
the result easily follows by Proposition \ref{proposition: computation of the order}.
\end{proof}

In principle, we can compute the order of any rational cuspidal divisor by Proposition \ref{proposition: computation of the order} above. Since $\Upsilon(N)$ is a scalar multiple of $\Lambda(N)^{-1}$ and is an integral matrix, it is easy to compute various invariants in the following.
\begin{definition}\label{definition: bbV and Gcd and pw}
For $C=\sum_{d\mid N} a_d \cdot (P_d) \in \Qdivv N$, we define 
\begin{equation*}
V(C)=\textstyle\sum_{\delta \mid N} V(C)_\delta \cdot {\bf e}(N)_\delta:=\Upsilon(N) \times \Phi_N(C) \in \cS_1(N).
\end{equation*}
In other words, for any divisor $\delta$ of $N$, we have 
\begin{equation*}
V(C)_\delta:=\textstyle\sum_{d \mid N} \Upsilon(N)_{\delta d} \times a_d \in \Z.
\end{equation*}
Let $\Gcd(C):=\gcd(V(C)_\delta : \delta \mid N)$ be the greatest common divisor of the entries of $V(C)$ 
and let
\begin{equation*}
\bbV(C)=\textstyle \sum_{\delta\mid N} \bbV(C)_\delta \cdot {\bf e}(N)_\delta:=\Gcd(C)^{-1} \times V(C) \in \cS_1(N).
\end{equation*}
Furthermore, let 
\begin{equation*}
\pw_p(C):=\textstyle\sum_{\tn{val}_p(\delta)\not\in 2\Z} \bbV(C)_{\delta},
\end{equation*}
where the sum runs over the divisors of $N$ whose $p$-adic valuations are odd. Finally, let
\begin{equation*}
\fh(C):=\begin{cases}
1 & \text{ if }~~ \pw_p(C) \in 2\Z ~~\text{ for all primes } p,\\
2 & \text{ if } ~~\pw_p(C) \not\in 2\Z ~~\text{ for some prime } p.
\end{cases}
\end{equation*}
\end{definition}

The main theorem of this section is the following, which gives an easy algorithm for computing the order of any degree $0$ rational cuspidal divisor.  
\begin{theorem}\label{theorem: computation of order}
For any $C\in \Qdiv N$, the order of $C$ is
\begin{equation*}
\num\left(\frac{\kappa(N)\times \fh(C)}{24\times \Gcd(C)}\right)=\frac{\kappa(N)}{\gcd(24\times \Gcd(C) \times \fh(C)^{-1}, ~\kappa(N))}.
\end{equation*}
In particular, if the genus of $X_0(N)$ is positive, then the order of $C$ divides $\frac{\kappa(N)}{12}$.
\end{theorem}
\begin{remark}\label{remark: kappa divisible by 24}
Since $\kappa(N)=N\prod_{p\mid N}(p-p^{-1})$ is multiplicative and $p^2-1$ is divisible by $24$ for any prime $p\geq 5$, it is easy to see that $\kappa(N)$ is divisible by $24$ if and only if $N \not\in \{1, 2, 3, 4, 8 \}$.
\end{remark}
\begin{proof}[Proof of Theorem \ref{theorem: computation of order}]
First, suppose that the genus of $X_0(N)$ is $0$, in which case 
\begin{equation*}
N \in \{1, 2, 3, 4, 5, 6, 7, 8, 9, 10, 12, 13, 16, 18, 25\}.
\end{equation*}
Then we can easily verify the formula, which we leave to the readers. (In fact, the order of any degree $0$ rational cuspidal divisor is $1$.) So we assume that the genus of $X_0(N)$ is positive, in which case $\frac{\kappa(N)}{24} \in \Z$ by the remark above. 

Next, let $C \in \Qdiv N$. Since $V(C) \in \cS_1(N)$, it is easy to check that $2V(C)$ satisfies conditions (0) and (4). Note that $2V(C)=\frac{\kappa(N)}{12} \cdot {\bf r}(C)$.
Since conditions (1), (2) and (3) for $n\cdot {\bf r}(C)$ are always satisfied for any integer $n\geq 1$ as in Corollary \ref{corollary: order 1 criterion}, the order of $C$ is a divisor of $\frac{\kappa(N)}{12}$ by Proposition \ref{proposition: computation of the order}. 
Accordingly, let $m$ be the largest divisor of $\frac{k(N)}{12}$ such that 
\begin{equation*}
{\bf r}(m):=\frac{\kappa(N)}{12m} \cdot {\bf r}(C)=\frac{2}{m} \cdot V(C)=k\cdot \bbV(C)
\end{equation*}
satisfies conditions (0) and (4), where $k=\frac{2\cdot \Gcd(C)}{m}$. Then the order of $C$ is $\frac{\kappa(N)}{12m}$.
Since the greatest common divisor of the entries of $\bbV(C)$ is $1$ and ${\bf r}(m) \in \cS_1(N)$, we have $k \in \Z$. Also, since ${\bf r}(m)=k\cdot \bbV(C)$, for any prime $p$ we have
\begin{equation*}
\textstyle k\cdot \pw_p(C)\equiv \sum_{\delta \mid N} {\bf r}(m)_\delta \cdot \tn{val}_p(\delta)\pmodo {2}.
\end{equation*}
Thus, $k\cdot \pw_p(C) \in 2\Z$ for all primes $p$ by condition (4), and hence we must have either $k\in 2\Z$ or $\pw_p(C) \in 2\Z$. So by definition, if $\fh(C)=1$ then there is no condition on $k$, and if $\fh(C)=2$ then there is a prime $p$ such that $\pw_p(C) \not\in 2\Z$, in which case $k$ must be even. 
Therefore $k$ must be divisible by $\fh(C)$ in both cases, and so $m$ is a divisor of $\frac{2\cdot \Gcd(C)}{\fh(C)}$. Since $m$ is given as a divisor of $\frac{\kappa(N)}{12}$, we have
\begin{equation*}
m=\gcd\left(\frac{\kappa(N)}{12}, \frac{2\cdot \Gcd(C)}{\fh(C)} \right).
\end{equation*}
This implies $12m=\gcd(\kappa(N), 24\cdot \Gcd(C) \cdot \fh(C)^{-1})$, and so the result follows.
\end{proof}

\ms
\subsection{Example I: The divisors defined by tensors}\label{section: Example I: The divisors defined by tensors}
Let $C \in \Qdiv N$ be a degree $0$ rational cuspidal divisor on $X_0(N)$. We say that $C$ \textit{is defined by tensors} if there are non-trivial divisors $N_1$ and $N_2$ of $N$ such that $N=N_1 N_2$, $\gcd(N_1, N_2)=1$ and $\Phi_N(C)=V_1 \motimes V_2$ for some $V_i \in \cS_2(N_i)$.
Let $C_i = \Phi_{N_i}^{-1}(V_i)$ so that $C_i \in \Qdivv {N_i}$. Since the degree of $C$ is the product of the degrees of $C_1$ and $C_2$, we further assume that the degree of $C_1$ is zero.
The main result of this section is the following.
\begin{theorem}\label{theorem: order defined by tensors}
Suppose that $C$ is defined by tensors and written as above. Then we have
\begin{equation*}
V(C) = V(C_1) \motimes V(C_2) \qa \bbV(C)=\bbV(C_1) \motimes \bbV(C_2).
\end{equation*}
Also, we have $\Gcd(C)=\Gcd(C_1)\cdot \Gcd(C_2)$ and 
\begin{equation*}
\pw_p(C)=\begin{cases}
\pw_p(C_1) \cdot \sum_{\delta_2 \mid N_2} \bbV(C_2)_{\delta_2} & \text{ if } ~~p \mid N_1, \\
~~~0 & \text{ otherwise}.
\end{cases}
\end{equation*}
\end{theorem}
\begin{proof}
By definition, we have
\begin{equation*}
\begin{split}
V(C)&=\Upsilon(N) \times \Phi_N(C) = (\Upsilon(N_1) \motimes \Upsilon(N_2)) \times (V_1 \motimes V_2)\\
&=(\Upsilon(N_1) \times V_1) \motimes (\Upsilon(N_2) \times V_2)=V(C_1) \motimes V(C_2).
\end{split}
\end{equation*}
Also, by the definition of the tensor product, the greatest common divisor of the entries of $v \otimes w$ is the product of the greatest common divisors of the entries of $v$ and $w$. Indeed, if we let $g$ (resp. $h$) be the greatest common divisor of the entries of $v$ (resp. $w$), then
all the entries of $v\otimes w$ are divisible by $gh$ and hence the greatest common divisor of the entries of $v\otimes w$ is a multiple of $gh$. Conversely, since there are integers $a_i$ and $b_j$ such that $g=\sum_i a_i v_i$ and $h= \sum_j b_j w_j$, where $v=(v_i)$ and $w=(w_j)$,
we have
\begin{equation*}
\textstyle  gh=(\sum_i a_i v_i)(\sum_j b_j w_j)=\sum_{i, j} a_i b_j (v_i w_j),
\end{equation*}
and so $gh$ is divisible by the greatest common divisor of the entries of $v\otimes w$. Thus, the first assertion follows.

Next, let $p$ be a prime divisor of $N_1$. For a divisor $\delta$ of $N$, let $\delta_i=\gcd(\delta, N_i)$. Then by definition, we have
\begin{equation*}
\begin{split}
\pw_p(C)&=\sum_{\tn{val}_p(\delta)\not\in 2\Z} \bbV(C)_\delta=\sum_{\tn{val}_p(\delta_1)\not\in 2\Z} \sum_{\delta_2 \mid N_2} \bbV(C_1)_{\delta_1} \cdot \bbV(C_2)_{\delta_2} \\
&=\sum_{\tn{val}_p(\delta_1)\not\in 2\Z} \bbV(C_1)_{\delta_1} \times \sum_{\delta_2 \mid N_2} \bbV(C_2)_{\delta_2}=\pw_p(C_1) \cdot \sum_{\delta_2 \mid N_2} \bbV(C_2)_{\delta_2}.
\end{split}
\end{equation*}
Similarly, if $p$ divides $N_2$, then we have $\pw_p(C)=\pw_p(C_2) \cdot \sum_{\delta_1 \mid N_1} \bbV(C_1)_{\delta_1}$. Since the degree of $C_1$ is zero, the sum of the entries of $V(C_1)$ is zero and hence the result follows.
\end{proof}

\begin{remark}
In the previous papers \cite{Yoo3, Yoo6}, the author computed the orders of certain rational cuspidal divisors, which are eigenvectors under the action of the Hecke operators. The secret to our success of the computations is that the rational cuspidal divisors in our consideration are defined by tensors, and hence we could use an inductive method based on Theorem \ref{theorem: order defined by tensors}.
The author did not know the work of Yazdani \cite{Yaz11} at the time of writing of the papers \cite{Yoo3, Yoo6}. After reading \cite{Yaz11}, the author decided to follow Yazdani's idea, which enormously simplifies our previous notation.
\end{remark}

\ms
\subsection{Example II: The order of the divisor $C_N$}\label{section: Example II: The order of CN}
In this subsection, we compute the order of $C_{N}$ for any positive integer $N$. 
\begin{theorem}\label{theorem: example order C_N}
Let $N$ be an integer greater than $1$. Then we have
\begin{equation*}
\Gcd(C_N)=\fg(N) \qa \fh(C_N) =\fh(N),
\end{equation*}
where $\fg(N)$ and $\fh(N)$ are defined as follows:
\begin{enumerate}
\item 
$\fg(1):=0$.
\item
$\fg(N)=\gcd(N+(-1)^{t-1}, \, p_i^2-1 : 1\leq i \leq t)$ if $N=\prod_{i=1}^t p_i$ is squarefree.
\item
$\fg(N):=\gcd(p, \, \fg(M))$ if $N=Mp^2$ with $\gcd(M, p)=1$ and $M$ squarefree.
\item
$\fg(N):=1$ otherwise.
\end{enumerate}
Also, $\fh(N)=2$ if one of the following holds, and $\fh(N)=1$ otherwise\footnote{This is a variant of $h$ in \cite[Th. 1.3]{Yoo3} and of $h(M, N, D)$ in \cite[Th. 4.3]{Yoo6}.}.
\begin{enumerate}
\item
$N=p$ for a prime $p$.
\item
$N=2^r$ for an odd integer $r$.
\item
$N=pq$ for two distinct odd primes $p$ and $q$ such that $\tn{val}_2(p-1)=\tn{val}_2(q-1)$ and $\tn{val}_2(p+1)=\tn{val}_2(q+1)$.
\item
$N=4p$ for a prime $p$ congruent $1$ modulo $4$.
\end{enumerate}
\end{theorem}

By Theorem \ref{theorem: computation of order}, we have the following.
\begin{theorem}\label{theorem: 3.18}
Let $N$ be a positive integer, and let $n(N)$ be the order of $C_N$.
 If $N=p^r$ is a prime power, then we have
 \begin{equation*}
 \begin{array}{|c|c|c|c|c|} \hline 
 & r=1 & r=2 & p=2, ~r\geq 3 \tn{ is odd} & \tn{otherwise} \\  \hline
 n(N) & \frac{p-1}{\gcd(12, \, p-1)}   & \frac{p^2-1}{\gcd(24, \, p^2-1)} & 2^{r-3} & \frac{p^{r-1}(p^2-1)}{24} \\ \hline
 \end{array}
 \end{equation*}
If $N$ is not a prime power, then we have the following.
\begin{enumerate}
\item
For an odd prime $q$, we have $n(2q)=\frac{q^2-1}{8\cdot \gcd(3, \, q+1)}$.
If $N=pq$ for two distinct odd primes $p$ and $q$, then 
\begin{equation*}
n(N)=\frac{(p^2-1)(q^2-1)}{12\cdot \gcd(p-1, q-1)\cdot \gcd(p+1, q+1)}.
\end{equation*}
\item
If $N=\prod_{i=1}^t p_i$ be a squarefree integer with $t\geq 3$, then
\begin{equation*}
n(N)=\frac{\prod_{i=1}^t (p_i^2-1)}{24\cdot \gcd(N+(-1)^{t-1}, \, p_i^2-1 : 1\leq i \leq t)}.
\end{equation*}
\item
If $N=Mp^2$ for a squarefree integer $M$ not divisible by a prime $p$, then
\begin{equation*}
n(N)=\begin{cases}
\frac{\kappa(N)}{24p} & \text{ if $~p$ is odd and $p$ divides $\fg(M)$},\\
\frac{\kappa(N)}{48} & \text{ if $~p=2$ and $M$ is not a prime congruent to $1$ modulo $4$},\\
\frac{\kappa(N)}{24} & \text{ otherwise}.
\end{cases}
\end{equation*}

\item
If $N$ is not of the form considered above, i.e., either $N$ is divisible by $p^3$ or by $p^2 q^2$, then $n(N)=\frac{\kappa(N)}{24}$.
\end{enumerate}
\end{theorem}
\begin{proof}[Proof of Theorems \ref{theorem: example order C_N} and \ref{theorem: 3.18}]
First, let $N=p^r$ be a prime power. If $r=1$, then we have 
\begin{equation*}
V(C_N)=\mat p {-1} {-1} p \times \vect 1 {-1}=\vect {p+1}{-p-1}.
\end{equation*}
Thus, $\Gcd(C_N)=p+1$ and $\fh(C_N)=2$ because $\pw_p(C_N)=-1$ is odd. Therefore the order of $C_N$ is the numerator of $\frac{p-1}{12}$.
If $r=2$, then we have
\begin{equation*}
V(C_N)=\left(\begin{smallmatrix}
p & -p & 0 \\
-1 & p^2+1 & -1 \\
0 & -p & p
\end{smallmatrix}\right)\times \left(\begin{smallmatrix}
1  \\
0 \\
-1
\end{smallmatrix}\right)=\left(\begin{smallmatrix}
p  \\
0 \\
-p
\end{smallmatrix}\right).
\end{equation*}
Thus, $\Gcd(C_N)=p$ and $\fh(C_N)=1$ because $\pw_p(C_N)=0 \in 2\Z$. Suppose that $r\geq 3$.
Then we have
\begin{equation*}
V(C_N)=(p, -1, \dots, 1, -p)^t,
\end{equation*}
where the dots denote zero entries. Thus, $\Gcd(C_N)=1$ and 
\begin{equation*}
\pw_p(C_N)=\begin{cases}
0 & \text{ if } ~~r \in 2\Z,\\
-1-p & \text{ if } ~~r\not\in 2\Z.
\end{cases}
\end{equation*}
Therefore $\fh(C_N)=2$ if and only if $p=2$ and $r\not\in 2\Z$. 
If $p$ is odd then $p^{r-1}(p^2-1)$ is divisible by $24$ and hence the result follows. 

From now on, we assume that $N$ is divisible by at least two primes. For simplicity, let
\begin{equation*}
\pow_p(C_N):=\Gcd(C_N) \cdot \pw_p(C_N)=\frac{\kappa(N)}{24}\times \rpw_p(C_N).
\end{equation*}
\begin{enumerate}
\item
Let $N=pq$, and assume that $q$ is an odd prime. Then we have
\begin{equation*}
V(C_N)=\left(\begin{smallmatrix}
pq & -q & -p & 1 \\
-q & pq & 1 & -p \\
-p & 1 & pq & -q \\
1 & -p & -q & pq
\end{smallmatrix}\right) \times\left(\begin{smallmatrix}
1 \\
0\\
0\\
-1
\end{smallmatrix}\right)
=\left(\begin{smallmatrix}
pq-1\\
p-q \\
q-p\\
1-pq
\end{smallmatrix}\right).
\end{equation*}
Thus, we have
\begin{equation*}
\Gcd(C_N)=\gcd(pq-1, p-q)
\end{equation*}
and
\begin{equation*}
\pow_p(C_N)=(p+1)(1-q) \qqa \pow_q(C_N)=(1-p)(q+1). 
\end{equation*}

If $p=2$, then $\Gcd(C_N)=\gcd(3, q+1)$, which is odd. Since $q$ is odd, $\pow_p(C_N)$ and $\pow_q(C_N)$ are both even, and hence $\fh(C_N)=1$. 

Suppose that $p$ is odd. If $\tn{val}_2(p-1) =\tn{val}_2(q-1)$ and $\tn{val}_2(p+1) = \tn{val}_2(q+1)$, then we have $\tn{val}_2(\Gcd(C_N))=\tn{val}_2((p-1)(q+1))$ (cf. Lemma \ref{lemma: squarefree t=2} below). Thus, we have $\fh(C_N)=2$ in this case.
Suppose that either $\tn{val}_2(p-1)\neq \tn{val}_2(q-1)$ or $\tn{val}_2(p+1) \neq \tn{val}_2(q+1)$. 
Then we have $\tn{val}_2(pq-1)=\tn{val}_2(p-q)$.\footnote{For instance, if $\tn{val}_2(p-1)>\tn{val}_2(q-1)$, then we have
\begin{equation*}
\tn{val}_2(pq-1)=\tn{val}_2(q(p-1)+(q-1))=\tn{val}_2(q-1)=\tn{val}_2(p-q).
\end{equation*}}
Therefore we have
\begin{equation*}
\tn{val}_2((p-1)(q+1))=\tn{val}_2((pq-1)+(p-q)) > \tn{val}_2(pq-1)=\tn{val}_2(p-q).
\end{equation*}
Similarly, we have $\tn{val}_2((p+1)(q-1))>\tn{val}_2(pq-1)$ and so $\fh(C_N)=1$. By Lemma \ref{lemma: squarefree t=2} below, we can simplify the formula and thus we have
 \begin{equation*}
n(pq)= \frac{(p^2-1)(q^2-1)}{12\cdot \gcd(p-1, q-1) \cdot \gcd(p+1, q+1)} \in \Z.
\end{equation*}

\item
Let $N=\prod_{i=1}^t p_i$ with $t\geq 3$. As above, for a divisor $d$ of $N$, let ${\bf r}(d)$ be the $d$-th column vector of the matrix $\Upsilon(N)$. By direct computation, we have
\begin{equation*}
V(C_N)_\delta={\bf r}(1)_\delta-{\bf r}(N)_{\delta}=\mu(\delta)(N/\delta+(-1)^{t-1}\delta),
\end{equation*}
where $\mu(d)$ is the M\"obius function. We claim that
\begin{equation*}
\Gcd(C_N)=\fg(N)=\gcd(N+(-1)^{t-1}, \, p_i^2-1 : 1\leq i \leq t).
\end{equation*}

Note that
\begin{equation*}
-p_i \cdot V(C_N)_{p_i}=N+(-1)^{t-1}p_i^2=N+(-1)^{t-1} + (-1)^{t-1}(p_i^2-1).
\end{equation*}
Thus, $\Gcd(C_N)$ divides $p_i^2-1$, and so $\fg(N)$.

Conversely, since $p_i^2 \equiv 1 \pmod {\fg(N)}$, we have $\delta^2 \equiv 1 \pmod {\fg(N)}$ for any divisor $\delta$ of $N$. Therefore we have
\begin{equation*}
\mu(\delta) \delta \cdot V(C_N)_\delta =  N + (-1)^{t-1} \delta^2 \equiv 0 \pmodo {\fg(N)}.
\end{equation*}
Since $\fg(N)$ is relatively prime to $N$, it divides $V(C_N)_\delta$ for any divisor $\delta$ of $N$, and hence $\Gcd(C_N)$. This proves the claim.

Next, let $d$ be a divisor of $N$, and $p=p_i$ for some $i$. Let $d'=\gcd(d, M)$, where $M=N/p$. Then we have
\begin{equation*}
\textstyle \sum_{\tn{val}_{p}(\delta)\not\in 2\Z} {\bf r}(d)_\delta=\sum_{\delta' \mid M} {\bf r}(d)_{p\delta'}=\Upsilon(p)_{p p^\epsilon}\sum_{\delta' \mid M} \Upsilon(M)_{\delta' d'},
\end{equation*}
where $\epsilon=\tn{val}_p(d)$. Since $M$ is squarefree, by Lemma \ref{lemma: entries of Upsilon matrix}(1) we have
\begin{equation*}
\textstyle \pow_{p_i}(C_N)=\sum_{\tn{val}_{p_i}(\delta)\not\in 2\Z} ({\bf r}(1)_\delta-{\bf r}(N)_\delta)=-(p_i+1)\prod_{j=1, j\neq i}^t (p_j-1).
\end{equation*}
Thus, we have $\fh(C_N)=1$ by Lemma \ref{lemma: squarefree t>2} below. Since $(p^2-1)(q^2-1)$ is always divisible by $24$ and $p_i^2-1 \equiv 0 \pmod {\fg(N)}$, we have 
$\frac{\kappa(N)}{24\cdot \fg(N)} \in \Z$, and therefore we have
\begin{equation*}
n(N)=\frac{\prod_{i=1}^t (p_i^2-1)}{24\cdot \gcd(N+(-1)^{t-1}, p_i^2-1 : 1\leq i \leq t)} \in \Z.
\end{equation*} 
\item
Let $N=Mp^2$ for a prime $p$ not dividing $M$. Let ${\bf r}'(d)$ be the $d$-th column vector of $\Upsilon(M)$. 
Then by direct computation, we have
\begin{equation*}
V(C_N)=\Upsilon(M) \motimes \left(\begin{smallmatrix}
p & -p & 0 \\
-1 & p^2+1 & -1 \\
0 & -p & p
\end{smallmatrix}\right) \times \left(\begin{smallmatrix}
{\bf e}(M)_1 \\
0\\
-{\bf e}(M)_M
\end{smallmatrix}\right)
= \left(\begin{smallmatrix}
p\cdot {\bf r}'(1) \\
-{\bf r}'(1)+{\bf r}'(M)\\
-p\cdot {\bf r}'(M)
\end{smallmatrix}\right).
\end{equation*}
Since ${\bf r}'(1)_{\rad(M)}=(-1)^k$, where $k$ is the number of prime divisors of $M$, $\Gcd(C_N)$ divides $p$. Since $V(C_M)={\bf r}'(1)-{\bf r}'(M)$, we have 
\begin{equation*}
\Gcd(C_N)=\gcd(p, \,\Gcd(C_M)).
\end{equation*}
If $M$ is squarefree, we have $\Gcd(C_M)=\fg(M)$ by the result above 
and hence $\Gcd(C_N)=\gcd(p, \, \fg(M))$.

Next, we have 
\begin{equation*}
\pow_p(C_N)=\textstyle\sum_{\delta' \mid M}(-{\bf r}'(1)_{\delta'}+{\bf r}'(M)_{\delta'})=0,
\end{equation*}
and for a prime divisor $\ell$ of $M$
\begin{equation*}
\textstyle \pow_\ell(C_N)=\sum_{\tn{val}_\ell(\delta')\not\in 2\Z}(p-1)({\bf r}'(1)_{\delta'}-{\bf r}'(M)_{\delta'})=(p-1)\cdot \pow_\ell(C_M).
\end{equation*}
Since $\pow_\ell(C_M)=\Gcd(C_M) \cdot \pw_\ell(C_M)$, we have $\fh(C_N)=1$ if $p$ is odd. Thus, the result follows when $p$ is odd.

Lastly, let $p=2$ and assume that $M$ is odd and squarefree. Then by definition, $\fg(M)$ is even, and therefore we have $\Gcd(C_N)=2$. As discussed above, we have $\pow_\ell(C_N)=\pow_\ell(C_M)$. If $\fh(C_M)=1$, then $\pow_\ell(C_N)$ is divisible by $4$ and so $\fh(C_N)=1$. Suppose that $\fh(C_M)=2$.
Since $M$ is odd and squarefree, $\pow_\ell(C_M)$ is divisible by $4$ unless $M$ is a prime congruent to $1$ modulo $4$, in which case we obtain $\fh(C_N)=2$ (cf. Remark \ref{remark: fg 2-adic valuation} below). Thus, the result follows.

\item
If $N$ is not a prime power, we have $\frac{\kappa(N)}{24} \in \Z$ (Remark \ref{remark: kappa divisible by 24}) and so it suffices to show that $\Gcd(C_N)=1$ and $\fh(C_N)=1$.

Suppose first that $N$ is divisible by $p^2q^2$. We may assume that $p$ is an odd prime.
Then by taking $M=N/{p^2}$ (resp. $M=N/{q^2}$) in (3) above, we obtain that $\Gcd(C_N)$ divides $p$ (resp. $q$). Therefore $\Gcd(C_N)=1$. Also, since $p$ is odd, we have $\pow_\ell(C_N) \in 2\Z$ for any prime divisor $\ell$ of $M=N/{p^2}$. Thus, we have $\fh(C_N)=1$ as well.

Suppose next that $N=Mp^r$ with $\gcd(M, p)=1$ and $r\geq 3$. 
Since we assume that $N$ is not a prime power, we have $M>1$.
If we use the same notation as in (3), then we have
\begin{equation*}
V(C_N)=\Upsilon(M) \motimes\Upsilon(p^r) \times \left(\begin{smallmatrix}
{\bf e}(M)_1\\
\bbO\\
{\bf e}(M)_M
\end{smallmatrix}\right)=\left(\begin{smallmatrix}
p\cdot {\bf r}'(1)\\
-{\bf r}'(1)\\
\bbO\\
{\bf r}'(M)\\
-p\cdot {\bf r}'(M)
\end{smallmatrix}\right),
\end{equation*}
where $\bbO$ is the zero matrix of suitable size.
Since ${\bf r}'(1)_{\rad(M)}=(-1)^k$, we have $\Gcd(C_N)=1$. Also, 
as above we have 
\begin{equation*}
\begin{split}
\pow_p(C_N)&=\sum_{\delta' \mid M} (-{\bf r}'(1)_{\delta'}+\epsilon \cdot {\bf r}'(M)_{\delta'})\\
&=(\epsilon-1) \times \sum_{\delta' \mid M}{\bf r}'(M)_{\delta'}=(\epsilon-1) \times \prod_{\ell \mid M}(\ell-1),
\end{split}
\end{equation*}
where $\epsilon=1$ if $r$ is even, and $\epsilon=-p$ if $r$ is odd. Since $N$ is divisible by an odd prime, we have $\pow_p(C_N) \in 2\Z$. Also, we have
\begin{equation*}
\pow_\ell(C_N)=(p-1)\times \pow_\ell(C_M)
\end{equation*}
and hence $\pow_\ell(C_N)\in 2\Z$ if $p$ is odd. Furthermore, if $p=2$ and $M>1$, then we have
\begin{equation*}
\pow_\ell(C_M)=\Gcd(C_M) \cdot \pw_\ell(C_M)\in 2\Z,
\end{equation*}
and hence $\fh(C_N)=1$. Indeed, since $M$ is odd we have the following:
\begin{itemize}[--]
\item
If $M$ is squarefree, then $\Gcd(C_M)=\fg(M) \in 2\Z$.
\item
If $M$ is exactly divisible by $\ell^2$, then $\pow_\ell(C_M) \in 2\Z$ as in (3).
\item
If $M$ is divisible by $\ell^3$, then $\pow_\ell(C_M) \in 2\Z$ as above.
\end{itemize}
\end{enumerate}
This completes the proof.
\end{proof}

\begin{remark}\label{remark: fg 2-adic valuation}
Let $M$ be an odd and squarefree integer. By definition, $\fg(M)$ is even. Suppose that $\fh(M)=2$. Then either $M$ is a prime or the product of two primes satisfying certain conditions above. In the latter case, we can prove that $\fg(M)$ is divisible by $8$ (cf. Lemma \ref{lemma: squarefree t=2}). Thus, $\fg(M)$ is divisible by $4$ unless $M$ is a prime congruent to $1$ modulo $4$, in which case $\fg(M)=M+1 \equiv 2 \pmod 4$.
\end{remark}

We finish this section by proving two lemmas used above.
\begin{lemma}\label{lemma: squarefree t=2}
Let $p$ and $q$ be two distinct odd primes. Then we have
\begin{equation*}
\gcd(pq-1, p-q)=\gcd(pq-1, p^2-1, q^2-1)=2^a \cdot \gcd(p-1, q-1) \cdot \gcd(p+1, q+1),
\end{equation*}
where $a=0$ if $\tn{val}_2(p-1)=\tn{val}_2(q-1)$ and $\tn{val}_2(p+1)=\tn{val}_2(q+1)$, and $a=-1$ otherwise.
\end{lemma}
\begin{proof}
Let $g=\gcd(pq-1, p-q)$. Then $\gcd(g, p)=\gcd(g, q)=1$ because $pq\equiv 1 \pmod g$. This directly implies the first equality because
\begin{equation*}
p(p-q)=p^2-pq=(p^2-1)+(1-pq)  \qqa q(p-q)=(pq-1)+(1-q^2).
\end{equation*}

Next, let $G=\gcd(pq-1, \, p^2-1, \, q^2-1)$ and $H=\gcd(p-1, q-1)\cdot \gcd(p+1, q+1)$.
If $\ell$ is an odd prime, then we easily have that $\tn{val}_\ell(G)=\tn{\val}_\ell(H)$ because $\ell$ cannot divide both $p-1$ and $p+1$. 

Finally, we compute $b=\tn{val}_2(G)-\tn{val}_2(H)$. Let $r=\tn{val}_2(p-1)$ and $s=\tn{val}_2(q-1)$, which are both at least $1$. 
Without loss of generality we may assume that $r\geq s\geq 1$. Note that $\tn{val}_2(p^2-1) \geq r+1$ and $\tn{val}_2(q^2-1)\geq s+1$.
\begin{itemize}[--]
\item
If $r>s$, then we have
\begin{equation*}
pq \equiv 1 \pmodo {2^s} \qqa pq \equiv q \not\equiv 1 \pmodo {2^{s+1}}.
\end{equation*}
Therefore $\tn{val}_2(G)=s$. Since $s\geq 1$, we have $r\geq 2$ and hence $\tn{val}_2(p+1)=1$. So we have $\tn{val}_2(H)=s+1$ and $b=-1$. 

\item
If $r=s\geq 2$, then $\tn{val}_2(p+1)=\tn{val}_2(q+1)=1$, and so $\tn{val}_2(p^2-1)=\tn{val}_2(q^2-1)=s+1$. Also, we have
\begin{equation*}
pq-1=(1+2^s)(1+2^s)-1 \equiv 0 \pmodo {2^{s+1}},
\end{equation*}
and therefore $\tn{val}_2(G)=s+1$. Since $\tn{val}_2(H)=s+1$, we have $b=0$.

\item
Assume that $r=s=1$. Let $x=\tn{val}_2(p+1)$ and $y=\tn{val}_2(q+1)$, which are both at least $2$. If $x>y$, then similarly as above, we have $\tn{val}_2(G)=y$ and $\tn{val}(H)=y+1$, and so $b=-1$. If $x=y\geq 2$, then we have
\begin{equation*}
pq-1=(-1+2^x)(-1+2^x)-1 \equiv 0 \pmodo {2^{x+1}}.
\end{equation*}
Therefore $\tn{val}_2(G)=\tn{val}_2(H)=x+1$ and hence $b=0$. 
\end{itemize}
This completes the proof.
\end{proof}

\begin{lemma}\label{lemma: squarefree t>2}
Let $N=\prod_{i=1}^t p_i$ be a squarefree integer with $t\geq 3$. Let
\begin{equation*}
\textstyle \fg(N)=\gcd(N+(-1)^{t-1}, \, p_i^2-1 : 1\leq i\leq t) \qa s_i(N)=(p_i+1)\prod_{j=1, j\neq i}^t (p_j-1).
\end{equation*}
Then $2\cdot \fg(N)$ divides $s_i(N)$ for all $1\leq i\leq t$.
\end{lemma}
\begin{proof}
It suffices to show that for any prime $\ell$ we have
\begin{equation*}
\tn{val}_\ell(2)+\tn{val}_\ell(\fg(N)) \leq \tn{val}_\ell(s_i(N)) ~\text{ for all }i.
\end{equation*}
Since $t\geq 3$, $s_i(N)$ is always even and hence the inequality above always holds if $\fg(N)$ is not divisible by $\ell$. Thus, we may assume that $\tn{val}_\ell(\fg(N))\geq 1$. 

First, suppose that $\ell$ is odd. 
Since either $\tn{val}_\ell(p_j-1)=0$ or $\tn{val}_\ell(p_j+1)=0$,  we have
\begin{equation*}
\tn{val}_\ell(\fg(N)) \leq \tn{val}_\ell(p_j^2-1)=\max(\tn{val}_\ell(p_j+1), \tn{val}_\ell(p_j-1)).
\end{equation*}
Therefore we have $\tn{val}_\ell(\fg(N))\leq \tn{val}_\ell(s_i(N))$ unless 
\begin{equation*}
\tn{val}_\ell(p_i-1)>0 \qa \tn{val}_\ell(p_j-1)=0 ~~\text{ for all } ~j\neq i,
\end{equation*}
in which case we have $\tn{val}_\ell(p_j+1)\geq 1$ for all $j\neq i$ because  $\tn{val}_\ell(\fg(N))\geq 1$. Thus, we have
\begin{equation*}
N+(-1)^{t-1}=\textstyle\prod_{k=1}^t p_i +(-1)^{t-1} \equiv (-1)^{t-1} +  (-1)^{t-1} \equiv 2 \cdot (-1)^{t-1} \pmodo \ell,
\end{equation*}
which is a contradiction to the assumption that $\tn{val}_\ell(\fg(N))\geq 1$ because $\ell$ is odd. 

Next, let $\ell=2$. Note that if $N$ is even then $\fg(N)$ is odd by definition. 
Thus, $N$ is odd because we assume that $\fg(N)$ is divisible by $\ell=2$, and hence all $p_i$ are odd. 
So we have $\tn{val}_2(s_i(N))\geq t$. Let $g=\tn{val}_2(\fg(N))$.
If $g \leq t-1$, then there is nothing to prove, so we further assume that $g\geq t$, which is at least $3$ by our assumption.
Since $\tn{val}_2(p_j^2-1)=\tn{val}_2(p_j-1)+\tn{val}_2(p_j+1) \geq g$, and either $\tn{val}_2(p_i-1)=1$ or $\tn{val}_2(p_i+1)=1$, we have $\tn{val}_2(p_i-1)\geq g-1 \geq 2$ or $\tn{val}_2(p_i+1) \geq g-1 \geq 2$.
Thus, we have $\tn{val}_2(s_i(N))\geq g+1$ unless 
\begin{equation*}
\tn{val}_2(p_i+1)=\tn{val}_2(p_j-1)=1 ~\text{ for all } ~j\neq i,
\end{equation*}
in which case we have
$\tn{val}_2(p_i-1) \geq 2$ and $\tn{val}_2(p_j+1)\geq 2$ for all $j\neq i$. Thus, we have
\begin{equation*}
N+(-1)^{t-1}=\textstyle\prod_{k=1}^t p_i +(-1)^{t-1}\equiv 2\cdot (-1)^{t-1} \not\equiv 0 \pmodo {4}.
\end{equation*}
This implies that $g=1$, which is a contradiction. This completes the proof.
\end{proof}

\ms
\subsection{Example III: The order of the divisor $C_d$}\label{section: Example II: The order of Cd}
In this subsection, we compute the order of $C_{d}$ for any positive integer $N$ and a non-trivial divisor $d$ of $N$ with $d\neq N$. By Theorem \ref{theorem: computation of order}, it suffices to prove the following.
\begin{theorem}\label{theorem: example order Cd}
Let $N$ be a positive integer and let $d$ be a non-trivial divisor of $N$. Then we have
\begin{equation*}
\Gcd(C_d) = \fg(N, d) \qa \fh(C_d)=\fh(N, d),
\end{equation*}
where $\fg(N, d)$ and $\fh(N, d)$ are defined as follows: First, we set 
\begin{equation*}
\fg(N, N):=\fg(N) \qa \fh(N, N):=\fh(N).
\end{equation*}
Next, suppose that $d\neq N$, and let $z=\gcd(d, N/d)$. Then we set
\begin{enumerate}
\item
$\fg(N, d):=\fg(d)$ if $z=1$.
\item
$\fg(N, d):=\gcd(p, \, \fg(d/p))$ if $z=p$ is a prime and $\tn{val}_p(d)=1$.
\item
$\fg(N, d):=\frac{z}{\rad(z)}$ otherwise.
\end{enumerate}
Also, $\fh(N, d):=2$ if one of the following holds, and $\fh(N, d):=1$ otherwise.
\begin{enumerate}
\item
$N=2^r$ for $r\geq 2$ and $d=2$.
\item
$N=2^r$ for $r\geq 2$ and $d=2^f$ for $f\in 2\Z$.
\item
$N=2^r p$ for $r\geq 2$ and $d=2p$, where $p\equiv 1 \pmod 4$ is a prime.
\item
$N=2^r d$ for $r\geq 1$ and $d=p$, where $p$ is an odd prime.
\item
$N=2^r d$ for $r\geq 1$ and $d=pq$, where $p$ and $q$ are two distinct odd primes such that $\tn{val}_2(p-1)=\tn{val}_2(q-1)$ and $\tn{val}_2(p+1)=\tn{val}_2(q+1)$.
\end{enumerate}
\end{theorem}

\begin{proof}
By Theorem \ref{theorem: example order C_N}, the result follows for $d=N$. Thus, we henceforth assume that $d\neq N$. During the proof, let $\bbO$ denote the zero matrix of suitable size.

First, let $N=p^r$ with $r\geq 2$. If $d=p$, then  we have
\begin{equation*}
V(C_d)=\Upsilon(p^r) \times \left(\begin{smallmatrix}
p-1\\
-1 \\
\bbO
\end{smallmatrix}\right)=\left(\begin{smallmatrix}
p^2 \\
-p^2-p\\
p \\
\bbO
\end{smallmatrix}\right)=p\left(\begin{smallmatrix}
p\\
-p-1\\
1 \\
\bbO
\end{smallmatrix}\right).
\end{equation*}
Thus, $\Gcd(C_d)=p$ and $\pw_p(C_d)=-p-1$. Therefore $\fh(C_d)=1$ if $p$ is odd, and $\fh(C_d)=2$ if $p=2$. 

Suppose that $d=p^2$ (so $r\geq 3$). Then we have
\begin{equation*}
V(C_d)=\Upsilon(p^r) \times \left(\begin{smallmatrix}
\varphi(p^{m(2)})\\
0\\
-1 \\
\bbO
\end{smallmatrix}\right)=\left(\begin{smallmatrix}
p^{m(2)}(p-1) \\
p^{m(2)-1}\\
-(p^2+1)p^{m(2)-1} \\
p^{m(2)}\\
\bbO
\end{smallmatrix}\right)=p^{m(2)-1}\left(\begin{smallmatrix}
p(p-1) \\
1\\
-p^2-1 \\
p\\
\bbO
\end{smallmatrix}\right).
\end{equation*}
Thus, $\Gcd(C_d)=p^{m(2)-1}=\frac{z}{\rad(z)}$ and $\pw_p(C_d)=p+1$.  Therefore $\fh(C_d)=2$ if $p=2$, and $\fh(C_d)=1$ if $p$ is odd.

Suppose that $d=p^f$ with $3\leq f\leq r-1$. Then we have
\begin{equation*}
V(C_d)=\Upsilon(p^r) \times \left(\begin{smallmatrix}
\varphi(p^{m(f)})\\
\bbO \\
-1 \\
\bbO
\end{smallmatrix}\right)=\left(\begin{smallmatrix}
p\varphi(p^{m(f)}) \\
-\varphi(p^{m(f)})\\
\bbO \\
p^{m(f)} \\
-(p^2+1)p^{m(f)-1} \\
p^{m(f)}\\
\bbO \\
\end{smallmatrix}\right)=p^{m(f)-1}\left(\begin{smallmatrix}
p(p-1) \\
1-p\\
\bbO \\
p \\
-p^2-1 \\
p\\
\bbO \\
\end{smallmatrix}\right)
\end{equation*}
because $\varphi(p^{m(f)})=p^{m(f)-1}(p-1)$. Thus, $\Gcd(C_d)=p^{m(f)-1}$ and 
\begin{equation*}
\pw_p(C_d)=\begin{cases}
p+1 & \text{ if }~~ f \in 2\Z,\\
-p(p+1) & \text{ if }~~ f \not\in 2\Z.\\
\end{cases}
\end{equation*}
Therefore $\fh(C_d)=2$ if $p=2$ and $f$ is even, and $\fh(C_d)=1$ otherwise. This proves the result for the case where $N$ is a prime power.

Next, let $N=Mp^r$ with $\gcd(M, p)=1$ and $r\geq 1$. We assume that $M>1$, so there is a prime divisor of $M$ different from $p$. For a divisor $d$ of $N$, let $d'=\gcd(M, d)$ and $p^f=\gcd(d, p^r)$, so that $d=d'p^f$.
Also, let $z'=\gcd(d', M/{d'})$. Note that 
\begin{equation*}
\varphi(z)\cdot {\bf e}(N)_1=\varphi(z') \cdot {\bf e}(M)_1 \motimes \varphi(p^{m(f)})\cdot {\bf e}(p^r)_1
\end{equation*}
and ${\bf e}(N)_d={\bf e}(M)_{d'} \motimes {\bf e}(p^r)_{p^f}$. For a divisor $k$ of $M$, let ${\bf r}'(k):=\Upsilon(M) \times {\bf e}(M)_k$ be the $k$-th column vector of the matrix $\Upsilon(M)$. Then we have
\begin{equation*}
V(C(M)_{d'})=\Upsilon(M) \times \Phi_M(C(M)_{d'})=\varphi(z')\cdot {\bf r}'(1)-{\bf r}'(d').
\end{equation*}
Here, we allow $d'=1$, in which case $V(C(M)_{d'})$ is the zero vector. From now on, let $\ell$ denote a prime divisor of $M$.
By Lemma \ref{lemma: entries of Upsilon matrix}(4), the greatest common divisor of the entries of ${\bf r}(d')$ is $\frac{z'}{\rad(z')}$. Since $\varphi(z')=\frac{z'}{\rad(z')} \prod_{\ell\mid z'} (\ell-1)$,
we easily have
\begin{equation}\label{equation: 0jae}
\Gcd(C(M)_{d'})=\frac{z'}{\rad(z')} \times \alpha ~~~ \text{ for some } \alpha \in \Z.
\end{equation}
Also, since the degree of $C(M)_{d'}$ is zero, we have
\begin{equation*}
\textstyle \sum_{{\delta'}\mid M} V(C(M)_{d'})_{\delta'}=\sum_{{\delta'} \mid M} (\varphi(z') \cdot {\bf r}'(1)_{\delta'}-{\bf r}'(d')_{\delta'})=0.
\end{equation*}
Furthermore, we have
\begin{equation*}
\textstyle \pow_\ell(C(M)_{d'})=\sum_{\tn{val}_\ell({\delta'})\not\in 2\Z} (\varphi(z')\cdot {\bf r}'(1)_{\delta'}-{\bf r}'(d')_{\delta'}).
\end{equation*}
This will be used below without further mention.

By direct computation, we have
\begin{equation*}
\Upsilon(N)\times \varphi(z)\cdot {\bf e}(N)_1=\varphi(z)(\Upsilon(M)\times {\bf e}(M)_1) \motimes \left(\begin{smallmatrix}
p \\
-1\\
\bbO \\
\end{smallmatrix}\right)=\left(\begin{smallmatrix}
p\varphi(z)\cdot {\bf r}'(1) \\
-\varphi(z)\cdot {\bf r}'(1) \\
\bbO \\
\end{smallmatrix}\right).
\end{equation*}
Similarly, if $f=0$ (resp. $f=r$), then
\begin{equation*}
\Upsilon(N)\times {\bf e}(N)_d= \left(\begin{smallmatrix}
p\cdot {\bf r}'(d') \\
-{\bf r}'(d') \\
\bbO \\
\end{smallmatrix}\right) \quad \left(\tn{resp}. ~~\left(\begin{smallmatrix}
\bbO \\
-{\bf r}'(d')\\
p\cdot {\bf r}'(d')
\end{smallmatrix}\right) \right),
\end{equation*}
and if $1\leq f\leq r-1$, then 
\begin{equation*}
\Upsilon(N)\times {\bf e}(N)_d=\left(\begin{smallmatrix}
\bbO \\
-p^{m(f)}\cdot {\bf r}'(d') \\
p^{m(f)-1}(p^2+1)\cdot {\bf r}'(d')\\
-p^{m(f)}\cdot {\bf r}'(d')\\
\bbO \\
\end{smallmatrix}\right)\begin{smallmatrix}
\phantom{\bbO}\\
\phantom{-p^{m(f)}\cdot {\bf r}'(d') }\\
\leftarrow \text{$p^f$-th entry.}\\
\phantom{-p^{m(f)}\cdot {\bf r}'(d') }\\
\phantom{\bbO}\\
\end{smallmatrix}
\end{equation*}
Using this computation, we have the following.

\vspace{2mm}
\noindent {\bf Case 1}. Assume that $f=0$ and $r\geq 1$. Then $m(f)=0$ and $d=d'$. We have
\begin{equation*}
V(C_d)=V(C(M)_{d}) \motimes (p, -1, \dots)^t.
\end{equation*}
Thus, we have
\begin{equation*}
\Gcd(C_d)=\Gcd(C(M)_d).
\end{equation*}

Next, by definition we have
\begin{equation*}
\textstyle  \pow_p(C_d)=-\sum_{\delta\mid M} V(C(M)_d)_{\delta}=0,
\end{equation*}
and 
\begin{equation*}
\textstyle \pow_\ell(C_d)=(p-1)\sum_{\tn{val}_\ell(\delta) \not\in 2\Z} V(C(M)_d)_\delta=
(p-1)\times \pow_\ell(C(M)_d).
\end{equation*}
Thus, we have
\begin{equation*}
\fh(C_d)=\begin{cases}
2 & \text{ if } ~~p=2 \qqa \fh(C(M)_d)=2,\\
1 & \text{ otherwise}.
\end{cases}
\end{equation*}

\vspace{2mm}
\noindent {\bf Case 2}. Assume that $f=1$ and $r\geq 2$. Then $m(f)=1$ and we have
\begin{equation*}
 V(C_d)=\left(\begin{smallmatrix}
 p\varphi(z)\cdot {\bf r}'(1)+p\cdot {\bf r}'(d')\\
 -\varphi(z)\cdot {\bf r}'(1)-(p^2+1)\cdot {\bf r}'(d')\\
 p\cdot {\bf r}'(d') \\
\bbO \\
 \end{smallmatrix}\right).
 \end{equation*}
Note first that by Lemma \ref{lemma: entries of Upsilon matrix}(4), $p$ does not divide the greatest common divisor of the entries of ${\bf r}'(d')$, and so $p^2$ does not divide $\Gcd(C_d)$.
Since $\varphi(z)=\varphi(z')(p-1)$ and $V(C(M)_{d'})=\varphi(z') \cdot {\bf r}'(1)-{\bf r}'(d')$, we have
\begin{equation}\label{37}
-\varphi(z)\cdot {\bf r}'(1)-(p^2+1)\cdot {\bf r}'(d')=(1-p)\cdot V(C(M)_{d'})-p(p+1)\cdot {\bf r}'(d').
\end{equation}
Thus, $p$ divides $\Gcd(C_d)$ if and only if $p$ divides $\Gcd(C(M)_{d'})$. 
Moreover, by Lemma \ref{lemma: entries of Upsilon matrix}(4) and (\ref{equation: 0jae}) we have
\begin{equation*}
\Gcd(C_d)= \frac{z'}{\rad(z')} \times \gcd(p, \, \alpha).
\end{equation*}

Next, since the sum of the entries of $V(C(M)_{d'})$ is zero, by (\ref{37}) we have
\begin{equation*}
\textstyle \pow_p(C_d)=\sum_{{\delta'} \mid M} (-\varphi(z)\cdot {\bf r}'(1)_{\delta'}-(p^2+1)\cdot {\bf r}'(d')_{\delta'})=-p(p+1)\sum_{{\delta'}\mid M}{\bf r}'(d')_{\delta'},
\end{equation*}
which is equal to $-\frac{z'}{\rad(z')} \times p(p+1)\prod_{\ell\mid M}(\ell-1)^{a(\ell)}$ by Lemma \ref{lemma: entries of Upsilon matrix}(1).

Lastly, as above we have
\begin{equation*}
\begin{split}
\pow_\ell(C_d)&=(p-1)^2 \sum_{\tn{val}_\ell({\delta'})\not\in 2\Z} (\varphi(z')\cdot {\bf r}'(1)_{\delta'}-{\bf r}'(d')_{\delta'})\\
&=(p-1)^2 \cdot \pow_\ell(C(M)_{d'})=\frac{z'}{\rad(z')} \times (p-1)^2 \cdot \alpha \cdot \pw_\ell(C(M)_{d'}).
\end{split}
\end{equation*}
Thus, $\fh(C_d)=2$ if and only if all the following hold.
\begin{itemize}[--]
\item
$p=2$.
\item
$\tn{val}_2(\alpha) \leq 1$.
\item
$\fh(C(M)_{d'})=2$.
\end{itemize}

\vspace{2mm}
\noindent {\bf Case 3}. Assume that $f=2$ and $r\geq 3$. We have
\begin{equation*}
V(C_d)= \left(\begin{smallmatrix}
p\varphi(z)\cdot {\bf r}'(1)\\
-\varphi(z)\cdot {\bf r}'(1)+p^{m(2)}\cdot {\bf r}'(d')\\
-p^{m(2)-1}(p^2+1) \cdot {\bf r}'(d')\\
p^{m(2)}\cdot {\bf r}'(d') \\
\bbO \\
\end{smallmatrix}\right).
\end{equation*}
By Lemma \ref{lemma: entries of Upsilon matrix}(4), all entries of $p^{m(2)-1}\cdot {\bf r}'(d')$ are divisible by $\frac{z}{\rad(z)}$. Since $\varphi(z)$ is also divisible by $\frac{z}{\rad(z)}$, all entries of $V(C_d)$ are divisible by $\frac{z}{\rad(z)}$. By comparing the entries of $p^{m(2)-1}(p^2+1) \cdot {\bf r}'(d')$ and $p^{m(2)}\cdot {\bf r}'(d')$, we have 
\begin{equation*}
\Gcd(C_d)=\frac{z}{\rad(z)}.
\end{equation*}

Next, since the sum of the entries of $V(C(M)_{d'})$ is zero, as above we have
\begin{equation*}
\begin{split}
\pow_p(C_d)&=\sum_{{\delta'}\mid M}(-\varphi(z)\cdot {\bf r}'(1)_{\delta'}+2p^{m(2)}\cdot {\bf r}'(d')_{\delta'})\\
&=p^{m(2)-1}(p+1)\sum_{{\delta'} \mid M} {\bf r}'(d')_{\delta'}=\frac{z}{\rad(z)} \times (p+1) \prod_{\ell\mid M}(\ell-1)^{a(\ell)}.
\end{split}
\end{equation*}

Lastly, as above we have
\begin{equation*}
\begin{split}
\pow_\ell(C_d)&=\sum_{\tn{val}_\ell({\delta'})\not\in 2\Z}((p-1)\varphi(z) \cdot {\bf r}'(1)_{\delta'}-p^{m(2)-1}(p-1)^2 \cdot {\bf r}'(d')_{\delta'})\\
&=p^{m(2)-1}(p-1)^2\sum_{\tn{val}_\ell({\delta'})\not\in 2\Z} (\varphi(z')\cdot {\bf r}'(1)_{\delta'}-{\bf r}'(d')_{\delta'}) \\
&=p^{m(2)-1}(p-1)^2 \cdot \pow_\ell(C(M)_{d'})\\
&=\frac{z}{\rad(z)} \times (p-1)^2 \cdot \alpha \cdot \pw_\ell(C(M)_{d'}).
\end{split}
\end{equation*}
Thus, $\fh(C_d)=2$ if and only if all the following hold.
\begin{itemize}[--]
\item
$p=2$.
\item
$\alpha$ is odd.
\item
$\fh(C(M)_{d'})=2$.
\end{itemize}

\vspace{2mm}
\noindent {\bf Case 4}. Assume that $3\leq f \leq r-1$. We have
\begin{equation*}
V(C_d)= \left(\begin{smallmatrix}
p\varphi(z)\cdot {\bf r}'(1)\\
-\varphi(z)\cdot {\bf r}'(1)\\
\bbO \\
p^{m(f)}\cdot {\bf r}'(d')\\
-p^{m(f)-1}(p^2+1) \cdot {\bf r}'(d')\\
p^{m(f)}\cdot {\bf r}'(d') \\
\bbO \\
\end{smallmatrix}\right)\begin{smallmatrix}
\phantom{p\varphi(z)\cdot {\bf r}'(1)}\\
\phantom{p\varphi(z)\cdot {\bf r}'(1)}\\
\phantom{\bbO}\\
\phantom{p^{m(f)}\cdot {\bf r}'(d')}\\
\leftarrow \tn{$p^f$-th entry.}\\
\phantom{p^{m(f)}\cdot {\bf r}'(d')}\\
\phantom{\bbO}\\
\end{smallmatrix}
\end{equation*}
As above, we have 
\begin{equation*}
\Gcd(C_d)=\frac{z}{\rad(z)}.
\end{equation*}

Also, if $f$ is odd (resp. even) then the computation of $\pow_q(C_d)$ for any prime $q$ is the same as Case 2 (resp. Case 3) above, and we have
\begin{equation*}
\pw_p(C_d)=\begin{cases}
-p(p+1)\prod_{\ell\mid M}(\ell-1)^{a(\ell)} & \text{ if } ~~f \not\in 2\Z,\\
(p+1)\prod_{\ell\mid M}(\ell-1)^{a(\ell)}& \text{ if } ~~f \in 2\Z,\\
\end{cases}
\end{equation*}
and
\begin{equation*}
\pw_\ell(C_d)=(p-1)^2 \cdot \alpha \cdot \pw_\ell(C(M)_{d'}).
\end{equation*}
Thus, $\fh(C_d)=2$ if and only if all the following hold.
\begin{itemize}[--]
\item
$p=2$.
\item
$\alpha$ is odd.
\item
$\fh(C(M)_{d'})=2$.
\end{itemize}

\vspace{2mm}
Now, we are ready to prove our theorem when $N$ is divisible by at least two primes.
Firstly, suppose that $z=1$. Then by Case (1), we have
\begin{equation*}
\Gcd(C_d)=\Gcd(C(d)_d)=\fg(d) \qqa \pw_p(C_d)= \textstyle \pw_p(C(d)_d) \times \prod_{\ell \mid N/d} (\ell-1).
\end{equation*}
Thus, $\fh(C_d)=2$ if and only if $N/d$ is a power of $2$ and $\fh(C(d)_d)=\fh(d)=2$. 
Since $z=1$, if $N/d$ is a power of $2$, then $d$ must be odd. By definition, we have $\fh(d)=1$ unless $d$ is squarefree. Thus, $\fh(C_d)=2$ if and only if $N=2^r d$ with $d$ odd squarefree and $\fh(d)=2$. (These are the cases (4) and (5) in the definition of $\fh(N, d)$.)

Secondly, suppose that $z$ is a power of a prime $p$. If we write $N=Mp^r$ with $\gcd(M, p)=1$, then $d=d'p^f$ for a divisor $d'$ of $M$ and $1\leq f\leq r-1$. Since $z'=\gcd(d', M/{d'})=1$, we have  $\Gcd(C(M)_{d'})=\Gcd(C(d')_{d'})=\fg(d')$. In the discussion below, we use the same notation as in Cases (1)--(4).
\begin{itemize}[--]
\item
Assume that $f\geq 2$. Then we are in Cases (3) and (4), and so we easily have $\Gcd(C_d)=\frac{z}{\rad(z)}$. Also, we always have $\fh(C_d)=1$. Indeed, if $p=2$, then we have either $\alpha \in 2\Z$ or $\fh(C(M)_{d'})=1$. 
More specifically, let $p=2$ (and so $M$ is odd).  If $d'\neq M$, then there is an odd prime $\ell$ dividing $M/{d'}$ because $M$ is odd. Since $z'=1$, by Case (1) we have $\fh(C(M)_{d'})=1$. If $d'=M$, then  $\Gcd(C(M)_{d'})=\fg(M)=\alpha$. Since $M$ is odd, either $\fh(C(M)_{d'})=\fh(M)=1$ or $\fg(M)=\alpha \in 2\Z$ (cf. Remark \ref{remark: fg 2-adic valuation}).  

\item
Assume that $f=1$. Then we are in Case (2). Since $d'=d/p$ and $z'=1$, we have $\Gcd(C(M)_{d'})=\Gcd(C(d')_{d'})=\fg(d')=\fg(d/p)$ by Case (1). As above, we have
\begin{equation*}
\Gcd(C_d)=\gcd(p, \alpha)=\gcd(p, \Gcd(C(M)_{d'}))=\gcd(p, \fg(d/p)).
\end{equation*}

Next, we compute $\fh(C_d)$. By the discussion above, $\fh(C_d)=2$ if and only if $p=2$, $\alpha \not\equiv 0 \pmod 4$, and $\fh(C(M)_{d'})=2$. Let $p=2$. If $d'\neq M$ then we have $\fh(C(M)_{d'})=1$ as above. 
Thus, we further assume that $d'=M$. Then $\fh(C(M)_{d'})=\fh(M)=2$ and $\fg(M) \not\equiv 0 \pmod 4$ if and only if $M$ is a prime congruent to $1$ modulo $4$ (cf. Remark \ref{remark: fg 2-adic valuation}). 
\end{itemize}

Finally, suppose that $z$ is divisible by two distinct primes $p$ and $q$. As above, let $N=Ap^rq^s$ with $\gcd(A, pq)=1$. 
If either $\tn{val}_p(d)\geq 2$ or $\tn{val}_q(d)\geq 2$, then as discussed above we have
\begin{equation*}
\Gcd(C_d)=\frac{z}{\rad(z)} \qa \fh(C_d)=1.
\end{equation*}
Thus, we further assume that $d=d'pq$, where $d'=\gcd(A, d)$. Note that $\frac{z}{\rad(z)}=\frac{z'}{\rad(z')}$, where $z'=\gcd(d', A/{d'})$.
As in (\ref{equation: 0jae}), we have
\begin{equation*}
\Gcd(C_d)=\frac{z}{\rad(z)} \times \beta=\frac{z'}{\rad(z')} \times \beta ~~ \text{ for some } \beta \in \Z.
\end{equation*}
If we take $M=Aq^s$ then by Case (2) we have $\beta=\gcd(p, \alpha)$, which is a divisor of $p$. 
By swapping the roles of $p$ and $q$, we obtain that $\beta$ is also a divisor of $q$. Therefore we have $\beta=1$. Since either $p$ or $q$ is odd, we also have $\fh(C_d)=1$ by Case (2). 

This completes the proof.
\end{proof}

\vspace{10mm}
\section{Motivational Examples}\label{chapter4}
We hope to find rational cuspidal divisors $Z(d)$ on $X_0(N)$ such that
\begin{equation*}
\scC(N) \simeq \moplus_{d\in \cD_N^0} \br{ \ov{Z(d)}}.
\end{equation*}
It turns out that it is indeed possible if $N$ is a prime power. On the other hand, if $N$ is divisible by at least two primes, then the author does not know how to find such divisors except a few easy cases. Instead, for any given prime $\ell$, we try to find rational cuspidal divisors $Z_\ell(d)$ such that
\begin{equation*}
\scC(N)[\ell^\infty] \simeq \moplus_{d \in \cD_N^0} \br{\ov{Z_\ell(d)}}.
\end{equation*}
Indeed, we can do more as follows: we construct rational cuspidal divisors $Z(d)$ such that
\begin{equation*}
\scC(N) \simeq \scC(N)^\sqf \moplus \left(\moplus_{d \in \cD_N^\nsqf} \br{ \ov{Z(d)}}\right),
\end{equation*}
where $\scC(N)^\sqf:=\br{\ov{Z(d)} : d \in \cD_N^\sqf}$. Also, for any $d\in \cD_N^\sqf$, we find a rational cuspidal divisor $Y^2(d)$ such that
\begin{equation*}
\scC(N)^\sqf[\ell^\infty] \simeq \moplus_{d\in \cD_N^\sqf}  \br{\ov{Y^2(d)}}[\ell^\infty].
\end{equation*}

\ms
The purpose of this section is to explain our initial ideas for the computation of $\scC(N)$ in great detail and to define rational cuspidal divisors $Z(d)$ and $Y^2(d)$. Also, we try to give a motivation behind the definition. We will not prove our claims since we will do in the proceeding sections.

\ms
From now on, a vector in $\cS_k(p^r)$ is written as $(a_0, \dots, a_r)$ so that its $p^i$-th entry is $a_i$, i.e., 
\begin{equation*}
\textstyle (a_0, \dots, a_r) := \sum_{i=0}^r a_i \cdot {\bf e}(p^r)_{p^i} \in \cS_k(p^r) ~~\text{ for both $k=1$ or $2$}.
\end{equation*}

\ms
\subsection{Case of level $p^r$}\label{section: case of odd prime power}
In this subsection, we explain our construction of $B_p(r, f)$ for any prime $p$ and $r\geq 2$.

\vspace{2mm}
For simplicity, let $(P_k):=(P(p^r)_{p^k})$ and let
\begin{equation*}
C_r(f):=C(p^r)_{p^f}=\varphi(p^{m(f)})\cdot (P_0)-(P_f),
\end{equation*}
where $m(f):=\min(f, \,r-f)$. To simplify our computation, we assume that $p\geq 7$. 

\subsubsection{$r=2$}
Since the group $\Qdiv {p^2}$ is generated by $C_2(f)$, it suffices to compute the intersection
\begin{equation*}
\br {\ov{C_2(1)}} \cap \br {\ov{C_2(2)}}.
\end{equation*}
Suppose that there is a relation between $C_2(f)$, i.e., $\sum_{f=1}^2 a_f \cdot \ov{C_2(f)}=0$. Then the order of the divisor $X=\sum_{f=1}^2 a_f \cdot C_2(f)$ is $1$, and so by Corollary \ref{corollary: order 1 criterion} we have
\begin{equation*}
{\bf r}(X) \in \cS_1(N) \qa \rpw_p(X) \in 2\Z.
\end{equation*}
Let $n_f$ be the order of the divisor $C_2(f)$ and let $b_f=\frac{a_f}{n_f}$. (Note that we can always take $b_f \in [0, 1)$.) By direct computation, we have
\begin{equation*}
V(C_2(1))=p(p, -p-1, 1) \qa V(C_2(2))=p(1, 0, -1).
\end{equation*}
Thus, we have $\Gcd(C_2(f))=p$ and $\fh(C_2(f))=1$, and so by Theorem \ref{theorem: computation of order} we have $n_1=n_2=\frac{p^2-1}{24}>1$. Also, by Lemma \ref{lemma: invertible lambda} we have
\begin{equation}\label{equation: relation r and bbV}
{\bf r}(a_f \cdot C_2(f))= \frac{a_f \cdot 24\cdot \Gcd(C_2(f))}{\kappa(p^2)} \times \bbV(C_2(f))=b_f\cdot {\fh(C_2(f))} \cdot \bbV(C_2(f)).
\end{equation}
Thus, we have
\begin{equation*}
{\bf r}(X)=(b_1 p + b_2, -b_1(p+1), b_1-b_2).
\end{equation*}
If we take $b_1=b_2=\frac{2}{p+1}$, then ${\bf r}(X)=(2, -2, 0) \in \cS_1(N)$ and $\rpw_p(X)=-2$. 
Thus, there is a relation between $C_2(f)$, and
this argument illustrates a way to find a relation among rational cuspidal divisors.
Indeed, this is informed by the fact that the greatest common divisor of the entries of
\begin{equation*}
\bbV(C_2(1))+\bbV(C_2(2))=(p+1, -p-1, 0)=(p+1)(1, -1, 0)
\end{equation*}
is $p+1$. Now, we may replace $C_2(1)$ by $C_2(1)+C_2(2)$, and let 
\begin{equation*}
\begin{split}
B_p(2, 1)&:=C_2(1)+C_2(2)=p(P_0)-(P_1)-(P_2),\\
B_p(2, 2)&:=C_2(2)=(P_0)-(P_2).
\end{split}
\end{equation*}
Note that
\begin{equation*}
\bbV(B_p(2, 1))=(1, -1, 0) \qa \bbV(B_p(2, 2))=(1, 0, -1).
\end{equation*}
Thus, if there is a relation as above, say $a_1 \cdot \ov{B_p(2, 1)}+a_2 \cdot \ov{B_p(2, 2)}=0$, then 
\begin{equation*}
{\bf r}(X)=(cb_1+b_2, -cb_1, -b_2) \in \cS_1(N) \qa -cb_1 \in 2\Z
\end{equation*}
for some integer $c$. Thus, we get $b_2=0$ and hence $b_1=0$ as well, which is a contradiction. Therefore we have
\begin{equation}\label{equation: level p2 no relations}
\br {\ov{B_p(2, 1)}} \cap \br {\ov{B_p(2, 2)}}=0.
\end{equation}
Since it is obvious that $\Qdiv {p^2}$ is generated by $B_p(2, f)$, we obtain
\begin{equation*}
\scC(p^2) \simeq \moplus \br{\ov{B_p(2, f)}}=\br {\ov{B_p(2, 1)}} \moplus \br {\ov{B_p(2, 2)}}.
\end{equation*}

\begin{observation*}
Note that $B_p(2, 1)=C_2(1)+C_2(2)$ is equal to $\alpha_p(p)^*(C_1(1))$. Thus, we may guess that ``nice'' rational cuspidal divisors can be obtained from lower levels by the degeneracy maps. Also, in the proof of (\ref{equation: level p2 no relations}), we easily have $b_i=0$ because there is a divisor $\delta$ of $N$ such that 
\begin{equation*}
\bbV(B_p(2, 1))_\delta=0 \qa \bbV(B_p(2, 2))_\delta=-1.
\end{equation*}
(Also, we crucially use the fact that $\fh(B_2(2, 2))=1$.)
As a generalization, we will obtain a simple criterion for ``linear independence'' in Section \ref{sec: criterion for linear independence}. Moreover, it turns out that if $D$ is a rational cuspidal divisor constructed from lower level using two degeneracy maps $\alpha_p(N)^*$ and $\beta_p(N)^*$, then most of the entries of $\bbV(D)$ are zeros (cf. Proposition \ref{prop: degeneracy maps revisited}). Thus, it is easy to apply our criterion with such divisors.
\end{observation*}

\subsubsection{$r=3$}
Since the group $\Qdiv {p^3}$ is generated by $C_3(f)$, we need to compute the intersections
\begin{equation*}
\br {\ov{C_3(i)}, ~\ov{C_3(j)}} \cap \br {\ov{C_3(6-i+j)}} ~~\text{ for any } i, j \in \{1, 2, 3 \} \text{ with } i\neq j.
\end{equation*}
Instead, as above we replace $C_3(1)$ by $\beta_p(p^2)^*(C_2(2))=-C_3(1)+p\cdot C_3(3)$. Also, we replace $C_3(2)$ by $\alpha_p(p^2)^*(B_2(2, 1))$.
Thus, let
\begin{equation*}
\begin{split}
B_p(3, 1)&:=\alpha_p(p^2)^*(B_2(2, 1))=p^2(P_0)-p(P_1)-(P_2)-(P_3),\\
B_p(3, 2)&:=(P_0)+(P_1)-p(P_3) ~\text{ and }~ B_p(3, 3):=(P_0)-(P_3).
\end{split}
\end{equation*}
Then it is easy to see that $\Qdiv {p^3}$ is generated by $B_p(3, f)$. Also, since
\begin{equation*}
\begin{split}
\bbV(B_p(3, 1))&=(1, -1, 0, 0),\\
\bbV(B_p(3, 2))&=(0, 1, 0, -1),\\
\bbV(B_p(3, 3))&=(p, -1, 1, -p),
\end{split}
\end{equation*}
and $\fh(B_p(3, 2))=\fh(B_p(3, 3))=1$, we can prove that there is no relation among them. Indeed, you can see $\bbV(B_p(3, f))_{p^2}=0, 0$ or $1$ for $f=1, 2$ or $3$, respectively. So as above we obtain
\begin{equation*}
\br{ \ov{B_p(3, 1)}, ~~ \ov{B_p(3, 2)} } \cap \br {\ov{B_p(3, 3)}} =0.
\end{equation*}
Therefore we have
\begin{equation*}
\scC(p^3) \simeq  \br{ \ov{B_p(3, 1)}, ~~ \ov{B_p(3, 2)} } \moplus \br {\ov{B_p(3, 3)}}.
\end{equation*}
Also, if you check the $p^3$-th entries of $\bbV(B_p(3, 1))$ and $\bbV(B_p(3, 2))$, then we obtain
\begin{equation*}
\br{ \ov{B_p(3, 1)}, ~~ \ov{B_p(3, 2)} } \simeq \br{ \ov{B_p(3, 1)} } \moplus \br {\ov{B_p(3, 2)}}.
\end{equation*}
Thus, we finally have
\begin{equation*}
\scC(p^3) \simeq  \moplus \br {\ov{B_p(3, f)}}.
\end{equation*}

\subsubsection{$r=4$}
Now, we continue as above without hesitation. Let
\begin{equation*}
\begin{split}
B_p(4, 1)&:=\alpha_p(p^3)^*(B_p(3, 1))=p^3(P_0)-p^2(P_1)-(P_2)-(P_3)-(P_4),\\
B_p(4, 3)&:=\beta_p(p^3)^*(B_p(3, 3))=(P_0)+(P_1)-p(P_4), \quad B_p(4, 4)=(P_0)-(P_4).
\end{split}
\end{equation*}
Next, we take $B_p'(4, 2):=\alpha_p(p^3)^*(B_p(3, 3))=p(P_0)-(P_3)-(P_4)$.
Then by direct computation, we have
\begin{equation*}
\begin{split}
\bbV(B_p(4, 1))&=(1, -1, 0, 0, 0), \\
\bbV(B_p'(4, 2))&=(p, -1, 1, -p, 0), \\
\bbV(B_p(4, 3))&=(0, p, -1, 1, -p), \\
\bbV(B_p(4, 4))&=(p, -1, 0, 1, -p).
\end{split}
\end{equation*}
The generation part (the group $\Qdiv {p^4}$ is generated by these divisors) is easy. 
On the other hand, there may exist a relation among $B_p'(4, 2)$, $B_p(4, 3)$ and $B_p(4, 4)$. Thus, we replace $B_p'(4, 2)$ by 
\begin{equation*}
X=\alpha_p(p^3)^*(B_p(3, 2))=p(P_0)+p(P_1)-p(P_3)-p(P_4)
\end{equation*}
so that the $p^2$-th entry of $\bbV(X)$ is zero. Since all the coefficients of $X$ are divisible by $p$ 
(cf. Lemma \ref{lemma: B2 definition}), the generation part is not achieved. Therefore we replace $X$ by $Y=\frac{1}{p}X$, which does not change the vector $\bbV(X)$, and then we can prove the generation part.
Now, by comparing the $p^2$-th entries we can conclude that 
\begin{equation*}
\scC(p^4) \simeq \br {\ov{B_p(4, 1)}, ~\ov{Y}, ~\ov{B_p(4, 4)}} \moplus \br {\ov{B_p(4, 3)}}.
\end{equation*}
However, there is a relation between $Y$ and $B_p(4, 4)$. Indeed, since the greatest common divisor of the entries of 
\begin{equation*}
V=\bbV(Y)+\bbV(B_p(4, 4))=(p, 0, 0, 0, -p)
\end{equation*}
is $p$, we can find a new divisor 
\begin{equation*}
B_p(4, 2):=Y+p^2 \cdot B_p(4, 4)=(p^2+1)(P_0)+(P_1)-(P_3)-(p^2+1)(P_4)
\end{equation*}
so that $\bbV(B_p(4, 2))$ is parallel to $V$. Now, we can prove that $\Qdiv {p^4}$ is generated by $B_p(4, f)$. Also, we have
\begin{equation*}
\begin{split}
\bbV(B_p(4, 1))&=(1, -1, 0, 0, 0), \\
\bbV(B_p(4, 2))&=(1, 0, 0, 0, -1), \\
\bbV(B_p(4, 3))&=(0, p, -1, 1, -p), \\
\bbV(B_p(4, 4))&=(p, -1, 0, 1, -p).
\end{split}
\end{equation*}
Thus, by comparing the $p^2$-th, $p^3$-th and $p^4$-th entries successively, we obtain 
\begin{equation*}
\scC(p^4) \simeq \moplus \br {\ov{B_p(4, f)}}.
\end{equation*}

\subsubsection{$r=5$}
As above, we take
\begin{equation*}
\begin{split}
B_p(5, 1)&:=\alpha_p(p^4)^*(B_p(4, 1))=p^4(P_0)-p^3(P_1)-p(P_2)-(P_3)-(P_4)-(P_5),\\
B_p(5, 2)&:=\beta_p(p^4)^*(B_p(4, 2))=(p^2+1)((P_0)+(P_1)-p(P_5))+(P_2)-p(P_4),\\
B_p(5, 3)&:=\alpha_p(p^4)^*(B_p(4, 4))=p(P_0)-(P_4)-(P_5),\\
B_p(5, 4)&:=\beta_p(p^4)^*(B_p(4, 4))=(P_0)+(P_1)-p(P_5) \qqa B_p(5, 5):=(P_0)-(P_5).\\
\end{split}
\end{equation*}
Then we can show that $\Qdiv {p^5}$ is generated by $B_p(5, f)$. Also, we have
\begin{equation*}
\begin{split}
\bbV(B_p(5, 1))&=(1, -1, 0, 0, 0, 0),\\
\bbV(B_p(5, 2))&=(0, 1, 0, 0, 0, -1),\\
\bbV(B_p(5, 3))&=(p, -1, 0, 1, -p, 0),\\
\bbV(B_p(5, 4))&=(0, p, -1, 0, 1, -p),\\
\bbV(B_p(5, 5))&=(p, -1, 0, 0, 1, -p).
\end{split}
\end{equation*}
Thus, by comparing the $p^3$-th, $p^2$-th, $p^4$-th and $p^5$-th entries successively, we obtain
\begin{equation*}
\scC(p^5) \simeq \moplus \br{\ov{B_p(5, f)}}.
\end{equation*}

\subsubsection{$r=6$}
As in the case of $r=4$, we can take
\small
\begin{equation*}
\begin{split}
B_p(6, 1)&:=\alpha_p(p^5)^*(B_p(5, 1))=p^5(P_0)-p^4(P_1)-p^2(P_2)-(P_3)-(P_4)-(P_5)-(P_6),\\
B_p(6, 2)&:=(p^4+p^2+1)((P_0)-(P_6))+(p^2+1)((P_1)-(P_5))+(P_2)-(P_4),\\
B_p(6, 3)&:=\beta_p(p^5)^*(B_p(5, 4))=(P_0)+(P_1)+(P_2)-p^2(P_6),\\
B_p(6, 4)&:=\alpha_p(p^5)^*(B_p(5, 5))=p(P_0)-(P_5)-(P_6),\\
B_p(6, 5)&:=\beta_p(p^5)^*(B_p(5, 5))=(P_0)+(P_1)-p(P_6) \qqa B_p(6, 6):=(P_0)-(P_6).
\end{split}
\end{equation*}
\normalsize
Note that $B_p(6, 2)$ is equal to $\frac{1}{p}\alpha_p(p^5)^*(B_p(5, 2))+p^4\cdot B_p(6, 6)$. Then we have
\begin{equation*}
\begin{split}
\bbV(B_p(6, 1))&=(1, -1, 0, 0, 0, 0, 0),\\
\bbV(B_p(6, 2))&=(1, 0, 0, 0, 0, 0, -1),\\
\bbV(B_p(6, 3))&=(0, 0, p, -1, 0, 1, -p),\\
\bbV(B_p(6, 4))&=(p, -1, 0, 0, 1, -p, 0),\\
\bbV(B_p(6, 5))&=(0, p, -1, 0, 0, 1, -p),\\
\bbV(B_p(6, 6))&=(p, -1, 0, 0, 0, 1, -p).
\end{split}
\end{equation*}
As above, by comparing the $p^3$-th, $p^4$-th, $p^2$-th, $p^5$-th and $p^6$-th entries successively, we obtain
\begin{equation*}
\scC(p^6) \simeq \moplus  \br{\ov{B_p(6, f)}}.
\end{equation*}


\subsubsection{$r\geq 7$}
As above, we define
\begin{equation*}
B_p(r, 1):=\alpha_p(p^{r-1})^*(B_p(r-1, 1)) \qqa B_p(r, r):=(P_0)-(P_r).
\end{equation*}
Also, if $f=r-2a \geq 3$, then we set
\begin{equation*}
B_p(r, f):=\pi_1(p^r, p^{r-a})^*(B_p(r-a, r-a)).
\end{equation*}
Furthermore, if $f=r+1-2a\geq 3$, then we set
\begin{equation*}
B_p(r, f):=\pi_2(p^r, p^{r-a})^*(p^{r-a})^*(B_p(r-a, r-a)).
\end{equation*}
Finally, we set $B_p(r, 2):=\beta_p(p^{r-1})^*(B_p(r-1, 2))$ if $r$ is odd, and otherwise
\begin{equation*}
B_p(r, 2):=\frac{1}{p} \times \pi_{12}(p^{r-2})^*(B_p(r-2, 2))+p^{r-2}\cdot B_p(r, r).
\end{equation*}
Then we will prove later that
\begin{equation*}
\scC(p^r) \simeq \moplus_{f=1}^r \br {\ov{B_p(r, f)}}.
\end{equation*}

\ms
\subsection{Case of level $2^r$}\label{section: case of level 2r} 
In this subsection, we explain our construction of $B^2(r, f)$ for any $r\geq 5$.

\ms
Let $N=2^r$ for some $r\geq 5$. As above, we hope to prove that
\begin{equation*}
\scC(2^r) \simeq \moplus \br{\ov{B_2(r, f)}}.
\end{equation*}
However, the arguments above (about linear independence) break down because $\fh(B_2(r, f))=2$ for some $f\geq 2$. Indeed, if $r\geq 6$ is even, then we can prove that $\br {\ov{C_1}} \cap \br {\ov{C_2}}  \simeq \zmod 2$, where $C_i=B_2(r, r-i)$. Note that since
\begin{equation*}
V(C_1)=2(0, 2, -1, \dots, 1, -2) \qa V(C_2)=2(2, -1, \dots, 1, -2, 0),
\end{equation*}
the order of $C_i$ is $2^{r-4}$ by Theorem \ref{theorem: computation of order}. If we let $X=C_1+C_2$, then we have
\begin{equation*}
V(X)=V(C_1)+V(C_2)=2(2, 1, -1, \dots, 1, -1, -2).
\end{equation*}
Since $\pw_2(X)=0$, the order of $X$ is $2^{r-5}$, which is one half of the order of $C_i$.
Also, since $\br{\ov{C_1}, \ov{C_2}} = \br{\ov{X}, \ov{C_1}}$, there is a relation between $C_i$. Moreover, we can prove that there is no relation between $X$ and $C_1$, which proves the claim. 
So we need another idea for finding new generators to remove a possible non-trivial intersection resulting from this phenomenon.\footnote{This seems to be a reason why Ling could not find all possible relations among $C_d$.} One may guess if we find new generators $D$ with $\pw_2(D) \in 2\Z$ or even better $\pw_2(D)=0$, then such a phenomenon will not occur. 
In our method, this simple idea will be crucial.

\vspace{2mm}
For simplicity, let $D_f:=C(2^r)_{2^f}$. 
\subsubsection{$r=5$}\label{subsection: p=2, r=5}
By direct computation, we have 
\begin{equation*}
\begin{array}{|c|c|c|c|c|} \hline
f & \bbV(D_f) & \Gcd & \pw_2(D_f) & \text{the order of }D_f\\ \hline
1 & (2, -3, 1, 0, 0, 0) & 2 & -3 & 2\\ \hline
2& (2, 1, -5, 2, 0, 0) & 2 & 3 & 2\\ \hline
3 & (2, -1, 2, -5, 2, 0) & 2 & -6 & 1\\ \hline
4 & (2, -1, 0, 2, -5, 2) & 1 & 3 & 4\\ \hline
5 & (2, -1, 0, 0, 1, -2)& 1 & -3 & 4\\ \hline
\end{array}
\end{equation*}
We can compute the orders of $D_1-2D_5$, $D_2-2D_5$, $D_3$ and $D_4+D_5$, which are all $1$. Thus, we have
\begin{equation*}
\scC(32) \simeq \br {\ov{D_5}} \simeq \zmod 4.
\end{equation*}

\subsubsection{$r=6$}
By direct computation, we have
\begin{equation*}
\begin{array}{|c|c|c|c|c|} \hline
f & \bbV(D_f) & \Gcd & \pw_2(D_f) & \text{the order of }D_f\\ \hline
1 & (2, -3, 1, 0, 0, 0, 0) & 2 & -3 & 4\\ \hline
2& (2, 1, -5, 2, 0, 0, 0) & 2 & 3 & 4 \\ \hline
3 & (2, -1, 2, -5, 2, 0, 0) & 4 & -6 & 1 \\ \hline
4 & (2, -1, 0, 2, -5, 2, 0) & 2 & 3 & 4 \\ \hline
5 & (2, -1, 0, 0, 2, -5, 2)& 1 & -6 & 4\\ \hline
6 & (2, -1, 0, 0, 0, 1, -2) & 1 & 0 & 4 \\ \hline
\end{array}
\end{equation*}
We can show that there is a relation between $D_1$ and $D_2$. Indeed, such a relation can be obtained from level $16$ because the genus of $X_0(16)$ is zero (cf. Lemma \ref{lemma: non-trivial relation in level 2r}).
This is very useful because we can ignore some divisors from the set of generators and our computation heavily relies on the number of generators. As a result, we can ignore the divisors $D_2$ and $D_5$ (and $D_3$ in this case because $\ov{D_3}=0$) from our computation. In other words, we have
\begin{equation*}
\scC(2^6)=\br {\ov{D_1}, ~\ov{D_4}, ~\ov{D_6} }.
\end{equation*}
As above, we can find a relation between $D_1$ and $D_4$, and so we replace the divisor $D_4$ by $D_1+D_4$. Then we can prove that
\begin{equation*}
\scC(2^6) \simeq \br { \ov{D_1}} \moplus \br{ \ov{D_1+D_4}, ~\ov{D_6}}.
\end{equation*}
Since $\fh(D_1+D_4)=1$, by comparing the $2^2$-th entries of $\bbV(D_1+D_4)$ and $\bbV(D_6)$ we further have
\begin{equation*}
\scC(2^6) \simeq \br { \ov{D_1}} \moplus \br{ \ov{D_1+D_4}} \moplus \br {\ov{D_6}}.
\end{equation*}

\subsubsection{$r=7$}\label{subsection: p=2, r=7}
By direct computation, we have
\begin{equation*}
\begin{array}{|c|c|c|c|c|} \hline
f & \bbV(D_f) & \Gcd & \pw_2(D_f) & \text{the order of }D_f\\ \hline
1 & (2, -3, 1, 0, 0, 0, 0, 0) & 2 & -3 & 8 \\ \hline
2& (2, 1, -5, 2, 0, 0, 0, 0) & 2 & 3 & 8 \\ \hline
3 & (2, -1, 2, -5, 2, 0, 0, 0) & 4 & -6 & 2 \\ \hline
4 & (2, -1, 0, 2, -5, 2, 0, 0) & 4 & 3 &  4 \\ \hline
5 & (2, -1, 0, 0, 2, -5, 2, 0)& 2 & -6 & 4\\ \hline
6 & (2, -1, 0, 0, 0, 2, -5, 2)& 1 & 3 & 16 \\ \hline
7 & (2, -1, 0, 0, 0, 0, 1, -2) & 1 & -3 & 16 \\ \hline
\end{array}
\end{equation*}
Again, as above $D_2$ and $D_6$ can be removed from the set of generators. Next, for any $3\leq f \leq 5$, we compute the intersection $\br{ \ov{D_1}} \cap \br{ \ov{D_f}}$, which is $0$ if $f$ is odd, and $\zmod 2$ otherwise. Thus, for even $f$ we replace $D_f$ with $E_f:=D_f+a_f \cdot D_1$ for (some suitable $a_f$) so that the order of $E_f$ is one half of that of $D_f$. (This process is similar to the construction of $X$ at the beginning of the section.)
Then we may insist that there is no relation among $D_1, D_3, E_4, D_5$ and $D_7$. However, since the computation seems highly complicated, for odd $f$ we replace $D_f$ by $E_f:=D_f-a_f \cdot D_1$ (for some suitable $a_f$) so that $\pw_2(E_f)=0$.  Also, we replace $D_1$ by $E_6:=D_1-2D_r$ so that $\pw_2(E_6)=0$. Finally, let $E_7=D_7$. Then we have
\begin{equation*}
\pw_2(E_7)=-3 \not\in 2\Z  \qa \pw_2(E_j)=0 ~~\text{ for all }  3\leq j<7.
\end{equation*}
This implies that 
\begin{equation*}
\br{\ov{E_f} : 3\leq f \leq 7} \simeq \br{\ov{E_f} : 3\leq f \leq 6} \moplus \br{ \ov{E_7}}.
\end{equation*}
By direct computation, we have
\begin{equation*}
\begin{split}
\bbV(E_3) &=(-2, 5, 0, -5, 2, 0, 0, 0),\\
\bbV(E_4)& =(4, -4, 1, 2, -5, 2, 0, 0),\\
\bbV(E_5)& =(-2, 5, -2, 0, 2, -5, 2, 0),\\
\bbV(E_6)& =(0, -2, 1, 0, 0, 0, -1, 2).\\
\end{split}
\end{equation*}
Suppose that there is a relation among $\{ E_f : 3\leq f \leq 6\}$. As above, we can find $b_f \in [0, 1)$ so that the order of $X=\sum_{f=3}^6 b_f \cdot E_f$ is $1$. Since 
${\bf r}(X)=\sum_{f=3}^6 b_f \cdot \bbV(E_f) \in \cS_1(2^7)$,
from the $2^2$-th and $2^6$-th entries we have $b_4 \in \Z$, and hence $b_4=0$. From the $2^3$-th entry, we have $-5b_3 \in \Z$, and so $b_3=0$. (Note that the denominator of $b_f$ is a power of $2$.) Also, from the $2^5$-th entry we have $-5b_5 \in \Z$, and therefore $b_5=0$. Finally, we have $b_6=0$ as well, which is a contradiction. Thus, we have
\begin{equation*}
\scC(2^7) \simeq  \moplus_{f=3}^7 \br{\ov{E_f}}.
\end{equation*}

\subsubsection{$r\geq 8$}
By direct computation, we have
\begin{equation*}
\begin{array}{|c|c|c|c|c|} \hline
f & \bbV(D_f) & \Gcd & \pw_2(D_f) \\ \hline
1 & (2, -3, 1, \dots) & 2 & -3  \\ \hline
2& (2, 1, -5, 2, \dots) & 2 & 3  \\ \hline
3\leq f\leq r-1 & (2, -1, \dots, 2, -5, 2, \dots) & 2^{m(f)-1} & \begin{cases}
3 & \text{ if } ~~f\in 2\Z\\
-6 & \text{ otherwise}
\end{cases}  \\ \hline
r & (2, -1, \dots, 1, -2) & 1 & \begin{cases}
0 & \text{ if } ~~f\in 2\Z\\
-3 & \text{ otherwise}
\end{cases}  \\ \hline
\end{array}
\end{equation*}
where $m(f)=\min(f, r-f)$ and the dots denote zero entries.
As discussed above, for any $3\leq f \leq r-2$, we replace $D_f$ by 
\begin{equation*}
E_f:=\begin{cases}
D_f-2^{m(f)-1}\cdot D_1 & \text{ if $~f$ is odd},\\
D_f+2^{m(f)-2}\cdot D_1 & \text{ if $~f$ is even}.
\end{cases}
\end{equation*}
If $r$ is odd, then we further replace $D_1$ by $D_1-2D_r$. Let
\begin{equation*}
(E_{r-1}, E_r)=\begin{cases}
(D_1-2D_r, D_r) & \text{ if $~r$ is odd},\\
(D_r, D_1) & \text{ if $~r$ is even},
\end{cases}
\end{equation*}
so that $\pw_2(E_r) \not\in 2\Z$ but $\pw_2(E_f)=0$ for all $3\leq f\leq r-1$. Then we have
\begin{equation*}
\scC(2^r) \simeq \br{ \ov{E_f} : 3\leq f\leq r-1} \moplus \br{ \ov{E_r}}.
\end{equation*}
Therefore it suffices to show that there is no relation among $\{E_f :3\leq f\leq r-1\}$. It turns out that we can easily prove (as in the case of $r=7$) that 
\begin{equation*}
\br{ \ov{E_f} : 3\leq f\leq r-1} \simeq \moplus_{f=3}^{r-1} \br{ \ov{E_f}}.
\end{equation*}

\begin{remark}
For an odd integer $f$ with $3\leq f \leq r-2$, the order of $E_f$ is the same as that of $D_f$. Thus, there is no relation among 
\begin{equation*}
\{E_r, E_{r-1}\} \cup  \{D_{2k-1}, E_{2k} : 2\leq k \leq (r-1)/2 \}.
\end{equation*}
It might be a good exercise to prove this directly.
\end{remark}

\ms
\subsection{Case of odd squarefree level}\label{section: example odd squarefree}
Let $N=\prod_{i=1}^t p_i$ be an odd squarefree integer with $t\geq 2$. In this subsection, we discuss our construction of  $Y^0(d)$ on $X_0(N)$ for any non-trivial divisors $d$ of $N$.

In this and next subsections, since all cusps of $X_0(N)$ are defined over $\Q$, we simply write $P_d$ for the rational cuspidal divisor $(P(N)_d)$. We use two orderings
$\prec$ and $\vtl$ on $\cD_N^0=\cD_N^\sqf$ (which are defined in Section \ref{section: the orderings}), and write
\begin{equation*}
\cD_N^0=\{d_1, \dots, d_\fm \}=\{\delta_1, \dots, \delta_\fm \}
\end{equation*}
so that $d_i \prec d_j$ (resp. $\delta_i \vtl \delta_j$) if and only if $i<j$. (Here, $\fm=\#\cD_N^0=2^t-1$.) For simplicity, we take $\delta_0:=1$, and write a vector in $\cS_1(N)$ as $(a_0, \dots, a_\fm)$ so that its $\delta_i$-th entry is $a_i$. Note that the ordering $\vtl$ in this subsection is given by the colexicographic order on $\Delta(t)$, i.e., $\prod_{i=1}^t p_i^{a_i} \vtl \prod_{i=1}^t p_i^{b_i}$ if and only if there is an index $h$ such that $a_h=0$, $b_h=1$ and $a_i=b_i$ for all $i>h$. Thus we have
\begin{equation*}
(\delta_0, \delta_1, \delta_2, \delta_3, \dots, \delta_\fm)=(1, p_1, p_2, p_1p_2, \dots, N).
\end{equation*}

\subsubsection{$t=2$}\label{subsection: odd t=2}
For simplicity, let $p=p_1$ and $q=p_2$ be two odd primes. In this subsection, we explain why we fail to find a basis for $\scC(pq)$ and how to construct the divisors $Y^0(d)$ for $d=p$, $q$ or $pq$.

\vspace{2mm}
First, the group $\Qdiv {pq}$ is generated by $C_p$, $C_q$ and $C_{pq}$. As in Section \ref{section: case of odd prime power}, we may replace $C_{pq}$ by 
\begin{equation*}
X=\alpha_q(p)^*(P_1-P_p)=qP_1+P_q-qP_p-P_{pq}=-C_q+q C_p+C_{pq}.
\end{equation*}
Then by direct computation, we have
\begin{equation}\label{equation: example level pq}
\begin{split}
\bbV(X)&=(1, -1, 0, 0)=(1, -1) \motimes (1, 0),\\
\bbV(C_p)&=(q, -q, -1, 1)=(1, -1) \motimes (q, -1),\\
\bbV(C_q)&=(p, -1, -p, 1)=(p, -1) \motimes (1, -1).\\
\end{split}
\end{equation}

Let $V:=\bbV(C_p)-\bbV(C_q)=(q-p, 1-q, p-1, 0)$. Then the greatest common divisor of the entries of $V$ is $\gcd(p-1, q-1)$, and so there may exist a relation between $C_p$ and $C_q$. Let $Y=g((q+1)C_p-(p+1)C_q)$, where $g=\gcd(p+1, \, q+1)^{-1}$. Then $\bbV(Y)$ is parallel to $V$, and so the divisors $X, Y, C_p$ (or $X, Y, C_q$) are ``good'' to apply our criteria for linear independence. On the other hand, $\Qdiv {pq}$ is not generated by them unless $g(p+1)=1$ (or $g(q+1)=1$). Even in this ``simple'' case, it is very difficult to find ``good'' generators without further assumption on $p$ and $q$. Thus, we study the $\ell$-primary subgroup of $\scC(pq)$ instead. 

Now, let $\ell$ be a given prime. By changing the role of $p$ and $q$ if necessary, we may assume that 
\begin{equation}\label{equation: assumption for level pq}
\tn{val}_\ell(p+1)\geq \tn{val}_\ell(q+1) \qa \tn{val}_\ell(p-1)\leq \tn{val}_\ell(q-1).
\end{equation}
This is always possible because $p$ and $q$ are both odd. Under this assumption, let
\begin{equation*}
D(p, q):=Y=g ((q+1)\cdot C_p-(p+1)\cdot C_q),
\end{equation*}
and we take 
\begin{equation*}
E_1=X, ~ E_2=D(p, q) ~~~\text{ and }~~~ E_3=C_q.
\end{equation*}
Since $g(q+1)$ is an $\ell$-adic unit, $\Qdiv {pq}$ is \textit{$\ell$-adically generated by} $E_i$, which means that 
\begin{equation*}
\textstyle \Qdiv {pq} \motimes_\Z \Z_\ell = \br{ E_i : 1\leq i \leq 3} \motimes_\Z \Z_\ell.
\end{equation*}
By direct computation, we have
\begin{equation*}
\begin{split}
\bbV(E_1)&=(1, -1, 0, 0),\\
\bbV(E_2)&=h(q-p, 1-q, p-1, 0),\\
\bbV(E_3)&=(p, -1, -p, 1),\\
\end{split}
\end{equation*}
where $h=\gcd(p-1, \, q-1)^{-1}$. Since $h(p-1)$ is an $\ell$-adic unit, the matrix  
\begin{equation*}
\fM_0=(\bbV(E_i)_{\delta_j})_{1\leq i, j\leq 3}
\end{equation*}
is \textit{lower $\ell$-unipotent}, which means that $\fM_0$ is a lower-triangular matrix whose diagonal entries are $\ell$-adic units. This is enough to conclude that for odd $\ell$, we have
\begin{equation*}
\scC(pq)[\ell^\infty] \simeq \moplus \br{\ov{E_i}}[\ell^\infty].
\end{equation*}
Moreover, since $p$ is odd, we have $\fh(E_3)=1$ and hence 
\begin{equation*}
\scC(pq)[2^\infty] \simeq \br{\ov{E_1}, ~\ov{E_2}}[2^\infty] \moplus \br{ \ov{E_3}}[2^\infty].
\end{equation*}
Also, since $\pw_q(E_2)\not\in \ell\Z$ (and $\ell=2$) and $\pw_q(E_1)=0$, we finally get
\begin{equation*}
\scC(pq)[2^\infty] \simeq \moplus \br{\ov{E_i}}[2^\infty].
\end{equation*}

\begin{remark}\label{remark: A0 and A1}
As you may see in (\ref{equation: example level pq}), the vectors $\bbV(C)$ are written as tensor products. Indeed, even more is true: we have
\begin{equation*}
\begin{split}
\Phi_{pq}(X)=(1, -1) \motimes (q, 1),\\
\Phi_{pq}(C_p)=(1, -1) \motimes (1, 0),\\
\Phi_{pq}(C_q)=(1, 0) \motimes (1, -1),
\end{split}
\end{equation*}
i.e., $X$, $C_p$ and $C_q$ are all defined by tensors. (However, $D(p, q)$ is not defined by tensors.) In Section \ref{section: case of odd prime power}, we did not define $B_p(1, 1)$ but it seems natural to do by 
\begin{equation*}
B_p(1, 1):=P_1-P_p=0-\infty,
\end{equation*}
which is a ``unique'' degree $0$ divisor of level $p$. By the consideration above, it is natural to define 
\begin{equation*}
A_p(1, 0):=P_0 \qa A_p(1, 1):=p(P_0)+(P_p),
\end{equation*}
which is not of degree $0$ though. Note that if you imagine a cuspidal divisor of level $1$, which is $\infty \in X_0(1)$, then we have $A_p(1, 1)=\alpha_p(1)^*(\infty)$.
\end{remark}

\subsubsection{$t=3$}\label{subsection: odd t=3}
As above, we first fix a prime $\ell$, and want to find rational cuspidal divisors $E_i$ such that the matrix 
\begin{equation*}
\fM_0:=(\bbV(E_i)_{\delta_j})_{1\leq i, j\leq 7}
\end{equation*}
is lower $\ell$-unipotent. Note that $\Qdiv {pqr}$ is generated by $C_{p_1}$, $C_{p_2}$, $C_{p_3}$, $C_{p_1p_2}$, $C_{p_1p_3}$, $C_{p_2p_3}$ and $C_{p_1p_2p_3}$. As in (\ref{equation: assumption for level pq}), we assume that
\begin{equation*}
\tn{val}_\ell(p_1+1)\geq \tn{val}_\ell(p_2+1) \geq \tn{val}_\ell(p_3+1)
\end{equation*}
and
\begin{equation*}
\tn{val}_\ell(p_1-1)\leq \tn{val}_\ell(p_2-1) \leq \tn{val}_\ell(p_3-1).
\end{equation*}
Under this assumption, we keep $C_{p_3}$ unchanged because it is the one with the highest order. As above, we can find a relation between $C_{p_1}$ and $C_{p_2}$. 
So we replace $C_{p_1}$ by $Y=D(p_1, p_2)$ above. Similarly, we replace $C_{p_2}$ by $Y_1=D(p_2, p_3)$. 
(Here, $D(p_i, p_j)$ is a divisor in level $p_ip_j$ but we regard it as a divisor in level $p_1p_2p_3$.) 
If we use the tensor notation, it turns out that
\begin{equation*}
\bY=\bD(p_1, p_2) \motimes \bA_{p_3}(1, 0)  \qa  \bY_1=\bA_{p_1}(1, 0) \motimes \bD(p_2, p_3).
\end{equation*}
Similarly, we consider the vector $\bD(p_1, p_3) \motimes \bA_{p_2}(1, 0)$. By considering all possible combinations of tensor products, and by computing the vectors $\bbV(D)$ for such divisors $D$ defined by tensors (using Theorem \ref{theorem: order defined by tensors}), we find natural candidates for $E_i$ which make the matrix $\fM_0$ lower $\ell$-unipotent as follows.
\begin{equation*}
\begin{array}{|c|c|c|c|} \hline 
i& d_i & \bE_i & \bbV(E_i) \\ \hline
1& p_1p_2p_3 & \bB_{p_1}(1, 1) \motimes \bA_{p_2}(1, 1) \motimes \bA_{p_3}(1, 1) & (1, -1, 0, 0, 0, 0, 0, 0)\\ \hline
2&p_1p_3 & \bD(p_1, p_2) \motimes \bA_{p_3}(1, 1) & (*, *, \star, 0, 0, 0, 0, 0) \\ \hline
3& p_2p_3 & \bA_{p_1}(1, 0) \motimes \bB_{p_2}(1, 1) \motimes \bA_{p_3}(1, 1) & (*, *, *, 1, 0, 0, 0, 0) \\ \hline
4&p_1p_2 & \bD(p_1, p_3) \motimes \bA_{p_2}(1, 1) & (*, *, 0, 0, \star, 0, 0, 0) \\ \hline
5&p_2 & \bA_{p_1}(1, 0) \motimes \bD(p_2, p_3) & (*, *, *, *, *, \star, 0, 0) \\ \hline
6&p_1 & \bD(p_1, p_2) \motimes \bA_{p_3}(1, 0) & (*, *, *, 0, *, *, \star, 0) \\ \hline
7& p_3 & \bA_{p_1}(1, 0) \motimes \bA_{p_2}(1, 0) \motimes \bB_{p_3}(1, 1) & (*, *, *, *, *, *, *, -1) \\ \hline
\end{array}
\end{equation*}
Here, $*$ denotes an arbitrary integer and $\star$ denotes an $\ell$-adic unit. 
Now, we can apply our criteria (for any prime $\ell$) and obtain
\begin{equation*}
\scC(pqr)[\ell^\infty] \simeq \moplus \br{ \ov{E_i}}[\ell^\infty].
\end{equation*}

\subsubsection{$t\geq 4$}\label{subsection: odd t>3}
As above, for a given prime $\ell$ we assume that
\begin{equation}\label{11}
\tn{val}_\ell(p_i+1) \geq \tn{val}_\ell(p_j+1) ~~\text{ for all } ~i<j,
\end{equation}
and
\begin{equation}\label{12}
\tn{val}_\ell(p_i-1) \leq \tn{val}_\ell(p_j-1) ~~\text{ for all }~ i<j.
\end{equation}

As above, for any $I=(f_1, \dots, f_t) \in \Delta(t)$, we construct a rational cuspidal divisor $Y^0(\fp_I)$, or equivalently a vector $\bY^0(\fp_I) \in \cS_2(N)$ by
\begin{equation*}
\bY^0(\fp_I):=\begin{cases}
\motimes_{i=1, \, i\neq m}^t \bA_{p_i}(1, f_i) \motimes \bB_{p_m}(1, 1) & \text{ if }~~ I \in \cE,\\
\motimes_{i=1, \, i\neq m, n}^t \bA_{p_i}(1, f_i) \motimes \bD(p_m, p_n) & \text{ otherwise},
\end{cases}
\end{equation*}
where $m=m(I)$ and $n=n(I)$. (For notation, see Section \ref{section: intro notation}.) 

If we take $E_i=Y^0(d_i)$, then we can prove that $\Qdiv N$ is $\ell$-adically generated by $E_i$ using our assumption (\ref{11}). Also, we can show the matrix 
\begin{equation*}
\fM_0=(\bbV(E_i)_{\delta_j})_{1\leq i, j \leq \fm}
\end{equation*}
is lower $\ell$-unipotent. The property that $\fM_0$ is lower-triangular is obtained by the definition of our orderings $\prec$ and $\vtl$. Also, the property that the diagonal entries of $\fM_0$ are $\ell$-adic units is deduced by our assumption (\ref{12}).
If $\ell$ is odd, they are enough to conclude that 
\begin{equation*}
\scC(N)[\ell^\infty] \simeq  \moplus_{i=1}^{\fm} \br{ \ov{E_i}}[\ell^\infty].
\end{equation*}
Even when $\ell=2$, since $p_i$ are odd, we can prove the above isomorphism.

\begin{remark}
For any $i$, if we write $C_{d_i}=\sum  a(j)\cdot Y^0(d_j)$ with $a(j) \in \Z_\ell$, 
then $a(i)$ is an $\ell$-adic unit, which is the reason behind the definition of $Y^0(d_i)$. 
\end{remark}

\ms
\subsection{Case of even squarefree level}\label{section: example even  squarefree level}
Let $N=\prod_{i=1}^t p_i$ be an even squarefree integer with $t\geq 2$. Also, let $u$ be an integer such that $p_u=2$.

We use the same notation as in the previous section. 
However, the orderings $\prec$ and $\vtl$ are different from the previous ones since $u\geq 1$. For instance, the ordering $\vtl$ in this subsection is given by a twisted colexicographic order on $\Delta(t)$. More precisely, we define $q_1=p_u$, $q_i=p_{i+1}$ if $i<u$, and $q_j=p_j$ if $j>u$. Then the ordering $\vtl$ is the colexicographic order given by $q_i$, i.e., $\prod_{i=1}^t q_i^{a_i} \vtl \prod_{i=1}^t q_i^{b_i}$ if and only if there is an index $h$ such that $a_h=0$, $b_h=1$ and $a_i=b_i$ for all $i>h$. Thus we have
\begin{equation*}
(\delta_0, \delta_1, \delta_2, \delta_3, \dots, \delta_{\fm})=\begin{cases}
(1, 2, p_1, 2p_1, \dots, N) & \text{ if } ~~u\geq 2,\\
(1, 2, p_2, 2p_2, \dots, N) & \text{ if } ~~u=1.
\end{cases}
\end{equation*}

\subsubsection{$t=2$}
Let $N=2p$. As in Section \ref{subsection: odd t=2}, we replace the divisor $C_{2p}$ by 
\begin{equation*}
U_1:=\alpha_p(2)^*(0-\infty)=pP_1+P_p-pP_2-P_{2p}=-C_p+pC_2+C_{2p},
\end{equation*}
which is of order $1$. (So we can ignore it.) Also, we use the divisor
\begin{equation*}
D(p, 2)=g(3\cdot C_p-(p+1) \cdot C_2) ~~\text{ or } ~~D(2, p)=-D(p, 2)
\end{equation*}
depending on the $\ell$-adic valuation of $p+1$, where $g=\gcd(3, \, p+1)^{-1}$. 
\begin{itemize}
\item
Case 1: Suppose that $\tn{val}_\ell(3) \geq \tn{val}_\ell(p+1)$. Then we have
\begin{equation*}
\textstyle \br{C_p, D(2, p)} \motimes_\Z \Z_\ell = \br{ C_2, C_p} \motimes_\Z \Z_\ell.
\end{equation*}
By direct computation as in Section \ref{subsection: odd t=2}, we have
\begin{equation*}
\begin{split}
\bbV(D(2, p))&=(p-2, 1-p, 1, 0),\\
\bbV(C_p)&=(2, -1, -2, 1).\\
\end{split}
\end{equation*}
If $\ell$ is odd, then we have
\begin{equation*}
\scC(2p)[\ell^\infty] \simeq \br{\ov{D(2, p)}}[\ell^\infty]  \moplus \br{ \ov{C_p}}[\ell^\infty].
\end{equation*}
But there is a problem\footnote{In fact, this is not a problem in this simple case because $\ell$ cannot be $2$. However, we consider this case as if $\ell=2$ since a similar problem occurs when $\tn{val}_2(N)$ is large enough.} if $\ell=2$: $\pw_p(C_p)=-\pw_p(D(2, p))=-1$. To fix this, we retreat our modification, and consider the divisor $C_2$ instead. Then we have $\pw_p(C_2)=0$ and hence we can prove that
\begin{equation*}
\scC(2p)[2^\infty] \simeq \br{ \ov{C_2}}[2^\infty] \moplus \br{ \ov{C_p}}[2^\infty].
\end{equation*}

\item
Case 2: Suppose that $\tn{val}_\ell(p+1) > \tn{val}_\ell(3)$.
Then we have 
\begin{equation*}
\textstyle \br{ C_2, D(p, 2)} \motimes_\Z \Z_\ell = \br{ C_2, C_p} \motimes_\Z \Z_\ell.
\end{equation*}
By direct computation, we have
\begin{equation*}
\begin{split}
\bbV(D(p, 2))&=(2-p, p-1, -1, 0),\\
\bbV(C_2)&=(p, -p, -1, 1).\\
\end{split}
\end{equation*}
Since $p$ is odd, $\fh(C_2)=1$. Thus, for any prime $\ell$ we have
\begin{equation*}
\scC(2p)[\ell^\infty] \simeq \br{\ov{C_2}}[\ell^\infty]\moplus \br{\ov{D(p, 2)}}[\ell^\infty].
\end{equation*}
\end{itemize}

 \subsubsection{$t=3$}
Let $N=2pq$. Fix a prime $\ell$ and assume that
\begin{equation*}
\tn{val}_\ell(p+1) \geq \tn{val}_\ell(q+1)  \qa  \tn{val}_\ell(p-1) \leq \tn{val}_\ell(q-1).
\end{equation*}
As above, we can replace the divisor $C_{2pq}$ by 
\begin{equation*}
U_2:=\pi_1(2pq, 2)^*(0-\infty)=\textstyle\sum_{d\mid pq} d(P_{d'}-P_{2d'})=C_{2pq}-C_{pq}+\cdots,
\end{equation*} 
where $d'=pq/d$. Since the order of $U_2$ is $1$, we ignore it. 
Let $\gamma=\frac{p-1}{\gcd(p-1, \, q-1)} \not\in \ell\Z$.

\begin{itemize}
\item
Case 1: Suppose that $\tn{val}_\ell(3)\geq \tn{val}_\ell(p+1)$. 
We take $p_1=2$, $p_2=p$ and $p_3=q$. 
Also, we take the divisors $E_i$ as in Section \ref{subsection: odd t=3}. 
Since our ordering $\vtl$ is the same as the previous one, the matrix $\fM_0$ is lower $\ell$-unipotent. Thus, we obtain the result for an odd prime $\ell$.

Now, suppose that $\ell=2$. Then a problem occurs as in Case 1 of the previous subsection. So we retreat our modification for the divisor $E_2$, i.e., replace $\bE_2$ by $\bX_1=\bB_{p_1}(1, 1) \motimes \bA_{p_2}(1, 0) \motimes \bA_{p_3}(1, 1)$.
As in Section \ref{subsection: odd t=3}, we have
\begin{equation*}
\begin{array}{|c|c|c|c|} \hline 
i& d_i & \bE_i & \bbV(E_i) \\ \hline
2& p_1p_3& \bB_{p_1}(1, 1) \motimes \bA_{p_2}(1, 0) \motimes \bA_{p_3}(1, 1) & (p, -p, -1, 1, 0, 0, 0, 0) \\ \hline
3& p_2p_3 & \bA_{p_1}(1, 0) \motimes \bB_{p_2}(1, 1) \motimes \bA_{p_3}(1, 1) & (2, -1, -2, 1, 0, 0, 0, 0) \\ \hline
4&p_1p_2 & \bD(p_1, p_3) \motimes \bA_{p_2}(1, 1) & (q-2, 1-q, 0, 0, 1, 0, 0, 0) \\ \hline
5& p_2 & \bA_{p_1}(1, 0) \motimes \bD(p_2, p_3) & (*, *, *, *, 2\gamma, -\gamma, 0, 0) \\ \hline
6& p_1 & \bD(p_1, p_2) \motimes \bA_{p_3}(1, 0) & (*, *, *, 0, *, *, -1, 0) \\ \hline
7& p_3 & \bA_{p_1}(1, 0) \motimes \bA_{p_2}(1, 0) \motimes \bB_{p_3}(1, 1) & (*, *, *, *, *, *, 2, -1) \\ \hline
\end{array}
\end{equation*}
Still, we have a problem because $\pw_{p_3}(E_4)$ and $\pw_{p_3}(E_5)$ are both odd. Thus, we replace $\bE_4$ by $\bX_2=\bB_{p_1}(1, 1) \motimes \bD(p_2, p_3)$ so that $\pw_{p_3}(E_4)=0$.
Then we can prove that
\begin{equation*}
\scC(2pq)[2^\infty] \simeq \moplus \br{ \ov{E_i}} [2^\infty].
\end{equation*}

\item
Case 2: Suppose that\footnote{So $\ell=3$, but we also consider the case $\ell=2$ for better understanding of the problem.} $\tn{val}_\ell(p+1)>\tn{val}_\ell(3) \geq \tn{val}_\ell(q+1)$. 
We take $p_1=p$, $p_2=2$ and $p_3=q$. 
Also, we take the divisors $E_i$ as in Section \ref{subsection: odd t=3}. 
Then the matrix $\fM_0$ is not lower $\ell$-unipotent, and so a modification is necessary. A problematic element is either $E_5$ or $E_6$.  We replace $\bE_6$ by $\bW=\bD(p_1, p_3) \motimes \bA_{p_2}(1, 0)$
and swap the role of $E_5$ and $E_6$. Then we have
\begin{equation*}
\begin{array}{|c|c|c|c|} \hline 
i & d_i & \bE_i & \bbV(E_i)\\ \hline
2 & p_1p_3 & \bD(p_1, p_2) \motimes \bA_{p_3}(1, 1) & (*, *, -1, 0, 0, 0, 0, 0) \\ \hline
3 & p_2p_3 & \bA_{p_1}(1, 0) \motimes \bB_{p_2}(1, 1) \motimes \bA_{p_3}(1, 1) & (p, -p, -1, 1, 0, 0, 0, 0) \\ \hline
4 & p_1p_2 & \bD(p_1, p_3) \motimes \bA_{p_2}(1, 1) & (*, *, *, *, \gamma, 0, 0, 0) \\ \hline
5 & p_1 & \bD(p_1, p_3) \motimes \bA_{p_2}(1, 0) & (*, *, *, *, 2\gamma, -\gamma, 0, 0) \\ \hline
6 & p_2 & \bA_{p_1}(1, 0) \motimes \bD(p_2, p_3) & (*, *, *, *, *, *, -1, 0) \\ \hline
7 & p_3 & \bA_{p_1}(1, 0) \motimes \bA_{p_2}(1, 0) \motimes \bB_{p_3}(1, 1) & (*, *, *, *, *, *, *, -1) \\ \hline
\end{array}
\end{equation*}
Unfortunately, we have $\pw_{p_3}(E_5)=\pw_{p_3}(E_4)=\gamma \not\in \ell\Z$, so our argument breaks down if $\ell=2$. We replace $\bE_4$ by $\bX_3=\bD(p_1, p_3) \motimes \bB_{p_2}(1, 1)$ so that $\pw_{p_3}(E_4)=0$. Then for any prime $\ell$, we can prove
\begin{equation*}
\scC(2pq)[\ell^\infty] \simeq \moplus \br{ \ov{E_i}} [\ell^\infty].
\end{equation*}

\item
Case 3: Suppose that $\tn{val}_\ell(q+1) > \tn{val}_\ell(3)$.
We take $p_1=p$, $p_2=2$ and $p_3=q$. 
Also, we take the divisors $E_i$ as in Section \ref{subsection: odd t=3}.  
As above, we replace $\bE_3$ by $\bU_3=\bA_{p_1}(1, 0) \motimes \bA_{p_2}(1, 1) \motimes \bB_{p_3}(1, 1)$.
By swapping the role of $E_2$ (resp. $E_5$) and $E_4$ (resp. $E_6$), we have
\begin{equation*}
\begin{array}{|c|c|c|c|} \hline 
i & d_i & \bE_i & \bbV(E_i) \\ \hline
2 & p_1p_2 & \bD(p_1, p_3) \motimes \bA_{p_2}(1, 1) & (*, *, -1, 0,  0, 0, 0, 0) \\ \hline
3 & p_2p_3 & \bA_{p_1}(1, 0) \motimes \bA_{p_2}(1, 1) \motimes \bB_{p_3}(1, 1) & (p, -p, -1, 1, 0, 0, 0, 0) \\ \hline
4 & p_1p_3 & \bD(p_1, p_2) \motimes \bA_{p_3}(1, 1) & (*, *, *, *, \gamma, 0, 0, 0) \\ \hline
5 & p_1 & \bD(p_1, p_2) \motimes \bA_{p_3}(1, 0) & (*, *, *, *, 2\gamma, -\gamma, 0, 0) \\ \hline
6 & p_2 & \bA_{p_1}(1, 0) \motimes \bD(p_2, p_3) &  (*, *, *, *, *, *, 1, 0) \\ \hline
7 & p_3 & \bA_{p_1}(1, 0) \motimes \bA_{p_2}(1, 0) \motimes \bB_{p_3}(1, 1) & (*, *, *, *, *, *, 1,-1)  \\ \hline
\end{array}
\end{equation*}
As we can see, the matrix $\fM_0$ is lower $\ell$-unipotent. 
However, as above we have $\pw_{p_2}(E_5)=\pw_{p_2}(E_4)=\gamma \not\in 2\Z$. Thus, we replace $\bE_4$ by $\bX_4=\bD(p_1, p_2) \motimes \bB_{p_3}(1, 1)$
so that $\pw_{p_2}(E_4)=0$. Then we can prove
\begin{equation*}
\scC(2pq)[\ell^\infty] \simeq \moplus \br{ \ov{E_i}} [\ell^\infty].
\end{equation*}
\end{itemize}

\begin{remark}
In Case 2, if we write $p$ as $\fp_I$ for some $I \in \Delta(3)$, then we have $I=(1, 0, 0)$ in our ordering of the prime divisors of $N$. Note that a new vector $\bW$ for such an $I$ is constructed as if positions of $p$ and $2$ are changed.
More precisely, if we rename $p_1=2$ and $p_2=p$ then our (previous) construction for $p=\fp_J$ with $J=(0, 1, 0) \in \Delta(t)$ is exactly $\bW$.\footnote{A similar idea is used for the definition of $\bU_i$. For instance, if we swap the role of $p_2$ and $p_3$ in Case 3, then our (previous) construction for $p_2p_3$ is exactly $\bU_3$.}
After this modification, a problematic one is exactly the one obtained by replacing the vector $\bA_2(1, 0)$ in $\bW$ by $\bA_2(1, 1)$. Since
$\bA_2(1, 1)=-\bB_2(1, 1)+2\bA_2(1, 0)$, we may replace $\bA_2(1, 1)$ by $\bB_2(1, 1)$ and obtain a new vector $\bX_3$. A similar idea is used for the definition of $\bX_i$.
\end{remark}

\subsubsection{$t\geq 4$}\label{subsection: even t>3}
As in Section \ref{subsection: odd t>3}, we hope to assume that
\begin{equation}\label{13}
\tn{val}_\ell(p_i+1) \geq \tn{val}_\ell(p_j+1) ~~\text{ for all }~ i<j,
\end{equation}
and
\begin{equation}\label{14}
\tn{val}_\ell(p_i-1) \leq \tn{val}_\ell(p_j-1) ~~\text{ for all }~ i<j.
\end{equation}
If $\ell=2$, then we cannot make both assumptions together because $N$ is even. Since it is not difficult to make the matrix $\fM_0$ lower $\ell$-unipotent, and since it seems difficult to prove that $\Qdiv N$ is $\ell$-adically generated by $E_i$, we keep the first assumption (\ref{13}), and modify the second assumption (\ref{14}) by
\begin{equation}\label{15}
\tn{val}_2(p_i-1) \leq \tn{val}_2(p_j-1)  ~~\text{ for all }~ i<j \text{ different from } u.
\end{equation}

\vspace{2mm}
Now, we try to construct divisors $E_i$ so that the matrix $\fM_0$ is lower $\ell$-unipotent. Let $d_i=\fp_I$ for some $I=(f_1, \dots, f_t) \in \Delta(t)$. (For notation, see Section \ref{section: intro notation}.) As the case of $\bU_i$, if $I \in \cE$ and $m(I)<u$, then we replace $\bE_i$ by 
\begin{equation*}
\motimes_{i=1, \, i\neq u}^t \bA_{p_i}(1, f_i) \motimes \bB_{p_u}(1, 1).
\end{equation*}
Then there are some problematic elements, which are exactly those for $I \in \cH_u$, and 
as the case of $\bW$ we replace $\bE_i$ by
\begin{equation*}
\motimes_{i=1, \, i\neq m, k}^t \bA_{p_i}(1, f_i) \motimes \bD(p_m, p_k).
\end{equation*}
Then we can indeed prove that the matrix $\fM_0$ is lower $\ell$-unipotent. So $E_i$ are the ones we are looking for if $\ell$ is odd. On the other hand, if $\ell=2$, then we cannot apply our criteria for linear independence, so we will replace some problematic elements as follows.  
\begin{itemize}
\item
Case 1: $u=s\geq 2$. In this case, problematic elements are exactly those for $I \in \cF'_s$, i.e., $I=E_s(n)$ for some $n\in \cI_s$. By direct computation, if $d_i=\fp_{E_s(n)}$ for some $n \in \cI_s$, then we have
\begin{equation*}
d_{i-1}=\fp_{E(n)}, ~\delta_{i-1}=\fp_{F(n)} \text{ and } \delta_i=\fp_{F_s(n)}.
\end{equation*}
Also, we have $\bE_{i-1}=\motimes_{i=2, \,i\neq n}^t \bA_{p_i}(1, 1) \motimes \bD(p_1, p_n)$ and 
\begin{equation*}
\pw_{p_n}(E_i)=\pw_{p_n}(E_{i-1}) \not\in 2\Z \text{ and } \pw_{p_n}(E_j)=0 \text{ for all } j<i-1.
\end{equation*}
Thus, we replace $\bE_{i-1}$ by
\begin{equation*}
\motimes_{i=2, \,i\neq n, s}^t \bA_{p_i}(1, 1) \motimes \bD(p_1, p_n) \motimes \bB_{p_s}(1, 1)
\end{equation*}
so that $\pw_{p_n}(E_{i-1})=0$ (cf. the cases of $\bX_3$ and $\bX_4$). 
\item
Case 2: $u=s=1$. As in Case 1, we replace $E_{i-1}$ by the following
\begin{equation*}
\motimes_{i=3, \,i\neq n}^t \bA_{p_i}(1, 1) \motimes \bB_{p_1}(1, 1) \motimes \bD(p_2, p_n) 
\end{equation*}
so that $\pw_{p_n}(E_{i-1})=0$ (cf. the case of $\bX_2$). Still, there is a problematic element, which is exactly $E_3$. Motivated by $\bX_1$, we replace $\bE_2$ by
\begin{equation*}
\bB_{p_1}(1, 1) \motimes \bA_{p_2}(1, 0) \motimes_{i=3}^t \bA_{p_i}(1, 1).
\end{equation*}
\end{itemize}
After such modifications, we can finally prove that for any prime $\ell$
\begin{equation*}
\scC(N)[\ell^\infty] \simeq \moplus_{i=1}^{\fm} \br{\ov{E_i}}[\ell^\infty].
\end{equation*}
(In fact, $\br{\ov{E_1}}$ is not necessary as it is trivial.)

\ms
\subsection{The definition of $Z(d)$ and $Y^2(d)$}\label{section: definition of Zd}
Let $N=\prod_{i=1}^t p_i^{r_i}$ be the prime factorization of $N$.
Motivated by the previous study (cf. (\ref{13}) and (\ref{15})), we make the following assumption.

\begin{assumption}[Assumption \ref{assumption 1.14}]\label{assumption: chapter4}
For a given prime $\ell$, by appropriately ordering the prime divisors of $N$, we assume the following.
\begin{equation}\label{equation: assumption 1}
\tn{val}_\ell(\gamma_i) \geq \tn{val}_\ell(\gamma_j) ~\text{ for any } 1\leq i < j \leq t, \text{ where } \gamma_i:=p_i^{r_i-1}(p_i+1).
\end{equation}
If $N$ is odd, we set $u=0$, and we define $u$ as the smallest positive integer such that $p_u=2$ if $N$ is even. If $\ell$ is odd, we set $s=0$, and if $\ell=2$ then we set $s=u$. We further assume that
\begin{equation}\label{equation: assumption 2}
\tn{val}_\ell(p_i-1) \leq \tn{val}_\ell(p_j-1) ~\text{ for any } 1\leq i < j \leq t \text{ different from } s.
\end{equation}
\end{assumption}

\vspace{2mm}
In this subsection, we define rational cuspidal divisors $Z(d)$ on $X_0(N)$ for any non-trivial divisors $d$ of $N$. 
To do so, we first define various vectors in $\cS_2(p^r)$ for a prime $p$ and an integer $r\geq 1$.

\begin{definition}\label{definition: inductive definition for A and B}
Let $\alpha(i)$ and $\beta(i)$ be the maps from $\cS_2(p^{r-i})$ to $\cS_2(p^r)$ defined by the degeneracy maps $\pi_1(p^r, \, p^{r-i})^*$ and $\pi_2(p^r, \, p^{r-i})^*$, respectively. For instance, $\alpha(i)$ is the composition of the maps:
\begin{equation*}
\xyh{6}
\xymatrix{
\Qdivv {p^{r-i}} \ar[r]_-{\pi_1(p^r, \, p^{r-i})^*} & \Qdivv {p^r} \ar[d]^-{\Phi_{p^r}} \\
\cS_2(p^{r-i}) \ar[u]^-{\Phi_{p^{r-i}}^{-1}} \ar@{-->}[r]^-{\alpha(i)} & \cS_2(p^r).
}
\end{equation*}

Also, let $\gamma$ be the map from $\cS_2(p^{r-2})$ to $\cS_2(p^r)$ defined by the degeneracy map $\frac{1}{p}\times \pi_{12}(p^{r-2})^*$ (which is well-defined by Lemma \ref{lemma: B2 definition}).

For any $0 \leq f \leq r$,  we define a vector $\bA_p(r, f)$ in $\cS_2(p^r)$ as follows:
\begin{equation*}
\bA_p(r, 0)_{p^j}:=\begin{cases}
1 & \text{ if } ~~j=0,\\
0 & \text{ otherwise}, \end{cases}  \qqa \bA_p(1, 1)_{p^j}:=\begin{cases}
p & \text{ if } ~~j=0,\\ 
1 & \text{ if } ~~j=1.\\
\end{cases}
\end{equation*}
Let $r\geq 2$. We then define
\begin{equation*}
\bA_p(r, r)_{p_j}:=\begin{cases}
1 & \text{ if } ~~j=0,\\
-1 & \text{ if } ~~j=r, \\
0 & \text{ otherwise}.
\end{cases}
\end{equation*}
Also, for any $1\leq f \leq r-1$, we define
\begin{equation*}
\bA_p(r, f):=\begin{cases}
\alpha(r-1)(\bA_p(1, 1)) & \text{ if }~~ f=1,\\
\beta(1)(\bA_p(r-1, 2)) & \text{ if }~~ f=2 \text{ and } r \not\in 2\Z, \\
\gamma(\bA_p(r-2, 2))+p^{r-2}\cdot \bA_p(r, r) & \text{ if }~~ f=2 \text{ and } r \in 2\Z, \\
\alpha(j)(\bA_p(r-j, r-j)) & \text{ if }~~ f=r-2j\geq 3,\\
\beta(j)(\bA_p(r-j, r-j)) & \text{ if }~~ f=r+1-2j\geq 3.\\
\end{cases}
\end{equation*}

Finally, for any $r\geq 1$ and $1\leq f \leq r$, we define a vector $\bB_p(r, f)$ in $\cS_2(p^r)^0$ by
\begin{equation*}
\bB_p(r, f) :=\begin{cases}
p^{r-1}(p+1) \cdot \bA_p(r, 0)-\bA_p(r, 1) & \text{ if } ~~f=1, \\
\bA_p(r, f)  & \text{ otherwise}.
\end{cases}
\end{equation*}
\end{definition}
\begin{remark}\label{remark: 4.8}
By definition, we easily have $\bA_p(r, 1)=\alpha(r)(1)$. 
Let $M>1$ be an integer relatively prime to $p$. For simplicity, let
\begin{equation*}
\pi_1^*=\pi_1(Mp^r, M)^* : \Qdivv {M} \to \Qdivv {Mp^r}.
\end{equation*}
Then by Lemma \ref{lemma: degeneracy maps pullback on RCD} and Remark \ref{remark: level 1, level M}, for any divisor $C \in \Qdivv M$ we easily have
\begin{equation*}
\Phi_{Mp^r}(\pi_1^*(C))=\Phi_M(C) \motimes \bA_p(r, 1).
\end{equation*}
\end{remark}

For $p=2$ and $r\geq 5$, we define another vector in $\cS_2(2^r)$ as follows.
\begin{definition}
For any $3\leq f \leq r-2$, we define
\begin{equation*}
\bE_f:=\begin{cases}
(0, 2^{m(f)-1}, \bbO_{f-2}, -1, \bbO_{r-f}) & \text{ if $~f$ is odd},\\
(3\cdot 2^{m(f)-2}, -2^{m(f)-2}, \bbO_{f-2}, -1, \bbO_{r-f}) & \text{ if $~f$ is even},
\end{cases}
\end{equation*}
where $m(f)=\min(f, \, r-f)$ and $\bbO_k$ is the zero vector of size $k$. Also, we define
\begin{equation*}
\bB^2(r, f):=\begin{cases}
\bE_{\frac{r+f-2}{2}} & \text{ if } ~~3\leq f \leq r-2 ~\text{ and }~ r-f \in 2\Z,\\
\bE_{\frac{r-f+3}{2}} & \text{ if }~~ 3\leq f \leq r-2 ~\text{ and }~ r-f \not\in 2\Z,\\
(1, -1, \bbO_{r-1}) & \text{ if } ~~f=r-1 ~\text{ and }~ r \in 2\Z,\\
(-1, -1, \bbO_{r-2}, 2) & \text{ if }~~ f=r-1 ~\text{ and }~ r \not\in 2\Z,\\
(1, \bbO_{r-1}, -1) & \text{ if }~~ f=r.
\end{cases}
\end{equation*}
\end{definition}

Now, we define vectors $\Z(d)$ for any non-trivial divisors $d$ of $N$. Let $d=\fp_I$ for some $I=(f_1, \dots, f_t) \in \Omega(t)$. First, we construct $\Z'(\fp_I) := \motimes_{i=1}^t \bA_{p_i}(r_i, f_i)$.
Then it is not of degree $0$ if $I \in \Delta(t)$. Thus, we replace them and define
\begin{equation*}
\Z^1(\fp_I):=\begin{cases}
\motimes_{i=1}^t \bA_{p_i}(r_i, f_i) & \text{ if }~~ I \in \square(t),\\
\motimes_{i=1, \, i\neq m}^t \bA_{p_i}(r_i, f_i) \motimes \bB_{p_m}(r_m, 1) & \text{ if } ~~I \in \Delta(t),
\end{cases}
\end{equation*}
where $m=m(I)$. Now, it seems like we are almost done, but as in Section \ref{section: case of level 2r} we have to replace some, so finally we define 
\begin{equation*}
\Z(\fp_I):=\begin{cases}
\motimes_{i=1}^t \bA_{p_i}(r_i, f_i)  & \text{ if } ~~I \in \square(t) \sm \cT_u,\\
\motimes_{i=1, \, i\neq u}^t \bA_{p_i}(r_i, 1) \motimes \bB^2(r_u, f_u) & \text{ if }~~ I \in \cT_u,\\
\motimes_{i=1, \, i\neq m}^t \bA_{p_i}(r_i, f_i) \motimes \bB_{p_m}(r_m, 1) & \text{ if }~~ I \in \Delta(t).\\
\end{cases}
\end{equation*}

\vspace{2mm}
\begin{remark}\label{remark: relation B2 and E}
Note that we easily have
\begin{equation*}
\{\bB^2(r, f) : 3 \leq f \leq r-2 \} = \{ \bE_k: 3\leq k \leq r-2 \}.
\end{equation*}
Thus, if we set
\begin{equation*}
\bE_{r-1}:= \bB^2(r, r-1+\epsilon) \qqa \bE_r:=\bB^2(r, r-\epsilon),
\end{equation*}
where $\epsilon=\frac{1+(-1)^r}{2}$,
then we have
\begin{equation*}
\{ \bB^2(r, f) : 3\leq f \leq r \} = \{ \bE_k : 3\leq k \leq r\}.
\end{equation*}
In other words, the vectors $\bB^2(r, f)$ are just rearrangements of $\bE_k$. Note that 
the definition of $B^2(r, f)$ is designed so that $\Gcd(B_2(r, f))=\Gcd(B^2(r, f))$ for any $3\leq f \leq r$, which can be checked directly. Thus, by Theorem \ref{theorem: order defined by tensors} we have
\begin{equation*}
\Gcd(Z^1(d))=\Gcd(Z(d)) ~~~~\text{ for any } d \in \cD_N^0.
\end{equation*}
\end{remark}

\vspace{2mm}
Next, we define rational cuspidal divisors $Y^2(d)$ for any non-trivial squarefree divisors $d$ of $N$. 
\begin{definition}
For any $1\leq i<j \leq t$, let
\small
\begin{equation*}
\bD(p_i^{r_i}, p_j^{r_j}):=\frac{\gamma_j}{\gcd(\gamma_i, \gamma_j)} \cdot \bB_{p_i}(r_i, 1) \motimes \bA_{p_j}(r_j, 0)-\frac{\gamma_i}{\gcd(\gamma_i, \gamma_j)} \cdot \bA_{p_i}(r_i, 0) \motimes \bB_{p_j}(r_j, 1),
\end{equation*}\normalsize
where $\gamma_i:=p_i^{r_i-1}(p_i+1)$. For $I=(f_1, \dots, f_t) \in \Delta(t)$ with $m=m(I)$, $n=n(I)$ and $k=k(I)$, we define $\bY^2(\fp_I) \in \cS_2(N)^0$ as follows. 
\begin{equation*}
\bY^2(\fp_I):=\begin{cases}
\motimes_{i=1, \, i\neq x}^t \bA_{p_i}(r_i, f_i) \motimes \bB_{p_x}(r_x, 1) & \text{ if } ~~I \in \cE,\\
\motimes_{i=1, \, i\neq y, s, n}^t \bA_{p_i}(r_i, 1) \motimes \bB_{p_s}(r_s, 1) \motimes \bD(p_y^{r_y}, p_n^{r_n})  & \text{ if } ~~I \in \cF_s,\\
\motimes_{i=3}^t \bA_{p_i}(r_i, 1) \motimes \bB_{p_1}(r_1, 1) \motimes \bA_{p_2}(r_2, 0)  & \text{ if } ~~I \in \cG_s,\\
\motimes_{i=1, \, i\neq m, k}^t \bA_{p_i}(r_i, f_i) \motimes \bD(p_m^{r_m}, p_k^{r_k})  & \text{ if } ~~I \in \cH_u,\\
\motimes_{i=1, \, i\neq m, n}^t \bA_{p_i}(r_i, f_i) \motimes \bD(p_m^{r_m}, p_n^{r_n}) & \text{ otherwise},
\end{cases}
\end{equation*}
where $x=\max(m, u)$ and $y=\max(1, 3-s)$.
\end{definition}

\begin{remark}\label{remark: Y0 Y1 Y2}
As we have already seen in Section \ref{subsection: odd t>3}, the definition of $Y^0(\fp_I)$ is simple. In fact, the divisors $Y^2(\fp_I)$ are constructed in three steps as in Section \ref{subsection: even t>3}. First, let
\begin{equation*}
\bY^0(\fp_I):=\begin{cases}
\Z(\fp_I) & \text{ if }~~ I \in \cE,\\
\motimes_{i=1, \, i\neq m, n}^t \bA_{p_i}(r_i, f_i) \motimes \bD(p_m^{r_m}, p_n^{r_n}) & \text{ otherwise}.
\end{cases}
\end{equation*}

Next, we replace some problematic elements when $u\geq 1$, and let
\begin{equation*}
\bY^1(\fp_I):=\begin{cases}
\motimes_{i=1, \, i\neq m}^t \bA_{p_i}(r_i, f_i) \motimes \bB_{p_x}(r_x, 1) & \text{ if } ~~I \in \cE,\\
\motimes_{i=1, \, i\neq m, k}^t \bA_{p_i}(r_i, f_i) \motimes \bD(p_m^{r_m}, p_k^{r_k})  & \text{ if }~~ I \in \cH_u,\\
\bY^0(\fp_I) & \text{ otherwise}.
\end{cases}
\end{equation*}

Finally, when $N$ is even and $\ell=2$, i.e., $s=u\geq 1$, 
we get $\bY^2(\fp_I)$ as above by replacing the elements $\bY^1(\fp_I)$ for $I \in \cF_s \cup \cG_s$. (If $s=0$, then $\cF_s=\cG_s=\emptyset$, and so we have $\bY^1(\fp_I)=\bY^2(\fp_I)$.)
\end{remark}

\vspace{10mm}
\section{Strategy for the computation}\label{chapter5}
In this section, we sketch our idea for computing the group $\scC(N)$ and provide necessary tools  for later use. 

\ms
\subsection{Criteria for linear independence}\label{sec: criterion for linear independence}
Throughout this section, let $C_i \in \Qdiv N$ for any $1\leq i \leq k$. Inspired by our investigation in section \ref{chapter4}, we have the following.

\begin{theorem}\label{thm: criterion 1}
Suppose that there is a divisor $\delta$ of $N$ such that
\begin{equation*}
|\bbV(C_k)_\delta|=1 \qa \bbV(C_i)_\delta=0 ~~ \text{ for all }~  i < k.
\end{equation*}
Suppose further that either $\fh(C_k)=1$ or there is a prime $p$ such that
\begin{equation*}
\pw_p(C_k) \not\in 2\Z \qa \pw_p(C_i)=0 ~~\text{ for all }~ i < k.
\end{equation*}
Then we have
\begin{equation*}
\lla \ov{C_i} : 1\leq i \leq k\rra \simeq \lla\ov{C_i} : 1\leq i \leq k-1\rra \moplus \lla\ov{C_k} \rra.
\end{equation*}
\end{theorem}

\begin{theorem}\label{thm: criterion 2}
Suppose that there is a divisor $\delta$ of $N$ such that
\begin{equation*}
\bbV(C_k)_\delta \in \Z_\ell^\times \qa \bbV(C_i)_\delta=0 ~~\text{ for all }~  i < k.
\end{equation*}
If $\ell$ is odd, then we have
\begin{equation*}
\lla \ov{C_i} : 1\leq i \leq k\rra[\ell^\infty] \simeq \lla\ov{C_i} : 1\leq i \leq k-1\rra[\ell^\infty] \moplus \lla\ov{C_k} \rra[\ell^\infty].
\end{equation*}
\end{theorem}

\begin{theorem}\label{thm: criterion 3}
Suppose that one of the following holds.
\begin{enumerate}
\item
$\fh(C_k)=1$ and there is a divisor $\delta$ of $N$ such that 
\begin{equation*}
\bbV(C_k)_\delta \in \Z_2^\times \qa \bbV(C_i)_\delta=0 ~~\text{ for all } i < k.
\end{equation*}
\item
There is a prime $p$ such that 
\begin{equation*}
\pw_p(C_k) \not\in 2\Z \qa \pw_p(C_i)=0 ~~\text{ for all }~ i < k.
\end{equation*}
\end{enumerate}
Then we have
\begin{equation*}
\lla \ov{C_i} : 1\leq i \leq k\rra[2^\infty] \simeq \lla\ov{C_i} : 1\leq i \leq k-1\rra[2^\infty] \moplus \lla \ov{C_k} \rra[2^\infty].
\end{equation*}
\end{theorem}

\ms
We finish this section by proving all the theorems above.
\begin{proof}
Suppose that there are integers $a_i$ such that
\begin{equation*}
a_1 \cdot \ov{C_1} + a_2 \cdot \ov{C_2} + \cdots + a_k \cdot \ov{C_k}=0 \in J_0(N).
\end{equation*}
Let $X=\sum_{i=1}^k a_i \cdot C_i \in \Qdiv N$. Then by Corollary \ref{corollary: order 1 criterion}, we have 
\begin{equation*}
{\bf r}(X) \in \cS_1(N) \qa \rpw_q(X) \in 2\Z ~~\text{ for all primes }q.
\end{equation*}

\ms
To prove Theorem \ref{thm: criterion 1}, it suffices to show that $a_k \cdot \ov{C_k}=0$.
Therefore by Corollary \ref{corollary: order 1 criterion}, it is enough to prove that
\begin{equation*}
{\bf r}(a_k \cdot C_k) \in \cS_1(N) \qa \rpw_q(a_k \cdot C_k) \in 2\Z ~~\text{ for all primes } q.
\end{equation*}

Note that for any $1\leq i <k$, we have ${\bf r}(a_i \cdot C_i)_\delta=0$ because we assume that $\bbV(C_i)_\delta=0$. Thus, we have
\begin{equation*}
{\bf r}(a_k \cdot C_k)_\delta=\textstyle\sum_{i=1}^k {\bf r}(a_i\cdot C_i)_\delta={\bf r}(X)_\delta \in \Z.
\end{equation*}
Since $|\bbV(C_k)_\delta|=1$, this implies that ${\bf r}(a_k \cdot C_k) \in \cS_1(N)$. Indeed, for any divisor $d$ of $N$, we have
\begin{equation}\label{equation: 999}
{\bf r}(a_k \cdot C_k)_d={\bf r}(a_k \cdot C_k)_\delta \times \frac{\bbV(C_d)_d}{\bbV(C_k)_{\delta}}={\bf r}(a_k \cdot C_k)_\delta \cdot \bbV(C_k)_d \cdot \bbV(C_k)_\delta \in \Z.
\end{equation}
Thus, we have $\Gcd(a_k \cdot C_k)=|{\bf r}(a_k \cdot C_k)_\delta|$ and so
\begin{equation*}
{\bf r}(a_k \cdot C_k)=|{\bf r}(a_k \cdot C_k)_\delta |\cdot \bbV(a_k \cdot C_k)=| {\bf r}(a_k \cdot C_k)_\delta| \cdot \bbV(C_k).
\end{equation*}
This implies that for any prime $q$, we have
\begin{equation}\label{equation: 1000}
\rpw_q(a_k \cdot C_k)=|{\bf r}(a_k \cdot C_k)_\delta |\cdot \pw_q(C_k). 
\end{equation}

Suppose first that $\fh(C_k)=1$, i.e., $\pw_q(C_k) \in 2\Z$ for any prime $q$. Then by (\ref{equation: 1000}), we have $\rpw_q(a_k \cdot C_k) \in 2\Z$.
Suppose next that there is a prime divisor $p$ of $N$ such that $\pw_p(C_k) \not\in 2\Z$ and $\pw_p(C_i)=0$ for all $1\leq i < k$. 
Since $\rpw_p(X) \in 2\Z$ and $\rpw_p(C_i)=0$ for all $1\leq i <k$, we have
\begin{equation*}
\rpw_p(a_k \cdot C_k)=\textstyle\sum_{i=1}^k \rpw_p(a_i \cdot C_i)=\rpw_p(X) \in 2\Z.
\end{equation*}
Therefore by (\ref{equation: 1000}), we have ${\bf r}(a_k \cdot C_k)_\delta \in 2\Z$ since $\pw_p(C_k) \not\in 2\Z$. Again by (\ref{equation: 1000}), we easily obtain $\rpw_q(a_k \cdot C_k) \in 2\Z$.
This completes the proof of Theorem \ref{thm: criterion 1}.

\ms
Next, we prove Theorem \ref{thm: criterion 2}.
As above, it suffices to show that there is an integer $a$ not divisible by $\ell$ such that 
\begin{equation*}
{\bf r}(a\cdot a_k \cdot C_k) \in \cS_1(N) \qa \rpw_q(a\cdot a_k \cdot C_k) \in 2\Z ~~\text{ for all primes } q.
\end{equation*}

As above, we have
\begin{equation*}
{\bf r}(a_k \cdot C_k)_\delta = \textstyle\sum_{i=1}^k {\bf r}(a_i \cdot C_i)_\delta={\bf r}(X)_\delta  \in \Z.
\end{equation*}
In particular, we have $\tn{val}_\ell({\bf r}(a_k \cdot C_k)_\delta) \geq 0$.
Since $\bbV(C_k)_\delta \in \Z_\ell^\times$, by the same argument as in (\ref{equation: 999}),
we have
$\tn{val}_\ell({\bf r}(a_k \cdot C_k)_d)\geq 0$ for any divisor $d$ of $N$. Thus, there is an integer $b \in \Z_\ell^\times$ such that ${\bf r}(b \cdot a_k \cdot C_k) \in \cS_1(N)$. Since $\ell$ is odd,  we can take $a=2b \in \Z_\ell^\times$. Indeed, since ${\bf r}(b \cdot a_k \cdot C_k) \in \cS_1(N)$, we also have 
${\bf r}(2b \cdot a_k \cdot C_k)\in \cS_1(N)$ and $\rpw_q(2b \cdot a_k \cdot C_k)=2\cdot \rpw_q(b \cdot a_k \cdot C_k) \in 2\Z~$ for any primes $q$, as desired.
This completes the proof of Theorem \ref{thm: criterion 2}.

\ms
Finally, we prove Theorem \ref{thm: criterion 3}. Again, it suffices to show that there is an odd integer $a$ such that 
\begin{equation*}
{\bf r}(a\cdot a_k \cdot C_k) \in \cS_1(N) \qa \rpw_q(a\cdot a_k \cdot C_k) \in 2\Z ~~\text{ for all primes } q.
\end{equation*}

Suppose first that $\fh(C_k)=1$ and there is a divisor $\delta$ of $N$ such that $\bbV(C_k)_\delta$ is odd and $\bbV(C_i)_\delta=0$ for all $i<k$. As above, there is an odd integer $b$ such that ${\bf r}(b\cdot a_k \cdot C_k) \in \cS_1(N)$. Also, similarly as (\ref{equation: 1000}), we have $\rpw_q(b \cdot a_k \cdot C_k) \in 2\Z$ for any prime $q$ because $\pw_q(C_k) \in 2\Z$. Thus, we can take $a=b$.
Suppose next that there is a prime $p$ such that $\pw_p(C_k)$ is odd and $\pw_p(C_i)=0$ for all $i<k$. Since ${\bf r}(a_k \cdot C_k)=b \cdot \bbV(C_k)$ for some $b\in \Q^\times$, we have $\rpw_p(a_k \cdot C_k)=b \cdot \pw_p(C_k)$. Since we have
\begin{equation*}
b\cdot \pw_p(C_k)=\textstyle\rpw_p(a_k \cdot C_k)= \sum_{i=1}^k \rpw_p(a_i \cdot C_i)=\rpw_p(X) \in 2\Z
\end{equation*}
and $\pw_p(C_k)$ is odd, we have $\tn{val}_2(b)\geq 1$. Now, we take $a$ as the denominator of $b$. 
Since $\tn{val}_2(b)\geq 1$, $a$ is odd and $ab \in 2\Z$. Thus, we have ${\bf r}(a\cdot a_k \cdot C_k)=ab \cdot \bbV(C_k) \in \cS_1(N)$. Also, since $\rpw_q(a \cdot a_k \cdot C_k)=ab \cdot \pw_q(C_k)$ for any prime $q$, we have $\rpw_q(a \cdot a_k \cdot C_k) \in 2\Z$, as desired.
This completes the proof of Theorem \ref{thm: criterion 3}.
\end{proof}


\ms
\subsection{First attempt}\label{section: first attempt}
Using suitable orderings $\prec$ and $\vtl$ on $\cD_N^0$, we write
\begin{equation*}
\cD_N^\sqf=\{d_1, \dots, d_\fm \}=\{\delta_1, \dots, \delta_\fm \}  ~~\text{ and } ~~\cD_N^0 = \{d_1, \dots, d_\fn \}=\{ \delta_1, \dots, \delta_\fn\},
\end{equation*}
so that $d_i \prec d_j$ (resp. $\delta_i \vtl \delta_j$) if and only if $i<j$, where $\fm=\#\cD_N^\sqf$ and $\fn=\#\cD_N^0$. (Thus, for any non-trivial squarefree divisor $a$ of $N$ and for any non-squarefree divisor $b$ of $N$, we have $a \prec b$ and $a \vtl b$.)

\begin{enumerate}
\item
Find a rational cuspidal divisor $D_i$ such that the images $\ov{D_i}$ span the group $\scC(N)$, or more generally 
\begin{equation*}
\cS_2(N)^0= \br{\bD_i : 1\leq i \leq \fn}.
\end{equation*}
\item
Show that the matrix $\fM:=(|\bbV(D_i)_{\delta_j}|)_{1\leq i, j \leq \fn}$ is \textit{lower unipotent}, i.e., 
$\fM$ is lower-triangular and $|\bbV(D_i)_{\delta_i}|=1$ for all $i$.
\item
Prove that $\fh(D_i)=1$ for all $i \neq 1$.
\end{enumerate}

Then by successively applying Theorem \ref{thm: criterion 1}, we easily have
\begin{equation*}
\scC(N) \simeq \moplus_{i=1}^\fn \br{\ov{D_i}}.
\end{equation*}
Although our requirements are quite strong, this strategy indeed works when $N$ is an odd prime power (Section \ref{section: 6.2}).

\ms
\subsection{Second attempt}\label{section: second attempt}
As above, we define two orderings $\prec$ and $\vtl$ on $\cD_N^0$. But since some arguments break down in general, we modify our previous strategy a little bit.

\begin{enumerate}
\item
Find a rational cuspidal divisor $D_i$ such that the images $\ov{D_i}$ span the group $\scC(N)$, or more generally 
\begin{equation*}
\cS_2(N)^0= \br{\bD_i : 1\leq i \leq \fn}.
\end{equation*}
\item
Show that the matrix $\fM$
 is of the form
\begin{equation*}
\arraycolsep=1.4pt\def\arraystretch{1.5}
\left( \begin{array}{c:c} 
\phantom{a} * \phantom{a} & \phantom{a} \bbO \phantom{a} \\  \hdashline
\phantom{a} * \phantom{a} & \phantom{a} U \phantom{a}
\end{array}\right) \begin{array}{cc}
\leftarrow &\text{ $\fm~$ rows}\\
\leftarrow &\text{ $(\fn-\fm)$ rows},
\end{array}
\end{equation*}
where $U$ is a lower unipotent matrix of size $\fn-\fm$, and $\bbO$ is the $\fm\times (\fn-\fm)$ zero matrix. 
\item
Prove that $\fh(D_i)=1$ for any $\fm<i\leq \fn$.
\end{enumerate}

We then have
\begin{equation*}
\scC(N)=\br{\ov{D_i} : 1\leq i \leq \fn} \simeq \scC(N)^\sqf \moplus \left(\moplus_{i=\fm+1}^\fn \br{ \ov{D_i} }\right),
\end{equation*}
where $\scC(N)^\sqf=\br{\ov{D_i} : 1\leq i \leq \fm}$. Indeed, the equality follows from (1), and the isomorphism follows by successively applying Theorem \ref{thm: criterion 1} as we have (2) and (3).

\vspace{3mm}
Now, we fix a prime $\ell$ and compute the $\ell$-primary subgroup of $\scC(N)^\sqf$.
\begin{enumerate}
\setcounter{enumi}{3}
\item
Find a rational cuspidal divisor $E_i$ such that
\begin{equation}\label{equation: step4}
\br{\bD_i : 1\leq i \leq \fm} \motimes_\Z \Z_\ell = \br{\bE_i : 1\leq i \leq \fm} \motimes_\Z \Z_\ell.
\end{equation}
\item
Show that the matrix $\fM_0=(\bbV(E_i)_{\delta_j})_{1\leq i, j \leq \fm}$ is \textit{lower $\ell$-unipotent}, i.e., $\fM_0$ is lower-triangular and $\bbV(E_i)_{\delta_i} \in \Z_\ell^\times$ for all $i$.
\end{enumerate}

If $\ell$ is odd, then by successively applying Theorem \ref{thm: criterion 2}, we have
\begin{equation*}
\scC(N)^\sqf [\ell^\infty] \simeq \moplus_{i=1}^\fm \br{\ov{E_i}}[\ell^\infty].
\end{equation*}

If $\ell=2$, then we need one more step.
\begin{enumerate}
\setcounter{enumi}{5}
\item
Prove that for any $1<i\leq \fm$, either $\fh(E_i)=1$ or there is a prime $p$ such that
\begin{equation*}
\pw_p(E_i) \not\in 2\Z \qa \pw_p(E_j)=0 \text{ for all } j<i.
\end{equation*}
\end{enumerate}

Then by successively applying Theorem \ref{thm: criterion 3}, we have
\begin{equation*}
\scC(N)^\sqf [2^\infty] \simeq \moplus_{i=1}^\fm \br{\ov{E_i}} [2^\infty].
\end{equation*}
By taking $D_i=Z(d_i)$ and $E_i=Y^0(d_i)$, our strategy works if $N$ is odd.

\ms
\subsection{Third attempt}\label{section: third attempt}
Let $N=\prod_{i=1}^t p_i^{r_i}$ be the prime factorization of $N$.
We fix a prime $\ell$ and make Assumption \ref{assumption: chapter4}. 
We assume that $N$ is even, i.e., $u\geq 1$. For simplicity, let $r=r_u$.

\ms
In the previous section, we in fact verify (1) and (2) for $D_i=Z^1(d_i)$.\footnote{Note that $Z(d)=Z^1(d)$ for any $d \in \cD_N^0$ if $N$ is odd.} So if we take $D_i=Z^1(d_i)$, then (1) and (2) can be verified without any change, but (3) is not fulfilled in general. Problematic elements are some of $Z^1(d_i)$ with $\fm< i < \fm+r$ as we have already seen in Section \ref{section: case of level 2r}. Thus, we ignore them at the moment, and by mimicking the strategy in the previous section, we obtain\footnote{Note that $Z^1(d_i)=Z(d_i)$ for any $i\geq \fm+r$.}
\begin{equation*}
\scC(N) \simeq \br{\ov{Z^1(d_i)} : 1\leq i < \fm+r} \moplus\left(\moplus_{j=\fm+r}^\fn \br{\ov{Z(d_j)}}\right).
\end{equation*}

Next, we deal with the group $\scC(N)^\sqf[\ell^\infty]$. If we take $E_i=Y^0(d_i)$ as in the previous section, then (4) works without any change because (\ref{equation: assumption 1}) guarantees that (\ref{equation: step4}) holds.
However, (5) works only when $u=1$. Note that the vector $\bD(p_i^{r_i}, p_j^{r_j})$ is designed so that $\bbV(D(p_i^{r_i}, p_j^{r_j}))_{p_j}\in \Z_\ell^\times$ for any $i<j$ under the assumption that $\tn{val}_\ell(p_i-1) \leq \tn{val}_\ell(p_j-1)$.
Since $p_u-1=1$, this property may not hold if $j=u$. This is why we take a twisted colexicographic order for the ordering $\vtl$ on $\cD_N^\sqf$ with the cost that the matrix $\fM_0$ is not lower-triangular any more. With this new ordering $\vtl$, we try to make the matrix $\fM_0$ lower $\ell$-unipotent. Fortunately, if we take $E_i=Y^1(d_i)$, then the matrix $\fM_0$ is lower $\ell$-unipotent. If $\ell$ is odd, by successively applying Theorem \ref{thm: criterion 2}, we have
\begin{equation*}
\scC(N)^\sqf[\ell^\infty] \simeq \moplus_{i=1}^\fm \br{\ov{E_i} } [\ell^\infty].
\end{equation*}

Suppose that $\ell=2$. We hope to verify (6) for $E_i=Y^1(d_i)$. If $u=s\geq 1$, then $p_s-1=1$ is odd, and so some arguments break down. As above, we replace some problematic elements and finally take $E_i=Y^2(d_i)$. Although the matrix $\fM_0$ is not lower-triangular any more, we can successively apply Theorem \ref{thm: criterion 3} and obtain that
\begin{equation}\label{000}
\scC(N)^\sqf[2^\infty] \simeq \moplus_{i=1}^\fm \br{\ov{E_i} } [2^\infty],
\end{equation}
as desired.

Lastly, we hope to prove
\begin{equation*}
\br{\ov{Z^1(d_i)} : 1\leq i < \fm+r} \simeq \scC(N)^\sqf \moplus\left(\moplus_{i=\fm+1}^{\fm+r-1} \br{\ov{Z(d_i)}}\right),
\end{equation*}
which will be done in Section \ref{section: final step}. 
This completes the computation of $\scC(N)$.

\ms
\subsection{Degeneracy maps revisited}\label{sec: degeneracy maps revisited}
Let $N=Mp^r$ with $\gcd(M, p)=1$. For simplicity, let $\alpha:=\alpha_p(N)^*$ (resp. $\beta:=\beta_p(N)^*$) be the degeneracy map 
from $\Div(X_0(N))$ to $\Div(X_0(Np))$.

\begin{proposition}\label{prop: degeneracy maps revisited}
Let $C \in \Qdivv N$. If $r=0$, then we have
\begin{equation*}
V(\alpha(C))=(p^2-1) \cdot V(C) \motimes \vect 1 0 \qqa V(\beta(C))=(p^2-1) \cdot V(C) \motimes \vect 0 1.
\end{equation*}
If $r\geq 1$, then for any divisor $d$ of $M$ we have
\begin{equation*}
V(\alpha(C))_{dp^f}=\begin{cases}
p\cdot V(C)_{dp^f} & \text{ if }~~ 0\leq f\leq r,\\
0 & \text{ if }~~ f=r+1,
\end{cases}
\end{equation*}
and
\begin{equation*}
V(\beta(C))_{dp^f}=\begin{cases}
0 & \text{ if } ~~ f=0,\\
p\cdot V(C)_{dp^{f-1}} & \text{ if }~~ 1\leq f\leq r+1.
\end{cases}
\end{equation*}
\end{proposition}

\begin{proof}
First, let $r=0$. Then we have $\Phi_{Np}(\alpha(C))=\Phi_N(C) \motimes \vect p 1$ by Lemma \ref{lemma: degeneracy maps pullback on RCD}. Thus, we have
\begin{equation*}
 \begin{split}
 V(\alpha(C))&=\Upsilon(Np)\times \Phi_{Np}(\alpha(C))=(\Upsilon(N) \times \Phi_N(C)) \motimes (\mat p {-1} {-1} p \times \vect p 1)\\
 &=V(C) \motimes \vect {p^2-1}{0},
 \end{split}
 \end{equation*}
 as desired. Similarly, we obtain the formula for $V(\beta(C))$.
 
 Next, let $r\geq 2$. Let $\Phi_N(C)=\sum_{d\mid M, ~0\leq f \leq r} a(d p^f) \cdot {\bf e}(N)_{dp^f}$. Also, for each $0\leq f \leq r$, let $A_f= \sum_{d\mid M} a(d p^f) \cdot {\bf e}(M)_d \in \cS_2(M)$ and write $\Phi_N(C)=(A_0, A_1, \dots, A_{r-1}, A_r)^t$, i.e., the $p^f$-th block vector is $A_f$. If we write 
\begin{equation*}
\Phi_{Np}(\alpha(C))=(A_0', A_1', \dots, A_r', A_{r+1}')^t
\end{equation*}
(so that the $p^f$-th block vector is $A_f'$), then by Lemma  \ref{lemma: degeneracy maps pullback on RCD} we have $A_{r+1}'=A_r$ and $A_f'=a_f \cdot A_f$ for any $0\leq f \leq r$, where $a_f=p$ if $f\leq [r/2]$, and $a_f=1$ otherwise.
For simplicity, let
 \begin{equation*}
 V(C)=
\Upsilon(N) \times \left(\begin{array}{c}
A_0 \\
\vdots \\
A_r
 \end{array}\right)\\
=:\left(\begin{array}{c}
E_0 \\
\vdots\\
E_r
 \end{array}\right)
 \end{equation*}
and
 \begin{equation*}
V(\alpha(C))=
 \Upsilon(Np) \times \left(\begin{array}{c}
A_0' \\
\vdots\\
A_{r+1}'
 \end{array}\right)\\
 =:\left(\begin{array}{c}
F_0 \\
\vdots\\
F_{r+1}
 \end{array}\right).
\end{equation*} 

To prove the result for $V(\alpha(C))$, we must show that $F_{r+1}=0$ and $F_f=pE_f$ for any $0\leq f \leq r$. First, we easily have $F_{r+1}=0$ because the last row of $\Upsilon(p^{r+1})$ is $(\dots, -p, p)$, where the dots denote zero entries. Next, let $B_f=\Upsilon(M) \times A_f$ for any $0\leq f \leq r$. Then we have
\begin{equation*}
E_0=pB_0-pB_1 \qa E_r=-p^{m(r-1)}B_{r-1}+pB_r,
\end{equation*}
where $m(f)=\min(f, \, r-f)$. Also, we have 
\begin{equation*}
F_0=p^2B_0-p^2B_1 \qqa F_r=-p^{n(r-1)}a_{r-1}B_{r-1}+(p^2+1)B_r-B_r,
\end{equation*}
where $n(f)=\min(f, \, r+1-f)$. Thus, we easily have $F_f=pE_f$ for $f=0$ or $r$.
Finally, suppose that $1\leq f \leq r-1$. Then we have
\begin{equation*}
E_f=-p^{m(f-1)}B_{f-1}+p^{m(f)-1}(p^2+1)B_f-p^{m(f+1)}B_{f+1}
\end{equation*}
and
\begin{equation*}
F_f=-p^{n(f-1)} a_{f-1}B_{f-1}+p^{n(f)-1}(p^2+1)a_fB_f-p^{n(f+1)}a_{f+1}B_{f+1}.
\end{equation*}
Note that if $i\leq [r/2]$, then $a_i=p$ and $m(i)=n(i)=i$. If $i>[r/2]$, then $a_i=1$ and $n(i)=r+1-i=m(i)+1$. Thus, we have $p^{n(i)}a_i=p^{m(i)+1}$ for all $1\leq i\leq r$, and so we have $F_f=pE_f$ for any $1\leq f\leq r-1$. This completes the proof of $V(\alpha(C))$. Similarly, we obtain the result for $V(\beta(C))$.

Lastly, let $r=1$. Since the result easily follows by the same argument as in the case of $r\geq 2$, we leave the details to the readers.
\end{proof}

\vspace{10mm}
\section{The structure of $\scC(N)$}\label{chapter6}
In this section, we compute the structure of the rational cuspidal divisor class group $\scC(N)$ of $X_0(N)$ based on the strategy in the previous section. More precisely, we prove the following, which is a combination of Theorems \ref{theorem: main theorem 1} and \ref{theorem: main theorem 2}.
\begin{theorem}\label{theorem: chapter 6 main}
Let $N$ be a  positive integer. Then we have
\begin{equation*}
\scC(N) \simeq \scC(N)^\sqf \moplus\left(\moplus_{d\in \cD_N^\nsqf} \br{\ov{Z(d)}}\right).
\end{equation*}
Also, for any prime $\ell$, we have
\begin{equation*}
\scC(N)^\sqf[\ell^\infty] \simeq \moplus_{d\in \cD_N^\sqf} \br{\ov{Y^2(d)}}[\ell^\infty].
\end{equation*}
Furthermore, the orders of $Z(d)$ and $Y^2(d)$ are $\fn(N, d)$ and $\fN(N, d)$, respectively.
\end{theorem}

\ms
Throughout this section, let $N=\prod_{i=1}^t p_i^{r_i}$ be the prime factorization of $N$,
\begin{equation*}
\textstyle\fm=\#\cD_N^\sqf=2^t-1 \qa \fn=\#\cD_N^0=\prod_{i=1}^t (r_i+1)-1.
\end{equation*}
As already mentioned, we write
\begin{equation*}
\textstyle (a_0, \dots, a_r) := \sum_{i=0}^r a_i \cdot {\bf e}(p^r)_{p^i} \in \cS_k(p^r) ~~\text{ for both $k=1$ or $2$}.
\end{equation*}
  
\ms
\subsection{The orderings $\prec$ and $\vtl$}\label{section: the orderings}
As already mentioned, we define two orderings $\prec$ and $\vtl$ on $\cD_N^0$, and write
\begin{equation*}
\cD_N^\sqf=\{d_1, \dots, d_\fm \}=\{\delta_1, \dots, \delta_\fm \}  ~~\text{ and } ~~\cD_N^0 = \{d_1, \dots, d_\fn \}=\{ \delta_1, \dots, \delta_\fn\}
\end{equation*}
so that $d_i \prec d_j$ (resp. $\delta_i \vtl \delta_j$) if and only if $i<j$. For simplicity, we set $\delta_0=1$. Since any divisor of $N$ can be written as $\fp_I$ for some $I=(f_1, \dots, f_t) \in \Omega(t)$, we define two orderings on $\Omega(t)$ instead.

\vspace{2mm}
Motivated by the idea in Section \ref{section: case of odd prime power}, we first define the following.
\begin{definition}\label{definition: prime power level orderings}
Let $r$ be a positive integer. On the set $\{0, 1, \dots, r\}$, we define the orderings $\prec_r$ and $\vtl_r$  as follows.
\begin{enumerate}
\item
If $r=1$, then we define 
\begin{equation*}
1 \prec_1 0 \qa 0 \vtl_1 1.
\end{equation*}
\item
If $r=2$, then we define
\begin{equation*}
1 \prec_2 0 \prec_2 2 \qa 0 \vtl_2 1 \vtl_2 2. 
\end{equation*}

\item
If $r=3$, then we define
\begin{equation*}
1 \prec_3 0 \prec_3 2 \prec_3 3 \qa 0 \vtl_3 1 \vtl_3 3 \vtl_3 2.
\end{equation*}

\item
If $r\geq 4$, then we define
\small
\begin{equation*}
\setlength{\arraycolsep}{2pt}
\begin{array}{ccccccccccccccccccc}
1 & \prec_r & 0 & \prec_r & 2 & \prec_r & r & \prec_r &r-1 & \prec_r & r-2  & \prec_r & \cdots & \prec_r & 5 & \prec_r & 4 & \prec_r & 3\\
0 & \vtl_r & 1  & \vtl_r &  r  & \vtl_r & r-1  & \vtl_r & 2  & \vtl_r &  r-2  & \vtl_r  &  \cdots  & \vtl_r & k-(-1)^r & \vtl_r & k+(-1)^r & \vtl_r& k,
\end{array}
\end{equation*}\normalsize
where $k=[(r+1)/2]$.
\end{enumerate}

We define a bijection $\iota_r$ on $\{0, 1, \dots, r\}$ so that $i \prec j$ if and only if $\iota_r(i) \vtl \iota_r(j)$. For instance, if $r$ is large enough, then we set $\iota_r(1)=0$, 
$\iota_r(0)=1$, $\iota_r(2)=r$, $\iota_r(r)=r-1$, $\dots$, $\iota_r(5)=k-(-1)^r$, $\iota_r(4)=k+(-1)^r$ and $\iota_r(3)=k$.
\end{definition}

Next, we define the orderings $\prec$ and $\vtl$ on $\square(t)$ as follows.
\begin{definition}
Let $I=(a_1, \dots, a_t)$ and $J=(b_1, \dots, b_t)$ be two elements in $\square(t)$.
We write $I \prec J$ (resp. $I \vtl J$) if and only if one of the following holds.
\begin{enumerate}
\item
$a_u \prec_{r_u} b_u$ (resp. $a_u \vtl_{r_u} b_u$) and $a_j=b_j$ for all $j$ different from $u$.
\item
There is an index $k\neq u$ such that 
$a_k \prec_{r_k} b_k$ (resp. $a_k \vtl_{r_k} b_k$) and
$a_j=b_j$ for all $j>k$ different from $u$.
\end{enumerate}
\end{definition}

Then, we define the orderings $\prec$ and $\vtl$ on $\Delta(t)$ as follows.
\begin{definition}
Let $I=(a_1, \dots, a_t)$ and $J=(b_1, \dots, b_t)$ be two elements in $\Delta(t)$.
We write $I \vtl J$ if and only if one of the following holds.
\begin{enumerate}
\item
$a_u=0$, $b_u=1$ and $a_j=b_j$ for all $j$ different from $u$.
\item
There is an index $k\neq u$ such that $a_k=0$, $b_k=1$ and
$a_j=b_j$ for all $j>k$  different from $u$.
\end{enumerate}
Also, we write $I \prec J$ if and only if $\iota(I) \vtl \iota(J)$, where the map $\iota$ is defined below.
\end{definition}

Finally, we define the orderings $\prec$ and $\vtl$ on $\Omega(t)$ as follows.
\begin{definition}
For two elements $I$ and $J$ in $\Omega(t)$, we write $I \prec J$ (resp. $I \vtl J$) if and only if one of the following holds.
\begin{enumerate}
\item
$I \in \Delta(t)$ and $J \in \square(t)$.
\item
$I, J \in \square(t)$ and $I \prec J$ (resp. $I \vtl J$).
\item
$I, J \in \Delta(t)$ and $I \prec J$ (resp. $I \vtl J$).
\end{enumerate}
\end{definition}

As mentioned above, we define a map $\iota : \Delta(t) \to \Delta(t)$ as follows.
\begin{definition}
For any $I=(a_1, \dots, a_t) \in \Delta(t)$ with $m=m(I)$, $n=n(I)$ and $x=\max(m, u)$, we set $\iota(I):=(b_1, \dots, b_t)$, where
\begin{enumerate}
\item
$b_i=1-a_i$ for all $i\neq x$ and $b_x=1$ if $I \in \cE$,
\item
$b_m=1$, $b_u=0$ and $b_i=1-a_i$ for all $i\neq m, u$ if $I \in \cH_u^1$,
\item
$b_i=1-a_i$ for all $i$ if $I \not\in \cE\cup \cH_u^1$.
\end{enumerate}
\end{definition}

\begin{remark}\label{remark: the map iota}
Let $I=(a_1, \dots, a_t) \in \Delta(t)$ with $m=m(I)$, $n=n(I)$ and $x=\max(m, u)$. Also, let $\iota(I)=(b_1, \dots, b_t)$. By definition, we have the following.
\begin{enumerate}
\item[$(1)$\phantom{'}]
If $I\in \cE$ and $x=m$, then $b_i=1$ for all $i\leq m$, and $b_j=0$ for all $j>m$.
\item[$(1)'$]
If $I \in \cE$ and $x=u>m$, then $b_i=1$ for all $i<m$ and $i=u$, and $b_j=0$ for all $j\geq m$ different from $u$.
\item[$(2)$\phantom{'}]
If $I \in \cH_u^1$, then $b_i=1$ for all $i\leq m$, and $b_i=0$ for all $i>m$.
\item[$(3)$\phantom{'}]
If $I \not\in \cE \cup \cH_u^1$, then $b_i=1-a_i$ for all $i$. In particular, $b_m=1-a_m=0$ and $b_n=1-a_n=1$. 
\end{enumerate}
In cases (1) and (2), there seems no difference on $b_i$. However, we have $m\geq u$ in case (1), but $m<u$ in case (2) as we have $m<n$ by definition.
\end{remark}

\begin{lemma}\label{lemma: iota}
The map $\iota$ is a bijection on $\Delta(t)$ for any $t$.
\end{lemma}
\begin{proof}
As above, let $I=(a_1, \dots, a_t)\in \Delta(t)$ with $m=m(I)$, $n=n(I)$ and $x=\max(m, u)$. Also, let $\iota(I)=(b_1, \dots, b_t)$. 

We first claim that $\iota(I) \in \Delta(t)$. By definition, it is easy to see that $b_i \in \{0, 1\}$, so it suffices to prove that $b_i=1$ for some $i$. 
In case $(1)$ (resp. $(1)'$, $(2)$, and $(3)$) in Remark \ref{remark: the map iota}, we have $b_m=1$ (resp. $b_u=1$, $b_m=1$, and $b_n=1$). Therefore the claim follows.

Next, we prove that the map $\iota$ is bijective. Since $\Delta(t)$ is a finite set, it suffices to show that $\iota$ is injective. Let $J=(e_1, \dots, e_t) \in \Delta(t)$ with $m'=m(J)$, $n'=n(J)$ and $x'=\max(m', u)$. 
Also, let $\iota(J)=(f_1, \dots, f_t)$. If $I$ and $J$ are both in the same case in Remark \ref{remark: the map iota}, then it is easy to prove that $I=J$ whenever $\iota(I)=\iota(J)$. So it suffices to prove that $b_i \neq f_i$ for some $i$ in the following cases.
\begin{enumerate}
\item[$(1)$\phantom{'}]
Assume that $I\in \cE$ and $x=m$. Then we have $u\leq m$. 
\begin{enumerate}
\item
Suppose that $J \in \cE$ and $x'=u> m'$. Since $m'<u \leq m$, we have $b_{m'}=1$. 
By definition, we have $e_{m'}=1$. Since $m'\neq x'$, we have $f_{m'}=1-e_{m'}=0$.
Thus, we have $b_{m'} \neq f_{m'}$.

\item
Suppose that $J \in \cH_u^1$. Then we have $m'<n'=u \leq m$.
Since $f_j=0$ for all $i>m'$, we have $f_m=0$. Since $b_m=1$, we have $b_m \neq f_m$.

\item
If $J \not\in \cE \cup \cH_u^1$, then $f_{m'}=0$ and $f_{n'}=1$ with $m'< n'$. 
If $m' \leq  m$, then $b_{m'}=1$. If $m'>m$ then $n'>m'>m$ and so $b_{n'}=0$. Thus, we have either $b_{m'} \neq f_{m'}$ or $b_{n'} \neq f_{n'}$. 
\end{enumerate}

\item[$(1)'$]
Assume that $I \in \cE$ and $x=u>m$. 
\begin{enumerate}
\item
If $J \in \cH_u^1$, then we have $f_u=0$. Since $b_u=1$, we have $b_u \neq f_u$.
\item
If $J \not\in \cE \cup \cH_u^1$, then $f_{m'}=0$ and $f_{n'}=1$ with $m'< n'$. 
If $n' \neq u$ then  either $b_{m'} \neq f_{m'}$ or $b_{n'} \neq f_{n'}$ as in case ($1$)-(c) above.
Suppose that $n'=u$. Since $J \not\in \cH_u^1$, we have $k'\leq t$, and so $f_{k'}=1-e_{k'}=1$. 
As above, we have either $b_{m'}\neq f_{m'}$ (if $m'<m$) or $b_{k'} \neq f_{k'}$ (otherwise).
\end{enumerate}

\item[$(2)$\phantom{'}]
Assume that $I \in \cH_u^1$. If $J \not\in \cE \cup \cH_u^1$, then we have $f_{m'}=0$ and $f_{n'}=1$. 
As in case ($1$)-(c) above, we have either $b_{m'} \neq f_{m'}$ or $b_{n'} \neq f_{n'}$. 
\end{enumerate}
This completes the proof.
\end{proof}

\begin{remark}\label{remark: iota isomorphism}
We define the map $\iota_{\square} : \square(t) \to \square(t)$ by
\begin{equation*}
\iota_{\square}(a_1, \dots, a_t) := (b_1, \dots, b_t), ~~ \text{ where } b_i = \iota_{r_i}(a_i).
\end{equation*}
Also, we define the map $\iota_{\Omega} : \Omega(t) \to \Omega(t)$ by $\iota_{\Omega}=\iota_{\square}$ if $I \in \square(t)$, and $\iota_{\Omega}=\iota$ otherwise.
Then by definition, the map $\iota_{\Omega}$ is an order-preserving bijection from $(\Omega(t), \prec)$ to $(\Omega(t), \vtl)$. In other words, we have $\iota_{\Omega}(d_i)=\delta_i$ for any $i$.
\end{remark}

\begin{remark}\label{remark: 2power ordering}
Suppose that $N$ is divisible by $4$ and let $r=r_u\geq 2$. For any $2\leq f\leq r$, let $I_f:=(a_1, \dots, a_t) \in \square(t)$ such that $a_u=f$ and $a_i=1$ for all $i \neq u$. Then by definition, $d_{\fm+1}=2\cdot \rad(N)=\fp_{I_2}$. Also, we have 
\begin{equation*}
\{ d_i : \fm+2 \leq i < \fm+r \}=\{ \fp_{I_f} : 3\leq f \leq r\},
\end{equation*}
which is equal to $\cT_u$ if $r\geq 5$. Furthermore, if $i \geq \fm+r$, then there is an index $h \neq u$ such that $\tn{val}_{p_h}(d_i)\geq 2$.
\end{remark}

The following will be used later.
\begin{lemma}\label{lemma: 6.8}
Suppose that $t\geq 2$ and $s\geq 1$.  If $n \in \cI_s$, then we have 
\begin{equation*}
\iota(E(n))=F(n) \qa \iota(E_s(n))=F_s(n).
\end{equation*}
Also, we have
\begin{equation*}
d_{2^{n-\epsilon}}=\fp_{E(n)} \qa d_{2^{n-\epsilon}+1}=\fp_{E_s(n)},
\end{equation*}
where $\epsilon=1$ if $s<n$, and $\epsilon=0$ if $n<s$. 
\end{lemma}
\begin{proof}
Let $I=E_s(n)$ and $J=E(n)$. By definition, $n(I)=\min(n, s) \leq t$ and $k(I)=\max(n, s) \leq t$. Therefore $I \not\in \cE \cup \cH_u^1$. Also, since $n(J)=n\leq t$ and $n\neq s$, we have $J \not\in \cE \cup \cH_u^1$. Thus, we have $\iota(I)=F_s(n)$ and $\iota(J)=F(n)$ by definition.

Now, consider a well-ordered set $(\Delta(t), \vtl)$.
By definition, the largest element less than $F_s(n)$ is $F(n)$. Furthermore, for any $K=(a_1, \dots a_t) \in \Delta(t)$, $K$ is smaller than $F(n)$ 
(or equivalently, $K \vtl F(n)$) if and only if $a_i=0$ for all $i\geq n$ different from $s$. Thus, the number of such elements is exactly $2^n-1$ if $n<s$, and $2^{n-1}-1$ otherwise. Therefore the result follows.
\end{proof}
\begin{lemma}\label{lemma: 6.7}
Suppose that $t\geq 2$. Then we have
\begin{equation*}
d_1=\rad(N)=\fp_{A(1)} \qa 
\delta_1=\begin{cases}
p_1=\fp_{F(1)} & \text{ if }  ~~u\leq 1,\\
p_u =\fp_{F(u)} & \text{ if } ~~u\geq 2.
\end{cases}
\end{equation*}
Also, we have
\begin{equation*}
d_2=\begin{cases}
\fp_{E(2)} & \text{ if }  ~~u\leq 1,\\
\fp_{E(u)} & \text{ if } ~~u\geq 2,
\end{cases}
\qa 
\delta_2=\begin{cases}
p_2=\fp_{F(2)} & \text{ if } ~~ u\leq 1,\\
p_1 =\fp_{F(1)} & \text{ if } ~~u\geq 2.
\end{cases}
\end{equation*}
Furthermore, we have
\begin{equation*}
d_3=\fp_{A(2)}=\fp_{E(1)}  \qa 
\delta_3=\begin{cases}
p_1p_2 & \text{ if }  ~~u\leq 1,\\
p_1p_u & \text{ if } ~~u\geq 2.
\end{cases}
\end{equation*}
\end{lemma}
\begin{proof}
By definition, it is easy to verify the formula for $\delta_i$. Also, by the definition of the map $\iota$ we can easily verify that $\iota(d_i)=\delta_i$ for $1\leq i \leq 3$.
\end{proof}

\ms
\subsection{Case of $t=1$ and $u=0$}\label{section: 6.2}
In this subsection, we prove the following.

\begin{theorem}\label{theorem: case of t=1, u=0}
For an odd prime $p$ and $r\geq 1$, we have
\begin{equation*}
\scC(p^r) \simeq \moplus_{f=1}^r \br{\ov{B_p(r, f)}}.
\end{equation*}
Also, the order of $B_p(r, f)$ is 
\begin{equation*}
\begin{cases}
\num\left(\frac{p-1}{12}\right) & \text{ if }~~ f=1,\\
\num\left(\frac{p^2-1}{24}\right) & \text{ if } ~~ f=2,\\
\frac{p^{r-1-j}(p^2-1)}{24} & \text{ if }~~ 3\leq f\leq r, \text{ where } j=[\frac{r+1-f}{2}].
\end{cases}
\end{equation*}
\end{theorem}

\vspace{2mm}
To begin with, we prove the following.
\begin{proposition}\label{proposition: generation A and B}
For any $r\geq 1$, we have
\begin{equation*}
\cS_2(p^r) = \lla \bA_p(r, f) : 0 \leq f \leq r \rra \qa \cS_2(p^r)^0=\lla \bB_p(r, f) : 1\leq f \leq r\rra.
\end{equation*}
\end{proposition}
\begin{proof}
By the result in Section \ref{section: actions on RCD}, we can explicitly compute $\bA_p(r, f)$. Indeed, we have $\bA_p(r, 1)_{p^k} =p^{\max(r-2k, \, 0)}$ by Lemma \ref{lemma: A1 definition} (and Remark \ref{remark: 4.8}). Also, by Lemmas \ref{lemma: degeneracy maps pullback on RCD}  and \ref{lemma: B2 definition}, we have the following. If $r\geq 3$ is odd, then we have 
\begin{equation*}
\bA_p(r, 2)_{p^k} =\begin{cases}
K_p(\frac{r-3}{2}) & \text{ if } ~~k=0,\\
K_p(\frac{r-1}{2}-k) & \text{ if } ~~0 < k \leq  \frac{r-1}{2},\\
0 & \text{ if } ~~k=\frac{r+1}{2},\\
-p\cdot K_p(k-\frac{r+3}{2}) & \text{ otherwise,}\\
\end{cases}
\end{equation*}
where $K_p(j)=\sum_{i=0}^j p^{2i}$. Also, if $r\geq 2$ is even, then 
\begin{equation*}
\bA_p(r, 2)_{p^k} =\begin{cases}
K_p(\frac{r-2}{2}-k) & \text{ if }~~ 0 \leq k <\frac{r}{2},\\
0 & \text{ if }~~ k=\frac{r}{2},\\
-K_p(k-\frac{r+2}{2}) & \text{ otherwise.}\\
\end{cases}
\end{equation*}
Moreover, by Lemma \ref{lemma: degeneracy maps pullback on RCD} we have the following.
If $3 \leq f=r-2a \leq r$, then
\begin{equation*}
 \bA_p(r, f)_{p^k} =\begin{cases}
 p^a & \text{ if }~~k=0, \\
-1 & \text{ if }~~ r-a \leq k \leq r,\\
0 & \text{ otherwise,}
\end{cases}
\end{equation*}
and if $3 \leq f=r+1-2a \leq r$, then
\begin{equation*}
 \bA_p(r, f)_{p^k} =\begin{cases}
1 & \text{ if }~~ 0 \leq k \leq a,\\
-p^a & \text{ if }~~ k=r, \\
0 & \text{ otherwise.}
\end{cases}
\end{equation*}

Using this description, we first prove that $\cS_2(p^r)\subset \lla \bA_p(r, f) : 0 \leq f \leq r\rra$. For simplicity, let ${\bf e}_k:={\bf e}(p^r)_{p^k}$.
Then it suffices to show that for any $0\leq k \leq r$, there are integers $a(k, f)$ such that
\begin{equation*}
{\bf e}_k=\textstyle\sum_{f=0}^r a(k, f) \cdot \bA_p(r, f).
\end{equation*}
This is obvious for $k=0$ because $\bA_p(r, 0)={\bf e}_0$. Since $\bA_p(r, r)={\bf e}_0-{\bf e}_r$, the claim follows for $k=r$. By direct computation, we easily find $a(k, f)$ for small $r$, so we assume that $r\geq 5$. Since $\bA_p(r, r-1)={\bf e}_0+{\bf e}_1-p\cdot {\bf e}_r$, the claim follows for $k=1$. Also, since $\bA_p(r, r-2)=p\cdot {\bf e}_0-{\bf e}_{r-1}-{\bf e}_r$, the result follows for $k=r-1$. Doing this successively, we can find $a(k, f)$ for any $k$ different from $m$ or $m+1$, where $m=[\frac{r-1}{2}]$. 
Note that $\bA_p(r, 2)_{p^m}=1$ and $\bA_p(r, 2)_{p^{m+1}}=0$. Thus, we obtain the result for $k=m$ using the vectors $\bA_p(r, f)$ with $f\neq 1$. Also, since $\bA_p(r, 1)_{p^{m+1}}=1$, the result for $k=m+1$ follows. Since the other inclusion is obvious, we obtain $\cS_2(p^r)=\lla \bA_p(r, f) : 0 \leq f\leq r \rra$.

Next, we prove that $\cS_2(p^r)^0 \subset \lla \bB_p(r, f) : 1\leq f \leq r \rra$. By Lemma \ref{lemma: group of rational cuspidal divisors}, the group $\cS_2(p^r)^0$ is generated by
\begin{equation*}
{\bf f}_k:=\varphi(p^{\min(k, \, r-k)}) \cdot {\bf e}_0-{\bf e}_k ~~\text{ for any } 1\leq k \leq r.
\end{equation*}
Since ${\bf e}_k=\sum_{f=0}^r a(k, f)\cdot \bA_p(r, f)$, we easily obtain
\begin{equation*}
{\bf f}_k =\textstyle\sum_{f=1}^r -a(k, f)\cdot \bB_p(r, f).
\end{equation*}
As the other inclusion is obvious, this completes the proof.
\end{proof}

Next, we define vectors $\bbA_p(r, f)$ and $\bbB_p(r, f)$ in $\cS_1(p^r)$.
\begin{definition}\label{defn: bA bB in prime level}
For any $0 \leq f \leq r$,  we define a vector $\bbA_p(r, f)$ in $\cS_1(p^r)$ by
\begin{equation*}
\bbA_p(r, f):=\begin{cases}
(p, -1, \bbO_{r-1}) & \text{ if }~~ f=0,\\
(1, \bbO_r) & \text{ if }~~ f=1, \\
(1, \bbO_{r-1}, -1) & \text{ if }~~ f=2 \qqa r \in 2\Z,\\
(0, 1, \bbO_{r-2}, -1) & \text{ if }~~ f=2 \qqa r \not\in 2\Z,\\
(p, -1, \bbO_{r-3-j}, 1, -p, \bbO_{j}) & \text{ if }~~ 3\leq  f=r-2j \leq r,\\
(\bbO_j, p, -1, \bbO_{r-3-j}, 1, -p) & \text{ if }~~ 3 \leq f=r+1-2j \leq r-1, \\
\end{cases}
\end{equation*}
where $\bbO_a=(0, \dots, 0)$ is the zero vector of size $a$. 
Also, for any $1\leq f \leq r$, we set
\begin{equation*}
\bbB_p(r, f):=\begin{cases}
(1, -1, \bbO_{r-1}) & \text{ if }~~ f=1,\\
\bbA_p(r, f) & \text{ if }~~ f\geq 2.
\end{cases}
\end{equation*}
\end{definition}

Then, the following is obvious from our construction.
\begin{lemma}\label{lemma: relation among bA and bbA}
For any $0\leq f \leq r$, we have
\begin{equation*}
\Upsilon(p^r) \times \bA_p(r, f)=g_p(r, f) \times \bbA_p(r, f),
\end{equation*}
where 
\begin{equation*}
g_p(r, f)=\begin{cases}
1 & \text{ if }~~ f=0,\\
p^{r-1}(p^2-1) & \text{ if }~~ f=1,\\
p^{r-1} & \text{ if } ~~f=2,\\
p^{j} & \text{ if } ~~3\leq  f \leq r \qqa j=\left[\frac{r+1-f}{2}\right].
\end{cases}
\end{equation*}
Also, for any $1\leq f\leq r$ we have
\begin{equation*}
\Upsilon(p^r) \times \bB_p(r, f)=\begin{cases}
p^{r-1}(p+1) \cdot \bbB_p(r, 1) & \text{ if }~~ f=1,\\
g_p(r, f)\cdot \bbB_p(r, f) & \text{ if }~~f\geq 2.
\end{cases}
\end{equation*}
\end{lemma}
\begin{proof}
If $r\leq 2$, we can easily verify the formulas by direct computation. For $r\geq 3$, by definition it is obvious that
\begin{equation*}
\Upsilon(p^r) \times \bA_p(r, f)=g_p(r, f) \times \bbA_p(r, f) ~~ \text{ for } f=0 ~\text{ or }~ r.
\end{equation*}
So by applying Proposition \ref{prop: degeneracy maps revisited}, we can easily prove the formulas at least for $f\neq 2$. 

Suppose the formula for $f=2$ holds for  $r-1$, i.e.,
\begin{equation*}
\Upsilon(p^{r-1}) \times \bA_p(r-1, 2)=p^{r-2} \times \bbA_p(r-1, 2).
\end{equation*}
If $r-1$ is even, then by Proposition \ref{prop: degeneracy maps revisited} we have
\begin{equation*}
\begin{split}
\Upsilon(p^r) \times \bA_p(r, 2)&=\Upsilon(p^r) \times (\beta_p(p^{r-1})^*(\bA_p(r-1, 2)))\\
&=p \cdot p^{r-2} \times (0, 1, \bbO_{r-2}, -1)
=p^{r-1} \times \bbA_p(r, 2).
\end{split}
\end{equation*}
If $r-1$ is odd, then we have
\begin{equation*}
\Upsilon(p^r) \times (\alpha_p(p^{r-1})^*(\bA_p(r-1, 2)))=p^{r-1} \times (0, 1, \bbO_{r-3}, -1, 0).
\end{equation*}
Also, we have 
\begin{equation*}
\Upsilon(p^{r}) \times (p^{r-2} \cdot \bA_p(r, r))=p^{r-2} \times (p, -1, \bbO_{r-3}, 1, -p).
\end{equation*}
Thus, by definition we have
\begin{equation*}
\Upsilon(p^r) \times \bA_p(r, 2)=p^{r-1} \times (1, \bbO_{r-1}, -1)=p^{r-1} \times \bbA_p(r, 2).
\end{equation*}
By induction, the formula for $f=2$ holds for any $r\geq 3$.
\end{proof}
\begin{remark}\label{remark: definition of cG}
For any prime $p$ and an integer $r\geq 1$, we have
\begin{equation*}
\kappa(p^r)=p^{r-1}(p^2-1) = g_p(r, f) \times \cG_p(r, f) ~~\text{ for any } 0\leq f \leq r.
\end{equation*}
This is the reason behind the definition of $\cG_p(r, f)$. Also, we have
\begin{equation*}
\kappa(p^r)=p^{r-1}(p+1) \times (p-1).
\end{equation*}
\end{remark}

By definition, it is easy to see that the greatest common divisor of the entries of $\bbA_p(r, f)$ (resp. $\bbB_p(r, f)$) is $1$. Thus, we have 
\begin{equation*}
\bbV(A_p(r, f))=\bbA_p(r, f) \qa \bbV(B_p(r, f))=\bbB_p(r, f).
\end{equation*}
Also, we can easily prove the following.
\begin{lemma}\label{lemma: bbA unipotent}
For any $0\leq f \leq r$, we have
\begin{equation*}
\bbA_p(r, f)_{p^{\iota_r(f)}}= \pm 1 \qa \bbA_p(r, f)_{p^k} =0 ~~ \text{ for all } \iota_r(f) \vtl_r k.
\end{equation*}
\end{lemma}
\begin{proof}
It is easy to check for small $r$, so suppose that $r$ is large enough. Then by definition, we have 
\begin{itemize}[--]
\item
$\bbA_p(r, 0)_{p}=-1$ and $\bbA_p(r, 0)_{p^k}=0$ for all $1 < k \leq r$.
\item
$\bbA_p(r, 1)_1=1$ and $\bbA_p(r, 1)_{p^k}=0$ for all $0 < k \leq r$.
\item
$\bbA_p(r, 2)_{p^r}=-1$ and $\bbA_p(r, 2)_{p^k}=0$ for all $1< k <r$.
\item
$\bbA_p(r, r)_{p^{r-1}}=1$ and $\bbA_p(r, r)_{p^k}=0$ for all $1<k<r-1$.
\item
$\bbA_p(r, r+1-2a)_{p^{1+a}}=-1$ and $\bbA_p(r, r+1-2a)_{p^k}=0$ for all $1+a<k<r-1$.
\item
$\bbA_p(r, r-2a)_{p^{r-1-a}}=1$ and $\bbA_p(r, r-2a)_{p^k}=0$ for all $1<k<r-1-a$.
\end{itemize}
In the last two items, $a$ is any positive integer satisfying $r+1-2a\geq 3$ and $r-2a\geq 3$, respectively. 
Thus, the result follows by the definition of the map $\iota_r$.
\end{proof}

As a corollary, the following is obvious.
\begin{corollary}\label{corollary: 6.14}
For each $1\leq i\leq r$ with $d_i=p^{f_i}$, let $D_i=B_p(r, f_i)$. 
Then the matrix $\fM=(|\bbV(D_i)_{\delta_j}|)_{1\leq i, j \leq r}$ is lower unipotent, or equivalently for any $1\leq i \leq r$, we have
\begin{equation*}
|\bbV(B_p(r, f_i))_{\delta_i}|=1  \qa  \bbV(B_p(r, f_i))_{\delta_j}=0 ~\text{ for all } j>i.
\end{equation*}
\end{corollary}

Finally, the following is easy to verify from the previous discussion.
\begin{lemma}\label{lemma: pw for B}
For any $1\leq f\leq r$, we have
\begin{equation*}
\pw_p(B_p(r, f))=\begin{cases}
\pm (p+1) & \text{ if }~~ f\geq 3 
~\text{ and }~ r-[\frac{r-f+1}{2}]\not\in 2\Z,\\
-1 & \text{ if }~~ f=1,\\
0 & \text{ otherwise}.
\end{cases}
\end{equation*}
\end{lemma}

\vspace{3mm}
\begin{proof}[Proof of Theorem \ref{theorem: case of t=1, u=0}]
As above, for each $1\leq i\leq r$ with $d_i=p^{f_i}$, we take $D_i=Z(d_i)=Z(\fp_{(f_i)})=B_p(r, f_i)$. For example, $D_1=Z(p)$ and $D_r=Z(p^3)$ (resp. $Z(p^2)$) if $r\geq 3$ (resp. $r=2$).

To prove the first assertion, we follow the strategy in Section \ref{section: first attempt}.
By Proposition \ref{proposition: generation A and B} and Corollary \ref{corollary: 6.14}, (1) and (2) are fulfilled, respectively. 
Since $p$ is odd, we have $\fh(D_i)=1$ for any $2\leq i \leq r$ by Lemma \ref{lemma: pw for B}, and therefore (3) is satisfied. This completes the proof of the first assertion.

The second assertion follows by Lemmas \ref{lemma: relation among bA and bbA} and \ref{lemma: pw for B}, and Theorem \ref{theorem: computation of order}.
\end{proof}

\ms
\subsection{Case of $t=u=1$} \label{section: case of t=u=1}
In this subsection, we prove the following.

\begin{theorem}\label{theorem: case of t=u=1}
If $r\leq 4$, then we have $\scC(2^r)=0$. Suppose that $r\geq 5$. Then we have
\begin{equation*}
\scC(2^r) \simeq \moplus_{f=3}^r \br{\ov{B^2(r, f)}}.
\end{equation*}
Also, the order of $B^2(r, f)$ is 
\begin{equation*}
\begin{cases}
1 & \text{ if }~~ 1\leq f \leq 2,\\
2^{r-3-j} & \text{ if } ~~ f=r+1-\gcd(2, r), \\
2^{r-4-j} & \text{ otherwise}, \text{ where } j=[\frac{r+1-f}{2}].
\end{cases}
\end{equation*}
\end{theorem}
As in Section \ref{section: case of level 2r}, let $D_f=C(2^r)_{2^f}$. 
\begin{proof}
Suppose that $r\leq 4$. Since the genus of $X_0(2^r)$ is zero, we have $\scC(2^r)=0$. Thus, we assume that $r\geq 5$. 
We first claim that\footnote{In contrast to Proposition \ref{proposition: generation A and B}, the vectors $\bE_k$ cannot generate $\cS_2(2^r)^0$ because the number of $\bE_k$ is smaller than $r$, the rank of $\cS_2(2^r)^0$. So we need at least two relations among $D_f$.}
\begin{equation*}
\scC(2^r) = \br{\ov{E_k} : 3\leq k \leq r}=\br{\ov{B^2(r, f)} : 3\leq f \leq r}.
\end{equation*}
Note that the second equality follows by Remark \ref{remark: relation B2 and E}.
Note also that the group $\Qdiv {2^r}$ is generated by $D_f$ for any $1\leq f \leq r$ by Lemma \ref{lemma: group of rational cuspidal divisors}.
Thus, by Lemma \ref{lemma: non-trivial relation in level 2r} below, it suffices to show that for any $1\leq f \leq r$ different from $2$ and $r-1$, there are integers $a(f, k)$ such that $D_f=\sum_{k=3}^r a(f, k) \cdot E_k$. This is obvious by definition. Indeed, we have $D_r=B^2(r, r)$ and 
\begin{equation*}
D_1=\begin{cases}
B^2(r, r-1)+2\cdot B^2(r, r) & \text{ if $r$ is odd},\\
B^2(r, r-1) &  \text{ if $r$ is even}.
\end{cases}
\end{equation*}
Since $\{B^2(r, r-1), B^2(r, r)\}=\{E_{r-1}, E_r\}$, the result for $f=1$ and $f=r$ follows. Also, for any $3\leq f \leq r-2$, we have
\begin{equation*}
D_f=\begin{cases}
E_f+2^{m(f)-1} \cdot D_1 & \text{ if $f$ is odd},\\
E_f-2^{m(f)-2} \cdot D_1 & \text{ if $f$ is even}.\\
\end{cases}
\end{equation*}
This completes the proof of the claim. 

Next, we claim that there is no relation among $E_k$. 
To begin with, we compute $\Gcd(E_k)$, $\bbV(E_k)$ and $\pw_2(E_k)$. 
If $r=5$, we can easily compute them and have $\scC(32) \simeq \zmod 4$ (cf. Section \ref{subsection: p=2, r=5})\footnote{By direct computation, $\ov{B^2(5, 3)}=\ov{B^2(5, 4)}=0$ and $\br{\ov{B^2(5, 5)}} \simeq \zmod 4$.}.
So we assume that $r\geq 6$. By direct computation, we have the following.
\small
\begin{equation*}
\begin{array}{|c|c|c|c|} \hline
&  \Gcd(E_k) & \bbV(E_k)  & \pw_2(E_k) \\ \hline
k=r-1, \text{ $r$ odd } & 2 & (0, 2, -1, \bbO_{r-4}, 1, -2) &0  \\ \hline
k=r-1, \text{ $r$ even } & 1 & (2, -1, \bbO_{r-3}, 1, -2) &0 \\ \hline
k=3 & 4 & (-2, 5, 0, -5, 2, \bbO_{r-4}) & 0 \\ \hline
4 \leq k \leq r-2, \text{ $k$ odd} & 2^{m(k)-1} & (-2, 5, -2, \bbO_{k-4}, 2, -5, 2, \bbO_{r-1-k}) & 0  \\ \hline
4 \leq k \leq r-2, \text{ $k$ even} & 2^{m(k)-1}&(4, -4, 1, \bbO_{k-4}, 2, -5, 2, \bbO_{r-1-k})  & 0 \\ \hline
k=r, \text{ $r$ odd }& 1 &  (2, -1, \bbO_{r-3}, 1, -2) & -3   \\ \hline
k=r, \text{ $r$ even }& 2 &  (2, -3, 1, \bbO_{r-2}) & -3   \\ \hline
\end{array}
\end{equation*} \normalsize

Since $\pw_2(E_r)=-3$ and $\pw_2(E_k)=0$ for all $3\leq k<r$, by Theorem \ref{thm: criterion 3} we have\footnote{Note that the orders of $E_k$ are powers of $2$, and so we can apply Theorem \ref{thm: criterion 3}.}
\begin{equation*}
\lla \ov{E_k} : 3\leq k \leq r \rra \simeq \lla \ov{E_k} : 3\leq k \leq r-1 \rra \moplus \lla \ov{E_r} \rra. 
\end{equation*}

Suppose that there are integers $a_k$ such that
\begin{equation}\label{002}
a_3 \cdot \ov{E_3}+a_4 \cdot \ov{E_4}+\cdots +a_{r-1} \cdot \ov{E_{r-1}} = 0 \in J_0(2^r).
 \end{equation}
For simplicity, let $n_k$ be the order of $E_k$. Since $\fh(E_k)=1$ for any $3\leq k\leq r-1$, by Theorem \ref{theorem: computation of order} we have
 \begin{equation}\label{041}
  n_k=\frac{2^{r-4}}{\Gcd(E_k)}=
  \begin{cases}
  2^{r-3-m(k)} & \text{ if }~~ 3\leq k \leq r-2,\\
  2^{r-4} & \tn{ if }~~ k=r-1 ~\text{ and }~ r \not\in 2\Z,\\
  2^{r-5} & \tn{ if }~~ k=r-1 ~\text{ and }~ r\in 2\Z.
  \end{cases}
  \end{equation}
Thus, as in (\ref{equation: relation r and bbV}) we have
\begin{equation*}
{\bf r}(E_k)=2^{4-r} \times  V(E_k)=n_k^{-1} \times \bbV(E_k).
\end{equation*}
Let $X=\sum_{k=3}^{r-1} a_k \cdot E_k$. Since $\ov{X}=0$, by Corollary \ref{corollary: order 1 criterion} we have
 \begin{equation*}
{\bf r}(X)=(x_0, x_1, \dots, x_r) \in \cS_1(2^r).
 \end{equation*}
Since (\ref{002}) does not change if we replace $a_k$ by $a_k-c \cdot n_k$ for any $c\in \Z$, we may assume that $0\leq a_k < n_k$. 
For simplicity, let $b_k= \frac{a_k}{n_k} \in [0, 1)$. 
As in Section \ref{subsection: p=2, r=7}, we have
\begin{equation*}
\begin{split}
x_j=\sum_{k=3}^{r-1} b_k \cdot \bbV(E_k)_{2^j}=\begin{cases}
-5b_3+ 2b_4 & \text{ if } ~~j=3,\\
2b_{k-1}-5 b_k+2b_{k+1} & \text{ if }~~ 4\leq j\leq r-3,\\
2b_{r-3}-5b_{r-2} & \text{ if }~~ j=r-2,\\
2b_{r-2}+b_{r-1} & \text{ if }~~ j=r-1,\\
-2b_{r-1} & \text{ if }~~ j=r.\\
\end{cases} 
\end{split}
 \end{equation*}
For an integer $c\geq 1$, let $\bbM(c)$ denote the set of rational numbers between $0$ and $1$ whose denominators are exactly $2^c$ when reduced to lowest terms, i.e.,
\begin{equation*}
\bbM(c):=\left\{\frac{a}{2^c} : 1\leq a < 2^c \qqa a \in \Z \sm 2\Z \right\}.
\end{equation*}
Since $n_k$ is a power of $2$, we have $b_k=0$ or $b_k \in \bbM(c)$ for some $c\geq 1$.
Suppose that $b_{r-1} \neq 0$. Since $x_r\in \Z$, we have $b_{r-1} \in \bbM(1)$. By the condition for $x_{r-1}$, we have $b_{r-2} \in \bbM(2)$. Also, by the condition for $x_{r-2}$ we have $b_{r-3} \in \bbM(3)$. Similarly, we can deduce that $b_{r-k} \in \bbM(k)$ for any $3\leq k \leq r-1$ by the conditions for $x_j$ with $4\leq j \leq r-3$. On the other hand, we then have $x_3=-5b_3+2b_4 \in \bbM(r-3)$, which is a contradiction to $x_j \in \Z$. Thus, we have $b_{r-1}=0$. 

Next, suppose that $b_s \neq 0$ for some $3\leq s\leq r-2$ and $b_i=0$ for all $s<i\leq r-1$. 
By the condition for $x_{s+1}$, we have $b_s \in \bbM(1)$. Similarly as above, we obtain $b_k \in \bbM(s-k+1)$ by the conditions for $x_j$ with $4\leq j\leq r-2$.
In particular, we have $b_4 \in \bbM(s-3)$ and $b_3 \in \bbM(s-2)$. As above, we then have $x_3 \in \bbM(s-2)$, which is a contradiction to $x_j \in \Z$. Thus, we have $b_i=0$ for all $3 \leq i \leq r-1$. 
This completes the proof of the claim.

Finally, by (\ref{041}) and the table above, we obtain the result for the order. This completes the proof.
\end{proof}

\begin{lemma}\label{lemma: non-trivial relation in level 2r}
For an integer $r\geq 5$, we have
\begin{equation*}
\ov{D_2}=-\ov{D_1}-\textstyle\sum_{f=3}^{r-4} 2^{\max(0, \, 2f-r)} \cdot \ov{D_f}+2^{r-4} \cdot \ov{D_r}
\end{equation*}
and
\begin{equation*}
\ov{D_{r-1}}=-\ov{D_r}-\textstyle\sum_{f=4}^{r-2} 2^{\max(0, \, r-2f)} \cdot \ov{D_f}.
\end{equation*}
\end{lemma}
\begin{proof}
Note that the genus of $X_0(16)$ is zero, and so $J_0(16)=0$. Thus, any pullbacks of the divisors on $X_0(16)$ by the degeneracy maps become trivial in $J_0(2^r)$. In particular, we have
\begin{equation*}
\ov{\pi_1(2^r, 2^4)^*(C(16)_{16})}=\ov{\pi_2(2^r, 2^4)^*(C(16)_{16})}=0 \in J_0(2^r).
\end{equation*}
By direct computation using Lemma \ref{lemma: degeneracy maps pullback on RCD}, we have
\begin{equation*}
\pi_1(2^r, 2^4)^*(C(16)_{16})=\textstyle\sum_{f=4}^{r} 2^{\max(0, \, r-2f)} \cdot D_f
\end{equation*}
and
\begin{equation*}
\pi_2(2^r, 2^4)^*(C(16)_{16})=2^{r-4}\cdot D_r-\textstyle\sum_{f=1}^{r-4} 2^{\max(0, \, 2f-r)} \cdot D_f.
\end{equation*}
This completes the proof.
\end{proof}

\ms
\subsection{First reduction}
From now on, we assume that $t\geq 2$. In this subsection, we prove the following.
\begin{theorem}\label{theorem: main theorem in 6.4}
We have
\begin{equation*}
\scC(N) \simeq \br{\ov{Z^1(d_i)} : 1\leq i < \fm+r}  \moplus\left(\moplus_{j=\fm+r}^\fn \br{\ov{Z(d_j)}}\right),
\end{equation*}
where $r=1$ if $u=0$ and $r=r_u$ if $u\geq 1$.
\end{theorem}

First, we claim the following.
\begin{proposition}\label{proposition: generation of Z}
For any $t\geq 2$, we have
\begin{equation*}
\cS_2(N)^0=\br{ \Z^1(d) : d \in \cD_N^0}.
\end{equation*}
\end{proposition}
\begin{proof}
We prove this by induction on $t$. 

Since the claim holds for $t=1$ by Proposition \ref{proposition: generation A and B}, we suppose that 
\begin{equation}\label{equation: 6.4}
\cS_2(M)^0 =\br{ \Z^1(M, d') : d' \in \cD_M^0}, ~~~\text{ where } M=\textstyle\prod_{i=1}^{t-1} p_i^{r_i}=N/{p_t^{r_t}}.
\end{equation}
Here, we use the notation $\Z^1(M, d')$ to emphasize that they are of level $M$ (and to distinguish them from $\Z^1(d')$, which are of level $N$). In other words, for any $J=(f_1, \dots, f_{t-1}) \in \Omega(t-1)$, we have
\begin{equation*}
\Z^1(M, \fp_J):=\motimes_{i=1, \,i\neq m}^{t-1} \bA_{p_i}^{f_i} \motimes \bB_m, ~~\text{ where } m=m(J).
\end{equation*}

To prove the assertion, it suffices to show that $\cS_2(N)^0 \subset \br{\Z^1(d) : d \in \cD_N^0}$ since the other inclusion is obvious. By Lemma \ref{lemma: group of rational cuspidal divisors}, it suffices to show that 
for any non-trivial divisor $\delta$ of $N$, there are integers $a(d)$ such that
\begin{equation*}
{\bf f}_\delta:=\Phi_N(C(N)_\delta) = \sum_{d \in \cD_N^0} a(d) \cdot \Z^1(d).
\end{equation*}

For simplicity, let $p=p_t$ and $r=r_t$. Also, let $f=\tn{val}_p(\delta)$ and $\delta'=\gcd(\delta, M)$.

If $\delta'\neq 1$, then by (\ref{equation: 6.4}) there are integers $b(d')$ such that
\begin{equation*}
{\bf f}(M)_{\delta'}:=\varphi(\gcd(\delta', M/{\delta'})) \cdot {\bf e}(M)_1-{\bf e}(M)_{\delta'}=\sum_{d' \in \cD_M^0} b(d') \cdot \Z^1(M, d').
\end{equation*}
Also by Proposition \ref{proposition: generation A and B}, we have ${\bf e}(p^r)_{p^f}=\sum_{k=0}^r c(k) \cdot \bA_p(r, k)$ for some $c(k) \in \Z$.
Furthermore, $\varphi(p^{\min(f, \,r-f)})\cdot {\bf e}(p^r)_1-{\bf e}(p^r)_{p^f}=\sum_{k=1}^r -c(k) \cdot \bB_p(r, k)$.

Now, we prove that for any $\delta\in \cD_N^0$, we have ${\bf f}_\delta \in \br{ \Z^1(d) : d \in \cD_N^0}$.
Suppose first that $f=0$, and so $\delta'\neq 1$. 
Then since $\Z^1(d)=\Z^1(M, d) \motimes \bA_p(r, 0)$ for any $d \in \cD_M^0$, we have
\begin{equation*}
{\bf f}_\delta={\bf f}(M)_{\delta'} \motimes \bA_p(r, 0)=\sum_{d \in \cD_M^0} b(d) \cdot \Z^1(d).
\end{equation*}

Next, suppose that $f\geq 1$ and $\delta'=1$. Let
\small
\begin{equation*}
\bX_1:={\bf e}(M)_1 \motimes (\varphi(p^{\min(f, \,r-f)})\cdot {\bf e}(p^r)_1-{\bf e}(p^r)_{p^f})={\bf e}(M)_1 \motimes \left(\sum_{k=1}^r -c(k) \cdot \bB_p(r, k) \right).
\end{equation*}\normalsize
Since ${\bf e}(M)_1=\motimes_{i=1}^{t-1} \bA_{p_i}(r_i, 0)$, we have 
\begin{equation*}
{\bf f}_\delta=\bX_1=\sum_{k=1}^r -c(k) \cdot \Z^1(p^k).
\end{equation*}

Finally, suppose that $\delta' \neq 1$ and $f\geq 1$. Let $\bX_2:={\bf f}(M)_{\delta'} \motimes {\bf e}(p^r)_{p^f}$. Since $\Z^1(M, d') \motimes \bA_p(r, k)=\Z^1(d'p^k)$ for any $d' \in \cD_M^0$, we have
\begin{equation*}
\begin{split}
\bX_2&=\left( \sum_{d' \in \cD_M^0} b(d') \cdot \Z^1(M, d') \right) \motimes \left( \sum_{k=0}^r c(k) \cdot \bA_p(r, k) \right)\\
&=\sum_{d=d'p^k \in \cD_N^0, \, d'\neq 1} b(d') \cdot c(k) \cdot \Z^1(d).
\end{split}
\end{equation*}
Since ${\bf f}_\delta=\varphi(\gcd(\delta', M/{\delta'})) \cdot\bX_1+\bX_2$, the result follows by induction.
\end{proof}

Next, we show the following.
\begin{proposition}\label{proposition: unipotent of Z}
For any non-squarefree divisor $d_i$ of $N$, we have
\begin{equation*}
|\bbV(Z^1(d_i))_{\delta_i}|= 1 \qa \bbV(Z^1(d_j))_{\delta_i}=0 \text{ for all } j<i.
\end{equation*}
\end{proposition}
\begin{proof}
It suffices to show that
\begin{enumerate}
\item
for any squarefree divisor $d$ of $N$ and any non-squarefree divisor $\delta$ of $N$, we have
$\bbV(Z^1(d))_\delta=0$, and
\item
for any non-squarefree divisors $d_i$ of $N$, we have
$|\bbV(Z^1(d_i))_{\delta_i}|=1$ and $\bbV(Z^1(d_i))_{\delta_j}=0$ for all $j>i$.
\end{enumerate}

Let $I=(f_1, \dots, f_t) \in \Omega(t)$ and $J=(a_1, \dots, a_t) \in \square(t)$. By definition, there is an index $h$ such that $a_h\geq 2$.

First, suppose that $I \in \Delta(t)$ with $m=m(I)$. Then by Theorem \ref{theorem: order defined by tensors}, we have
\begin{equation*}
\bbV(Z^1(\fp_I))=\textstyle\motimes_{i=1,\, i\neq m}^t \bbA_{p_i}(r_i, f_i) \motimes \bbB_{p_m}(r_m, 1).
\end{equation*}
Since $f_h \in \{0, 1\}$, $\bbA_{p_h}(r_h, f_h)_{p_h^{a_h}}=\bbB_{p_h}(r_h, 1)_{p_h^{a_h}}=0$ by Lemma \ref{lemma: bbA unipotent}, and so
\begin{equation*}
\bbV(Z^1(\fp_I))_{\fp_J}=\textstyle\prod_{i=1, \, i\neq m}^t \bbA_{p_i}(r_i, f_i)_{p_i^{a_i}} \times \bbB_{p_m}(r_m, 1)_{p_m^{a_m}}=0.
\end{equation*}

Next, suppose that $I \in \square(t)$. Let $\fp_I=d_i$ and $\delta_j=\fp_J$. If we write $\delta_i=\fp_K$ for some $K=(b_1, \dots, b_t)$, then we have $b_i=\iota_{r_i}(f_i)$ (cf. Remark \ref{remark: iota isomorphism}).
Since $|\bbA_{p_i}(r_i, f_i)_{p_i^{b_i}}|=1$ by Lemma \ref{lemma: bbA unipotent}, we have
\begin{equation*}
|\bbV(Z^1(d_i))_{\delta_i}|=\textstyle\prod_{i=1}^t |\bbA_{p_i}(r_i, f_i)_{p_i^{b_i}}|= 1.
\end{equation*}
Also, if $i<j$ then $K \vtl J$ by definition, and hence there is an index $h$ such that $b_h \vtl_{r_h} a_h$. Since $\bbA_{p_h}(r_h, f_h)_{p_h^{a_h}}=0$ by Lemma \ref{lemma: bbA unipotent}, we have
\begin{equation*}
\bbV(Z^1(d_i))_{\delta_j}=\bbV(Z^1(\fp_I))_{\fp_J}=\textstyle\prod_{i=1}^t \bbA_{p_i}(r_i, f_i)_{p_i^{a_i}}=0.
\end{equation*}
This completes the proof.
\end{proof}

Finally, we prove the following.
\begin{proposition}\label{proposition: gcd and h of Z}
Let $I=(f_1, \dots, f_t) \in \Omega(t)$. Then we have
\begin{equation*}
\Gcd(Z(\fp_I))=\begin{cases}
\prod_{i=1}^t g_{p_i}(r_i, f_i) & \text{ if }~~ I \in \square(t),\\
\prod_{i=1, \, i\neq m}^t g_{p_i}(r_i, f_i) \times p_m^{r_m-1}(p_m+1) & \text{ if }~~ I \in \Delta(t),
\end{cases}
\end{equation*}
where $m=m(I)$ and $g_{p_i}(r_i, f_i)$ is defined in Lemma \ref{lemma: relation among bA and bbA}. 
Furthermore, we have $\fh(Z(\fp_I))=2$ if and only if one of the following holds.
\begin{enumerate}
\item
$I=A(1)$.
\item
$u\geq 1$ and $I=E(u)$.
\item
$u\geq 1$, $3\leq r_u\leq 4$, $f_u=3$ and $f_i=1$ for all $i \neq u$.
\item
$u\geq 1$, $r_u\geq 5$, $f_u=r_u+1 -\gcd(2, r_u)$ and $f_i=1$ for all $i \neq u$.
\end{enumerate}
\end{proposition}
\begin{proof}
First, let $I=(f_1, \dots, f_t) \in \square(t)$. By definition, there is an index $h$ such that $f_h\geq 2$.
Note that by Theorem \ref{theorem: order defined by tensors} and Lemma \ref{lemma: relation among bA and bbA}, we have
\begin{equation*}
\Gcd(Z^1(\fp_I))=\textstyle\prod_{i=1}^t \Gcd(A_{p_i}(r_i, f_i))=\prod_{i=1}^t g_{p_i}(r_i, f_i).
\end{equation*}
As already mentioned in Remark \ref{remark: relation B2 and E}, we have $\Gcd(Z^1(\fp_I))=\Gcd(Z(\fp_I))$ and so the first assertion follows in this case.
Note that the summation of all entries of $\bbA_{p_i}(r_i, f_i)$ is $(p_i-1)^{1-f_i}$ if $f_i \in \{0, 1\}$, and $0$ otherwise. 
Thus, by Theorem \ref{theorem: order defined by tensors} we have $\pw_{p_n}(Z^1(\fp_I))=0$ unless $n=h$ and $f_i\leq 1$ for all $i\neq h$. 
Moreover, by Lemma \ref{lemma: pw for B} we have $\pw_{p_h}(Z^1(\fp_I))\not\in 2\Z$ if and only if $h=u\geq 1$, $f_u\geq 3$, $r_u-[\frac{r_u-f_u+1}{2}]$ is odd, and $f_i=1$ for all  $i \neq u$.
Thus, we obtain the result for $I \not\in \cT_u$ as $Z^1(\fp_I)=Z(\fp_I)$.
Suppose further that $I \in \cT_u$, i.e., $h=u\geq 1$, $r_u\geq 5$, $f_u\geq 3$ and $f_i=1$ for all $i \neq u$. Note that for any $r\geq 5$ (and $3\leq f \leq r$), $\pw_2(B^2(r, f)) \not\in 2\Z$ if and only if $f=r+1-\gcd(2, r)$. Therefore the result follows in this case.

Next, let $I=(f_1, \dots, f_t) \in \Delta(t)$ with $m=m(I)$. Similarly as above, we have
\begin{equation*}
\begin{split}
\Gcd(Z(\fp_I))&=\textstyle\prod_{i=1, \,i\neq m}^t \Gcd(A_{p_i}(r_i, f_i)) \times \Gcd(B_{p_m}(r_m, 1))\\
&=\textstyle\prod_{i=1, \, i\neq m}^t g_{p_i}(r_i, f_i) \times p_m^{r_m-1}(p_m+1).
\end{split}
\end{equation*}
Also, since the summation of all entries of $\bbB_p(r, 1)$ is zero, we have $\pw_{p_n}(Z(\fp_I))=0$ unless $n=m$. Furthermore, we have 
$\pw_{p_m}(Z(\fp_I))=-\prod_{i=1, \,i\neq m}^t (p_i-1)^{1-f_i}$.
Thus, $\pw_{p_m}(Z(\fp_I)) \not\in 2\Z$ if and only if one of the following holds.
\begin{itemize}[--]
\item
$f_i=1$ for all $i$, i.e., $I=A(1)$.
\item
$f_u=0$ and $f_i=1$ otherwise if $u\geq 1$, i.e., $I=E(u)$.
\end{itemize}
This completes the proof.
\end{proof}

\begin{corollary}\label{corollary: gcd and h of Z}
For any $d \in \cD_N^0$, we have
\begin{equation*}
\cG(N, d)=\frac{\kappa(N)}{\Gcd(Z(d))} \qa \cH(N, d)=\fh(Z(d)).
\end{equation*}
\end{corollary}
\begin{proof}
Note that $\kappa(N)=\prod_{i=1}^t \kappa(p_i^{r_i})$. So by Remark \ref{remark: definition of cG}, we have the first equality. Also, the second equality follows from  Proposition \ref{proposition: gcd and h of Z} (and the definition of $\cH(N, d)$).
\end{proof}

Combining all the results above, we now prove Theorem \ref{theorem: main theorem in 6.4}.

By Proposition \ref{proposition: generation of Z}, it suffices to show that
\begin{equation*}
\br{\ov{Z^1(d_i)} : 1 \leq i \leq \fn} \simeq \br{\ov{Z^1(d_i)} : 1\leq i < \fm+r} \moplus\left( \moplus_{j=\fm+r}^\fn \br{\ov{Z(d_j)}}\right).
\end{equation*}

First, suppose that $N$ is odd, i.e., $u=0$. Then $Z^1(d)=Z(d)$ for any $d \in \cD_N^0$ and so by Proposition \ref{proposition: gcd and h of Z}, we have $\fh(Z^1(d))=1$ for any $d \in \cD_N^\nsqf$. Hence by Proposition \ref{proposition: unipotent of Z}, we can successively apply Theorem \ref{thm: criterion 1} for $C_i=Z^1(d_i)$ and obtain
\begin{equation*}
\br{\ov{Z^1(d_i)} : 1 \leq i \leq \fn} \simeq \br{\ov{Z^1(d_i)} : 1\leq i < \fm+1} \moplus\left(\moplus_{j=\fm+1}^\fn \br{\ov{Z(d_j)}}\right).
\end{equation*}

Next, suppose that $N$ is even, i.e., $u\geq 1$. Let $r=r_u\geq 1$ and 
let $d_j=\fp_J$ for some $J=(f_1, \dots, f_t) \in \Omega(t)$. 
Suppose that $\fm+r \leq j \leq \fn$. Then there is an index $h \neq u$ such that $f_h\geq 2$ and $J \not\in \cT_u$ (cf. Remark \ref{remark: 2power ordering}). Thus, we have  $Z^1(d_j)=Z(d_j)$ and so $\fh(Z^1(d_j))=\fh(Z(d_j))=1$ by Proposition \ref{proposition: gcd and h of Z}.
Similarly as above, we obtain
\begin{equation*}
\begin{split}
\br{\ov{Z^1(d_i)} : 1\leq i \leq \fn} &\simeq \br{ \ov{Z^1(d_i)} : 1\leq i < \fm +r} \moplus\left(\moplus_{j=\fm+r}^\fn \br{\ov{Z^1(d_j)}}\right)\\
&= \br{ \ov{Z^1(d_i)} : 1\leq i < \fm +r} \moplus\left(\moplus_{j=\fm+r}^\fn \br{\ov{Z(d_j)}}\right).
\end{split}
\end{equation*}
This completes the proof. \qed

\begin{remark}
When $N$ is divisible by $32$, one may try to find a ``better'' replacement of $\bB_2(r, f)$ than $\bB^2(r, f)$, which can be used to directly prove
\begin{equation}\label{112}
\br{ \ov{Z^1(d_i)} : 1\leq i < \fm +r}\simeq \br{ \ov{Z^1(d_i)} : 1\leq i \leq \fm}\moplus\left(\moplus_{j=\fm+1}^{\fm+r-1} \br{ \ov{Z(d_j)}}\right).
\end{equation}
Then it seems necessary to find a replacement $D \in \Qdiv {2^r}$ of $B_2(r, r)$ with $\pw_2(D)\in 2\Z$. But then the order of $D$ is at most $2^{r-4}$ by Theorem \ref{theorem: computation of order}, whereas that of $B_2(r, r)$ is $2^{r-3}$ if $r\geq 5$ is odd. Thus, it seems impossible to find such a divisor, and unfortunately we only have an indirect proof of the first assertion of Theorem \ref{theorem: chapter 6 main}, which will be presented in Section \ref{section: final step}. 

On the other hand, if $N$ is not divisible by $32$, then (\ref{112}) is obviously true because the order of $Z^1(d_j)$ is $1$ for any $\fm<j < \fm+r$. 
\end{remark}

\ms
\subsection{Generation}\label{section: generation}
As above, we assume that $t\geq 2$. In this subsection, we prove the following.
\begin{theorem}\label{proposition: generation Y2}
For each $i=0$, $1$ or $2$, we have
\begin{equation*}
\br{ \Z(\fp_I) : I \in \Delta(t) } \motimes_\Z \Z_\ell = \br{ \bY^i(\fp_I) : I \in \Delta(t) } \motimes_\Z \Z_\ell.
\end{equation*}
\end{theorem}

For simplicity, we use the following abbreviated notation.
\begin{notation}
For any $1\leq i<j \leq t$, let
\begin{itemize}[--]
\item
$\bA_i^{f_i}:=\bA_{p_i}(r_i, f_i), ~~\bB_i:=\bB_{p_i}(r_i, 1) \qa \bD(i, j):=\bD(p_i^{r_i}, p_j^{r_j})$.
\item
$\bY^i(I):=\bY^i(\fp_I) \qa \Z(I):=\Z(\fp_I)$.
\item
$\gamma_i:=p_i^{r_i-1}(p_i+1) \qa G^i_j:=\gcd(\gamma_i, \, \gamma_j)$.
\item
$U^i_j=\gamma_j \cdot (G^i_j)^{-1} \in \Z_\ell^\times \qa W^i_j=\gamma_i \cdot (G^i_j)^{-1} \in \Z$.
\end{itemize}

For any $I=(f_1, \dots, f_t) \in \Delta(t)$ and a subset $S$ of $\{1, \dots, t\}$, let
\begin{equation*}
I^S:=(a_1, \dots, a_t) ~~\text{ such that }~~ \begin{cases}
a_i=f_i & \text{ if } ~~i \not\in S, \\
a_i=1-f_i & \text{ if } ~~i \in S, 
\end{cases} 
\end{equation*}
and 
\begin{equation*}
\bA(I)^S:=\motimes_{i=1, \, i\not\in S}^t \bA_i^{f_i}.
\end{equation*}
For instance, for any $I=(f_1, \dots, f_t) \in \Delta(t)$ with $m(I)=m$ and $n(I)=n \leq t$, we can write
\begin{equation*}
\begin{split}
\bY^0(I)&=\bA(I)^{\{m, n\}} \motimes \bD(m, n)=\bA(I)^{\{m, n, k\}} \motimes \bD(m, n) \motimes \bA_k^{f_k},\\
\Z(I)&=\bA(I)^{\{m\}} \motimes \bB_m=\bA(I)^{\{m, n\}} \motimes \bB_m \motimes \bA_n^0.
\end{split}
\end{equation*}

For any $1\leq k \leq t$, let 
\begin{equation*}
\Delta(t)^k:=\{ I=(f_1, \dots, f_t) \in \Delta(t) : f_k=1\},
\end{equation*}
and for any $I=(f_1, \dots, f_t) \in \Delta(t)^k$, let 
\begin{equation*}
\bX_k(I) :=  \bA(I)^{\{k\}} \motimes \bB_k = \motimes_{i=1, \, i\neq k}^t \bA_i^{f_i} \motimes \bB_k.
\end{equation*}

Also, let $\fY^i := \br{ \bY^i(\fp_I) : I \in \Delta(t) }$, $\fZ^0:=\br{ \Z(\fp_I) : I \in \Delta(t) }$ and 
\begin{equation*}
\fZ^1:=\br{ \bX_k(I) : I \in \Delta(t)^k ~~\text{ for any } 1\leq k \leq t }.
\end{equation*}
Furthermore, let 
\begin{equation*}
\fY^i_\ell :=\fY^i \motimes_\Z \Z_\ell \qa \fZ^i_\ell:=\fZ^i \motimes_\Z \Z_\ell.
\end{equation*}
From now on, we regard $\fY^i$ and $\fZ^i$ as subsets of $\fY^i_\ell$ and $\fZ^i_\ell$, respectively. 
\end{notation}

\vspace{3mm}
We will prove the following by which one can easily deduce the theorem.
\begin{enumerate}[(A)]
\item \label{a}
$\fZ^0 \subset \fY^0_\ell$.
\item \label{b}
$\fY^0 \subset \fY^1_\ell$.
\item \label{c}
$\fY^1 \subset \fY^2_\ell$.
\item \label{d}
$\fY^2 \subset \fZ^1$.
\item \label{e}
$\fZ^1=\fZ^0$.
\end{enumerate}

\vspace{2mm}
Note that for each $i$, we have $\bB_i=\gamma_i \cdot \bA_i^0-\bA_i^1$. Also, we have 
\begin{equation*}
\begin{split}
\bD(i, j)&=U^i_j \cdot  \bB_i \motimes \bA_j^0-W^i_j \cdot \bA_i^0 \motimes \bB_j,\\
\bA_i^1 \motimes \bB_j&=\bB_i \motimes \bA_j^1-G^i_j \cdot \bD(i, j).
\end{split}
\end{equation*}
Furthermore, we have
\begin{equation*}
\bB_i \motimes \bB_j=\gamma_j \cdot \bB_i \motimes \bA_j^0-\bB_i \motimes \bA_j^1=\gamma_i \cdot \bA_i^0 \motimes \bB_j-\bA_i^1 \motimes \bB_j.
\end{equation*}
These equalities are frequently used below without further mention.

\vspace{3mm}
\begin{proof}[Proof of \ref{a}]
Let $I \in \Delta(t)$ with $m(I)=m$ and $n(I)=n$. It suffices to show that
$\Z(I) \in \fY^0_\ell$. We prove it by (backward) induction on $m$. 

First, if $m=t$ then $I \in \cE$ and hence $\Z(I)=\bY^0(I) \in \fY^0$. 

Next, assume that $m<t$ and $\Z(J) \in \fY^0_\ell$ for any $J \in \Delta(t)$ with $m(J)\geq m+1$. 
If $n=t+1$, then $I\in \cE$ and hence $\Z(I)=\bY^0(I) \in \fY^0$. Assume further that $n\leq t$. Let $J=I^{\{m, n\}}$ and let
\begin{equation*}
\bX:=\bX_n(J)=\bA(J)^{\{n\}} \motimes \bB_n=\bA(I)^{\{m, n\}} \motimes \bA_m^0 \motimes \bB_n.
\end{equation*}
Note that since $m(J) \geq m+1$, we have $\Z(J) \in \fY^0_\ell$ by  induction hypothesis. 
If $n=m+1$, then we have $\bX=\Z(J) \in \fY^0_\ell$. Suppose that $n>m+1$. 
Since $\bA_{m+1}^1 \motimes \bB_n = \bB_{m+1} \motimes \bA_n^1-G^{m+1}_n \cdot \bD(m+1, n)$, we then have
\begin{equation*}
\bX=\bA(I)^S \motimes \bA_m^0 \motimes \bA_{m+1}^1 \motimes \bB_n=\Z(J)-G^{m+1}_n \cdot \bY^0(I^{\{m\}}) \in \fY^0_\ell,
\end{equation*}
where $S=\{m, m+1, n\}$.
Therefore we always have $\bX \in \fY^0_\ell$, and so
\begin{equation*}
U^m_n \cdot \Z(I)=\bA(I)^{\{m, n\}} \motimes (\bD(m, n)+W^m_n \cdot \bA_m^0 \motimes \bB_n)=\bY^0(I)+W^m_n \cdot \bX \in \fY^0_\ell.\\ 
\end{equation*}
Since $U^m_n \in \Z_\ell^\times$,  we have $\Z(I) \in \fY^0_\ell$ and the result follows by induction. 
\end{proof} 

\begin{remark}\label{remark: Z with Y0}
By modifying the induction hypothesis, we can prove that $\bX$ is written as a $\Z_\ell$-linear combination of $\bY^0(J)$ for $J \in \Delta(t)$ satisfying $m(J)\geq m+1$.
Thus, for each $I \in \Delta(t)$ with $m=m(I)$, we can write 
\begin{equation*}
\Z(I)=\textstyle \sum_{J \in \Delta(t)} a(J) \cdot \bY^0(J) ~~~~\text{ with } a(J) \in \Z_\ell,
\end{equation*}
where $a(J)=0$ unless $J=I$ or $m(J)\geq m+1$. Also, we have $a(I) \in \Z_\ell^\times$. This is the reason behind the definition of $\bY^0(I)$.
\end{remark}

\begin{proof}[Proof of \ref{b}]
Let $I \in \Delta(t)$ with $m=m(I)$, $n=n(I)$ and $k=k(I)$. As above, it suffices to show that
$\bY^0(I) \in \fY^1_\ell$. 

First, let $I\not\in \cE \cup \cH_u$. Then $\bY^0(I) =\bY^1(I) \in \fY^1$.

Next, let $I \in \cE$, i.e., $I=A(m)$. If $m\geq u$ then $\bY^0(I)=\bY^1(I) \in \fY^1$. Suppose that $m<u$. By definition, we have $\bY^1(I)=\bA(I)^{\{m, u\}} \motimes \bA_m^1 \motimes \bB_u$,
\begin{equation*}
 \bY^1(I^{\{m\}})=\bA(I)^{\{m, u\}}\motimes \bA_m^0 \motimes \bB_u  ~~\text{ and }~~\bY^1(I^{\{u\}})=\bA(I)^{\{m, u\}} \motimes \bD(m, u).
 \end{equation*}
Thus, we have
\begin{equation*}
\begin{split}
\bY^0(I)&=\bA(I)^{\{m, u\}} \motimes \bB_m \motimes \bA_u^1=\bA(I)^{\{m, u\}}\motimes \bB_m \motimes (\gamma_u \cdot \bA_u^0-\bB_u)\\
&=\gamma_u \cdot \bA(I)^{\{m, u\}} \motimes \bB_m \motimes \bA_u^0-\gamma_m \cdot \bY^1(I^{\{m\}})+\bY^1(I).
\end{split}
\end{equation*}
Since $U^m_u \cdot \bB_m \motimes \bA_u^0=\bD(m, u)+W^m_u \cdot \bA_m^0\motimes \bB_u$, we have
\begin{equation*}
U^m_u \cdot \bA(I)^{\{m, u\}} \motimes \bB_m \motimes \bA_u^0 =  \bY^1(I^{\{u\}})+W^m_u \cdot \bY^1(I^{\{m\}}) \in \fY^1.
\end{equation*}
Since $U^m_u \in \Z_\ell^\times$, we have $\bY^0(I) \in \fY^1_\ell$, as desired.

Lastly, let $I=(f_1, \dots, f_t) \in \cH_u$, i.e., $m<n=u<k \leq t$. For simplicity, 
for any $J \in \cH_u$ with $m(J)=m'$ and $k(J)=k'$, let
\begin{equation*}
\bU^{\epsilon_1, \epsilon_2}(J):=\bA(J)^{\{m', u, k'\}} \motimes \bA_{m'}^{\epsilon_1} \motimes \bA_u^{\epsilon_2} \motimes \bB_{k'} ~\text{ for any } \epsilon_i \in \{0, 1\}.
\end{equation*}
We prove that $\bY^0(I) \in \fY^1_\ell$ by (backward) induction on $m$, or more generally 
\begin{equation*}
\{ \bY^0(I), \bU^{0, 0}(I), \bU^{0, 1}(I), \bU^{1, 0}(I)\} \subset \fY^1_\ell.
\end{equation*}

\noindent {\bf Step 1}. Assume that $m=u-1$. 

(1) Let $J=I^{\{m, u\}}$. Then we have $m(J)=u$ and $n(J)=k$. Thus, we have
\begin{equation*}
\bY^1(J)=U^u_k \cdot \bA(I)^{\{m, u\}} \motimes \bA_m^0 \motimes \bB_u-W^u_k \cdot \bA(I)^{\{m, k\}} \motimes \bA_m^0 \motimes \bB_k.
\end{equation*}
Since $I \in \cH_u$, we have
\begin{equation*}
\bY^1(I)=U^m_k \cdot \bA(I)^{\{m\}} \motimes \bB_m-W^m_k \cdot \bA(I)^{\{m, k\}} \motimes \bA_m^0 \motimes \bB_k.
\end{equation*}
Note that $(U_k^m)^{-1} \cdot U^m_u \cdot W^m_k  \cdot U^u_k=W^m_u \cdot W^u_k=\gamma_m \cdot \gamma_u \cdot (G^m_u \cdot G^u_k)^{-1}$, and therefore 
\begin{equation*}
\begin{split}
\bY^0(I)&=U^m_u \cdot \bA(I)^{\{m\}} \motimes \bB_m-W^m_u \cdot \bA(I)^{\{m, u\}} \motimes \bA_m^0 \motimes \bB_u\\
&=(U^m_k)^{-1} \cdot U^m_u \cdot \bY^1(I)-(U^u_k)^{-1} \cdot W^m_u \cdot \bY^1(J) \in \fY^1_\ell.
\end{split}
\end{equation*}

(2) Note that $\bY^0(J) = \bY^1(J)$ for any $J \in \Delta(t)$ with $m(J) \geq u$. Therefore by Remark \ref{remark: Z with Y0}, for any $J \in \Delta(t)$ with $m(J) \geq u$, we have
\begin{equation}\label{a1}
\textstyle\Z(J) \in \br{ \bY^1(K) : K \in \Delta(t) \text{ with } m(K)\geq u} \motimes_\Z \Z_\ell \subset \fY^1_\ell.
\end{equation}

Let $J=I^{\{m, k\}}$. Then we have $m(J)=u+1$ and so $\Z(J) \in \fY^1_\ell$ by (\ref{a1}). 

If $k=u+1$, then we have $\bU^{0, 0}(I)=\Z(J) \in \fY^1_\ell$. 

If $k>u+1$, then for $S=\{m, u, u+1, k\}$ we have
\begin{equation*}
\begin{split}
\bU^{0, 0}(I)&=\bA(I)^S \motimes \bA_m^0 \motimes \bA_u^0 \motimes (\bA_{u+1}^1 \motimes \bB_k) \\
&=\bA(I)^S \motimes \bA_m^0 \motimes \bA_u^0 \motimes (\bB_{u+1} \motimes \bA_k^1-G^{u+1}_k \cdot \bD(u+1, k))\\
&=\Z(J)-G^{u+1}_k \cdot \bY^1(I^{\{m\}}) \in \fY^1_\ell.
\end{split}
\end{equation*}

(3) Let $S=\{m, u, k\}$. Since $m=u-1$, we have $m(I^S)=u$. Hence by (\ref{a1}), we have 
\begin{equation*}
\Z(I^S)=\bA(I)^S \motimes \bA_m^0 \motimes \bB_u \motimes \bA_k^1 \in \fY^1_\ell,
\end{equation*}
and therefore
\begin{equation*}
\begin{split}
\bU^{0, 1}(I)&=\bA(I)^S \motimes \bA_m^0 \motimes (\bB_u \motimes \bA_k^1-G^u_k \cdot \bD(u, k))\\
&=\Z(I^S)-G^u_k \cdot \bY^1(I^{\{m, u\}}) \in \fY^1_\ell.
\end{split}
\end{equation*}

(4) Let $S=\{m, u, k\}$ and let $\bV:=\bA(I)^S \motimes \bB_m \motimes \bA_u^0 \motimes \bA_k^1$. Then we have
$U^m_u \cdot \bV=\bY^0(I^{\{k\}})+W^m_u \cdot \Z(I^S)$ because 
$U^m_u \cdot \bB_m \motimes \bA_u^0=\bD(m, u)+W^m_u \cdot \bA_m^0 \motimes \bB_u$.

If $I^{\{k\}}\in \cH_u$, then $\bY^0(I^{\{k\}}) \in \fY^1_\ell$ by (1) because $m(I^{\{k\}})=m=u-1$. 

If $I^{\{k\}} \not\in \cH_u$, then we have $I^{\{k\}} \in \cH_u^1$ since $n(I^{\{k\}})=u$.
Thus, $\bY^0(I^{\{k\}}) \in \fY^1_\ell$ by the previous cases. 
Also, by (3) we have $\Z(I^S) \in \fY^1_\ell$. Therefore we always have $\bV \in \fY^1_\ell$. Hence, we have
\begin{equation*}
\begin{split}
\bU^{1, 0}(I)&=\bA(I)^S \motimes \bA_m^1 \motimes \bA_u^0 \motimes \bB_k=\bA(I)^S \motimes  \bA_u^0  \motimes \bA_m^1  \motimes \bB_k \\
&=\bA(I)^S \motimes \bA_u^0 \motimes (\bB_m \motimes \bA_k^1-G^m_k \cdot \bD(m, k))=\bV-G^m_k \cdot \bY^1(I) \in \fY^1_\ell.
\end{split}
\end{equation*}
This completes the proof of the claim for $m=u-1$. 

\vspace{2mm}
\noindent {\bf Step 2}. Assume that $m\leq u-2$. For any $J \in \cH_u$ with $m'=m(J)\geq m+1$, assume  further that 
\begin{equation*}
\{ \bY^0(J), \bU^{0, 0}(J), \bU^{0, 1}(J), \bU^{1, 0}(J)\} \subset \fY^1_\ell.
\end{equation*}
Since $I^{\{m\}} \in \cH_u$ with $m(I^{\{m\}})=m+1$ and $k(I^{\{m\}})=k$, we have 
\begin{equation*}
\{ \bY^0(I^{\{m\}}), \bU^{0, 0}(I^{\{m\}}), \bU^{0, 1}(I^{\{m\}}), \bU^{1, 0}(I^{\{m\}})\} \subset \fY^1_\ell.
\end{equation*}

 (1) Let $S=\{ m, m+1, u \}$ and $T=\{m, m+1, u, k\}$, and let
\begin{equation*}
\begin{split}
\bV^1 &:=\bA(I)^S \motimes \bA_m^0 \motimes \bB_{m+1} \motimes \bA_u^1 = \bA(I^{\{m\}})^{\{m+1, u\}}\motimes \bB_{m+1} \motimes \bA_u^1,\\
\bV^2 &:=\bA(I)^S \motimes \bA_m^0 \motimes \bA_{m+1}^1 \motimes \bB_u = \bA(I^{\{m\}})^{\{m+1, u\}} \motimes \bA_{m+1}^1 \motimes \bB_u.
\end{split}
\end{equation*}
Let $J=I^{\{m, u\}}$. Since $m(J)=m+1$ and $n(J)=k$, we have
\begin{equation*}
\begin{split}
\bY^1(J)&=\bA(I)^T \motimes \bA_m^0 \motimes \bA_u^1 \motimes (U^{m+1}_k \cdot \bB_{m+1} \motimes \bA_k^0-W^{m+1}_k \cdot \bA_{m+1}^0 \motimes \bB_k)\\
&=U^{m+1}_k \cdot \bV^1-W^{m+1}_k \cdot \bU^{0, 1}(I^{\{m\}}).
\end{split}
\end{equation*}
Since $U^{m+1}_k\in \Z_\ell^\times$, we have $\bV^1 \in \fY^1_\ell$, and therefore
\begin{equation*}
\bV^2=\bV^1-G^{m+1}_u \cdot \bY^0(I^{\{m\}}) \in \fY^1_\ell.
\end{equation*}
Note that
\begin{equation*}
\begin{split}
\bV^3&:=\bA(I)^{\{m\}} \motimes \bB_m=\bA(I)^{\{m, k\}} \motimes (U^m_k)^{-1}(\bD(m, k)+W^m_k \cdot \bA_m^0 \motimes \bB_k)\\
&= (U^m_k)^{-1}(\bY^1(I)+W^m_k \cdot \bU^{\{1, 0\}}(I^{\{m\}})) \in \fY^1_\ell.
\end{split}
\end{equation*}
Thus, we have
\begin{equation*}
\begin{split}
\bY^0(I)&=\bA(I)^{\{m, u\}} \motimes (U^m_u \cdot \bB_m \motimes \bA_u^0-W^m_k \cdot \bA_m^0 \motimes \bB_u)\\
&=U^m_u \cdot \bV^3 -W^m_k \cdot \bV^2 \in \fY^1_\ell.
\end{split}
\end{equation*}

(2) Note that $\bU^{\{0, 0\}}(I)=\bU^{\{1, 0\}}(I^{\{m\}}) \in \fY^1$.

(3) Let $S=\{m, m+1, u, k\}$, and let
\begin{equation*}
\bV^4:=\bA(I)^S \motimes \bA_m^0 \motimes \bB_{m+1} \motimes \bA_u^1 \motimes \bA_k^1.
\end{equation*}

If $f_i=1$ for all $i>k$, then $\bV^4=\bY^0(A(m+1)) \in \fY^1_\ell$ by the result above. 

Suppose that $f_{k'}=0$ for some $k<k'\leq t$. We take $k'$ as small as possible, i.e., $f_j=1$ for all $k<j<k'$. Let $J=I^{\{m, k\}}$. Then we have $m(J)=m+1$, $n(J)=u$ and $k(J)=k'$, and so we have $J \in \cH_u$. By induction hypothesis we have 
$\bY^0(J)\in \fY^1_\ell$ and $\bU^{0, 1}(J) \in \fY^1_\ell$.  Note that
\begin{equation*}
\begin{split}
\bY^1(J^{\{u\}})&=\bA(J)^T \motimes \bA_m^0\motimes \bA_u^1 \motimes (U^{m+1}_{k'} \cdot \bB_{m+1} \motimes \bA_{k'}^0 -W^{m+1}_{k'} \cdot \bA_{m+1}^0 \motimes \bB_{k'})\\
&=U^{m+1}_{k'} \cdot \bV^4-W^{m+1}_{k'} \cdot \bU^{0, 1}(J), ~ \text{ where }~T=\{ m, m+1, u, k' \}.
\end{split}
\end{equation*}
Thus, we have $\bV^4 \in \fY^1_\ell$ because $U^{m+1}_{k'} \in \Z_\ell^\times$. Note also that
\begin{equation*}
\bA(I)^S \motimes \bA_m^0 \motimes \bA_{m+1}^1 \motimes \bB_u \motimes\bA_k^1=\bV^4-G^{m+1}_u \cdot \bY^0(J) \in \fY^1_\ell.
\end{equation*}
Therefore we have
\begin{equation*}
\begin{split}
\bU^{0, 1}(I)&=\bU^{1, 1}(I^{\{m\}})=\gamma_u \cdot \bU^{1, 0}(I^{\{m\}})-\bA(I)^S \motimes \bA_m^0 \motimes \bA_{m+1}^1 \motimes \bB_u \motimes \bB_k\\
&=\gamma_u \cdot \bU^{1, 0}(I^{\{m\}})-\gamma_k \cdot \bV^2+\bA(I)^S \motimes \bA_m^0 \motimes \bA_{m+1}^1 \motimes \bB_u \motimes\bA_k^1 \in \fY^1_\ell.
\end{split}
\end{equation*}

(4) For simplicity, let $J=I^{\{k\}}$ and $S=\{ m, m+1, u \}$. As for $\bV^2$ and $\bV^3$, let 
\begin{equation*}
\begin{split}
\bV^5&:=\bA(J)^S \motimes \bA_m^0 \motimes \bA_{m+1}^1 \motimes \bB_u,\\
\bV^6&:=\bA(J)^{\{m\}} \motimes \bB_m=\bA(I)^{\{m, k\}} \motimes \bB_m \motimes \bA_k^1.
\end{split}
\end{equation*}

Suppose that $J\in \cH_u$. Since $m(J)=m$, by the same argument as in (1), we can easily prove that
\begin{equation*}
\bY^0(J) \in \fY^1_\ell \qa \bV^5 \in \fY^1_\ell.
\end{equation*}
Thus, we have
\begin{equation*}
U^m_u \cdot \bV^6 =\bA(J)^{\{m, u\}} \motimes (\bD(m, u)+W^m_u \cdot \bA_m^0 \motimes \bB_u)=\bY^0(J)+W^m_u \cdot \bV^5 \in \fY^1_\ell.
\end{equation*}

Suppose that $J \not\in \cH_u$, i.e., $f_i=1$ for all $i>k$. Then we have $\bY^0(J)=\bY^1(J)$. Since $\bY^0(A(m+1)) \in \fY^1_\ell$, we have
\begin{equation*}
\begin{split}
U^m_u \cdot \bV^6 &=\bY^1(J)+W^m_u \cdot \bA(J)^S \motimes \bA_m^0 \motimes (\bB_{m+1} \motimes \bA_u^1-G^{m+1}_u \cdot \bD(m+1, u))\\
&=\bY^1(J)+W^m_u \cdot \bY^0(A(m+1))-W^m_u \cdot G^{m+1}_u \cdot \bY^1(J^{\{m\}}) \in \fY^1_\ell.
\end{split}
\end{equation*}

Thus, whether $J \in \cH_u$ or not, we always have $\bV^6 \in \fY^1_\ell$ as $U^m_u \in \Z_\ell^\times$.
Therefore we have
\begin{equation*}
\bU^{1, 0}(I)=\bA(I)^{\{m, k\}} \motimes (\bB_m \motimes \bA_k^1-G^m_k \cdot \bD(m, k))=\bV^6-G^m_k \cdot \bY^1(I) \in \fY^1_\ell.
\end{equation*}
By induction, the result follows. 
\end{proof}

\begin{proof}[Proof of \ref{c}] 
It suffices to show that $\bY^1(I) \in \fY^2_\ell$  for any $I \in \cF_s \cup \cG_s$.

If $s=0$ then $\cF_s =\cG_s = \emptyset$ and the claim vacuously holds.

Suppose that $s=1$. Then we have
\begin{equation*}
\cF_s \cup \cG_s =\{ E(n) : 2\leq n \leq t\}.
\end{equation*}
Let $I=E(2)$. Then we have
\begin{equation*}
\begin{split}
\bY^1(I)&=\bA(I)^{\{1, 2\}} \motimes \bD(1, 2)=\bA(I)^{\{1, 2\}} \motimes (U^1_2 \cdot \bB_1 \motimes \bA_2^0 - W^1_2 \cdot \bA_1^0 \motimes \bB_2)\\
&=U^1_2\cdot \bY^2(I)-W^1_2 \cdot \bY^2(A(2)) \in \fY^2.
\end{split}
\end{equation*}
Let $I=E(n)$ for some $3\leq n\leq t$. Also, let $J=I^{\{1\}}$, $K=I^{\{1, 2\}}$, $S=\{ 1, 2, n \}$, and let
\begin{equation*}
\begin{split}
\bV^1&:=\bA(I)^S \motimes \bA_1^0 \motimes \bA_2^1 \motimes \bB_n=\bA(J)^{\{n\}} \motimes \bB_n,\\
\bV^2&:=\bA(I)^S \motimes \bA_1^0 \motimes \bA_2^0 \motimes \bB_n=\bA(K)^{\{n\}} \motimes \bB_n,\\
\bV^3&:=\bA(I)^S \motimes \bA_1^0 \motimes \bB_2 \motimes \bA_n^0=\bA(J)^{\{2\}}\motimes \bB_2.\\
\end{split}
\end{equation*}
Finally, let
\begin{equation*}
\bW^{\epsilon_1, \epsilon_2}:=\bA(I)^S \motimes \bB_1 \motimes \bA_2^{\epsilon_1} \motimes \bA_n^{\epsilon_2}~~\text{ for any } \epsilon_i \in \{0, 1\}.
\end{equation*}

Note that 
\begin{equation*}
\bY^1(I)=\bA(I)^S \motimes \bA_2^1 \motimes \bD(1, n)= U^1_n \cdot \bW^{1, 0}-W^1_n \cdot \bV^1.
\end{equation*}
Thus, it suffices to show that $\bW^{1, 0} \in \fY^2_\ell$ and $\bV_1 \in \fY^2_\ell$. More generally, 
we claim that $\bV^i \in \fY^2_\ell$ and $\bW^{\epsilon_1, \epsilon_2} \in \fY^2_\ell$ for all $i$ and $\epsilon_j$.
Indeed, we have
\begin{equation*}
\bV^1=\bA(J)^{\{2, n\}} \motimes (\bB_2 \motimes \bA_n^1 -G^2_n \cdot \bD(2, n))=\bY^2(A(2))-G^2_n \cdot \bY^2(J) \in \fY^2.
\end{equation*}
If $n=3$ then $\bV^2=\bY^2(A(3))\in \fY^2$. If $n>3$ then we have
\begin{equation*}
\bV^2=\bA(K)^{\{3, n\}} \motimes (\bB_3 \motimes \bA_n^1-G^3_n \cdot \bD(3, n))=\bY^2(A(3))-G^3_n \cdot \bY^2(K) \in \fY^2.
\end{equation*}
Also, since $U^2_n \in \Z_\ell^\times$ and 
\begin{equation*}
U^2_n \cdot \bV^3 = \bA(J)^{\{2, n\}} \motimes (\bD(2, n)+W^2_n \cdot \bA_2^0 \motimes \bB_n)=\bY^2(J)+W^2_n \cdot \bV^2 \in \fY^2,
\end{equation*}
we have $\bV^3 \in \fY^2_\ell$.

Next, we have $\bW^{1, 1} =\bY^2(A(1)) \in \fY^2$. Also, since $U^1_2 \in \Z_\ell^\times$ and 
\begin{equation*}
\bY^2(I^{\{2\}})=\bA(I)^S \motimes \bD(1, 2) \motimes \bA_n^0=U^1_2 \cdot \bW^{0, 0}-W^1_2 \cdot \bV^3,
\end{equation*}
we have $\bW^{0, 0} \in \fY^2_\ell$. Furthermore, we have $\bW^{0, 1} \in \fY^2_\ell$ because $U^1_2 \in \Z_\ell^\times$ and 
\small
\begin{equation*}
U^1_2 \cdot \bW^{0, 1}=\bA(I)^S \motimes (\bD(1, 2)+W^1_2 \cdot \bA_1^0 \motimes \bB_2) \motimes \bA_n^1=\bY^1(E(2))+W^1_2 \cdot \bY^2(A(2)) \in \fY^2.
\end{equation*}
\normalsize 
Finally, let 
\begin{equation*}
\bV^4:=\bA(I)^S \motimes \bB_1 \motimes \bB_2 \motimes \bA_n^0.
\end{equation*}
Then since $U^2_n \in \Z_\ell^\times$ and 
\begin{equation*}
\bY^2(I)=\bA(I)^S \motimes \bB_1 \motimes \bD(2, n)=U^2_n \cdot \bV^4 -W^2_n \cdot (\gamma_n \cdot \bW^{0, 0}-\bW^{0, 1}),
\end{equation*}
we have $\bV^4 \in \fY^2_\ell$. Thus, we have
\begin{equation*}
\bW^{1, 0}=\bA(I)^S \motimes \bB_1 \motimes (\gamma_2 \cdot \bA_2^0-\bB_2) \motimes \bA_n^0 = \gamma_2 \cdot \bW^{0, 0}-\bV^4 \in \fY^2_\ell.
\end{equation*}
This completes the proof for $s=1$. 

Suppose that $s\geq 2$. Then $\cG_s=\emptyset$ and 
\begin{equation*}
\cF_s  =  \{ E(n) : n \in \cI_s \}=\{ E(n) : 2\leq n \leq t ~\text{ and }~ n\neq s\}.
\end{equation*}

For some $n\in \cI_s$, let $I=E(n)$, $J=E_s(n)$ and $S=\{1, n, s \}$. 
Note that if $n>s$ then $J \in \cH_u$. Thus, whether $n<s$ or not, we always have 
\begin{equation*}
\bY^2(J)=\bY^1(J)=\bA(I)^S \motimes \bD(1, n) \motimes \bA_s^0.
\end{equation*}
Therefore we have
\begin{equation*}
\begin{split}
\bY^1(I)&=\bA(I)^S \motimes \bD(1, n) \motimes \bA_s^1=\bA(I)^S \motimes \bD(1, n) \motimes (\gamma_s \cdot \bA_s^0-\bB_s)\\
&=\gamma_s \cdot \bY^2(J)-\bY^2(I) \in \fY^2.
\end{split}
\end{equation*} 
This completes the proof. 
\end{proof}

\begin{proof}[Proof of \ref{d}]
The claim is obvious by the equalities at the beginning of the section, so we leave the details to the readers.
\end{proof}

\begin{proof}[Proof of \ref{e}]
Since $\fZ^0 \subset \fZ^1$ is obvious, we only prove $\fZ^1 \subset \fZ^0$. 
Let $I = (f_1, \dots, f_t) \in \Delta(t)^k$ with $m=m(I)$. We prove that $\bX_k(I) \in \fZ^0$ by (backward) induction on $m$. By definition, we have $m\leq k$. 

If $m=k$ then $\bX_k(I)=\Z(I) \in \fZ^0$. 

Next, suppose that $m<k$ and $\bX_k(J) \in \fZ^0$ for any $J$ with $m(J)\geq m+1$. Since $m(I^{\{m\}})\geq m+1$ and $I^{\{m\}} \in \Delta(t)^k$, we have $\bX_k(I^{\{m\}}) \in \fZ^0$ by induction hypothesis. Thus, we have
\begin{equation*}
\begin{split}
\bX_k(I)&=\bA(I)^{\{m, k\}} \motimes \bA_m^1 \motimes \bB_k=\bA(I)^{\{m, k\}} \motimes (\gamma_m \cdot \bA_m^0-\bB_m) \motimes \bB_k\\
&=\gamma_m \cdot \bX_k(I^{\{m\}}) -\gamma_k \cdot \Z(I^{\{k\}})+\Z(I) \in \fZ^0.
\end{split}
\end{equation*}
By induction the result follows. 
\end{proof}

\begin{remark}\label{remark: independence of ell}
Let $N=\prod_{i=1}^t p_i^{r_i}=\prod_{j=1}^t q_j^{s_j}$ be two prime factorizations of $N$, namely, $p_i$ and $q_j$ are rearrangements of each other.
As above, let $u$ be the index such that $q_u=2$. (If $N$ is odd, then we set $u=0$.)
For any $J=(f_1, \dots, f_t) \in \Delta(t)$, we define
\begin{equation*}
\Z^0(\fq_J):=\motimes_{i=1, ~i\neq m}^t \bA_{q_i}(s_i, f_i) \motimes \bB_{q_m}(s_m, 1),
\end{equation*}
where $\fq_J=\prod_{j=1}^t q_j^{f_j}$. Then by the same argument as in the proof of \ref{e}, we have 
\begin{equation*}
\br{\Z^0(\fq_J) : J \in \Delta(t)}=\fZ^1=\br{\Z(\fp_I) : I \in \Delta(t)}. 
\end{equation*}
Thus, we have 
\begin{equation*}
\scC(N)^\sqf=\br{ \ov{Z^0(d)} : d \in \cD_N^\sqf}.
\end{equation*}
In other words, the definition of $\scC(N)^\sqf$ does not depend on the ordering of the prime divisors of $N$.
\end{remark}

\ms
\subsection{Linear independence}
As above, we assume that $t\geq 2$. In this subsection, we prove the following. 
\begin{theorem}\label{theorem: main 6.6}
We have
\begin{equation*}
\scC(N)^\sqf [\ell^\infty] \simeq \moplus_{d \in \cD_N^\sqf} \br{\ov{Y^2(d)}}[\ell^\infty].
\end{equation*}
\end{theorem}

For simplicity, we use the following abbreviated notation.
\begin{notation}
For any $1\leq i<j \leq t$, let
\begin{itemize}[--]
\item
$\bbA_i^{f_i}:=\bbA_{p_i}(r_i, f_i) \qa \bbB_i:=\bbB_{p_i}(r_i, 1)$.
\item
$\gamma_i:=p_i^{r_i-1}(p_i+1) \qa G^i_j:=\gcd(\gamma_i, \gamma_j)$.
\item
$g^i_j:=\gamma_i \cdot \gamma_j \cdot \gcd(p_i-1, \, p_j-1) \cdot (G^i_j)^{-1}$.
\item
$u^i_j:=\gcd(p_i-1, \, p_j-1)^{-1}\cdot (p_i-1) ~~\text{ if } ~~ j \neq u \qa u^i_u:=-1$.
\item
$w^i_j:=\gcd(p_i-1, \, p_j-1)^{-1} \cdot (1-p_j) ~~\text{ if } ~~ j \neq u \qa w^i_u:=p_i-1$.
\item
$\bbD(i, j):=\bbV(D(p_i^{r_i}, p_j^{r_j})) \in \cS_1(p_i^{r_i}p_j^{r_j})$.
\item
For each $I \in \Delta(t)$, let $Y^i(I):=Y^i(\fp_I)$.
\item
For each $J \in \Delta(t)$ and $V \in \cS_1(N)$, let $V_J:=V_{\fp_J}$.
\end{itemize}
\end{notation}

Note that we have $u^i_j \in \Z_\ell^\times$ by Assumption \ref{assumption: chapter4}. Since
\begin{equation*}
G^i_j \cdot \bD(p_i^{r_i}, p_j^{r_j})=\gamma_j \cdot \bB_{p_i}(r_i, 1) \motimes \bA_{p_j}(r_j, 0)-\gamma_i \cdot \bA_{p_i}(r_i, 0) \motimes \bB_{p_j}(r_j, 1),
\end{equation*}
by Lemma \ref{lemma: relation among bA and bbA} we have
\begin{equation*}
\gamma_i^{-1} \cdot \gamma_j^{-1} \cdot G^i_j \cdot 
V(D(p_i^{r_i}, p_j^{r_j}))_\delta=\begin{cases}
p_j-p_i & \text{ if } ~~\delta=1,\\
1-p_j & \text{ if } ~~\delta=p_i,\\
p_i-1 & \text{ if } ~~\delta=p_j,\\
0 & \text{ otherwise}.
\end{cases}
\end{equation*}
Thus, we have 
\begin{equation*}
\Gcd(D(p_i^{r_i}, p_j^{r_j}))=g^i_j \qa \bbD(i, j)=(-u^i_j-w^i_j, w^i_j, u^i_j, \bbO) \in \cS_1(p_i^{r_i}p_j^{r_j}),
\end{equation*}
in particular we have $\bbD(i, j)_{p_ip_j}=0$. Also, if $j \neq u$ (resp. $j= u$), then $\bbD(i, j)_{p_j}=u^i_j \in \Z_\ell^\times$ (resp. $\bbD(i, j)_{p_i}=u^i_j \in \Z_\ell^\times$).
Note that $u^i_j$ is odd if $\ell=2$. However, $u^i_j$ might be even if $\ell$ is odd. In that case, $w^i_j$ is odd because $u^i_j$ and $w^i_j$ are relatively prime. Therefore either $u^i_j$ or $w^i_j$ is odd.

\vspace{3mm}
To begin with, we prove the matrix $\fM_0=(\bbV(Y^1(d_i))_{\delta_j})$ is lower $\ell$-unipotent.
\begin{proposition}
For any $I \in \Delta(t)$, we have
\begin{equation*}\label{proposition: unipotent of Y1}
\bbV(Y^1(I))_{\iota(I)} \in \Z_\ell^\times \qa \bbV(Y^1(J))_{\iota(I)}=0 ~~\text{ for all } J \prec I.
\end{equation*}
\end{proposition}
\begin{proof}
As above, it suffices to show that for any $I \in \Delta(t)$, we have
\begin{equation*}
\bbV(Y^1(I))_{\iota(I)} \in \Z_\ell^\times \qa \bbV(Y^1(I))_J=0 ~~\text{ for all } \iota(I) \vtl J.
\end{equation*}

Let $I=(f_1, \dots, f_t) \in \Delta(t)$ with $m=m(I)$, $n=n(I)$ and $k=k(I)$. Also, let 
$\iota(I)=(a_1, \dots, a_t)$ and $J=(b_1, \dots, b_t) \in \Delta(t)$.
Assume that $\iota(I) \vtl J$. Then by definition, there is an index $h$ such that $b_h=1$ and $a_h=0$. 
We further assume 
\begin{enumerate}
\item
$a_i=b_i$ for all $i>h$ different from $u$ if $h\neq u$, and
\item
$a_i=b_i$ for all $i$ different from $u$ if $h=u$.
\end{enumerate}

Then by Theorem \ref{theorem: order defined by tensors}, we can easily compute $\bbV(Y^1(I))$, and so we can prove the claim as in the proof of Proposition \ref{proposition: unipotent of Z}. More precisely, we proceed as follows: Let $K=(c_1, \dots, c_t) \in \Delta(t)$. 
\begin{enumerate}
\item
Assume that $I \in \cE$, and let $x=\max(m, u)$. Then we have
\begin{equation*}
\bbV(Y^1(I))_{K}=\textstyle\prod_{i=1, \, i\neq x}^t (\bbA_i^{f_i})_{p_i^{c_i}} \times (\bbB_{p_x})_{p_x^{c_x}}.
\end{equation*}
Note that 
\begin{equation*}
(\bbA_i^1)_{p_i}=0, ~~(\bbA_i^{f_i})_{p_i^{1-f_i}}=(-1)^{1-f_i} \qqa (\bbB_x)_{p_x}=-1.
\end{equation*}
Therefore we  have $|\bbV(Y^1(I))_{\iota(I)}|=1$. Also, since $I \in \cE$ we have $a_x=1$, and hence $h \neq x$.
By the definition of $\iota$, we have $f_h=1-a_h=1$. Since $f_h=b_h=1$ and $(\bbA_h^1)_{p_h}=0$, we have $\bbV(Y^1(I))_{J}=0$, as claimed.

\item
Assume that $I \in \cH_u$. Then we have
\begin{equation*}
\bbV(Y^1(I))_{K}=\textstyle\prod_{i=1, \, i\neq m, k}^t (\bbA_i^{f_i})_{p_i^{c_i}} \times \bbD(m, k)_{p_m^{c_m} p_k^{c_k}}.
\end{equation*}
Note that we have $n=u$ and so $k\neq u$. Thus, we have
\begin{equation*}
\bbD(m, k)_{p_k}=u^m_k \in \Z_\ell^\times \qa \bbD(m, k)_{p_m p_k}=0.
\end{equation*}
Since $I \not\in \cE \cup \cH_u^1$, we have $a_i=1-f_i$ for all $i$. 
Therefore we have 
\begin{equation*}
\bbV(Y^1(I))_{\iota(I)}=\textstyle\prod_{i=1, \, i\neq m, k}^t (\bbA_i^{f_i})_{p_i^{1-f_i}} \times \bbD(m, k)_{p_k} = \pm u^m_k \in \Z_\ell^\times.
\end{equation*}
Since $a_k=1-f_k=1$, we have $h\neq k$. 
If $h\neq m$, then $\bbV(Y^1(I))_J=0$ because $(\bbA_h^1)_{p_h}=0$ as above. If $h=m$, then $b_h=b_m=1$ and $a_i=b_i$ for all $i>m$ different from $u$, in particular $a_k=b_k=1$. Thus, we have $\bbV(Y^1(I))_J=0$ as $\bbD(m, k)_{p_m p_k}=0$ and $b_m=b_k=1$.

\item
Assume that $I \not\in \cE \cup \cH_u$. Then we have
\begin{equation*}
\bbV(Y^1(I))_{K}=\textstyle\prod_{i=1, \, i\neq m, n}^t (\bbA_i^{f_i})_{p_i^{c_i}} \times \bbD(m, n)_{p_m^{c_m} p_n^{c_n}}.
\end{equation*}
\begin{enumerate}
\item
Suppose that $I \in \cH_u^1$, i.e., $m<n=u$ and $f_i=1$ for all $i>u$. By definition, $a_m=1$, $a_u=0$ and $a_i=1-f_i$ for all $i \neq m, u$. 
Since $\bbD(m, u)_{p_m}=u^m_u=-1$, we have $|\bbV(Y^1(I))_{\iota(I)}|=1$. Also, since $a_m=1$ we have $h\neq m$. If $h\neq u$, then we have $\bbV(Y^1(I))_J=0$ as above. If $h=u$ then $b_h=b_u=1$ and $a_i=b_i$ for all $i$ different from $u$, and so $a_m=b_m=1$. 
Thus, we have $\bbV(Y^1(I))_J=0$ as $\bbD(m, u)_{p_m p_u}=0$ and $b_m=b_u=1$.

\item
Suppose that $I \not\in \cH_u^1$. Since $I \not\in \cH_u \cup \cH_u^1$, we have $n\neq u$. Also, since $I \not\in \cE \cup \cH_u^1$, we have $a_i=1-f_i$ for all $i$. 
Therefore we have  $\bbV(Y^1(I))_{\iota(I)}=\pm u^m_n \in \Z_\ell^\times$. Note that $f_h=1-a_h=1$, and so $h\neq n$.
As above, if $h \neq m$ then we have $f_h=b_h=1$, and so we have $\bbV(Y^1(I))_J=0$ as $(\bbA_h^1)_{p_h}=0$. Also, if $h=m$ then $b_h=b_m=1$ and $b_n=a_n=1-f_n=1$ because $n>m$ and $n\neq u$. 
Thus, we have $\bbV(Y^1(I))_J=0$ as $\bbD(m, n)_{p_m p_n}=0$ and $b_m=b_n=1$.
\end{enumerate}
\end{enumerate}
This completes the proof.
\end{proof}

Next, we compute $\Gcd(Y^1(I))$ and $\fh(Y^1(I))$.
\begin{proposition}\label{proposition: gcd and h of Y1}
Let $I=(f_1, \dots, f_t) \in \Delta(t)$ with $m=m(I)$, $n=n(I)$ and $k=k(I)$.
Also, let $x=\max(m, u)$.
Then we have 
\begin{equation*}
\Gcd(Y^1(I))=\begin{cases}
\prod_{i=1, \, i\neq x}^t g_{p_i}(r_i, f_i) \cdot \gamma_x & \text{ if } ~~ I \in \cE,\\
\prod_{i=1, \, i\neq m, k}^t g_{p_i}(r_i, f_i) \cdot g^m_k & \text{ if } ~~ I \in \cH_u,\\
\prod_{i=1, \, i\neq m, n}^t g_{p_i}(r_i, f_i) \cdot g^m_n & \text{ otherwise.}\\
\end{cases}
\end{equation*}
Also, $\fh(Y^1(I))=2$ if and only if $I \in \cF_u^1 \cup \cG_u^1 \cup \{A(1)\}$.
\end{proposition}
\begin{proof}
The first assertion easily follows by
the same argument as in the proof of Proposition \ref{proposition: gcd and h of Z}.  

To prove the second assertion, we note that the summation of all entries of $\bbB_i$ (resp. $\bbD(i, j)$) is zero. Thus, we have $\pw_{p_h}(Y^1(I))=0$ 
unless 
\begin{enumerate}
\item
$I \in \cE$ and $h=x$.
\item
$I \in \cH_u$ and either $h=m$ or $k$.
\item
$I \not\in \cE\cup \cH_u$ and either $h=m$ or $n$.
\end{enumerate}

Since either $u^i_j$ or $w^i_j$ is odd, and since $p_i-1$ is even unless $i=u$, we easily have the following.
\begin{enumerate}
\item
For $I \in \cE$, $\fh(Y^1(I))=2$ if and only if one of the following holds.
\begin{enumerate}
\item
$m=1$, i.e., $I=A(1)$.
\item
$u=1$ and $m=2$, i.e., $I=A(2)=E(1)$.
\end{enumerate}
\item
For $I \in \cH_u$ (and so $u\geq 2$), $\fh(Y^1(I))=2$ if and only if $m=1$ and $f_j=1$ for all $j>k$, i.e., $I=E_u(k)$ for some $k>u$. 

\item
For $I \not\in \cE \cup \cH_u$, $\fh(Y^1(I))=2$ if and only if one of the following holds.
\begin{enumerate}
\item
$m=1$ and $f_j=1$ for all $j>n$, i.e., $I=E(n)$ for some $n\geq 2$. (Here, $n$ might be equal to $u$.)
\item
$m=1$, $n<u$, $f_u=0$ and $f_j=1$ for all $j>n$ different from $u$, i.e., $I=E_u(n)$ for some $2\leq n < u$. 
\item
$u=1$, $m=2$ and $f_j=1$ for all $j>n$, i.e., $I=E_u(n)$ for some $3\leq n\leq t$.
\end{enumerate}
\end{enumerate}
Thus, the second assertion follows by the definition of $\cF_u^1$ and $\cG_u^1$.
\end{proof}

\vspace{3mm}
From now on, let
\begin{equation*}
H':=(\cF_u^1 \cup \cG_u^1) \sm (\cF_s \cup \cG_s) \qa H:=H' \cup \{A(1)\}.
\end{equation*}

\begin{lemma}\label{lemma: unipotent of Y2}
Let $\ell=2$. Then for any $I \in \Delta(t) \sm H$, we have
\begin{equation*}
\bbV(Y^2(I))_{\iota(I)} \not\in 2\Z \qa \bbV(Y^2(J))_{\iota(I)}=0 ~~\text{ for all } J \prec I. 
\end{equation*}
\end{lemma}
\begin{proof}
Since $\ell=2$, we have $s=u$. Also, since $Y^2(I)=Y^1(I)$ unless $I \in \cF_s \cup \cG_s$, by Proposition \ref{proposition: unipotent of Y1} it suffices to prove the assertion for $I \in \cF_s \cup \cG_s$. If $s=0$, there is nothing to prove and so we assume that $s\geq 1$.

For simplicity, let $F=\{2^{n-\epsilon} : n \in \cI_s\}$, where $\epsilon=1$ if $s<n$ and $0$ otherwise. Also, let $G=\{2\}$ if $s=1$, and $G=\emptyset$ otherwise. Then by Lemma \ref{lemma: 6.8}, $I \in \cF_s$ if and only if $\fp_I=d_i$ for some $i\in F$. Also, by Lemma \ref{lemma: 6.7} $I \in \cG_s$ if and only if $\fp_I=d_i$ for some $i\in G$.

Now, we claim the following. For any $i \in F \cup G$, we have
\begin{enumerate}
\item
$d_{i+1}=\fp_K$ for some $K \in H$,
\item
$\bbV(Y^2(d_i))_{\delta_i}=-\bbV(Y^2(d_i))_{\delta_{i+1}} \not\in 2\Z$, and
\item
$\bbV(Y^2(d_i))_{\delta_j}=0$ for any $j>i+1$.
\end{enumerate}
Indeed, if $i \in F$, i.e., $i=2^{n-\epsilon}$ for some $n \in \cI_s$, then by Lemma \ref{lemma: 6.8} we have $K=E_s(n) \in H$, $\delta_i=\fp_{F(n)}$ and $\delta_{i+1}=\fp_{F_s(n)}$. Also by Theorem \ref{theorem: order defined by tensors}, we have
\begin{equation*}
\bbV(Y^2(d_i))=\motimes_{i=1, \, i\neq y, s, n}^t \bbA_i^1 \motimes \bbB_s \motimes \bbD(y, n).
\end{equation*}
Therefore we have $\bbV(Y^2(d_i))_{F(n)}=-\bbV(Y^2(d_i))_{F_s(n)}=u^y_n \not\in 2\Z$. Suppose that $j>i+1$ and $\delta_j=\fp_J$ for some $J=(f_1, \dots, f_t) \in \Delta(t)$. Since $\delta_{j+1}=\fp_{F_s(n)}$, we have $F_s(n) \vtl J$, and so $f_h=1$ for some $h>n$ different from $s$. Thus, we have $\bbV(Y^2(d_i))_J=0$ as $(\bbA_h^1)_{p_h}=0$.
If $i \in G$, then $s=1$ and $i=2$. By Lemma \ref{lemma: 6.7}, we have $K=A(2)=E(1) \in H$, $\delta_2=p_2$ and $\delta_3=p_1p_2$.
Similarly as above, we have 
\begin{equation*}
\bbV(Y^2(I))=\motimes_{i=3}^t \bbA_i^1 \motimes \bbB_{s} \motimes \bbA_2^0,
\end{equation*}
and so $\bbV(Y^2(d_2))_{p_2}=-\bbV(Y^2(d_2))_{p_1p_2}=-1$. Also, as above $\bbV(Y^2(d_2))_{\delta_j}=0$ for all $j>3$. Therefore the claim follows.

Now, the assertion for $I \in \cF_s \cup \cG_s$ easily follows by the claim. Indeed, if $I \in \cF_s \cup \cG_s$, then $d_i=\fp_I$ for some $i \in F\cup G$ and so $\bbV(Y^2(I))_{\iota(I)} \not \in 2\Z$. Suppose that $J \prec I$. By definition, we have $\iota(J) \vtl \iota(I)$. If $J \not\in \cF_s \cup \cG_s$, then $Y^2(J)=Y^1(J)$ and so $\bbV(Y^2(J))_{\iota(I)}=0$ by Proposition \ref{proposition: unipotent of Y1}. If $J \in \cF_s \cup \cG_s$ with $d_j=\fp_J$, then we have $j \in F \cup G$. 
By the claim, $d_{j+1}=\fp_K$ for some $K \in H$. Since $I \not\in H$, we have $j+1\neq i$. Also, since $J \prec I$, we have $j<i$ and so $j+1<i$. Thus, $\bbV(Y^2(J))_{\iota(I)}=0$ by the claim.
This completes the proof.
\end{proof}

\begin{lemma}\label{lemma: gcd and h of Y2}
Let $I \in \cF_s \cup \cG_s$. Then we have 
\begin{equation*}
\Gcd(Y^2(I))=\begin{cases}
\prod_{i=1, \, i\neq s, y, n}^t g_{p_i}(r_i, 1)\cdot \gamma_s \cdot g^y_n & \text{ if }~~ I \in \cF_s,\\
\prod_{i=3}^t g_{p_i}(r_i, 1)\cdot \gamma_s \cdot g_{p_2}(r_2, 0)  & \text{ if }~~ I \in \cG_s.
\end{cases}
\end{equation*}
Also, we have $\fh(Y^2(I))=1$.
\end{lemma}
\begin{proof}
By Theorem \ref{theorem: order defined by tensors}, the first assertion easily follows. For the second assertion, we note that the summation of all entries of $\bbB_s$ (resp. $\bbD(y, n)$) is zero. 
Thus, if $I \in \cF_s$, then $\pw_{p_h}(Y^2(I))=0$ for any $h$. Also, if $I \in \cG_s$, then $\pw_{p_1}(Y^2(I))=1-p_2$ and $\pw_{p_h}(Y^2(I))=0$ for all $h\geq 2$. Since $p_2$ is odd, the second assertion follows.
\end{proof}

\begin{corollary}\label{corollary: gcd and h of Y2}
For any $d \in \cD_N^\sqf$, we have 
\begin{equation*}
\scG(N, d)=\frac{\kappa(N)}{\Gcd(Y^2(d))} \qa \scH(N, d)=\fh(Y^2(d)).
\end{equation*}
\end{corollary}
\begin{proof}
Note that $\kappa(p_i^{r_i}p_j^{r_j})=g^i_j \times \cG(p_i^{r_i}, p_j^{r_j})$. 
Thus, by the same argument as Corollary \ref{corollary: gcd and h of Z}, the assertion follows from Proposition \ref{proposition: gcd and h of Y1} and Lemma \ref{lemma: gcd and h of Y2}. 
\end{proof}

Lastly, we prove the following.
\begin{lemma}\label{lemma: h=2 for Y2}
Let $\ell=2$. Then for any $I \in H'$, there is an index $1\leq h \leq t$ different from $u$ such that 
\begin{equation*}
\pw_{p_h}(Y^2(I)) \not\in 2\Z \qa \pw_{p_h}(Y^2(J))=0 ~~ \text{ for any } J \prec I.
\end{equation*}
\end{lemma}
\begin{proof}
Since $\ell=2$, we have $s=u$. For simplicity,
let $I=(f_1, \dots, f_n) \in \Delta(t)$ and $J=(a_1, \dots, a_t) \in \Delta(t)$. Assume that $I \in H'$ and $J \prec I$. 

First, suppose that $s=0$. Then $I=E(n)$ for some $2\leq n \leq t$. By definition, we have $\iota(I)=F(n)$. In this case, we can take $h=n$. (Since $u=0$, we have $h\neq u$.) Indeed, as above we have 
$\pw_{p_n}(Y^2(I))=\pw_{p_n}(Y^1(I))=u^1_n \not\in 2\Z$. Also, we claim that 
$\bbV(Y^2(J))_\delta=0$ for any $\delta \in \cD_N^\sqf$ divisible by $p_n$, or equivalently 
\begin{equation*}
\bbV(Y^2(J))_{K}=0 ~~\text{ for any } K=(c_1, \dots, c_t) \in \Delta(t) \text{ with } c_n=1,
\end{equation*}
which clearly implies that $\pw_{p_n}(Y^2(J))=0$. Since $c_n=1$, by definition $F(n) \vtl K$ or $F(n)=K$. Also, since $\iota(J) \vtl \iota(I)=F(n)$, we have $\iota(J) \vtl K$. Thus, the claim follows by Proposition \ref{proposition: unipotent of Y1} as $Y^2(J)=Y^1(J)$. (Note that $\cF_s \cup \cG_s= \emptyset$.)

Next, suppose that $s\geq 1$. Then we have $H'=\{ E(s), E_s(n) : n \in \cI_s \}$.
Suppose first that $I=E(s)$. Then $d_k=\fp_{E(s)}$ by Lemma \ref{lemma: 6.7}, where $k=\max(2, 4-s)$. Hence by direct computation, we can take $h=\max(1, 3-s)$. (Thus, we have $h\neq u$.) Suppose next that $I=E_s(n)$ for some $n\in \cI_s$. 
Then we can take $h=n$. (Since $n \in \cI_s$, we have $h\neq u$.) Indeed, whether $n<s$ or not\footnote{If $n>s$, then we have $n(I)=s$ and $k(I)=n$, i.e., $I \in \cH_s$.}, we have 
\begin{equation*}
\bbV(Y^2(I))=\motimes_{i=1, \, i\neq u, n, s} \bbA_i^1 \motimes \bbA_s^0 \motimes \bbD(y, n),
\end{equation*}
where $y=\max(1, 3-s)$. Thus, we have $\pw_{p_n}(Y^2(I))=u^y_n \not\in 2\Z$ as above. Also, if $J \in \cF_s \cup \cG_s$, then we have $\pw_{p_n}(Y^2(I))=0$ as $n\in \cI_s$. Furthermore, if $J \not\in \cF_s \cup \cG_s$, then we have $\iota(J) \vtl F(n)$. Thus, we have $\iota(J) \vtl K$ for any $K=(c_1, \dots, c_t) \in \Delta(t)$ with $c_n=1$. 
Since $Y^2(J)=Y^1(J)$, by Proposition \ref{proposition: unipotent of Y1} we have $\bbV(Y^2(J))_K=0$, and so $\pw_{p_n}(Y^2(J))=0$, as desired.
This completes the proof.
\end{proof}

Combining all the results above, we now prove Theorem \ref{theorem: main 6.6}.

By Theorem \ref{proposition: generation Y2}, it suffices to show that
\begin{equation*}
 \br{ \ov{Y^2(I)} : I \in \Delta(t) } [\ell^\infty] \simeq \moplus_{I \in \Delta(t)} \br{\ov{Y^2(I)}}[\ell^\infty].
\end{equation*}

Suppose first that $\ell$ is odd. Then we have $Y^1(I)=Y^2(I)$. Thus, the assertion follows by successively applying Theorem \ref{thm: criterion 2} thanks to Proposition \ref{proposition: unipotent of Y1}. 

Suppose next that $\ell=2$, and so $s=u$. By Theorem \ref{thm: criterion 3}, it suffices to show the following. For any $I \in \Delta(t) \sm \{A(1)\}$, one of the following holds\footnote{Note that $A(1)$ is the smallest element in $(\Delta(t), \prec)$ by Lemma \ref{lemma: 6.7}.}. 
\begin{enumerate}
\item
$\fh(Y^2(I))=1$, $\bbV(Y^2(I))_{\iota(I)} \not\in 2\Z$ and $\bbV(Y^2(J))_{\iota(I)}=0$ for all $J \prec I$.
\item
There is an index $1\leq h \leq t$ such that
\begin{equation*}
\pw_{p_h}(Y^2(I)) \not\in 2\Z \qa \pw_{p_h}(Y^2(J))=0 ~~ \text{ for any } J \prec I.
\end{equation*}
\end{enumerate}
Note that (1) follows by Lemma \ref{lemma: unipotent of Y2} because $\fh(Y^2(I))=1$ if and only if $I \in H$ (Corollary \ref{corollary: gcd and h of Y2}).
 And (2) follows by Lemma \ref{lemma: h=2 for Y2}. This completes the proof. \qed

\ms
\subsection{Proof of Theorem 6.1}\label{section: final step}
In this subsection, we finish the proof of Theorem \ref{theorem: chapter 6 main}.
\begin{proof}[Proof of Theorem \ref{theorem: chapter 6 main}]
By Theorems \ref{theorem: case of t=1, u=0} and \ref{theorem: case of t=u=1}, it suffices to prove the theorem for $t\geq 2$. 
First, by Theorem \ref{theorem: computation of order} and Corollary \ref{corollary: gcd and h of Z} (resp. \ref{corollary: gcd and h of Y2}), the order of $Z(d)$ (resp. $Y^2(d)$) is $\fn(N, d)$ (resp. $\fN(N, d)$). Thus, it suffices to show the first assertion.

For simplicity, we use the same notation as in the previous section. 
Also, let $r=1$ if $u=0$ and $r=r_u$ if $u\geq 1$.
Since $Z^1(d_i)=Z(d_i)$ for any $1\leq i \leq \fm$, by Theorems \ref{theorem: main theorem in 6.4} and \ref{theorem: main 6.6}, it suffices to show that
\begin{equation*}
\br{\ov{Z(d_i)}, \, \ov{Z^1(d_j)} : 1\leq i \leq \fm, \, \fm < j < \fm+r}\simeq \scC(N)^\sqf \moplus\left( \moplus_{j=\fm+1}^{\fm+r-1} \br{\ov{Z(d_j)}}\right).
\end{equation*}
Since the claim vacuously holds for $r=1$, we assume that $r\geq 2$. 
For any $2\leq f \leq r$, let $I_f:=(f_1, \dots, f_t)$ with $f_u=f$ and $f_i=1$ for all $i \neq u$.
Then by Remark \ref{remark: 2power ordering}, we have 
\begin{equation*}
\{d_j : \fm+1 \leq j \leq \fm+r-1 \} = \{ \fp_{I_f} : 2\leq f \leq r \}.
\end{equation*}
Thus, it is enough to show that
\begin{equation}\label{00jae}
\br{\ov{Z(d_i)}, \, \ov{Z^1(I_f)} : 1\leq i \leq \fm, \, 2\leq f \leq r}  \simeq \scC(N)^\sqf \moplus\left(\moplus_{f=2}^{r} \br{\ov{Z(I_f)}}\right).
\end{equation}

For simplicity, let 
\begin{equation*}
\pi_1^*:=\pi_1(N, 2^r)^* : J_0(2^r) \to J_0(N).
\end{equation*}
Then by Remark \ref{remark: 4.8}, we have $\ov{Z^1(I_f)}=\pi_1^*(\ov{B_2(r, f)})$.
So if $r\leq 4$, then we have $\ov{B_2(r, f)}=0$ because the genus of $X_0(2^r)$ is zero. Thus, we have $\ov{Z(I_f)}=\ov{Z^1(I_f)}=0$ for any $2\leq f \leq r$, and so (\ref{00jae}) obviously holds. 
Accordingly, we assume $r\geq 5$. By direct computation, the order of $Z(d_{\fm+1})=Z(I_2)$ is $\num(\frac{2^2-1}{24})=1$, i.e., $\ov{Z(d_{\fm+1})}=0$, and so (\ref{00jae}) is equivalent to
\begin{equation}\label{000jae}
\br{\ov{Z(d_i)}, \, \ov{Z^1(I_f)} : 1\leq i \leq \fm, \, 3\leq f \leq r} \simeq \scC(N)^\sqf \moplus\left( \moplus_{f=3}^{r} \br{\ov{Z(I_f)}}\right).
\end{equation}

We prove (\ref{000jae}) by the following two claims:
\begin{enumerate}
\item
$\scC(N)^\sqf \cap \br{\ov{Z^1(I_f)} : 3\leq f \leq r}=0$.
\item
$\br{\ov{Z^1(I_f)} : 3\leq f \leq r} \simeq \moplus_{f=3}^r \br{\ov{Z(I_f)}}$.
\end{enumerate}

For simplicity, let $G:=\br{\ov{Z^1(I_f)} : 3\leq f \leq r}$.
Then we have
\begin{equation*}
G = \br{\pi_1^*(\ov{B_2(r, f)}) : 3\leq f \leq r}=\pi_1^* \left(\br{\ov{B_2(r, f)} : 3 \leq f \leq r}\right)=\pi_1^*(\scC(2^r)).
\end{equation*}
Note that $G$ is a $2$-group as $\scC(2^r)$ is. Thus, we have
\begin{equation*}
\scC(N)^\sqf \cap G =\scC(N)^\sqf[2^\infty] \cap G.
\end{equation*}
By Theorem \ref{theorem: main 6.6}, we have
\begin{equation*}
\scC(N)^\sqf[2^\infty] \cap G \simeq \left(\moplus_{I \in \Delta(t)} \br{\ov{Y^2(I)}}[2^\infty] \right)\cap G \simeq \moplus_{I \in \Delta(t)} \left( \br{\ov{Y^2(I)}}[2^\infty] \cap G\right).
\end{equation*}
Hence it suffices to show that $\br{\ov{Y^2(I)}}[2^\infty] \cap G=0$ for any $I \in \Delta(t)$. If $I=A(1)$, then the claim follows because the order of $Y^2(A(1))$ is $1$.\footnote{The divisor $Y^2(A(1))$ comes from level $2$. Indeed, we have $Y^2(A(1))=\pi_1(N, 2)^*(0-\infty)$.} 
If $I \in \Delta(t) \sm \{A(1)\}$, then as in Proposition \ref{proposition: unipotent of Z}, we can prove that $\bbV(Z^1(I_f))_{\delta}=0$
for any $\delta$ divisible by $p_h$ for some $h \neq u$. Thus, we obtain $\br{\ov{Y^2(I)}}[2^\infty] \cap G=0$ by Theorem \ref{thm: criterion 3} because of the following:
\begin{enumerate}
\item
Suppose that $I \in \Delta(t) \sm H$.
Since $I \neq A(1)$, $\fp_{\iota(I)}$ is divisible by $p_h$ for some $h \neq u$, and thus we have $\bbV(Z^1(I_f))_{\iota(I)}=0$. As already discussed, we have $\bbV(Y^2(I))_{\iota(I)} \not\in 2\Z$.

\item
Suppose that $I \in H'$. Then by Lemma \ref{lemma: h=2 for Y2}, $\pw_{p_h}(Y^2(I))\not\in 2\Z$ for some $h\neq u$. Also, we have $\pw_{p_h}(Z^1(I_f))=0$ since $\bbV(Z^1(I_f))_\delta=0$ for any $\delta$ divisible by $p_h$.
\end{enumerate}
This completes the proof of the first claim.

Next, we prove the second claim.
Note that by Theorem \ref{theorem: case of t=u=1}, we have
\begin{equation*}
G = \pi_1^*(\scC(2^r))\simeq \pi_1^*\left(\moplus_{f=3}^r \br{\ov{B^2(r, f)}} \right).
\end{equation*}
Note also that since $\pi_1^*$ is injective (cf. \cite[Rem. 2.7]{Yoo6}), we have 
\begin{equation*}
 \pi_1^*\left(\moplus_{f=3}^r \br{\ov{B^2(r, f)}} \right) \simeq \moplus_{f=3}^r \br{\pi_1^*(\ov{B^2(r, f)})}.
\end{equation*}
Finally, since $\cT_u = \{ I_f : 3\leq f \leq r\}$, we have 
\begin{equation*}
\Z(I_f) = \motimes_{i=1, \, i\neq u} \bA_{p_i}(r_i, 1) \motimes \bB^2(r, f),
\end{equation*}
and so $\ov{Z(I_f)}=\pi_1^*(\ov{B^2(r, f)})$ by Remark \ref{remark: 4.8}. 
This completes the proof. 
\end{proof}

\begin{remark}
Unfortunately, we cannot use Theorem \ref{thm: criterion 3} to prove 
\begin{equation*}
\br{ \ov{Y^2(I)}}[2^\infty]  \cap G = 0, ~~ \text{ where } ~I=A(1).
 \end{equation*}
 In fact, if there were a rational cuspidal divisor $X$  satisfying $\bbV(X)=\bbV(Y^2(I))$ and the order of $X$ is even, then 
 we could prove that
 \begin{equation*}
 \br{\ov{X}}[2^\infty]  \cap G \simeq \zmod 2.
 \end{equation*}
Thus, it is crucial that the order of $Y^2(I)$ is $1$ (or at least odd). 
\end{remark}

\ms
\subsection{Acknowledgments}
I thank Ken Ribet for generously sharing his idea and encouragement. I also thank Myungjun Yu for comments and corrections to an earlier version of this manuscript.
This work was supported by National Research Foundation of Korea(NRF) grant funded by the Korea government(MSIT) (No. 2019R1C1C1007169 and 
No. 2020R1A5A1016126).

\bibliographystyle{amsalpha}


\end{document}